\documentclass[11pt]{amsart}
\usepackage[utf8]{inputenc}
\usepackage[T1]{fontenc}

\usepackage{amsmath,amsthm,amsfonts,amssymb,mathrsfs,pb-diagram,amscd}
\usepackage{amsmath}
\usepackage{amsxtra}
\usepackage{amscd}
\usepackage{amsthm}
\usepackage{amsfonts}
\usepackage{amssymb}
\usepackage{eucal}
\usepackage{color}
\usepackage[normalem]{ulem}
\usepackage[enableskew,vcentermath]{youngtab}
\usepackage{bm}

\usepackage{graphics}
\usepackage{mathrsfs}
\usepackage{enumerate}

\usepackage[pagebackref=false,hidelinks]{hyperref}

\makeatletter
\let\LS@setaddresses\@setaddresses
\def\@setaddresses{\par\vbox{\LS@setaddresses}}
\makeatother

\DeclareMathOperator\Std{Std}
\DeclareMathOperator\res{res}

\allowdisplaybreaks[4]

\newtheorem{thm}{Theorem}[section]
\theoremstyle{plain}

\newtheorem{lem}[thm]{Lemma}

\newtheorem{prop}[thm]{Proposition}
\newtheorem{cor}[thm]{Corollary}

\theoremstyle{definition}
\newtheorem{defn}[thm]{Definition}
\newtheorem{example}[thm]{Example}

\theoremstyle{remark}
\newtheorem{rem}[thm]{Remark}

\definecolor{A}{rgb}{.75,1,.75}

\numberwithin{equation}{section}

\newcommand{\Z}{\mathbb Z}
\newcommand{\N}{\mathbb N}

\newcommand{\mHcn}{\mathcal{H}_{c}(n)} 
\newcommand{\mHfcn}{\mathcal{H}^f_{c}(n)} 
\newcommand{\mhcn}{\mathfrak{H}_{c}(n)} 
\newcommand{\mhgcn}{\mathfrak{H}^g_{c}(n)} 

\newcommand{\mt}{\mathfrak{t}}
\newcommand{\ms}{\mathfrak{s}}
\newcommand{\mfku}{\mathfrak{u}}
\newcommand{\mfkv}{\mathfrak{v}}

\newcommand{\supp}{\text{supp}}

\newcommand{\undla}{\underline{\lambda}}
\newcommand{\undmu}{\underline{\mu}}
\newcommand{\undQ}{\underline{Q}}

\newcommand{\Add}{{\rm Add}}
\newcommand{\Rem}{{\rm Rem}}

{\vskip-\lastskip\medskip
	\noindent
	{\em #1.}\enspace
}%
{\qed\par\medskip
}

\usepackage{enumitem}

\newlist{caselist}{enumerate}{1}
\setlist[caselist,1]{
	label=\textbf{Case \arabic{section}.\arabic{subsection}.\arabic*:},
	ref=\arabic{section}.\arabic{subsection}.\arabic*,
	before=\setcounter{caselisti}{0}
}

\newcounter{case}

\newenvironment{symbols}{
    \section*{Index of notation}
    \begin{description}
}{
    \end{description}
}

\newcommand{\symitem}[3]{\item[#1] #2 \hfill\rlap{\pageref{#3}}}
\newcommand{\symitemtwo}[4]{\item[#1] #2 \hfill\rlap{\pageref{#3}, \pageref{#4}}}

\begin{document}

	\title[cyclotomic Hecke-Clifford algebras]{On (super)symmetrizing forms and Schur elements of cyclotomic Hecke-Clifford algebras}
	\author{Shuo Li}\address{Mathematics Research Center, School of Science and Engineering\\
		The Chinese University of Hong Kong, Shenzhen 518172\\
	    P.R. China}
       \address{School of Mathematics and Statistics\\
		Beijing Institute of Technology\\
		Beijing, 100081\\
       P.R. China}
	\email{lishuo@cuhk.edu.cn}

    \author{Lei Shi}\address{Academy of Mathematics and Systems Science\\
    	Chinese Academy of Sciences, Beijing 100190\\
    	P.R.China}
    	\address{Max-Planck-Institut f\"ur Mathematik\\
	    Vivatsgasse 7, 53111 Bonn\\
	    Germany}
    \email{leishi202406@163.com}

\begin{abstract}
For certain cyclotomic Hecke-Clifford algebras of even and odd level, we construct supersymmetrizing and symmetrizing forms, respectively. The odd-level construction requires an explicit invertibility condition; at even level, a further modification also yields symmetrizing forms under an additional invertibility condition. The existence of these symmetrizing and supersymmetrizing forms does not require semisimplicity. We introduce Schur elements for supersymmetric superalgebras and obtain explicit product formulae, in terms of residues, for the Schur elements of cyclotomic Hecke-Clifford and Sergeev algebras under their separation conditions. These results yield new symmetrizing forms on level-one Hecke-Clifford algebras and on cyclotomic quiver Hecke algebras of affine types $A$ and $C$. For the cyclotomic quiver Hecke superalgebras of twisted affine type $A$, we obtain symmetrizing forms at odd level and supersymmetrizing forms at even level, and determine their $\Z_2$-parities. Finally, the two-strand, level-four cyclotomic odd nilHecke algebra gives a counterexample to ordinary symmetricity in general.
\end{abstract}
	\maketitle
	
	\setcounter{tocdepth}{1}
	\tableofcontents
	
\section{Introduction}

Cyclotomic Hecke-Clifford algebras combine Hecke generators with odd Clifford generators and impose a polynomial relation on an affine generator. Their degenerate counterparts are cyclotomic Sergeev algebras. Brundan and Kleshchev introduced these algebras to study spin representation theory and crystals of twisted type $A^{(2)}$ \cite{BK}; connections with twisted type $D^{(2)}$ were developed in \cite{T}. The quiver Hecke-Clifford realization of Kang, Kashiwara and Tsuchioka \cite{KKT} gives these algebras $\Z$-graded structures and connects them with the categorification results of \cite{KKO1,KKO2}. This paper studies trace forms and Schur elements for cyclotomic Hecke-Clifford and Sergeev algebras, with applications to cyclotomic quiver Hecke algebras and superalgebras.

The ordinary, non-super theory provides the background for these questions. Cyclotomic Hecke algebras were introduced in \cite{AK,BM:cyc,C} and play an important role in the representation theory of affine Hecke algebras \cite{Ariki:can}. The symmetrizing form constructed in \cite{MM} and the formulae for its Schur elements in \cite{CJ,GIM,HLL,Ma2} provide tools for semisimplicity criteria and modular representation theory. Related calculations for quiver Hecke algebras support the construction of integral bases \cite{EM,HL}. Thus a trace form and its Schur elements are linked parts of the ordinary theory. We seek their counterparts in the presence of Clifford generators, where the $\Z_2$-grading changes both the trace identity and the simple-module conventions.

There are two trace identities relevant to a superalgebra $A$. For elements $x,y\in A$ that are homogeneous in the $\Z_2$-grading, an ordinary trace and a supertrace satisfy, respectively,
\[
 t(xy)=t(yx),\qquad t(xy)=(-1)^{|x||y|}t(yx).
\]
We seek forms for which the multiplication pairing $(x,y)\mapsto t(xy)$ is non-degenerate. Initially, all such forms are even: $t$ vanishes on the odd part of $A$. They are called symmetrizing and supersymmetrizing forms, respectively; the precise convention over an integral domain is given in Definition~\ref{(super)symmetrizing form}. Supertraces of odd parity will enter only in the quiver Hecke superalgebra application below. Here parity always refers to the $\Z_2$-grading, not to the integer grading of a quiver Hecke superalgebra.

Schur elements measure the coefficients of a trace on simple modules. For an ordinary split semisimple algebra $A$ over a field ${\rm K}$, a symmetrizing form has the character expansion
\[
 t=\sum_{V\in\operatorname{Irr}(A)}s_V^{-1}\chi_V.
\]
The nonzero scalar $s_V\in {\rm K}$ is the Schur element of $V$ relative to $t$.
In the superalgebra setting the character expansion requires additional conventions, including the distinction between simple modules of types $\texttt{M}$ and $\texttt{Q}$. We recall the symmetric theory of \cite{WW2} and introduce its supersymmetric counterpart in Section~\ref{basics}.

\subsection{Main results}\label{intro:main-results}

There are already trace constructions in the super setting. Wan and Wang constructed a symmetrizing form on the level-one Hecke-Clifford algebra over $\mathbb C(q)$ \cite[Theorem 5.8]{WW2}.

For the cyclotomic algebras below, let ${\rm R}$ be an integral domain of characteristic different from $2$, and write ${\rm R}^\times$ for its units.\label{pag:R} Fix $m\geq0$ and $n\geq1$. We write $r$ for the level and assume $r\geq1$.
First consider the cyclotomic Sergeev algebras over ${\rm R}$. For $\undQ=(Q_1,\ldots,Q_m)\in{\rm R}^m$, the defining polynomial is one of
\begin{align}
 g(x_1)&=\prod_{i=1}^m\bigl(x_1^2-Q_i(Q_i+1)\bigr),& r&=2m,\label{deg-cyc-even}\\
 g(x_1)&=x_1\prod_{i=1}^m\bigl(x_1^2-Q_i(Q_i+1)\bigr),& r&=2m+1.\label{deg-cyc-odd}
\end{align}
Here $r=\deg g$ is the level. The standard Frobenius form is supersymmetric for \eqref{deg-cyc-even} and symmetric for \eqref{deg-cyc-odd} \cite[Theorem 1.2]{LS1}. The non-degenerate case has a different obstruction: its standard Frobenius form need not satisfy either trace identity. Our first question is therefore how to modify that form to obtain a non-degenerate trace or supertrace for the cyclotomic Hecke-Clifford algebras considered here.

For cyclotomic Hecke-Clifford algebras over ${\rm R}$, take $q\in{\rm R}^\times\setminus\{\pm1\}$ with $q+q^{-1}\in{\rm R}^\times$ and $\undQ=(Q_1,\ldots,Q_m)\in({\rm R}^\times)^m$. Put $\mathtt q(u)=2(qu+(qu)^{-1})/(q+q^{-1})$ for $u\in{\rm R}^\times$. The defining Laurent polynomial is one of
\begin{align}
 f(X_1)&=\prod_{i=1}^m\bigl(X_1+X_1^{-1}-\mathtt q(Q_i)\bigr),
 &&r=2m,\label{cyc-even}\\
 f(X_1)&=(X_1-1)\prod_{i=1}^m\bigl(X_1+X_1^{-1}-\mathtt q(Q_i)\bigr),
 &&r=2m+1.\label{cyc-odd}
\end{align}
The algebra $\mHfcn$ is the quotient of the affine Hecke-Clifford algebra by $f(X_1)=0$. Its level $r$ is the degree of the ordinary polynomial $X_1^m f(X_1)$; multiplying by $X_1^m$ does not change the ideal.
\textbf{For the results below, we restrict $f$ to \eqref{cyc-even} and \eqref{cyc-odd}. Within this restriction, even and odd level correspond to \eqref{cyc-even} and \eqref{cyc-odd}, respectively. The polynomial with the two additional factors $(X_1-1)(X_1+1)$ in \eqref{cyc-two-special-factors} is not included.}

\begin{thm}\label{Nondengerate}
(i) For $f$ in \eqref{cyc-even}, $\mHfcn$ admits an even supersymmetrizing form. If $\prod_{j=1}^n(X_j-X_j^{-1})$ is invertible in $\mHfcn$, then it also admits an even symmetrizing form.

(ii) For $f$ in \eqref{cyc-odd}, if $\prod_{j=1}^n(1+X_j)$ is invertible in $\mHfcn$, then $\mHfcn$ admits an even symmetrizing form.
\end{thm}

The supersymmetrizing form at even level and the symmetrizing form at odd level, both denoted by $t_{r,n}$, are constructed explicitly in Corollaries~\ref{supersym. form. non-dege} and~\ref{sym. form. non-dege}, respectively. Corollary~\ref{sym. form. even. non-dege} gives the further modification at even level.
Seminormal bases allow us to establish the required identities at generic parameters and then descend them to ${\rm R}$. For $f$ in \eqref{cyc-odd}, Proposition~\ref{Gamma condition} gives an explicit sufficient parameter condition for the required invertibility.

\textbf{We emphasize that the existence results in Theorem~\ref{Nondengerate} do not require semisimplicity. The separation conditions imposed below are needed for the computation of Schur elements, not for the existence of these forms.}

The formulae for Schur elements of cyclotomic Hecke algebras suggest the next question: can the Schur elements associated with the forms constructed here also be computed as explicit products?
For the computation of Schur elements, we now work over an algebraically closed field $\mathbb K$ of characteristic different from $2$, and put $\mathbb K^*=\mathbb K\setminus\{0\}$.\label{pag:K} We take the defining parameters in $\mathbb K$ and retain the corresponding parameter restrictions above; in particular, the Sergeev parameters lie in $\mathbb K^m$, whereas the Hecke-Clifford parameters $\undQ$ lie in $(\mathbb K^*)^m$.
The simple modules under the separation conditions of \cite{SW,LS2} are indexed by ordinary multipartitions for \eqref{cyc-even}, and by a strict partition together with ordinary components for \eqref{cyc-odd}. We denote the corresponding index sets by $\mathscr P_n^f$. The notation includes the chosen parameter order and factor presentation of $f$. For the Sergeev polynomials $g$ in \eqref{deg-cyc-even} and~\eqref{deg-cyc-odd}, the analogous sets are $\mathscr P_n^g$.

We impose the separation conditions $P_n^f(q^2,\undQ)\ne0$ and $P_n^g(\undQ)\ne0$, respectively. These require separation for every corresponding shape of size $n+1$; see Propositions~\ref{separate formula} and~\ref{separate formula dege}. In what follows, we compute the Schur elements under these assumptions.

We use the specific simple-module representatives $\mathbb D(\undla)$ and $D(\undla)$ constructed in Section~\ref{Seminormal form}. Their Schur elements are taken relative to the forms $t_{r,n}$ of Corollaries~\ref{supersym. form. non-dege} and~\ref{sym. form. non-dege} and the standard Sergeev forms $\mathtt t_{r,n}$ of \cite[Theorem 1.2]{LS1}. In the supersymmetric case, fixing these representatives also fixes the sign convention under parity shift.

\begin{thm}\label{Schur-Non-dengerate}\label{Schur-dengerate}
Assume the corresponding separation condition holds over $\mathbb K$.

(i) For $f$ in \eqref{cyc-even} or~\eqref{cyc-odd}, the Schur elements of $\mathbb D(\undla)$, $\undla\in\mathscr P_n^f$, are nonzero and are given by the explicit product formulae in terms of residues in Corollaries~\ref{schur-even-f} and~\ref{schur-odd-f}.

(ii) For $g$ in \eqref{deg-cyc-even} or~\eqref{deg-cyc-odd}, the Schur elements of $D(\undla)$, $\undla\in\mathscr P_n^g$, are nonzero and are given by the explicit product formulae in terms of residues in Corollaries~\ref{schur-even-g} and~\ref{schur-odd-g}.
\end{thm}

These products involve residues of addable and removable boxes along a standard-tableau path; the final values are independent of the chosen tableau. The computation starts with the trace of a primitive idempotent. We first reduce its Clifford part, and then use the cyclotomic Mackey decomposition to calculate recursively in the rank. The extra Clifford data make this reduction and the choice of Mackey terms more delicate than their ordinary counterparts in \cite{HL,HLL}. Sections~\ref{dot=0} and~\ref{dot=s} carry out the even- and odd-level calculations separately.

\subsection{Applications and a counterexample}\label{intro:applications}

The symmetrizing and supersymmetrizing forms constructed above also have the following applications.

\begin{cor}\label{intro:application-overview}
(i) At level one, $f=X_1-1$, the Hecke-Clifford algebra $\mathcal H(n)$ admits the symmetrizing form $t_{1,n}$ whenever $q$ is not a primitive $4l$-th root of unity for $1\leq l\leq n$.

(ii) Under the realization and parameter conditions of Section~\ref{application:quiver}, the cyclotomic quiver Hecke algebras of types $A^{(1)}_{e-1}$ and $C^{(1)}_e$, $e>1$, admit symmetrizing forms induced by $t_{r,n}$ and its modification in Corollary~\ref{sym. form. even. non-dege}.
\end{cor}

The constructions are given in Propositions~\ref{Hecke-Cliffod sym} and~\ref{new trace form}. At $q=1$, the symmetrizing form and Schur elements of the level-one Hecke-Clifford algebra specialize to scalar multiples of the corresponding form and Schur elements of the Sergeev algebra; see Proposition~\ref{application:q1comparison}.

For twisted affine type $A_{2e}^{(2)}$, the same method must keep track of the parity of the form. Let $e\geq1$, let $h_0,\ldots,h_e$ be the simple coroots, and let $\Lambda$ be a dominant integral weight. With the unique odd vertex labelled $0$, set
\[
 K=h_0+2\sum_{i=1}^e h_i,\qquad r=\langle K,\Lambda\rangle.
\]
For a positive root-lattice element $\alpha$, write $\ell_\alpha$ for the total multiplicity of its odd simple roots. In this application, supersymmetrizing forms are allowed to have either $\Z_2$-parity: a form of parity $\varepsilon$ vanishes on $A_\delta$ for $\delta\ne\varepsilon$. Section~\ref{LS1:subsec:twisted-affine-A} specifies the cyclotomic quiver Hecke superalgebra ${\rm R}_\alpha^\Lambda$ and the two realizations used below.

\begin{thm}\label{twisted-A-overview}
Let ${\rm R}_\alpha^\Lambda\ne0$ be of type $A_{2e}^{(2)}$ over an algebraically closed field $\mathbb K$. Assume either
\begin{enumerate}
\item[(1)] $\operatorname{char}\mathbb K=0$, or $\operatorname{char}\mathbb K=p>0$ with $p\nmid2(2e+1)$, and choose $q\in\mathbb K^*$ such that $q^2$ is a primitive $(2e+1)$-th root of unity; or
\item[(2)] $2e+1=p$ is prime and $\operatorname{char}\mathbb K=p$.
\end{enumerate}
At odd level $r$, ${\rm R}_\alpha^\Lambda$ admits a symmetrizing form. At even level $r$, it admits a supersymmetrizing form of $\Z_2$-parity $\overline{\ell_\alpha}$.
\end{thm}

Proposition~\ref{LS1:prop:twisted-affine-A} gives the precise realization-based statement and proof.
The even-level conclusion asserts the existence of a supersymmetrizing form of the stated $\Z_2$-parity; it does not assert that all even-level blocks fail to be symmetric.

This distinction also clarifies the boundary of ordinary symmetricity. Kleshchev and Livesey conjectured that every cyclotomic quiver Hecke superalgebra is symmetric \cite[Remark 4.2.58]{KL:RoCK}. The odd-level conclusion above confirms this in the stated twisted type-$A$ setting. In general, however, ordinary symmetricity fails. Let $\operatorname{ONH}_2^4$ be the two-strand, level-four cyclotomic odd nilHecke algebra; its generators and relations are recalled in Section~\ref{LS1:sec:counterexample}.

\begin{prop}\label{LS1:prop:ONH-not-symmetric}
Over a field of characteristic different from $2$, $\operatorname{ONH}_2^4$ admits no non-degenerate ordinary trace. In particular, it is not a symmetric superalgebra.
\end{prop}

The proof uses the monomial basis obtained in \cite{HuShi:super} for the cyclotomic odd nilHecke algebra.
This counterexample concerns the general conjecture and is separate from the positive twisted affine type-$A$ result.

\enlargethispage{2\baselineskip}
Building on the results of this paper, the second author establishes semisimplicity criteria for cyclotomic Hecke-Clifford superalgebras in \cite[Theorem 1.3]{Sh}. We also use these results to construct generalized graded cellular bases for certain cyclotomic quiver Hecke-Clifford superalgebras in \cite[Theorem 1.2]{LS3}. In a forthcoming paper, we will use the symmetrizing forms constructed in this paper to study the centers and cocenters of non-degenerate cyclotomic Hecke-Clifford algebras.

The paper proceeds from the algebra and combinatorics in Section~\ref{preli} to seminormal forms in Section~\ref{Seminormal form}. Sections~\ref{embedding} and~\ref{frobenius} provide the results on restriction and the calculations of Frobenius forms needed in Sections~\ref{dot=0} and~\ref{dot=s}, where we construct the symmetrizing and supersymmetrizing forms and derive formulae for their Schur elements. Section~\ref{application} develops the applications, and Section~\ref{LS1:sec:counterexample} proves the counterexample.

\noindent\begin{minipage}{\textwidth}
\bigskip
\centerline{\bf Acknowledgements}
\bigskip
The research is supported by the National Natural Science Foundation of China
(No. 12431002).
The first author is partially supported by the National Natural Science Foundation of China (Nos. 12341101 and 12771460).
The second author is partially supported by the Postdoctoral Fellowship Program of CPSF under Grant Number GZB20250717.
Both authors thank Huansheng Li for his many valuable discussions on the computations of Schur elements.
Additionally, they thank Jinkui Wan for her suggestions and feedback.
\bigskip
\end{minipage}

	\section{Preliminaries}\label{preli}
Let $\N=\{1,2,\ldots\}$ and $[n]=\{1,\ldots,n\}$ for $n\in\N$.\label{pag:N}\label{pag:[n]}
Let $\mathfrak S_n$ be the symmetric group on $[n]$, with simple transpositions $s_1,\ldots,s_{n-1}$.\label{pag:Sn}
	\subsection{(Super)symmetric superalgebras and Schur elements}\label{basics}
   	We recall some basic notions of superalgebras, referring the
   reader to~\cite[\S 2-b]{BK}. Let us denote by
   $|v|\in\mathbb{Z}_2$ \label{pag:||} the parity of a homogeneous vector $v$ of an
   ${\rm R}$-vector superspace. By a superalgebra, we mean a
   $\mathbb{Z}_2$-graded associative ${\rm R}$-algebra.  Let $\mathcal{A}$ be an
   ${\rm R}$-superalgebra. By an $\mathcal{A}$-module, we mean a $\mathbb{Z}_2$-graded
   left $\mathcal{A}$-module. A homomorphism $f:V\rightarrow W$ of
   $\mathcal{A}$-modules $V$ and $W$ means a linear map such that $
   f(av)=(-1)^{|f||a|}af(v).$  Note that this and other such
   expressions only make sense for homogeneous $a, f$ and the meaning
   for arbitrary elements is to be obtained by extending linearly from
   the homogeneous case. A non-zero element $e\in\mathcal{A}$ is called a (super) primitive idempotent if
	it is an idempotent with $|e|=\bar{0}$ and it cannot be decomposed as the sum of
	two nonzero orthogonal idempotents with parity $\bar{0}.$
	
	Let $\mathcal{A}$ be a finite-dimensional superalgebra over $\mathbb{K}$ and $V$ be a finite dimensional
	$\mathcal{A}$-module. Let $\Pi
	V$ \label{pag:parity shift} be the same underlying vector space but with the opposite
	$\mathbb{Z}_2$-grading. The new action of $a\in\mathcal{A}$ on $v\in\Pi
	V$ is defined in terms of the old action by $a\cdot
	v:=(-1)^{|a|}av$. Note that the identity map on $V$ defines
	an isomorphism from $V$ to $\Pi V$. A superalgebra analog of Schur's Lemma states that the endomorphism
	algebra of a finite dimensional irreducible module over $\mathcal{A}$ is either one dimensional or two dimensional. In the
	former case, we call the module of {\em type }\texttt{M} while in
	the latter case the module is called of {\em type }\texttt{Q}.

	\begin{example}\label{simple algebra}
		(1) Let $V$ be a superspace with superdimension $(m,n)$ over the field $\mathbb{K}$, then $\mathcal{M}_{m,n}:={\text{End}}_{\mathbb{K}}(V)$ is a simple superalgebra with simple module $V$ of {\em type }\texttt{M}.  \\
		(2) Let $V$ be a superspace with superdimension $(n,n)$ over the field $\mathbb{K}$. We define
$\mathcal{Q}_n:=\left\{\begin{pmatrix}
A  & B\\
-B & A
\end{pmatrix}\biggm| A,B\in  M_n\right\}\subset \mathcal{M}_{n,n},$ then $\mathcal{Q}_n$ is a simple superalgebra with simple module $V$ of {\em type }\texttt{Q}.
\end{example}
Denote by $\mathcal{J}(\mathcal{A})$  the usual (non-super) Jacobson radical of $\mathcal{A}.$ We call $\mathcal{A}$ semisimple if $\mathcal{J}(\mathcal{A})=0$.

\begin{lem}\cite[Lemma 12.2.9]{K2}\label{semisimple}
	A finite dimensional superalgebra $\mathcal{A}$ over $\mathbb{K}$ is semisimple if and only if it is a direct sum of some simple superalgebras. Moreover, any finite dimensional simple superalgebra  over $\mathbb{K}$ is isomorphic to some $\mathcal{M}_{m,n}$ or $\mathcal{Q}_n$.
\end{lem}
	
Let $\mathcal A=\mathcal A_{\bar0}\oplus\mathcal A_{\bar1}$ be an ${\rm R}$-superalgebra that is free of finite rank.
An ${\rm R}$-linear map $t:\mathcal A\to{\rm R}$ is non-degenerate if the matrix $(t(b_1b_2))_{b_1,b_2\in\mathcal B}$ has unit determinant for a homogeneous basis $\mathcal B$.
\begin{defn}\label{(super)symmetrizing form}\cite[Definition 2.1]{LS1}, \cite[Sections 4.1, 5.1]{WW2}
		  	(i) The superalgebra $\mathcal{A}$ is called symmetric if there is an even, non-degenerate ${\rm R}$-linear map $t:\mathcal{A} \rightarrow {\rm R}$
		  	 such that $t(xy)=t(yx)$ for any $x, y\in \mathcal{A}$. In this case, we call $t$ a symmetrizing form on $\mathcal{A}.$
		  	
		  	(ii) The superalgebra $\mathcal{A}$ is called supersymmetric if there is an even, non-degenerate ${\rm R}$-linear map $t:\mathcal{A} \rightarrow {\rm R}$
		   such that  $t(xy)=(-1)^{|x| |y|}t(yx)$ for any homogeneous $x, y\in \mathcal{A}$. In this case, we call $t$ a supersymmetrizing form on $\mathcal{A}.$
		  \end{defn}

Now suppose the ${\rm R}$-superalgebra $\mathcal{A}$ is (super)symmetric with a (super)symmetrizing form $t$. Let $\mathcal{B}$ be a $\Z_2$-homogeneous basis of $\mathcal{A}.$ We denote by $\mathcal{B}^{\vee}=\{b^{\vee}\mid b\in \mathcal{B}\}$ the dual basis, which is also homogeneous and satisfies $t(b^\vee b')=\delta_{b,b'}$ for any $b,b'\in\mathcal{B}.$ Suppose ${V,V'}$ are two $\mathcal{A}$-modules. For any homogeneous map $f\in{\rm Hom}_{{\rm R}}(V, {V'})$, we define $I(f)\in {\rm Hom}_{{\rm R}}(V, { V'})$ by
\begin{equation}\label{eq:averaging-map}
I(f)(v):=\begin{cases} \sum_{b\in \mathcal{B}}(-1)^{|f||b|}b^{\vee} f(bv), &\text{if $t$ is symmetric};\\
	\sum_{b\in \mathcal{B}}(-1)^{|f||v|+|b|}b^{\vee} f(bv), &\text{if $t$ is supersymmetric},
\end{cases}
\end{equation}
for homogeneous $v\in V$, extending linearly.

As noted in \cite{WW2}, when $t$ is symmetric, $I(f)$ is independent of the choice of the homogeneous basis $\mathcal{B}$ and $I(f)\in{\rm Hom}_{\mathcal{A}}(V, {V'})$. Similarly, we can follow essentially the same proof as for \cite[Lemma 7.1.10]{GF} to see that when $t$ is supersymmetric, $I(f)$ is independent of the choice of the homogeneous basis $\mathcal{B}$ and $I(f)\in{\rm Hom}_{\mathcal{A}}(V, { V'})$.


\begin{defn}\cite[Section 1.1.2]{CW}
	Let $V$ be a superspace with superdimension $(m,n)$ over the field $\mathbb{K}$. For $f\in{\rm End}_{\mathbb{K}}(V)$, we write $f=f_{\bar{0}}+f_{\bar{1}}$, where $f_{a}$ is the $a$-th component of $f$. We define the supertrace of $f$ as follows$$
	{\rm suptr}(f):={\rm tr}(f_{\bar{0}}\downarrow_{V_{\bar{0}}})-{\rm tr}(f_{\bar{0}}\downarrow_{V_{\bar{1}}}).
	$$
\end{defn}

\begin{rem}\label{rem:supertrace-conjugation}
It's clear that ${\rm suptr}(f)=0$ for any $f\in{\rm End}_{\mathbb{K}}(V)_{\bar{1}}$ and ${\rm suptr}(g^{-1}fg)=(-1)^{|g|}{\rm suptr}(f)$
for any homogeneous isomorphism $g\in {\rm Hom}_{{\mathbb{K}}}(W,V),$ where $W$ is a superspace.
\end{rem}

We denote by ${\rm Irr}(\mathcal{A})$ the complete set of non-isomorphic irreducible $\mathcal{A}$-modules.
\begin{defn}	Suppose $\mathcal{A}$ is a split semisimple superalgebra over $\mathbb{K}$, then we can write $\mathcal{A}$ as a direct sum of simple superalgebras:$$
	\mathcal{A}=\bigoplus_{V\in{\rm Irr}(\mathcal{A})}H(V).
	$$ For any split irreducible $\mathcal{A}$-module $V$, we define $$
	\chi_V(a):={\rm tr}(a\downarrow_{V}),\qquad \chi'_V(a):={\rm suptr}(a\downarrow_{V}),\qquad \forall a\in\mathcal{A}.
	$$
\end{defn}	

By definition, $\chi_V$ is the usual character of an irreducible $\mathcal{A}$-module $V$.

\begin{lem}\cite[Lemma 5.2]{WW2}
	Suppose $\mathcal{A}$ is symmetric with a symmetrizing form $t$.	Let $V$ be a split irreducible $\mathcal{A}$-module. Then there exists a unique element $s_V\in \mathbb{K}$ such that  \begin{align}\label{schur1}I(f)=s_V{\rm tr}(f){\rm id}_V,\qquad\text{for $f\in{{\rm End}_{\mathbb{K}}(V)}_{\bar{0}}$}.
	\end{align} Furthermore, $s_V$ only depends on the isomorphism class of $V$.
\end{lem}

The element $s_V$ is called the {\bf Schur element} of $V$ with respect to $t$. \label{pag:Schur element of V}
For any $V\in{\rm Irr}(\mathcal{A}),$ we write
$$
\delta(V):=\begin{cases}
	0, &\text{if $V$ is of {\em type }\texttt{M}};\\
	1, &\text{if $V$ is of {\em type }\texttt{Q}}.
	\end{cases}
$$

\begin{prop}\label{schur formula 1}\cite[Proposition 5.4]{WW2}
Suppose $\mathcal{A}$ is a split semisimple superalgebra over $\mathbb{K}$ and $\mathcal{A}$ is symmetric with a symmetrizing form $t$. Then the Schur element $s_V$ for every irreducible $\mathcal{A}$-module $V$ is non-zero. Moreover,
\begin{align*}
t=\sum_{V\in{\rm Irr}(\mathcal{A})}\frac{1}{2^{\delta(V)}s_V}\chi_V.
\end{align*}
\end{prop}

Now we consider the analog in supersymmetric algebra $\mathcal{A}.$
\begin{lem}
Suppose $\mathcal{A}$ is supersymmetric with a supersymmetrizing form $t$.	Let $V$ be a split irreducible $\mathcal{A}$-module of {\em type }\texttt{M}. Then there exists a unique element $s_V\in \mathbb{K}$ such that \begin{align}\label{schur2}
	I(f)=s_V{\rm suptr}(f){\rm id}_V,\qquad\text{for $f\in{{\rm End}_{\mathbb{K}}(V)}_{\bar{0}}$}.
\end{align}Furthermore, up to sign, $s_V$ only depends on the isomorphism class of $V$.
	\end{lem}
The element $s_V$ is called the {\bf Schur element} of $V$ with respect to $t$.

\begin{proof}
Let $\mathcal{B}$ be a $\Z_2$-homogeneous basis for $\mathcal{A}$, and let $\mathcal{B}^{\vee}$ be its homogeneous dual basis, so that $t(b^\vee b')=\delta_{b,b'}$. Choose a homogeneous basis $\{e_i\mid 1\leq i\leq l\}$ of $V$. For any $a\in\mathcal{A}$, we write $ae_i=\sum_{s=1}^l a_{si}e_s$ for $1\leq i\leq l$. For $i,j\in[l]$ with $|e_i|=|e_j|$, the definition \eqref{eq:averaging-map} and the matrix-unit action give \begin{align*}
	I(E_{ij})(e_k)&=\sum_{\substack{b\in\mathcal{B}\\|b|=|e_i|+|e_k|}}	(-1)^{|b|} b^{\vee}E_{ij}(be_k)\\
	&=\sum_{\substack{b\in\mathcal{B}\\|b|=|e_i|+|e_k|}}	(-1)^{|b|} b^{\vee}E_{ij}(\sum_{s=1}^lb_{sk}e_s)\\
	&=\sum_{\substack{b\in\mathcal{B}\\|b|=|e_i|+|e_k|}}	(-1)^{|b|} b^{\vee}b_{jk}e_i\\
	&=\sum_{\substack{b\in\mathcal{B}\\|b|=|e_i|+|e_k|}}	(-1)^{|b|} \sum_{s=1}^lb^{\vee}_{si}b_{jk}e_s.
	\end{align*}
	On the other hand, since $I(E_{ij})\in {\rm End}_{\mathcal{A}}(V)\cong \mathbb{K}$, we deduce that $I(E_{ij})(e_k)=c_{ij}e_k$ for some $c_{ij}\in \mathbb{K}$. Comparing coefficients of $e_s$, we deduce \begin{align}\label{orthogonal 1}
		\sum_{\substack{b\in\mathcal{B}\\|b|=|e_i|+|e_k|}}	(-1)^{|b|} b^{\vee}_{si}b_{jk}=\delta_{sk}c_{ij},\qquad i,j,s,k\in[l].
		\end{align}
		Similarly, we substitute $\mathcal{B}$ and $\mathcal{B}^\vee$ by $(\mathcal{B}^\vee)':=\{(-1)^{|b|}b^\vee|~b\in\mathcal{B}\}$ and $\mathcal{B}$ respectively. Following the same computation, we obtain\begin{align}\label{orthogonal 2}
			\sum_{\substack{b\in\mathcal{B}\\|b|=|e_i|+|e_k|}}	 b_{si}b^{\vee}_{jk}=\delta_{sk}c_{ij},\qquad i,j,s,k\in[l].
		\end{align}
	
	We claim that $c_{ij}=0$ if $i\neq j$. Actually, we can choose $s=k$ in \eqref{orthogonal 1}, we then deduce that \begin{align}\label{orthogonal 3}
		(-1)^{|e_i|+|e_s|}\cdot\sum_{\substack{b\in\mathcal{B}\\|b|=|e_i|+|e_s|}} b^{\vee}_{si}b_{js}=c_{ij}.
	\end{align} On the other hand, by \eqref{orthogonal 2}, we have  \begin{align*}
	\sum_{\substack{b\in\mathcal{B}\\|b|=|e_s|+|e_j|}}	 b^{\vee}_{si}b_{js}=\delta_{ij}c_{ss}=0.
	\end{align*} This, combined with \eqref{orthogonal 3} and the assumption $|e_i|=|e_j|,$ proves our claim.
	
	Next, we compare $c_{ii}$ with $c_{ss}$ for $i\neq s$. We choose $i=j,s=k$ in \eqref{orthogonal 1}. Since $\mathcal{B}$ is a homogeneous basis, we deduce that \begin{align}\label{orthogonal 4}
		(-1)^{|e_i|+|e_s|}\cdot\sum_{\substack{b\in\mathcal{B}\\|b|=|e_i|+|e_s|}}	b^{\vee}_{si}b_{is}=c_{ii}.
	\end{align} On the other hand, by \eqref{orthogonal 2}, we have \begin{align*}
	\sum_{\substack{b\in\mathcal{B}\\|b|=|e_i|+|e_s|}}	b^{\vee}_{si}b_{is}=c_{ss}.
	\end{align*} This, combined with \eqref{orthogonal 4}, implies that $c_{ii}=c_{ss}$ for $|e_i|=|e_s|$ and $-c_{ii}=c_{ss}$ for $|e_i|\neq|e_s|.$
	
Set $s_V:=(-1)^{|e_1|}c_{11}$. Equations~\eqref{orthogonal 1}--\eqref{orthogonal 4} give $c_{ij}=0$ for $i\ne j$ and $c_{ii}=(-1)^{|e_i|}s_V$. By linearity on the even matrix units, these identities prove \eqref{schur2}.

Now let $f:V\to W$ be a homogeneous $\mathcal A$-module isomorphism. Write $I_V$ and $I_W$ for the averaging maps on the two modules. For any even $h\in\operatorname{End}_{\mathbb K}(W)$, definition~\eqref{eq:averaging-map} gives
\begin{equation}\label{eq:schur-averaging-conjugation}
 f^{-1}I_W(h)f
 =\sum_{b\in\mathcal B}(-1)^{|b|}f^{-1}b^\vee hbf
 =I_V(f^{-1}hf).
\end{equation}
The second equality uses $f^{-1}b^\vee=(-1)^{|f||b|}b^\vee f^{-1}$ and $bf=(-1)^{|f||b|}fb$; the two signs cancel. Applying \eqref{schur2} to \eqref{eq:schur-averaging-conjugation}, and then Remark~\ref{rem:supertrace-conjugation}, yields
\begin{align*}
 s_W\operatorname{suptr}_W(h)\operatorname{id}_V
 &=s_V\operatorname{suptr}_V(f^{-1}hf)\operatorname{id}_V\\
 &=(-1)^{|f|}s_V\operatorname{suptr}_W(h)\operatorname{id}_V.
\end{align*}
Taking $h$ to be the projection onto a homogeneous basis vector gives $s_W=(-1)^{|f|}s_V$, as required.
	\end{proof}
	
	\begin{example}\label{supersym schur}
		We consider $\mathcal{M}_{m,n}$ in Example \ref{simple algebra} (1). It's easy to check that ${\rm suptr}: A\mapsto {\rm suptr}(A)$ is a supersymmetrizing form on $\mathcal{M}_{m,n}$ and any other supersymmetrizing form is a scalar multiple of ${\rm suptr}$. Then $\mathcal{B}^\vee=\{E_{ij}\mid i,j\in [m]\}\bigsqcup \{-E_{ij}\mid m+1\leq i,j\leq m+n\}\bigsqcup \{-E_{ij}\mid m+1\leq i\leq m+n,1\leq j\leq m\}\bigsqcup \{E_{ij}\mid 1\leq i\leq m,\,m+1\leq j\leq m+n\}$ is a dual basis of $\mathcal{B}=\{E_{ij}\mid i,j\in [m]\}\bigsqcup \{E_{ij}\mid m+1\leq i,j\leq m+n\}\bigsqcup \{E_{ij}\mid 1\leq i\leq m,\,m+1\leq j\leq m+n\}\bigsqcup \{E_{ij}\mid m+1\leq i\leq m+n,1\leq j\leq m\}$. Now for any $A\in(\mathcal{M}_{m,n})_{\bar{0}}$, we can compute \begin{align*}I(A)(e_k)=&\sum_{i,j\in[m]}E_{ji}A(E_{ij}e_k)+ \sum_{m+1\leq i,j\leq m+n}(-E_{ji})A(E_{ij}e_k)\\
			&\qquad+\sum_{\substack{1\leq i\leq m,\\m+1\leq j\leq m+n}}E_{ji}A(E_{ij}e_k)+\sum_{\substack{1\leq j\leq m,\\m+1\leq i\leq m+n}}(-E_{ji})A(E_{ij}e_k)\\
			&\qquad\qquad={\rm suptr}(A) e_k.
			\end{align*} Hence we conclude $s_V=1$ and $I(A)={\rm suptr}(A)\operatorname{id}_V$ for any $A\in(\mathcal{M}_{m,n})_{\bar{0}}$.
		\end{example}
	
The supercocenter of $\mathcal{A}$ is defined by
$${\rm SupTr}(\mathcal{A}):=\mathcal{A}/[\mathcal{A}, \mathcal{A}]^{-},$$
where $[\mathcal{A}, \mathcal{A}]^{-}$ is the $\mathbb{K}$-span of all supercommutators $[x, y]^{-}:=xy-(-1)^{|x| |y|}yx,$ $x, y \in \mathcal{A}.$
\begin{prop}\label{supersym. schur formula 2}
	Suppose $\mathcal{A}$ is a split semisimple superalgebra over $\mathbb{K}$ and $\mathcal{A}$ is supersymmetric with a supersymmetrizing form $t$. Then every irreducible $\mathcal{A}$-module $V$ is of {\em type }\texttt{M}, and the Schur element $s_V$ is non-zero.
Moreover, we have \begin{align*}
		t=\sum_{\substack{V\in{\rm Irr}(\mathcal{A})}}\frac{1}{s_V}\chi'_V.
	\end{align*}
\end{prop}
\begin{proof}
Since $\mathcal{A}$ is a split semisimple superalgebra over $\mathbb{K}$, we can write $\mathcal{A}$ as a direct sum of simple superalgebras:$$
	\mathcal{A}=\bigoplus_{V\in{\rm Irr}(\mathcal{A})}H(V).
	$$
By definition, $t\in  {\rm Hom}_{\mathbb{K}}({\rm SupTr}({\mathcal{A}}),\mathbb{K})$. Moreover, by \cite[Example 2.5]{LS1} and Example \ref{supersym schur}, we deduce that $\{\chi'_{V}\mid V\in{\rm Irr}(\mathcal{A}),\text{$V$ is of \em type }\texttt{M} \}$ forms a $\mathbb{K}$-basis of ${\rm Hom}_{\mathbb{K}}({\rm SupTr}({\mathcal{A}}),\mathbb{K})$. Hence $$
t=\sum_{\substack{V\in{\rm Irr}(\mathcal{A})\\ \text{$V$ is of \em type }\texttt{M}}}d_V\chi'_V,
$$
where $d_V\in \mathbb{K}$ and the restriction of $t$ to the block $H(V)$ is $d_V\chi'_V$. This implies that $d_V\neq 0$ for each $V$ of {\em type }\texttt{M} and there is no irreducible $\mathcal{A}$-module of {\em type }\texttt{Q}, since $t$ is non-degenerate. We only need to prove $d_Vs_V=1$ for each $V$ of {\em type }\texttt{M}. Replacing \cite[Example 5.3]{WW2} with Example \ref{supersym schur} in the computation of \cite[Proposition 5.4]{WW2},  one can
prove the above equation by the same argument as that used in the proof of \cite[Proposition 5.4]{WW2}.
	\end{proof}

	\subsection{Affine Hecke-Clifford algebra $\mHcn$}\label{basic-Non-dege}
	
	Let $q\in{\rm R^\times}\setminus\{\pm 1\}$ such that $q+q^{-1}\in{\rm R^\times}.$
    We define $\epsilon:=q-q^{-1}\in{\rm R}\setminus\{0\}.$
    The non-degenerate affine Hecke-Clifford algebra $\mHcn$ \label{pag:AHCA} is
	the superalgebra over ${\rm R}$ generated by even generators
	$T_1,\ldots,T_{n-1},$ $X_1^{\pm 1},\ldots,X_n^{\pm 1}$ and odd generators
	$C_1,\ldots,C_n$ subject to the following relations\footnote{We remark that there is a typo in \cite[Relation (2.6)]{LS2}, where the ``$1+C_iC_{i+1}$'' should be ``$1-C_iC_{i+1}$'' as here \eqref{PX2}. In fact, the relation \eqref{PX2} is equivalent to the \eqref{PX1}, see \cite[Before (2.19)]{BK}.}
	\begin{align}
		T_i^2=\epsilon T_i +1,\quad T_iT_j =T_jT_i, &\quad
		T_iT_{i+1}T_i=T_{i+1}T_iT_{i+1}, \quad|i-j|>1,\label{Braid}\\
		X_iX_j&=X_jX_i, X_iX^{-1}_i=X^{-1}_iX_i=1, \quad 1\leq i,j\leq n, \label{Poly}\\
		C_i^2=1,C_iC_j&=-C_jC_i, \quad 1\leq i\neq j\leq n, \label{Clifford}\\
		T_iX_i&=X_{i+1}T_i-\epsilon(X_{i+1}+C_iC_{i+1}X_i),\label{PX1}\\
		T_iX_{i+1}&=X_iT_i+\epsilon(1-C_iC_{i+1})X_{i+1},\label{PX2}\\
		T_iX_j&=X_jT_i, \quad j\neq i, i+1, \label{PX3}\\
		T_iC_i=C_{i+1}T_i, T_iC_{i+1}&=C_iT_i-\epsilon(C_i-C_{i+1}),T_iC_j=C_jT_i,\quad j\neq i, i+1, \label{PC}\\
		X_iC_i=C_iX^{-1}_i, X_iC_j&=C_jX_i,\quad 1\leq i\neq j\leq n.
		\label{XC}
	\end{align}
	
	For $\alpha=(\alpha_1,\ldots,\alpha_n)\in\mathbb{Z}^n$ and
	$\beta=(\beta_1,\ldots,\beta_n)\in\mathbb{Z}_2^n$, we set
	$X^{\alpha}=X_1^{\alpha_1}\cdots X_n^{\alpha_n},$
	$C^{\beta}=C_1^{\beta_1}\cdots C_n^{\beta_n}$ and define
	$\supp(\beta):=\{1 \leq k \leq n:\beta_{k}=\bar{1}\},$ $|\beta|:=\sum_{i=1}^{n}\beta_i \in \mathbb{Z}_2.$ \label{pag:suppot and sum}
	Then we have the
	following.
	\begin{lem}\cite[Theorem 2.2]{BK}\label{lem:PBWNon-dege}
		The set $\{X^{\alpha}C^{\beta}T_w~|~ \alpha\in\mathbb{Z}^n,
		\beta\in\mathbb{Z}_2^n, w\in \mathfrak{S}_n\}$ forms an ${\rm R}$-basis of $\mHcn.$
	\end{lem}
	
	Let $\mathcal{A}_n$ \label{pag:subalg An} be the subalgebra generated by even generators $X_1^{\pm 1},\ldots,X_n^{\pm 1}$ and odd generators $C_1,\ldots,C_n$. By Lemma~\ref{lem:PBWNon-dege}, $\mathcal{A}_n$ actually can be identified with the superalgebra generated by even generators $X^{\pm 1}_1,\ldots,X^{\pm 1}_n$ and odd generators $C_1,\ldots,C_n$ subject to relations \eqref{Poly}, \eqref{Clifford}, \eqref{XC}. {\bf Clifford algebra} $\mathcal{C}_n$ \label{pag:Clifford algebra} is the subalgebra of $\mathcal{A}_n$ generated by  odd generators $C_1,\ldots,C_n$ subject to relations \eqref{Clifford}.
	
Now we quickly review the representation theory of Clifford superalgebra. In the rest of this subsection, we assume that
${\rm R}=\mathbb{K}$ is the algebraically closed field of characteristic different from $2$. For any $a\in \mathbb{K}$, we fix a solution of the equation $x^2=a$ and denote it by $\sqrt{a}$.

Let $A$ be any algebra and $a_1,a_2,\ldots,a_p\in A$, we define the ordered product \label{pag:ordered product} as
$$\overrightarrow{\prod_{i=1,2,\ldots, p}}a_i:=a_1 a_2 \ldots a_p.$$ Then we have the following.
\begin{lem}\cite[Lemma 2.4]{LS2}\label{lem:clifford rep}
	
	(1)  $\mathcal{C}_n$ is a simple superalgebra with the unique simple (super)module of type $\texttt{Q}$ if $n$ is odd, of type $\texttt{M}$ if $n$ is even. Let
	$$I_n:=\begin{cases}
		\{1\}, &\text{if $n=1$;}\\
		\Biggl\{2^{-\lfloor n/2 \rfloor}\cdot\overrightarrow{\prod}_{k=1,\cdots,{\lfloor n/2 \rfloor}}(1+(-1)^{a_k} \sqrt{-1}C_{2k-1}C_{2k})\Biggm|a_k\in\Z_2,\,1\leq k\leq {\lfloor n/2 \rfloor} \Biggr\}, &\text{if $n>1$,}
	\end{cases}$$
	where $\lfloor n/2 \rfloor$ \label{pag:round down} denotes the greatest integer less than or equal to $n/2.$
	Then the set $I_n$ forms a complete set of super primitive idempotents for $\mathcal{C}_n.$

	(2) Let $\gamma_1,\gamma_2\in I_n$. There is a unique monomial of the form $C^{b_1}_1C^{b_3}_3\cdots C^{b_{2{\lfloor n/2 \rfloor}-1}}_{2{\lfloor n/2 \rfloor}-1}$, where $b_1,b_3,\cdots, b_{2{\lfloor n/2 \rfloor}-1}\in \Z_2$ such that $\gamma_1,\gamma_2$ are conjugate via $C^{b_1}_1C^{b_3}_3\cdots C^{b_{2{\lfloor n/2 \rfloor}-1}}_{2{\lfloor n/2 \rfloor}-1}$. Moreover, if $n$ is odd, there is also a unique monomial of the form $C^{b_1}_1C^{b_3}_3\cdots C^{b_{2{\lfloor n/2 \rfloor}-1}}_{2{\lfloor n/2 \rfloor}-1}C_{2{\lceil n/2 \rceil}-1}$, where $b_1,b_3,\cdots, b_{2{\lfloor n/2 \rfloor}-1}\in \Z_2$ such that $\gamma_1,\gamma_2$ are conjugate via $C^{b_1}_1C^{b_3}_3\cdots C^{b_{2{\lfloor n/2 \rfloor}-1}}_{2{\lfloor n/2 \rfloor}-1}C_{2{\lceil n/2 \rceil}-1},$
where $\lceil n/2 \rceil$ \label{pag:round up} denotes the smallest integer greater than or equal to $n/2$.
	
\end{lem}

Now we recall the intertwining elements in $\mHcn$.
\label{pag:example}
	Given $1\leq i<n$, one can define the intertwining element $\tilde{\Phi}_i$ \label{pag:BK intertwining element} in $\mHcn$ as follows:
\begin{align}\label{intertwinNon-dege}
\tilde{\Phi}_i:=z^2_i T_i+\epsilon\frac{z^2_i}{X_i X^{-1}_{i+1}-1}-\epsilon\frac{z^2_i}{X_i X_{i+1}-1}C_i C_{i+1},
\end{align}
where $z_i:= X_i+X^{-1}_i-X_{i+1}-X^{-1}_{i+1}= X^{-1}_i(X_iX_{i+1}-1)(X_iX^{-1}_{i+1}-1).$

These elements satisfy the following properties (cf. \cite[(3.7),Proposition 3.1]{JN} and \cite[(4.11)-(4.15)]{BK})
	\begin{align}
		\tilde{\Phi}^2_i&=z^2_i\bigl(z^2_i-\epsilon^2 (X^{-1}_i X^{-1}_{i+1}(X_i X_{i+1}-1)^2-X^{-1}_iX_{i+1}(X_i X^{-1}_{i+1}-1)^2)\bigr)\label{Sqinter},\\
		\tilde{\Phi}_i X^\pm_i&=X^\pm_{i+1}\tilde{\Phi}_i, \tilde{\Phi}_iX^\pm_{i+1}=X^\pm_i\tilde{\Phi}_i,
		\tilde{\Phi}_i X^\pm_l=X^\pm_l\tilde{\Phi}_i \label{Xinter},\\
		\tilde{\Phi}_i C_i&=C_{i+1}\tilde{\Phi}_i, \tilde{\Phi}_i C_{i+1}=C_i \tilde{\Phi}_i,
		\tilde{\Phi}_iC_l=C_l\tilde{\Phi}_i \label{Cinter},\\
		\tilde{\Phi}_j \tilde{\Phi}_i&=\tilde{\Phi}_i \tilde{\Phi}_j,
		\tilde{\Phi}_i\tilde{\Phi}_{i+1}\tilde{\Phi}_i=\tilde{\Phi}_{i+1}\tilde{\Phi}_i \tilde{\Phi}_{i+1}\label{Braidinter}
	\end{align}
	for all admissible $i,j,l$ with $l\neq i, i+1$ and $|j-i|>1$.
	Define $\tilde{\Phi}_{w}:=\tilde{\Phi}_{i_1}\cdots \tilde{\Phi}_{i_p}$ for
	$w\in \mathfrak{S}_n$ and $w=s_{i_1}\cdots s_{i_p}$ is a reduced expression.

Jones and Nazarov (\cite[(3.6)]{JN}) also introduced the intertwining elements of the following version, \label{pag:JN intertwining element}
	\begin{align}\label{universal-Phi}
\Phi_i:= z_{i}^{-2} \tilde{\Phi}_i&=T_{i}+\frac{\epsilon}{X_{i}X_{i+1}^{-1}-1}-\frac{\epsilon}{X_{i}X_{i+1}-1}C_{i}C_{i+1}\\
&\in \mathbb{K}(X_1,\ldots,X_n)\otimes_{\mathbb{K}[X_1^{\pm 1},\ldots, X_n^{\pm 1}]} \mHcn, \nonumber
    \end{align}
	for $i=1,\ldots,n-1.$
	For any $i=1,2,\ldots,n-1$ and $x,y \in \mathbb{K}^*$ satisfying $y\neq x^{\pm 1},$ let (\cite[(3.13)]{JN})
    \label{pag:Phi function}
	\begin{align}\label{Phi function}
		\Phi_i(x,y):=T_{i}+\frac{\epsilon}{x^{-1}y-1}-\frac{\epsilon}{xy-1}C_{i}C_{i+1} \in \mHcn.
	\end{align}

	For any pair of $(x,y)\in (\mathbb{K}^*)^2$ and $y\neq x^{\pm1}$, we consider the following idempotency condition on $(x,y)$
	\begin{align}\label{invertible}
		\frac{x^{-1}y}{(x^{-1}y-1)^2}+\frac{xy}{(xy-1)^2}=\frac{1}{\epsilon^2}.
	\end{align}
	Note that the equation \eqref{invertible} holds for the pair $(x,y)$ if and only if it holds for one of these four pairs $(x^{\pm 1},y^{\pm 1}).$

For any invertible $\iota \in \mathbb{K}$, we define
\label{pag:q-function and b-function}
\begin{align}\label{substitution0}
	\mathtt{q}(\iota):=2\frac{q\iota+(q\iota)^{-1}}{q+q^{-1}}, \quad \mathtt{b}_{\pm}(\iota):=\frac{\mathtt{q}(\iota)}{2}\pm \sqrt{\frac{\mathtt{q}(\iota)^2}{4}-1}.
\end{align}
Clearly, $\mathtt{b}_{\pm}(\iota)$ is the solution of equation $x+x^{-1}=\mathtt{q}(\iota),$ thus $\mathtt{b}_{+}(\iota)\mathtt{b}_{-}(\iota)=1.$
According to \cite{JN}, via the substitution
\begin{align}\label{substitute}
x+x^{-1}=2\frac{qu+q^{-1}u^{-1}}{q+q^{-1}}=\mathtt{q}(u),\qquad\qquad y+y^{-1}=2\frac{qv+q^{-1}v^{-1}}{q+q^{-1}}=\mathtt{q}(v)
\end{align}
the condition \eqref{invertible}  is  equivalent to the condition which states that $u,v$ satisfy one of the following four equations
\begin{align}\label{invertible2}
v=q^2u,\quad v=q^{-2}u,\quad v=u^{-1},\quad v=q^{-4}u^{-1}.
\end{align}

	\subsection{Cyclotomic quotients}\label{basic-cyc-Non-dege}
	
	To define the cyclotomic Hecke-Clifford algebra $\mHfcn$ over ${\rm R},$ \label{pag:CHCA} we fix $m\geq 0,$ $\underline{Q}=(Q_1,Q_2,\ldots,Q_m)\in({\rm R^\times})^m$ \label{pag:Q-parameters} and take a $f=f(X_1)\in {\rm R}[X_1^\pm]$ satisfying \cite[(3.2)]{BK}.
We use the Laurent polynomials in \eqref{cyc-even} and \eqref{cyc-odd}, and also the third form considered in \cite{SW}:
\begin{align}\label{cyc-two-special-factors}
 f(X_1)&=(X_1-1)(X_1+1)\prod_{i=1}^m
 \bigl(X_1+X_1^{-1}-\mathtt q(Q_i)\bigr),& r&=2m+2.
\end{align}
In each case, $r$ is the degree of $X_1^m f(X_1)$, namely $2m$, $2m+1$ or $2m+2$. Since $X_1$ is invertible, this ordinary representative generates the same ideal as $f$.
	
	The non-degenerate cyclotomic Hecke-Clifford algebra $\mHfcn$ is defined as $$\mHfcn:=\mHcn/\mathcal{I}_f,
	$$ where $\mathcal{I}_f$ is the two-sided ideal of $\mHcn$ generated by $f(X_1)$. The integer $r$ \label{pag:nondege level} just specified is called the level of $\mHfcn.$ We shall denote the images of $X^{\alpha}, C^{\beta}, T_w$ in the cyclotomic quotient $\mHfcn$ still by the same symbols. Then we have the following due to \cite{BK}.
	
	\begin{lem}\cite[Theorem 3.6]{BK}
		The set $\{X^{\alpha}C^{\beta}T_w~|~ \alpha\in\{0,1,\cdots,r-1\}^n,
		\beta\in\mathbb{Z}_2^n, w\in {\mathfrak{S}_n}\}$ forms an ${\rm R}$-basis of $\mHfcn$.
	\end{lem}

\subsection{Partitions, tableaux and separation conditions}
\label{pag:The types of combinatorics}
The factor presentation of $f$ determines which components of a shape are ordinary partitions and which are strict. Let $\mathscr P_n$ be the partitions of $n$, let $\mathscr P_n^m$ be the $m$-multipartitions, and let $\mathscr{SP}_n$ be the strict partitions. Write $\ell(\mu)$ for the number of nonzero parts of $\mu\in\mathscr P_n$. We retain the empty-component convention for $m=0$. For the fixed presentation of $f$, set
\[
\mathscr P_n^f:=\begin{cases}
 \mathscr P_n^m,& f\text{ as in }\eqref{cyc-even},\\
 \displaystyle\bigsqcup_{a+b=n}(\mathscr{SP}_a\times\mathscr P_b^m),& f\text{ as in }\eqref{cyc-odd},\\
 \displaystyle\bigsqcup_{a+b+c=n}(\mathscr{SP}_a\times\mathscr{SP}_b\times\mathscr P_c^m),& f\text{ as in }\eqref{cyc-two-special-factors}.
\end{cases}
\]
\label{pag:multipartition}
We write these shapes, respectively, as
\[
 (\lambda^{(1)},\ldots,\lambda^{(m)}),\quad
 (\lambda^{(0)},\lambda^{(1)},\ldots,\lambda^{(m)}),\quad
 (\lambda^{(0_-)},\lambda^{(0_+)},\lambda^{(1)},\ldots,\lambda^{(m)}).
\]
Thus the strict components retain labels $0$, or $0_-,0_+$ in that order. We identify ordinary partitions with Young diagrams and strict partitions with shifted Young diagrams. A box in row $i$, column $j$ of component $l$ is denoted by $(i,j,l)$; the component labels are those in the corresponding display above.
We also use the notation $\alpha=(i,j,l)\in \undla$ if the diagram of $\undla$ has a box $\alpha$ on the $l$-th component of row $i$ and column $j$. We use $\Std(\undla)$ \label{pag:standard tableaux} to denote the set of standard tableaux of shape $\undla$. One can also regard each $\mathfrak{t}\in\Std(\undla)$ as a bijection $\mathfrak{t}:\undla\rightarrow \{1,2,\ldots, n\}$ satisfying $\mathfrak{t}((i,j,l))=k$ if the box occupied by $k$ is located in the $i$-th row, $j$-th column in the $l$-th component $\lambda^{(l)}$.
For $0\leq k\leq n,$ let $\mt\downarrow_{k}$ be the subtableau of $\mt$ that contains the numbers $\{1, 2,\dots,k\}$.
In particular, $\mt\downarrow_{0}$ is the empty tableau.
We use $\mathfrak{t}^{\undla}$ (resp. $\mathfrak{t}_{\undla}$) to denote the standard tableaux obtained by inserting the symbols $1,2,\ldots,n$ consecutively by rows (resp. columns) from the first (resp. last) component of $\undla$.

We use $\Add(\undla)$ and $\Rem(\undla)$ to denote the set of addable boxes of $\undla$ and the set of removable boxes of $\undla$ respectively. For $\mt\in\Std(\undla),$ we define $\Add(\mt):=\Add(\undla)$ and $\Rem(\mt):=\Rem(\undla).$ \label{pag:Add and Rem}

\label{pag:diag of undlam}
\begin{defn}(\cite[Definition 2.5]{SW})
	Let $\undla\in\mathscr{P}^{f}_{n}$.  We define
	$$
	\mathcal{D}_{\undla}:=\{(a,a,l)\in\undla\mid l\text{ labels a strict component}\}
	$$
	For any $\mathfrak{t}\in\Std(\undla), $ we define
	\begin{align}\label{Dt}
		\mathcal{D}_{\mt}:=\{\mathfrak{t}(a,a,l)|(a,a,l)\in\mathcal{D}_{\undla}\}.
	\end{align}
\end{defn}

\begin{example} Take $m=1$ and $f$ as in \eqref{cyc-odd}. Let $\undla=(\lambda^{(0)},\lambda^{(1)})\in \mathscr{P}^{f}_{5}$, where the strict and ordinary components have diagrams
	$$
	\lambda^{(0)}=\young(\,\, ,:\,), \lambda^{(1)}=\young(\,,\,).
	$$
	Then
	$$
	\mathfrak{t}^{\undla}=\Biggl(\young(12,:3), \young(4,5)\Biggr).
	$$
	and  an example of standard tableau is as follows:
	$$
	\mathfrak{t}=\Biggl(\young(13,:5), \young(2,4)\Biggr)\in \Std(\undla).
	$$ We have $$\mathcal{D}_{\undla}=\{(1,1,0),(2,2,0)\},\qquad \mathcal{D}_{\mathfrak{t}}=\{1,5\}.
	$$
\end{example}

\label{pag:adimssible}
The symmetric group $\mathfrak{S}_n$ acts on the set of tableaux of shape $\underline{\lambda}$ in the natural way.  For any standard tableau $\mathfrak{t}\in \Std(\undla)$, if $s_l \mathfrak{t}$ is still standard, we call the simple transposition $s_l$ {\bf admissible} with respect to $\mathfrak{t}$.

\begin{lem}\label{admissible transposes}\cite[Lemma 2.8]{SW}
	Let $\undla\in\mathscr{P}^{f}_{n}$. For any $\ms,\mt \in \Std(\undla),$ we denote by $d(\ms,\mt)\in \mathfrak{S}_n$ the unique element
	such that $\ms=d(\ms,\mt)\mt.$ Then we have
	$$s_{k_i}\text{ is admissible with respect to } s_{k_{i-1}}\ldots s_{k_1}\mathfrak{t},\quad \forall i=1,2,\ldots,p$$
	for any reduced expression $d(\ms,\mt)=s_{k_p}\cdots s_{k_1}.$
\end{lem}

For a finite set $S$, we denote by $\sharp S$ the number of elements in $S$. \label{pag:cardinality}
{\bf In the rest of this subsection, we fix $\undla\in\mathscr{P}^{f}_{n}$.
Let $t:=\sharp \mathcal{D}_{\undla}$.}
\begin{defn}
	We denote \begin{align}
		\mathcal{D}_{\mt^{\undla}}&:=\{\mt^{\undla}(a,a,l)\mid (a,a,l)\in\mathcal{D}_{\undla}\}=\{i_1<i_2<\cdots<i_t\}\label{stanard D},\\
		\mathcal{OD}_{\mt^{\undla}}&:=\{i_1,i_3,\cdots,i_{2{\lceil t/2 \rceil}-1}\}\subset \mathcal{D}_{\mt^{\undla}}\label{standard OD}
	\end{align} and \label{pag:dundla}
	\begin{align*}
		d_{\undla}:=
		\begin{cases}
			1, & \text{ if $t$ is odd}, \\
			0, & \text{ otherwise}.
	\end{cases}
\end{align*}
\end{defn}
\label{pag:Dt,ODt,Z2ODt}		
\begin{defn}\label{Dt,ODt,Z2ODt}
  For each $\mt\in \Std(\undla),$ we define
	\begin{align*}
		\mathcal{D}_{\mt}&:=d(\mt,\mt^{\undla}) (\mathcal{D}_{\mt^{\undla}}),\nonumber\\
		\mathcal{OD}_{\mt}&:=d(\mt,\mt^{\undla}) (\mathcal{OD}_{\mt^{\undla}}),\nonumber\\
		\Z_2(\mathcal{OD}_{\mt})&:=\{\alpha_{\mt} \in \mathbb{Z}_{2}^{n} \mid \supp(\alpha_{\mt}) \subseteq \mathcal{OD}_{\mt}\},\\
		\Z_2([n]\setminus \mathcal{D}_{\mt})&:=\{\beta_{\mt} \in \mathbb{Z}_{2}^{n} \mid \supp(\beta_{\mt}) \subseteq [n]\setminus \mathcal{D}_{\mt}\}.\nonumber
	\end{align*}
\end{defn}

\label{pag:decomposition of OD}
\begin{defn}\label{decomposition of OD}
	For any $\mt\in\Std(\undla),$ let ${d(\mt,\mt^{\undla})}\in \mathfrak{S}_n$ such that $\mt={d(\mt,\mt^{\undla})}\mt^{\undla}$. We define
	\begin{align}
		\mathbb{Z}_2(\mathcal{OD}_{\mt})_{\bar{0}}:=\{\alpha_{\mt} \in\Z_2(\mathcal{OD}_{\mt}) \mid d(\mt,\mt^{\undla})(i_t)\notin \supp(\alpha_{\mt})\},\nonumber\\
		\mathbb{Z}_2(\mathcal{OD}_{\mt})_{\bar{1}}:=\{\alpha_{\mt} \in\Z_2(\mathcal{OD}_{\mt}) \mid d(\mt,\mt^{\undla})(i_t)\in \supp(\alpha_{\mt})\}.\nonumber
	\end{align}
\end{defn}

That is, if $d_{\undla}=0$ (i.e., $t$ is even), then $\Z_2(\mathcal{OD}_{\mt})_{\bar{0}}=\Z_2(\mathcal{OD}_{\mt})$ and $\Z_2(\mathcal{OD}_{\mt})_{\bar{1}}=\emptyset;$  if $d_{\undla}=1$ (i.e., $t$ is odd), then $\Z_2(\mathcal{OD}_{\mt})_{\bar{0}}$ and $\Z_2(\mathcal{OD}_{\mt})_{\bar{1}}$ are both non-empty and there is a natural bijection between $\Z_2(\mathcal{OD}_{\mt})_{\bar{0}}$ and $\Z_2(\mathcal{OD}_{\mt})_{\bar{1}}$ which sends $\alpha_{\mt}\in \Z_2(\mathcal{OD}_{\mt})_{\bar{0}}$ to $\alpha_{\mt}+e_{{d(\mt,\mt^{\undla})}(i_t)}\in \Z_2(\mathcal{OD}_{\mt})_{\bar{1}}.$ In particular, we have
$$ \Z_2(\mathcal{OD}_{\mt})=\Z_2(\mathcal{OD}_{\mt})_{\bar{0}}\sqcup \Z_2(\mathcal{OD}_{\mt})_{\bar{1}}.$$ For any $\alpha_{\mt} \in \Z_2(\mathcal{OD}_{\mt})_{\bar{0}},$ we also use $\alpha_{\mt,\bar{0}}=\alpha_{\mt}$  to emphasize that $\alpha_{\mt}\in Z_2(\mathcal{OD}_{\mt})_{\bar{0}}$ and if $d_{\undla}=1$, we define $\alpha_{\mt,\bar{1}}:=\alpha_{\mt}+e_{{d(\mt,\mt^{\undla})}(i_t)}\in \Z_2(\mathcal{OD}_{\mt})_{\bar{1}}.$

For $\undla\in\mathscr{P}^{f}_{n}$, we associate $\mt\in\Std(\undla)$
with a certain idempotent $\gamma_{\mt}\in\mathcal{C}_n,$ \label{pag:gamma t} which will be used in the rest of this paper.
\begin{align}\label{gamma t}
\gamma_{\mt}:=2^{-\lfloor t/2 \rfloor}\cdot\overrightarrow{\prod_{k=1,\cdots,{\lfloor t/2 \rfloor}}}\biggl(1+ \sqrt{-1}C_{{d(\mt,\mt^{\undla})}(i_{2k-1})}C_{{d(\mt,\mt^{\undla})}(i_{2k})}\biggr)\in\mathcal{C}_n,
\end{align}
where $t=\sharp \mathcal{D}_{\undla}.$
We set $Q_0=Q_{0_+}=1, Q_{0_-}=-1$. Now we can define the residues of boxes.
\begin{defn}
	Suppose  $\undla\in\mathscr{P}^{f}_{n}$ and $(i,j,l)\in \undla,$ we define the residue of box $(i,j,l)$ with respect to the parameter $\undQ=(Q_1,Q_2,\ldots,Q_m)\in({\rm R^\times})^m$ as follows:
	\label{pag:nondeg residue}
	\begin{equation}\label{eq:residue}
		\res(i,j,l):=Q_lq^{2(j-i)}\in {\rm R^\times}.
	\end{equation}
	If $\mathfrak{t}\in \Std(\undla)$ and $\mathfrak{t}(i,j,l)=k$, we set
	\begin{align}
		\res_\mathfrak{t}(k)&:=Q_lq^{2(j-i)}\in {\rm R^\times},\label{resNon-dege-1}\\
		\res(\mathfrak{t})&:=(\res_\mathfrak{t}(1),\cdots,\res_\mathfrak{t}(n))\in ({\rm R^\times})^m,\label{resNon-dege-2}
	\end{align}
	then the $\mathtt{q}$-sequence of $\mt$ is
	\begin{align}\label{resNon-dege-3}
		\mathtt{q}(\res(\mathfrak{t})):=(\mathtt{q}(\res_{\mathfrak{t}}(1)), \mathtt{q}(\res_{\mathfrak{t}}(2)),\ldots, \mathtt{q}(\res_{\mathfrak{t}}(n)))\in {\rm R}^m.
	\end{align}
\end{defn}

Recall $[n]=\{1,2,\ldots,n\}$ for $n\in\mathbb{N}.$ In the rest of this subsection, we assume that ${\rm R}=\mathbb{K}$ is an algebraically closed field of characteristic different from $2$. We will recall the separation condition \cite[Definition 3.9]{SW} on the choice of the parameters $\underline{Q}$ and $f$, where $r$ is the level.
\begin{defn}\label{defn:separate}\cite[Definition 3.9]{SW}
Let $\undQ=(Q_1,\ldots,Q_m)\in(\mathbb{K}^*)^m$.  Assume $\undla\in\mathscr{P}^{f}_{n}$. Then $(q,\undQ)$ is said to be {\em separate} with respect to $\undla$ if for any $\mathfrak{t}\in \Std(\undla)$, the $\mathtt{q}$-sequence for $\mathfrak{t}$ defined via \eqref{resNon-dege-3} satisfies the following condition:
$$
\mathtt{q}(\res_{\mathfrak{t}}(k))\neq\mathtt{q}(\res_{\mathfrak{t}}(k+1)) \text{ for any } k=1,\ldots,n-1.
$$
\end{defn}

We write $P_n^f(q^2,\undQ)$ for the explicit Laurent polynomial in \cite[Proposition 3.11]{SW} corresponding to the fixed presentation of $f$.\label{pag:nondege Pioncare poly} We use its separation criterion below, rather than its full expression. In particular, the condition concerns all shapes of size $n+1$, not just size $n$.
\begin{prop}\label{separate formula}\cite[Proposition 3.11]{SW}
Let $n\geq 1,\,m\geq 0$,  $\undQ=(Q_1,\ldots,Q_m)\in(\mathbb{K}^*)^m$. Then $(q,\undQ)$ is separate with respect to $\underline{\mu}$ for any $\underline{\mu}\in\mathscr{P}^{f}_{n+1}$ if and only if $P_n^{f}(q^2,\undQ)\neq 0$.
\end{prop}

\begin{lem}\cite[Lemma 2.7]{LS2}\label{important condition1}
		Let $\undQ=(Q_1,\ldots,Q_m)\in(\mathbb{K}^*)^m$. Suppose $P_n^{f}(q^2,\undQ)\neq 0$.
		Then
	\begin{enumerate}
		\item
		 For any $\undla\in\mathscr{P}^{f}_{n},$ $\mathfrak{t}\in\Std(\undla)$, we have  $\mathtt{b}_{\pm}(\res_{\mathfrak{t}}(k))\notin \{\pm 1\}$ for $k\notin \mathcal{D}_{\mathfrak{t}}$;
		\item For any $\undla,\underline{\mu} \in\mathscr{P}^{f}_{n},$ $\mt\in\Std(\undla), \mt' \in \Std(\underline{\mu}),$
if $\mt\neq \mt',$ then we have $\mathtt{q}(\res(\mt))\neq \mathtt{q}(\res(\mt'))$;
		\item For any $\undla\in\mathscr{P}^{f}_{n},$ $\mathfrak{t}\in\Std(\undla)$ and $k\in [n-1]$, the four pairs $(\mathtt{b}_{\pm}(\res_{\mathfrak{t}}(k)),\mathtt{b}_{\pm}(\res_{\mathfrak{t}}(k+1)))$ do not satisfy \eqref{invertible} if $k,k+1$ are not in the adjacent diagonals of $\mathfrak{t}$.
	\end{enumerate}
	\end{lem}
	The following is key in the computation of Frobenius form $\tau_{r,n}$ on seminormal basis.
	\begin{prop}\label{compare eigenvalue}
			Let $\undQ=(Q_1,\ldots,Q_m)\in(\mathbb{K}^*)^m$. Suppose $P_n^{f}(q^2,\undQ)\neq 0$ and  $\undla\in\mathscr{P}^{f}_{n-1}$. Then
				\begin{enumerate}
					\item for any $\gamma_1\neq \gamma_2\in \Rem(\undla)$, we have $\mathtt{q}(\res(\gamma_1))\neq \mathtt{q}(\res(\gamma_2))$;
					\item for any $\gamma_1\in \Add(\undla),\,\gamma_2\in \Rem(\undla)$, we have $\mathtt{q}(\res(\gamma_1))\neq \mathtt{q}(\res(\gamma_2))$;
					\item for any $\gamma_1\neq \gamma_2\in \Add(\undla)$, we have $\mathtt{q}(\res(\gamma_1))\neq \mathtt{q}(\res(\gamma_2))$.
				\end{enumerate}
		\end{prop}
		\begin{proof}
			We only prove (3), since the proofs of (1) and (2) are similar. We fix any $\mt\in\Std(\undla)$. Note that $\gamma_2\in\Add(\undla\cup\{\gamma_1\})$. Hence there exists $\ms\in\Std(\undla\cup\{\gamma_1,\,\gamma_2\})$, such that $\ms\downarrow_{n-1}=\mt$ and $\ms^{-1}(n)=\gamma_1,\,\ms^{-1}(n+1)=\gamma_2$. Since $P_n^{f}(q^2,\undQ)\neq 0$, we can apply the separation condition to see that $\mathtt{q}(\res(\gamma_1))\neq \mathtt{q}(\res(\gamma_2))$.
			\end{proof}

When $q,Q_1,\ldots,Q_m$ are algebraically independent over $\Z$, and $\mathbb{K}$ is the algebraic closure of $\mathbb{Q}(q,Q_1,\ldots,Q_m)$, i.e., for the generic cyclotomic Hecke-Clifford algebra, the condition $P_n^{f}(q^2,\undQ)\neq 0$ clearly holds.

Suppose that the condition $P_n^{f}(q^2,\undQ)\neq 0$ holds in $\mathbb{K}.$ Then for each $\undla\in\mathscr{P}^{f}_{n},$ we can associate $\undla$ with a simple $\mHfcn$-module $\mathbb{D}(\undla),$ see
\cite[Theorem 4.5]{SW}. Furthermore, we have the following result.
\label{pag:nondege simple module}
	\begin{thm}\label{semisimple:non-dege}\cite[Theorem 4.10]{SW}
	Let $\undQ=(Q_1,Q_2,\ldots,Q_m)\in(\mathbb{K}^*)^m$.  Assume $P_n^{f}(q^2,\undQ)\neq 0$. Then $\mHfcn$ is a (split) semisimple algebra and
	$$
	\{\mathbb{D}(\undla)\mid \undla\in\mathscr{P}^{f}_{n}\}$$ forms a complete set of pairwise non-isomorphic irreducible $\mHfcn$-modules. Moreover,  $\mathbb{D}(\undla)$ is of type $\texttt{M}$ if and only if $\sharp \mathcal{D}_{\undla}$ is even and is of type  $\texttt{Q}$ if and only if $\sharp \mathcal{D}_{\undla}$  is odd.
	\end{thm}			
	In the next Section, we shall introduce $\mathbb{D}(\undla)$ in detail, with a basis given, see Proposition \ref{actions of generators on L basis}.
    By Theorem \ref{semisimple:non-dege}, we have the following $\mHfcn$-module isomorphism:
	$$\mHfcn\cong\bigoplus_{\undla\in\mathscr{P}^{f}_{n}}\mathbb{D}(\undla)^{\oplus 2^{n-\bigl\lceil\frac{\sharp\mathcal{D}_{\mt^{\undla}}}{2}\bigr\rceil}\sharp\Std(\undla) }=\bigoplus_{\undla\in\mathscr{P}^{f}_{n}}\mathbb{D}(\undla)^{\oplus 2^{n-\sharp\left(\mathcal{OD}_{\mt^{\undla}}\right)}\sharp\Std(\undla) }.$$
	So the block decomposition is
	$$\mHfcn=\bigoplus_{\undla \in \mathscr{P}^{f}_{n}} B_{\undla},$$ and for each $\undla \in \mathscr{P}^{f}_{n}$,
we have
$$B_{\undla}
\cong \mathbb{D}(\undla)^{\oplus 2^{n-\sharp\left(\mathcal{OD}_{\mt^{\undla}}\right)}\sharp\Std(\undla) }$$
as $B_{\undla}$-modules.

\section{Seminormal forms}\label{Seminormal form}
The trace calculations will use primitive idempotents and the multiplication of seminormal basis elements. We recall their construction from the simple modules, and then express the relevant coefficients as products over boxes. Except in the last subsection, we fix $\undQ\in(\mathbb K^*)^m$ and $f$ with $P_n^f(q^2,\undQ)\ne0$, using the residues in \eqref{eq:residue}. The last subsection gives the degenerate counterparts.
\subsection{Simple modules} \label{Non-dege-simplemodule}
In this subsection, we fix $\undla\in\mathscr{P}^{f}_{n}$.

\label{pag:nubetak}
\begin{defn}	
Let $\beta=(\beta_1,\ldots,\beta_n) \in \mathbb{Z}_{2}^{n}.$
\begin{enumerate}
\item For $k\in [n],$ we define
\begin{align*}
\nu_{\beta}(k):=
		\begin{cases}
			-1, & \text{ if } \beta_k=\bar{1}, \\
			1, & \text{ if } \beta_k=\bar{0},
		\end{cases}
\qquad
\delta_{\beta}(k):=\frac{1-\nu_{\beta}(k)}{2}=
			\begin{cases}
				1, & \text{ if } \beta_k=\bar{1}, \\
				0, & \text{ if } \beta_k=\bar{0}.
			\end{cases}
\end{align*}
\label{pag:deltabetak}

\item For $0\leq i \leq n+1,$ we define
      $$\quad |\beta|_{<i}:=\sum_{1\leq k <i}\beta_k,\quad |\beta|:=|\beta|_{<n+1}.$$
      Similarly, we can also define $|\beta|_{\leq i},$ $|\beta|_{> i}$ and $|\beta|_{\geq i}.$
\end{enumerate}
\end{defn}

\begin{defn}  For any $i\in [n], \mt\in \Std(\undla),$ we define
             \label{pag:nondege eigenvalues}
			 $$\mathtt{b}_{\mt,i}:=\mathtt{b}_{-}(\res_{\mt}(i))\in\mathbb{K}^*.$$
            For any $i\in [n-1]$, we define
			 \begin{align}
			 	\delta(s_i\mt)
			 	:=\begin{cases}
			 		1, & \text{ if } s_i\mt \in \Std(\undla), \\
			 		0, & \text{ otherwise}.\nonumber
			 	\end{cases}
			 \end{align} and
            \label{pag:nondege coeffi cti}
			\begin{align}\label{nondege coeffi cti}
				\mathtt{c}_{\mt}(i):=1-\epsilon^2 \biggl(\frac{\mathtt{b}_{\mt,i}^{-1}\mathtt{b}_{\mt,i+1}}{(\mathtt{b}_{\mt,i}^{-1}\mathtt{b}_{\mt,i+1}-1)^2}
				+\frac{\mathtt{b}_{\mt,i}\mathtt{b}_{\mt,i+1}}{(\mathtt{b}_{\mt,i}\mathtt{b}_{\mt,i+1}-1)^2}\biggr)\in \mathbb{K}.
			\end{align}
			\end{defn}
			Since $\mt \in \Std(\undla),$ $\mathtt{b}_{\mt,i}\neq \mathtt{b}_{\mt,i+1}^{\pm 1}$ by Definition \ref{defn:separate} and Proposition \ref{separate formula}, which immediately implies that $\mathtt{c}_{\mt}(i)$ is well-defined. If $s_i$ is admissible with respect to $\mt$, i.e., $\delta(s_i\mt)=1$, then $\mathtt{c}_{\mt}(i)\in \mathbb{K}^{*}$ by
		the third part of Lemma \ref{important condition1}. It is clear that $\mathtt{c}_{\mt}(i)=\mathtt{c}_{s_i\mt}(i).$	
		
\begin{prop}\cite[Theorem 4.8]{LS2}\label{actions of generators on L basis}
The simple $\mHfcn$-supermodule $\mathbb{D}(\undla)$ has a $\mathbb{K}$-basis of the form
$$\bigsqcup_{\mt\in \Std(\undla)}\Biggl\{C^{\beta_{\mt}}C^{\alpha_{\mt}}v_{\mt}\biggm|\begin{matrix}\beta_{\mt} \in \Z_2([n]\setminus \mathcal{D}_{\mt})  \\
					\alpha_{\mt}\in \Z_2(\mathcal{OD}_{\mt})
					\end{matrix}\Biggr\},$$
where the parity $|v_{\mt}|=\bar{0},$ and the actions of $X_i$ and $C_j$ $(i,j\in[n])$ on $\mathbb D(\undla)$ are given in the following.
					
			Let $\mt\in\Std(\undla),\,\beta_{\mt}\in\Z_2([n]\setminus \mathcal{D}_{\mt})$ and $\alpha_{\mt}\in \Z_2(\mathcal{OD}_{\mt})$.
			\begin{enumerate}
				\item For each $i\in [n],$ we have
				\begin{align}\label{X eigenvalues}
					X_i\cdot C^{\beta_{\mt}}C^{\alpha_{\mt}}v_{\mt}
					=\mathtt{b}_{\mt,i}^{-\nu_{\beta_{\mt}}(i)}C^{\beta_{\mt}}C^{\alpha_{\mt}}v_{\mt}.
				\end{align}
				\item We have $\gamma_{\mt} v_{\mt}=v_{\mt}.$
				\item For each $i\in [n],$ we have
				\begin{align}
					&C_i\cdot C^{\beta_{\mt}}C^{\alpha_{\mt}}v_{\mt} \label{Caction}\\
					&=\begin{cases}
						(-1)^{|\beta_{\mt}|_{<i}} C^{\beta_{\mt}+e_i}C^{\alpha_{\mt}}v_{\mt}, & \text{ if } i\in [n]\setminus \mathcal{D}_{\mt}, \\
						(-1)^{|\beta_{\mt}|+|\alpha_{\mt}|_{<i}} C^{\beta_{\mt}}C^{\alpha_{\mt}+e_{i}}v_{\mt}, & \text{ if  $i={d(\mt,\mt^{\undla})}(i_p) \in \mathcal{D}_{\mt}$ ,where $p$ is odd},\\
								(-\sqrt{-1})(-1)^{|\beta_{\mt}|+|\alpha_{\mt}|_{\leq {d(\mt,\mt^{\undla})}(i_{p-1})}}&\\
							\qquad\qquad	\cdot C^{\beta_{\mt}}C^{\alpha_{\mt}+e_{{d(\mt,\mt^{\undla})}(i_{p-1})}}v_{\mt}, &\text{ if  $i={d(\mt,\mt^{\undla})}(i_p) \in \mathcal{D}_{\mt}$ ,where $p$ is even}.\nonumber
					\end{cases}
				\end{align}
			
			\end{enumerate}
		\end{prop}

	\subsection{Primitive idempotents of $\mHfcn$}\label{primitiveidem}
The tableau and Clifford choices in a simple-module basis also index the primitive idempotents. We first collect these choices into triples and describe how they restrict when the last entry is removed.
    \label{pag:Tri}
	\begin{defn}\label{Tri}
For $a\in \Z_2,\,\undla\in\mathscr{P}^{f}_{n}$,  we define
		$${\rm Tri}_{a}(\undla):=\bigsqcup_{\mt\in \Std(\undla)}\{\mt \}\times  \Z_2(\mathcal{OD}_{\mt})_{a}\times  \Z_2([n]\setminus \mathcal{D}_{\mt}),$$
		and $${\rm Tri}(\undla):={\rm Tri}_{\bar{0}}(\undla)\sqcup {\rm Tri}_{\bar{1}}(\undla).$$
		\end{defn}

		Notice that ${\rm Tri}(\undla)={\rm Tri}_{0}(\undla)$ when $d_{\undla}=0.$ For any ${\rm T}=(\mt, \alpha_{\mt}, \beta_{\mt})\in {\rm Tri}_{\bar{0}}(\undla),$ we denote
		$${\rm T}_{a}=(\mt, \alpha_{\mt,a}, \beta_{\mt})\in {\rm Tri}_{a}(\undla),\quad a\in \mathbb{Z}_{2},$$
		when $d_{\undla}=1.$

\begin{defn}
	Let $\undla\in\mathscr{P}^{f}_{n}$ and ${\rm T}=(\mt,\alpha_{\mt},\beta_{\mt})\in{\rm Tri}(\undla).$
	We define $\alpha_{\mt\downarrow_{n-1}}\in \mathbb{Z}_2(\mathcal{OD}_{\mt\downarrow_{n-1}})_{\overline{0}},$ such that $\alpha_{\mt}\downarrow_{n-1}=\alpha_{\mt\downarrow_{n-1},a},$ for some $a\in\Z_2.$ Suppose $\mt\downarrow_{n-1}\in \Std(\underline{\mu})$
for some $\underline{\mu}\in\mathscr{P}^{f}_{n-1}$. We define
	$${\rm T}\downarrow_{n-1}:=(\mt\downarrow_{n-1},\alpha_{\mt\downarrow_{n-1}},\beta_{\mt}\downarrow_{n-1})\in{\rm Tri}_{\overline{0}}(\underline{\mu}).$$
	Then $$({\rm T}\downarrow_{n-1})_a=(\mt\downarrow_{n-1},\alpha_{\mt\downarrow_{n-1},a},\beta_{\mt}\downarrow_{n-1})=(\mt\downarrow_{n-1},\alpha_{\mt}\downarrow_{n-1},\beta_{\mt}\downarrow_{n-1})\in{\rm Tri}_{a}(\underline{\mu}).$$
\end{defn}

Hence for any $\undla\in\mathscr{P}^{f}_{n}$, we have the map:
\begin{align*}
	\downarrow_{n-1}:	{\rm Tri}(\undla)&\rightarrow \bigsqcup_{\substack{\underline{\mu}\in\mathscr{P}^{f}_{n-1}\\\underline{\mu}\subset \undla}}{\rm Tri}_{\bar{0}}(\underline{\mu})\\
	{\rm T}&\mapsto {\rm T}\downarrow_{n-1}.
\end{align*}
In particular, let ${\rm T}=(\mt,\alpha_{\mt},\beta_{\mt})\in{\rm Tri}(\undla)$. Suppose $\mt\downarrow_{n-1}\in \Std(\underline{\mu})$ and $d_{\underline{\mu}}=0,$ we have $${\rm T}\downarrow_{n-1}=(\mt\downarrow_{n-1},\alpha_{\mt}\downarrow_{n-1},\beta_{\mt}\downarrow_{n-1})\in{\rm Tri}_{\overline{0}}(\underline{\mu}).$$

\begin{example}
Take $m=1$ and $f$ as in \eqref{cyc-odd}. Let $\undla=(\lambda^{(0)},\lambda^{(1)})\in \mathscr{P}^{f}_{5}$, with diagrams
$$
\lambda^{(0)}=\young(\,\, ,:\,),\qquad \lambda^{(1)}=\young(\,,\,).
$$
Then
$$
\mathfrak{t}=\Biggl(\young(12,:5), \young(3,4)\Biggr)\in\Std(\undla).
$$
	We have $\alpha_{\mathfrak{t}}=\left(\bar{1},\bar{0},\bar{0},\bar{0},\bar{0}\right)\in\Z_2(\mathcal{OD}_{\mathfrak{t}})_{\bar{1}}$ and $\beta_{\mathfrak{t}}=\left(\bar{0},\bar{1},\bar{0},\bar{0},\bar{0}\right)\in	\Z_2([n]\setminus \mathcal{D}_{\mathfrak{t}}).$ Then ${\rm T}=(\mathfrak{t},\alpha_{\mathfrak{t}},\beta_{\mathfrak{t}})\in {\rm Tri}_{\bar{1}}(\undla)$ and ${\rm T}\downarrow_{4}=(\mathfrak{t}\downarrow_{4},\alpha_{\mathfrak{t}\downarrow_{4}},\beta_{\mathfrak{t}}\downarrow_{4})\in{\rm Tri}_{\overline{0}}(\underline{\mu}),$ where $\undmu=\left(\young(\,\,),\young(\,,\,)\right)\in \mathscr{P}^{f}_{4}$, $\mathfrak{t}\downarrow_{4}=\Biggl(\young(12), \young(3,4)\Biggr)\in\Std(\undmu),$ $\alpha_{\mathfrak{t}\downarrow_{4}}=\left(\bar{0},\bar{0},\bar{0},\bar{0}\right)\in\Z_2(\mathcal{OD}_{\mathfrak{t}\downarrow_{4}})_{\bar{0}}$ and $\beta_{\mathfrak{t}}\downarrow_{4}=\left(\bar{0},\bar{1},\bar{0},\bar{0}\right)$.
\end{example}

Now we can define the primitive idempotents.
\label{pag:primitive idempotents and blocks}
		\begin{defn}\label{primitive idempotents and blocks}
          For $k\in[n]$, let
         $$\mathtt{B}(k):=\{ \mathtt{b}_{\pm}(\res_{\ms}(k)) \mid \ms \in \Std(\mathscr{P}^{f}_{n}) \}.$$
			For any ${\rm T}=(\mt, \alpha_{\mt}, \beta_{\mt})\in {\rm Tri}_{\bar{0}}(\undla),$ we define
			\begin{align}\label{definition of primitive idempotents. non-dege}
				F_{\rm T}:=\left(C^{\alpha_{\mt}}\gamma_{\mt}(C^{\alpha_{\mt}})^{-1} \right) \cdot \left(\prod_{k=1}^{n}\prod_{\mathtt{b}\in \mathtt{B}(k)\atop \mathtt{b}\neq \mathtt{b}_{+}(\res_{\mt}(k))}\frac{X_k^{\nu_{\beta_{\mt}}(k)}-\mathtt{b}}{\mathtt{b}_{+}(\res_{\mt}(k))-\mathtt{b}}\right)\in \mHfcn.
			\end{align}
We define
\begin{align}
				F_{\undla}&:=\sum_{{\rm T}\in {\rm Tri}_{\bar{0}}(\undla)} F_{\rm T}.
			\end{align}

		\end{defn}

\begin{defn}
For $a\in \Z_2,$ we denote
	$${\rm Tri}_{a}(\mathscr{P}^{f}_{n}):=\bigsqcup_{\undla\in \mathscr{P}^{f}_{n}}{\rm Tri}_{a}(\undla),$$
	and $${\rm Tri}(\mathscr{P}^{f}_{n}):={\rm Tri}_{\bar{0}}(\mathscr{P}^{f}_{n})\sqcup {\rm Tri}_{\bar{1}}(\mathscr{P}^{f}_{n}).$$	
\end{defn}
	
\begin{lem}\label{idempotent action. non-dege}\cite[Lemma 4.14]{LS2}	
Let ${\rm T}=(\mt, \alpha_{\mt}, \beta_{\mt})\in {\rm Tri}_{\bar{0}}(\undla),$
and ${\rm S}=(\ms, \alpha_{\ms}', \beta_{\ms}')\in {\rm Tri}(\mathscr{P}^{f}_{n})$. We have
\begin{align}\label{F. eq3}
F_{\rm T}\cdot C^{\beta_{\ms}^{'}} C^{\alpha_{\ms}^{'}} v_{\ms}=
\begin{cases}
C^{\beta_{\ms}^{'}} C^{\alpha_{\ms}^{'}}  v_{\ms}, & \text{ if } d_{\undla}=0 \text{ and } {\rm S}={\rm T}, \\
C^{\beta_{\ms}^{'}} C^{\alpha_{\ms}^{'}}  v_{\ms}, & \text{ if $d_{\undla}=1$ and } {\rm S}={\rm T}_{a}\text{ for some $a\in \mathbb{Z}_2,$ } \\
0, & \text{ otherwise.}
\end{cases}
\end{align}
\end{lem}
			
		\begin{thm}\cite[Theorem 4.16]{LS2}\label{primitive idempotents}
			Suppose $P_n^{f}(q^2,\undQ)\neq 0$. Then we have the following.
			
			(a) $\{F_{\rm T} \mid {\rm T}\in {\rm Tri}_{\bar{0}}(\mathscr{P}^{f}_{n})\}$ is a complete set of (super) primitive orthogonal idempotents of $\mHfcn.$
			
			(b) $\{F_{\undla} \mid \undla \in \mathscr{P}^{f}_{n} \}$ is a complete set of (super) primitive central idempotents of $\mHfcn.$

		\end{thm}
		
		\begin{lem}\cite[Lemma 4.18]{LS2}
Let $\undla\in\mathscr{P}^{f}_{n}$ and ${\rm T}=(\mt,\alpha_{\mt},\beta_{\mt})\in{\rm Tri}_{\bar{0}}(\undla),$ we have
			\begin{equation}\label{comm. form of FT}
				F_{\rm T}=\prod_{k=1}^{n}\prod_{\mathtt{b}\in \mathtt{B}(k)\atop \mathtt{b}\neq \mathtt{b}_{+}(\res_{\mt}(k))}\frac{X_k^{\nu_{\beta_{\mt}}(k)}-\mathtt{b}}{\mathtt{b}_{+}(\res_{\mt}(k))-\mathtt{b}} \cdot
				\left(C^{\alpha_{\mt}}\gamma_{\mt}(C^{\alpha_{\mt}})^{-1} \right)\in \mHfcn.
			\end{equation}
		\end{lem}
	The following two Lemmas will be used frequently in the sequel.
	\begin{lem}\label{change sign in diagona}	
		Let $\undla\in\mathscr{P}^{f}_{n}$ and ${\rm T}=(\mt,\alpha_{\mt},\beta_{\mt})\in{\rm Tri}_{\bar{0}}(\undla).$ Then for any $i \in \mathcal{D}_{\mt},$ we have
		\begin{equation}\label{changesign1}
			\prod_{k=1}^{n}\prod_{\mathtt{b}\in \mathtt{B}(k)\atop \mathtt{b}\neq \mathtt{b}_{+}(\res_{\mt}(k))}\frac{X_k^{\nu_{\beta_{\mt}}(k)}-\mathtt{b}}{\mathtt{b}_{+}(\res_{\mt}(k))-\mathtt{b}} =\prod_{k=1}^{n}\prod_{\mathtt{b}\in \mathtt{B}(k)\atop \mathtt{b}\neq \mathtt{b}_{+}(\res_{\mt}(k))}\frac{X_k^{\nu_{\beta_{\mt}}(k)\cdot (-1)^{\delta_{k,i}}}-\mathtt{b}}{\mathtt{b}_{+}(\res_{\mt}(k))-\mathtt{b}}.
		\end{equation}
        \end{lem}
\begin{proof}
	It has been obtained in \cite[proof of Lemma 4.18]{LS2}.
	\end{proof}

	   \begin{lem}\label{change sign in diagonal}
Let $\undla\in\mathscr{P}^{f}_{n}$ and ${\rm T}=(\mt,\alpha_{\mt},\beta_{\mt})\in{\rm Tri}_{\bar{0}}(\undla).$
Then for any $i\in \mathcal{D}_\mt$, we have \begin{equation}\label{CFC=F}
				C_i F_{\rm T}=F_{\rm T}C_i.
			\end{equation}
		\end{lem}
		\begin{proof}
			We have
			\begin{align}
				C_i F_{\rm T}
				&=C_i \left(C^{\alpha_{\mt}}\gamma_{\mt}(C^{\alpha_{\mt}})^{-1} \right)
				\prod_{k=1}^{n}\prod_{\mathtt{b}\in \mathtt{B}(k)\atop \mathtt{b}\neq \mathtt{b}_{+}(\res_{\mt}(k))}\frac{X_k^{\nu_{\beta_{\mt}}(k)}-\mathtt{b}}{\mathtt{b}_{+}(\res_{\mt}(k))-\mathtt{b}}\\
				&=\left(C^{\alpha_{\mt}}\gamma_{\mt}(C^{\alpha_{\mt}})^{-1} \right)
				\prod_{k=1}^{n}\prod_{\mathtt{b}\in \mathtt{B}(k)\atop \mathtt{b}\neq \mathtt{b}_{+}(\res_{\mt}(k))}\frac{X_k^{\nu_{\beta_{\mt}}(k)(-1)^{\delta_{k,i}}}-\mathtt{b}}{\mathtt{b}_{+}(\res_{\mt}(k))-\mathtt{b}}C_i\nonumber\\
				&=\left(C^{\alpha_{\mt}}\gamma_{\mt}(C^{\alpha_{\mt}})^{-1} \right)
				\prod_{k=1}^{n}\prod_{\mathtt{b}\in \mathtt{B}(k)\atop \mathtt{b}\neq \mathtt{b}_{+}(\res_{\mt}(k))}\frac{X_k^{\nu_{\beta_{\mt}}(k)}-\mathtt{b}}{\mathtt{b}_{+}(\res_{\mt}(k))-\mathtt{b}}C_i\nonumber\\
				&=F_{{\rm T}} C_i,\nonumber
			\end{align}
			where in the first equality we have used the definition of $F_{{\rm T}}$ and in the third equality we have used Lemma \ref{change sign in diagona}. This completes the proof of the Lemma.
		\end{proof}
		
		\subsection{Seminormal bases of $\mHfcn$}\label{Seminormalbase}
		 In this subsection, we fix $\undla\in\mathscr{P}^{f}_{n}$. We will recall the construction of seminormal bases for block $B_{\undla}$. The following definition is crucial in our construction of seminormal bases.
	\begin{defn}\cite[Definition 4.21]{LS2}
    \label{pag:Phist and cst}	
For any $\ms,\mt \in \Std(\undla),$ we fix a reduced expression $d(\ms,\mt)=s_{k_p}\cdots s_{k_1}$. We define
		\begin{align}\label{Phist}
			\Phi_{\ms,\mt}:=\overleftarrow{\prod_{i=1,\ldots,p}}\Phi_{k_{i}}(\mathtt{b}_{s_{k_{i-1}}\cdots s_{k_1}\mathfrak{t},k_{i}}, \mathtt{b}_{s_{k_{i-1}}\cdots s_{k_1}\mathfrak{t},k_{i}+1})  \in \mHfcn
		\end{align}
		and the coefficient
		\begin{align}\label{c-coefficients. non-dege.}
			\mathtt{c}_{\ms,\mt}:=\prod_{i=1,\ldots,p}\sqrt{\mathtt{c}_{s_{k_{i-1}}\cdots s_{k_1}\mathfrak{t}}(k_{i})}  \in \mathbb{K}.
		\end{align}
		\end{defn}
		By Lemma \ref{admissible transposes} and the third part of Lemma \ref{important condition1}, $\mathtt{c}_{\ms,\mt}\in  \mathbb{K}^*$.  By \cite[Lemma 4.22]{LS2}, $\Phi_{\ms,\mt}$ is independent of the reduced expression of $d(\ms,\mt)$.

		\begin{lem}\label{Phist. lem}\cite[Lemma 4.23]{LS2}
			Let $\ms, \mt \in \Std(\undla).$
			\begin{enumerate}
			\item $\Phi_{\ms,\mt} \cdot v_{\mt}=\mathtt{c}_{\ms,\mt} v_{\ms}.$
			
		\item $\Phi_{\ms,\mt} C_{d(\mt,\mt^{\undla})(a)}= C_{d(\ms,\mt^{\undla})(a)}\Phi_{\ms,\mt}$ for any $a\in \mathcal{D}_{\mt^{\undla}}.$

		\item $\mathtt{c}_{\mt,\ms}=\mathtt{c}_{\ms,\mt}.$
			
		\item $\Phi_{\ms,\mt}\Phi_{\mt,\ms}=(\mathtt{c}_{\ms,\mt})^2.$
			
			\end{enumerate}
		\end{lem}

By definition, we have $\Phi_{\mt,\mt}=1$ and $\mathtt{c}_{\mt,\mt}=1.$

Now we can define the seminormal bases. The following definitions and results have full statements in \cite{LS2}.
\noindent\textbf{For the rest of this subsection, we restrict $f$ to \eqref{cyc-even} or \eqref{cyc-odd}. The defining polynomial in \eqref{cyc-two-special-factors} is not needed for these calculations of seminormal bases.}
        \label{pag:nondege seminormal basis}
		\begin{defn}\cite[Definition 4.24]{LS2} Let $\mathfrak{w}\in\Std(\undla)$.
						
			(1) For any ${\rm S}=(\ms, \alpha_{\ms}', \beta_{\ms}')\in {\rm Tri}_{\bar{0}}(\undla), {\rm T}=(\mt, \alpha_{\mt}, \beta_{\mt})\in {\rm Tri}(\undla),$ we define
		\begin{align*}
							f_{{\rm S},{\rm T}}^{\mathfrak{w}}
							:=F_{\rm S}C^{\beta_{\ms}'}C^{\alpha_{\ms}'}\Phi_{\ms,\mathfrak{w}}\Phi_{\mathfrak{w},\mt}
							(C^{\alpha_{\mt}})^{-1} (C^{\beta_{\mt}})^{-1} F_{\rm T}\in F_{\rm S}\mHfcn F_{\rm T},
						\end{align*} and	\begin{align*}
						f_{{\rm S},{\rm T}}
						:=F_{\rm S}C^{\beta_{\ms}'}C^{\alpha_{\ms}'}\Phi_{\ms,\mt}
						(C^{\alpha_{\mt}})^{-1} (C^{\beta_{\mt}})^{-1} F_{\rm T} \in F_{\rm S}\mHfcn F_{\rm T},
						\end{align*}

		\label{pag:nondege cT}	
	   (2)  For any ${\rm T}=(\mt, \alpha_{\mt}, \beta_{\mt})\in {\rm Tri}(\undla),$ we define
            $$\mathtt{c}_{\rm T}^{\mathfrak{w}}
            =\mathtt{c}_{\mt}^{\mathfrak{w}}:=\mathtt{c}_{\mt,\mathfrak{w}}\mathtt{c}_{\mathfrak{w},\mt}=(\mathtt{c}_{\mt,\mathfrak{w}})^2\in \mathbb{K}^*.$$
		\end{defn}

The following Lemma is crucial for our main Theorem.
\begin{lem}\cite[Lemma 4.25]{LS2}
We fix $\mathfrak{w}\in\Std(\undla)$.
\begin{enumerate}	
	\item Suppose $d_{\undla}=0$. For any ${\rm S}=(\ms, \alpha_{\ms}', \beta_{\ms}'),
	{\rm T}=(\mt, \alpha_{\mt}, \beta_{\mt}),{\rm U}=(\mfku,\alpha_{\mfku}^{''},\beta_{\mfku}^{''})\in {\rm Tri}(\undla),$ we have \begin{equation}\label{idempotent1}
		f_{{\rm T}, {\rm T}}=F_{\rm T},\qquad 	f_{{\rm T}, {\rm T}}^{\mathfrak{w}}=\mathtt{c}_{\rm T}^{\mathfrak{w}}F_{\rm T},
	\end{equation}
	 \begin{equation}\label{SNB. eq1}
					f_{{\rm S},{\rm T}}^\mathfrak{w} \cdot C^{\beta_{\mfku}^{''}} C^{\alpha_{\mfku}^{''}} v_{\mfku} =\begin{cases}
						\mathtt{c}_{\ms,\mathfrak{w}}\mathtt{c}_{\mathfrak{w},\mt}C^{\beta_{\ms}^{'}} C^{\alpha_{\ms}^{'}} v_{\ms},& \text{if ${\rm T}={\rm U}$},\\
						0,&\text{otherwise,}
						\end{cases}				
				\end{equation} and  \begin{equation}\label{SNB. eq2}
				f_{{\rm S},{\rm T}} \cdot C^{\beta_{\mfku}^{''}} C^{\alpha_{\mfku}^{''}} v_{\mfku} =\begin{cases}
					\mathtt{c}_{\ms,\mt}C^{\beta_{\ms}^{'}} C^{\alpha_{\ms}^{'}} v_{\ms},& \text{if ${\rm T}={\rm U}$},\\
					0,&\text{otherwise.}
				\end{cases}				
				\end{equation}
			\item Suppose $d_{\undla}=1$.  For any $a\in \mathbb{Z}_{2}$ and ${\rm S}=(\ms, \alpha_{\ms}', \beta_{\ms}')\in {\rm Tri}_{\bar{0}}(\undla), {\rm T}_{a}=(\mt, \alpha_{\mt,a}, \beta_{\mt})\in {\rm Tri}_{a}(\undla), {\rm U}=(\mfku,\alpha_{\mfku}^{''},\beta_{\mfku}^{''})\in {\rm Tri}(\undla),$ we have \begin{equation}\label{idempotent2}
					f_{{\rm T}_{\bar{0}}, {\rm T}_{\bar{0}}}=F_{{\rm T}_{\bar{0}}},\qquad	f_{{\rm T}_{\bar{0}}, {\rm T}_{\bar{0}}}^{\mathfrak{w}}=\mathtt{c}_{\rm T}^{\mathfrak{w}}F_{{\rm T}_{\bar{0}}},
			\end{equation} \begin{equation}\label{SNB. eq3}
				f_{{\rm S},{\rm T}_a}^\mathfrak{w}  \cdot C^{\beta_{\mfku}^{''}} C^{\alpha_{\mfku}^{''}} v_{\mfku} =\begin{cases}
			\mathtt{c}_{\mathfrak{s},\mathfrak{w} }\mathtt{c}_{\mathfrak{w},\mt }C^{\beta_{\ms}'} C^{\alpha_{\ms,a+b}'  }v_{\ms}, & \text{if  ${\rm U}={\rm T}_b$,
					for some $b\in \Z_2$,}\\
				0,&\text{otherwise,}
			\end{cases}				
			\end{equation} and  \begin{equation}\label{SNB. eq4}
				f_{{\rm S},{\rm T}_a} \cdot C^{\beta_{\mfku}^{''}} C^{\alpha_{\mfku}^{''}} v_{\mfku} =\begin{cases}
						\mathtt{c}_{\mathfrak{s},\mt} C^{\beta_{\ms}'} C^{\alpha_{\ms,a+b}'  }v_{\ms}, & \text{if  ${\rm U}={\rm T}_b$,
					for some $b\in \Z_2$,}\\
				0,&\text{otherwise.}
				\end{cases}				
			\end{equation}
				\end{enumerate}
			\end{lem}

		\begin{thm}\label{seminormal basis}\cite[Theorem 4.26]{LS2}
			Suppose $P_n^{f}(q^2,\undQ)\neq 0$. We fix $\mathfrak{w}\in\Std(\undla)$. Then the following two sets
			\begin{align}\label{Non-deg seminormal1}
			\left\{ f_{{\rm S},{\rm T}}^\mathfrak{w} \Biggm|
			{\rm S}=(\ms, \alpha_{\ms}', \beta_{\ms}')\in {\rm Tri}_{\bar{0}}(\undla),
			{\rm T}=(\mt, \alpha_{\mt}, \beta_{\mt})\in {\rm Tri}(\undla)
			\right\}
		\end{align} and \begin{align}\label{Non-deg seminormal2}
			\left\{ f_{{\rm S},{\rm T}} \Biggm|
		{\rm S}=(\ms, \alpha_{\ms}', \beta_{\ms}')\in {\rm Tri}_{\bar{0}}(\undla),
		{\rm T}=(\mt, \alpha_{\mt}, \beta_{\mt})\in {\rm Tri}(\undla)
		\right\}
			\end{align}  form two $\mathbb{K}$-bases of the block $B_{\undla}$ of $\mHfcn$.
			
			Moreover, for ${\rm S}=(\ms, \alpha_{\ms}', \beta_{\ms}')\in {\rm Tri}_{\bar{0}}(\undla),
			{\rm T}=(\mt, \alpha_{\mt}, \beta_{\mt})\in {\rm Tri}(\undla),$ we have \begin{equation}\label{fst and re. fst}
			f_{{\rm S},{\rm T}}
			=\frac{\mathtt{c}_{\ms,\mt}}{\mathtt{c}_{\ms,\mathfrak{w} }\mathtt{c}_{\mathfrak{w},\mt }} f_{{\rm S},{\rm T}}^\mathfrak{w}\in F_{\rm S}\mHfcn F_{\rm T}.
			\end{equation} The multiplications of basis elements in \eqref{Non-deg seminormal1} are given as follows.
			
			(1) Suppose $d_{\undla}=0.$  Then for any
			${\rm S}=(\ms, \alpha_{\ms}', \beta_{\ms}'),
			{\rm T}=(\mt, \alpha_{\mt}, \beta_{\mt}),
			{\rm U}=(\mfku,\alpha_{\mfku}^{''},\beta_{\mfku}^{''}),
			V=(\mfkv,\alpha_{\mfkv}^{'''},\beta_{\mfkv}^{'''})\in {\rm Tri}(\undla),$ we have
			\begin{align}\label{Non-deg multiplication1}
				f_{{\rm S},{\rm T}}^\mathfrak{w} f_{{\rm U},V}^\mathfrak{w}
				=\delta_{{\rm T},{\rm U}} \mathtt{c}_{\rm T}^\mathfrak{w} f_{{\rm S},V}^\mathfrak{w}.
			\end{align}

			(2) Suppose $d_{\undla}=1.$ Then for any $a,b\in \mathbb{Z}_2$ and
			\begin{align*}
				{\rm S}&=(\ms, \alpha_{\ms}', \beta_{\ms}')\in {\rm Tri}_{\bar{0}}(\undla), \quad
				{\rm T}_{a}=(\mt, \alpha_{\mt,a}, \beta_{\mt})\in {\rm Tri}_{a}(\undla),\nonumber\\
				{\rm U}&=(\mfku,\alpha_{\mfku}^{''},\beta_{\mfku}^{''})\in {\rm Tri}_{\bar{0}}(\undla), \quad
				V_{b}=(\mfkv,{\alpha_{\mfkv,b}^{'''}},\beta_{\mfkv}^{'''})\in {\rm Tri}_{b}(\undla),\nonumber
			\end{align*} we have
			\begin{align}\label{Non-deg multiplication2}
				f_{{\rm S},{\rm T}_{a}}^\mathfrak{w} f_{{\rm U},V_{b}}^\mathfrak{w}
				=\delta_{{\rm T}_{\bar{0}},{\rm U}}\mathtt{c}_{\rm T}^\mathfrak{w} f_{{\rm S},V_{a+b}}^\mathfrak{w}.
			\end{align}
		
		\end{thm}

\begin{rem}
Assume $d_{\undla}=0,$ then for any
${\rm S}=(\ms, \alpha_{\ms}', \beta_{\ms}'),
{\rm T}=(\mt, \alpha_{\mt}, \beta_{\mt}),
{\rm U}=(\mfku,\alpha_{\mfku}^{''},\beta_{\mfku}^{''}),
V=(\mfkv,\alpha_{\mfkv}^{'''},\beta_{\mfkv}^{'''})\in {\rm Tri}(\undla),$ we have
\begin{align*}
f_{{\rm S},{\rm T}}f_{{\rm U},V}
=\delta_{{\rm T},{\rm U}} \frac{\mathtt{c}_{\ms,\mt}\mathtt{c}_{\mt,\mathfrak{v}} }{\mathtt{c}_{\ms,\mfkv}}f_{{\rm S},V}.
\end{align*}
\end{rem}
			
Recall in \eqref{stanard D} and \eqref{standard OD}, we have set \begin{align*}
	\mathcal{D}_{\mt^{\undla}}&=\{i_1<i_2<\cdots<i_t\},\\
	\mathcal{OD}_{\mt^{\undla}}&=\{i_1,i_3,\cdots,i_{2{\lceil t/2 \rceil}-1}\}\subset \mathcal{D}_{\mt^{\undla}}.
\end{align*}

\begin{defn}\label{F(i,T,S)}\cite[Definition 4.33]{LS2}
Let $\undla\in\mathscr{P}^{f}_{n}$.
\begin{enumerate}
	\item
	Suppose $d_{\undla}=0.$ Let ${\rm S}=(\ms, \alpha_{\ms}', \beta_{\ms}'),
	{\rm T}=(\mt, \alpha_{\mt}, \beta_{\mt})\in {\rm Tri}(\undla).$ For each $i\in [n-1],$ we define the element
$F(i,{\rm T},{\rm S})\in\mHfcn$ as follows.

\begin{enumerate}
	\item If $i,i+1\in [n]\setminus \mathcal{D}_{\mt},$
	$$F(i,{\rm T},{\rm S}):=(-1)^{\delta_{\beta_{\mt}}(i)}f_{(\mt,\alpha_{\mt},\beta_{\mt}+e_{i}+e_{i+1}),{\rm S}};
	$$
	\item
	if $i={d(\mt,\mt^{\undla})}(i_p)\in \mathcal{D}_{\mt}, i+1\in [n]\setminus \mathcal{D}_{\mt},$
	$$F(i,{\rm T},{\rm S}):=\begin{cases}(-1)^{|\beta_{\mt}|_{>i}+|\alpha_{\mt}|_{<i}}f_{(\mt,\alpha_{\mt}+e_{i},\beta_{\mt}+e_{i+1}),{\rm S}}, & \text{ if  $p$ is odd},\\
		(-\sqrt{-1})(-1)^{|\beta_{\mt}|_{>i}+|\alpha_{\mt}|_{\leq {d(\mt,\mt^{\undla})}(i_{p-1})}}f_{(\mt,\alpha_{\mt}+e_{{d(\mt,\mt^{\undla})}(i_{p-1})},\beta_{\mt}+e_{i+1}),{\rm S}}, &\text{ if  $p$ is even};
	\end{cases}\nonumber
	$$
	\item
	if $i+1={d(\mt,\mt^{\undla})}(i_p)\in \mathcal{D}_{\mt}, i\in [n]\setminus \mathcal{D}_{\mt},$ $$F(i,{\rm T},{\rm S}):=\begin{cases}(-1)^{|\beta_{\mt}|_{\geq i}+|\alpha_{\mt}|_{<i+1}} f_{(\mt,\alpha_{\mt}+e_{i+1},\beta_{\mt}+e_{i}),{\rm S}}, & \text{ if  $p$ is odd},\\
		(-\sqrt{-1})(-1)^{|\beta_{\mt}|_{\geq i}+|\alpha_{\mt}|_{\leq {d(\mt,\mt^{\undla})}(i_{p-1})}}f_{(\mt,\alpha_{\mt}+e_{{d(\mt,\mt^{\undla})}(i_{p-1})},\beta_{\mt}+e_{i}),{\rm S}}, &\text{ if $p$ is even}.\end{cases}  \nonumber
	$$
	\end{enumerate}
		\item	Suppose $d_{\undla}=1.$
	Let $a\in \Z_2$, ${\rm T}=(\mt, \alpha_{\mt}, \beta_{\mt})\in {\rm Tri}_{\bar{0}}(\undla),$
	${\rm S}_{a}=(\ms, {\alpha_{\ms,a}'}, \beta_{\ms}')\in {\rm Tri}_{a}(\undla). $ For each $i\in [n-1],$ we define $F(i,{\rm T},{\rm S}_a)$ as follows.
	
	\begin{enumerate}
		\item If $i,i+1\in [n]\setminus \mathcal{D}_{\mt},$
		$$F(i,{\rm T},{\rm S}_a):=(-1)^{\delta_{\beta_{\mt}}(i)}f_{(\mt,\alpha_{\mt},\beta_{\mt}+e_{i}+e_{i+1}),{\rm S}_a};
		$$
		\item
		if $i={d(\mt,\mt^{\undla})}(i_p)\in \mathcal{D}_{\mt}, i+1\in [n]\setminus \mathcal{D}_{\mt},$
		$$F(i,{\rm T},{\rm S}_a):=\begin{cases}(-1)^{|\beta_{\mt}|_{>i}+|\alpha_{\mt,\bar{1}}|_{<i}}	f_{(\mt,\alpha_{\mt}+e_i,\beta_{\mt}+e_{i+1}),{\rm S}_a}, & \text{ if  $p\neq t$ is odd},\\
			(-1)^{|\beta_{\mt}|_{>i}+|\alpha_{\mt}|}	f_{(\mt,\alpha_{\mt},\beta_{\mt}+e_{i+1}),{\rm S}_{a+\bar{1}}}, & \text{ if  $p=t$},\\
			(-\sqrt{-1})(-1)^{|\beta_{\mt}|_{>i}+|\alpha_{\mt,\bar{1}}|_{\leq {d(\mt,\mt^{\undla})}(i_{p-1})}}&\\
			\qquad\qquad	\qquad\qquad \cdot f_{(\mt,\alpha_{\mt}+e_{{d(\mt,\mt^{\undla})}(i_{p-1})},\beta_{\mt}+e_{i+1}),{\rm S}_a}, &\text{ if  $p$ is even};
		\end{cases}
		$$
		\item
		if $i+1={d(\mt,\mt^{\undla})}(i_p)\in \mathcal{D}_{\mt}, i\in [n]\setminus \mathcal{D}_{\mt},$ $$F(i,{\rm T},{\rm S}_a):=\begin{cases}(-1)^{|\beta_{\mt}|_{\geq i}+|\alpha_{\mt,\bar{1}}|_{<i+1}}	f_{(\mt,\alpha_{\mt}+e_{i+1},\beta_{\mt}+e_{i}),{\rm S}_a},& \text{ if  $p\neq t$ is odd},\\
			(-1)^{|\beta_{\mt}|_{\geq i}+|\alpha_{\mt}|}	f_{(\mt,\alpha_{\mt},\beta_{\mt}+e_{i}),{\rm S}_{a+\bar{1}}},& \text{ if  $p=t$},\\
			(-\sqrt{-1})(-1)^{|\beta_{\mt}|_{\geq i}+|\alpha_{\mt,\bar{1}}|_{\leq {d(\mt,\mt^{\undla})}(i_{p-1})}}&\\
			\qquad\qquad\qquad\qquad\cdot	f_{(\mt,\alpha_{\mt}+e_{{d(\mt,\mt^{\undla})}(i_{p-1})},\beta_{\mt}+e_{i}),{\rm S}_a},&\text{ if $p$ is even}.\end{cases}
		$$
	\end{enumerate}
\end{enumerate}
\end{defn}
			The action of the generators on the seminormal basis is given by the following.
		\begin{prop}\label{generators action on seminormal basis}
			Let $\undla\in\mathscr{P}^{f}_{n}$.
		Suppose ${\rm T}=(\mt, \alpha_{\mt}, \beta_{\mt})\in {\rm Tri}_{\bar{0}}(\undla)$
        and ${\rm S}=(\ms, \alpha_{\ms}', \beta_{\ms}')\in {\rm Tri}(\undla).$ Then we have the following.
			\begin{enumerate}
					\item For each $i\in [n],$ we have
			\begin{align*}
				X_i \cdot f_{{\rm T},{\rm S}}
				    =\mathtt{b}_{\mt,i}^{-\nu_{\beta_{\mt}}(i)} f_{{\rm T},{\rm S}},\qquad
			        f_{{\rm T},{\rm S}}\cdot X_i
				    =\mathtt{b}_{\ms,i}^{-\nu_{\beta_{\ms}'}(i)} f_{{\rm T},{\rm S}}.
			\end{align*}
			
			\item For each $i\in [n-1],$ we have
			\begin{align*}
					 &T_i\cdot f_{{\rm T},{\rm S}}
					    \in-\frac{\epsilon}{\mathtt{b}_{\mt,i}^{-\nu_{\beta_{\mt}}(i)}\mathtt{b}_{\mt,i+1}^{\nu_{\beta_{\mt}}(i+1)}-1} f_{{\rm T},{\rm S}} 			+\frac{\epsilon}{\mathtt{b}_{\mt,i}^{\nu_{\beta_{\mt}}(i)}\mathtt{b}_{\mt,i+1}^{\nu_{\beta_{\mt}}(i+1)}-1}F(i,{\rm T},{\rm S})+\mathbb{K}f_{s_i\cdot{\rm T},{\rm S}}.
					\end{align*}
			\end{enumerate}
		\end{prop}

\begin{proof}
	The first part follows from \cite[Proposition 4.32]{LS2} and \cite[Proposition 4.34, 1(a),2(a)]{LS2}. The second part follows from \cite[Proposition 4.34, 1(c),2(c)]{LS2}
	\end{proof}

\subsection{The combinatorial formulae of $\mathtt{c}$-coefficients}\label{gamma coefficient}

In this subsection, we shall present some combinatorial formulae for the important coefficients $\mathtt{c}_{\mt}^{\mt^{\undla}}=\mathtt{c}_{\mt,\mt^{\undla}}\mathtt{c}_{\mt^{\undla},\mt}$ and
$\mathtt{c}_{\mt}^{\mt_{\undla}}=\mathtt{c}_{\mt,\mt_{\undla}}\mathtt{c}_{\mt_{\undla},\mt},$ for any fixed
$\undla\in\mathscr{P}^{f}_{n}$ and $\mt\in\Std(\undla).$

These coefficients enter the comparison of seminormal bases at consecutive ranks. Their product formulae also yield the tableau-independent scalar $\kappa_{\undla}$ used in the formulae for Schur elements below.

Recall the coefficients $\mathtt{c}_{\mt}(i)$ appeared in \eqref{nondege coeffi cti}. For any $i\in[n-1]$ such that $s_i\mt\in\Std(\undla)$ and $d(s_i\mt,\mt^{\undla})=s_i d(\mt,\mt^{\undla})>d(\mt,\mt^{\undla}),$ we have
\begin{align}\label{recurrence relations of c}
	\mathtt{c}_{s_i\mt}^{\mt^{\undla}}=\mathtt{c}_{\mt}^{\mt^{\undla}}\cdot\mathtt{c}_{\mt}(i), \text{ and }
	\mathtt{c}_{s_i\mt}^{\mt_{\undla}}\cdot\mathtt{c}_{\mt}(i)=\mathtt{c}_{\mt}^{\mt_{\undla}}.
\end{align}

\begin{lem}\label{rewrite cti}
	Let $\undla\in\mathscr{P}^{f}_{n}$ and $\mt\in\Std(\undla).$
	For any $i\in[n-1],$ we can rewrite the coefficient $\mathtt{c}_{\mt}(i)$ as follows
	\begin{align*}
		\mathtt{c}_{\mt}(i)=\frac{\left(\mathtt{q}(\res_{\mt}(i))-\mathtt{q}(q^2\res_{\mt}(i+1))\right)\left(\mathtt{q}(\res_{\mt}(i))-\mathtt{q}(q^{-2}\res_{\mt}(i+1))\right)}{\left(\mathtt{q}(\res_{\mt}(i))-\mathtt{q}(\res_{\mt}(i+1))\right)^2}.
	\end{align*}
\end{lem}
\begin{proof}
	For any $\iota\in\mathbb{K}^*,$ we have the following identities
	\begin{align}\label{eq1. rewrite cti}
		\mathtt{q}(q^2\iota)+\mathtt{q}(q^{-2}\iota)=(\epsilon^2+2)\mathtt{q}(\iota),\qquad
		\mathtt{q}(q^2\iota)\mathtt{q}(q^{-2}\iota)=\mathtt{q}(\iota)^2+4\epsilon^2.
	\end{align}
	Write $\iota'=\res_{\mt}(i)$ and $\iota=\res_{\mt}(i+1)$. Then we have
	\begin{align*}
		\mathtt{c}_{\mt}(i)
		&=1-\epsilon^2 \biggl(\frac{\mathtt{b}_{\mt,i}^{-1}\mathtt{b}_{\mt,i+1}}{(\mathtt{b}_{\mt,i}^{-1}\mathtt{b}_{\mt,i+1}-1)^2}
		+\frac{\mathtt{b}_{\mt,i}\mathtt{b}_{\mt,i+1}}{(\mathtt{b}_{\mt,i}\mathtt{b}_{\mt,i+1}-1)^2}\biggr)\\
		&=\frac{\left(\mathtt{q}(\iota')-\mathtt{q}(\iota)\right)^2-\epsilon^2(\mathtt{b}_{\mt,i}^{-1}\mathtt{b}_{\mt,i+1}^{-1}(\mathtt{b}_{\mt,i}\mathtt{b}_{\mt,i+1}-1)^2
			+\mathtt{b}_{\mt,i}\mathtt{b}_{\mt,i+1}^{-1}(\mathtt{b}_{\mt,i}^{-1}\mathtt{b}_{\mt,i+1}-1)^2)}{\left(\mathtt{q}(\iota')-\mathtt{q}(\iota)\right)^2}\\
		&=\frac{\left(\mathtt{q}(\iota')-\mathtt{q}(\iota)\right)^2-\epsilon^2\left(\mathtt{q}(\iota')\mathtt{q}(\iota) -4\right)}{\left(\mathtt{q}(\iota')-\mathtt{q}(\iota)\right)^2}\\
		&=\frac{\mathtt{q}(\iota')^2-(\epsilon^2+2)\mathtt{q}(\iota')\mathtt{q}(\iota)+\mathtt{q}(\iota)^2 +4\epsilon^2}{\left(\mathtt{q}(\iota')-\mathtt{q}(\iota)\right)^2}\\
		&=\frac{\mathtt{q}(\iota')^2-\mathtt{q}(\iota')\left(\mathtt{q}(q^2\iota)+\mathtt{q}(q^{-2}\iota)\right)+\mathtt{q}(q^2\iota)\mathtt{q}(q^{-2}\iota)}{\left(\mathtt{q}(\iota')-\mathtt{q}(\iota)\right)^2}\\
		&=\frac{\left(\mathtt{q}(\iota')-\mathtt{q}(q^2\iota)\right)\left(\mathtt{q}(\iota')-\mathtt{q}(q^{-2}\iota)\right)}{\left(\mathtt{q}(\iota')-\mathtt{q}(\iota)\right)^2},
	\end{align*}
	where in the penultimate equality we have used \eqref{eq1. rewrite cti}.
\end{proof}

\begin{defn}
	(1) \cite[Before Remark 3B.1]{EM}
For any two boxes $x=(i,j,l)$ and $y=(a,b,c),$ we write $y<x$ if and only if
$c<l; \text{ or } c=l \text{ and } a<i; \text{ or } c=l, a=i \text{ and } b<j.$
	
	(2) For any $\undla\in\mathscr{P}^{f}_{n}$ and $\mt\in\Std(\undla),$
	we define
	\begin{align*}
		\mathscr{A}_{\mt}^{\rhd}(k)&:=\{\alpha\in{\rm Add}(\mt\downarrow_{k-1})\mid \alpha>\mt^{-1}(k)\},\\
		\mathscr{R}_{\mt}^{\rhd}(k)&:=\{\beta\in{\rm Rem}(\mt\downarrow_{k-1})\mid \beta>\mt^{-1}(k)\},\\
		\mathscr{A}_{\mt}^{\lhd}(k)&:=\{\alpha\in{\rm Add}(\mt\downarrow_{k-1})\mid \alpha<\mt^{-1}(k)\},\\
		\mathscr{R}_{\mt}^{\lhd}(k)&:=\{\beta\in{\rm Rem}(\mt\downarrow_{k-1})\mid \beta<\mt^{-1}(k)\},
	\end{align*}
	for $k\in[n].$
\end{defn}

\label{pag:diagonalnodes}
\begin{defn}\label{diagonalnodes}
For the fixed factor presentation of $f$, let
\[
 \mathcal D:=\{(i,i,l)\mid i\geq1,\ l\text{ labels a strict component}\}.
\]
\end{defn}
Thus $\mathcal D$ records all possible strict diagonal positions, whereas $\mathcal D_{\undla}=\mathcal D\cap\undla$ records the boxes in a particular shape.

Now we can derive the following closed formulae for $\mathtt{c}_{\mt}^{\mt^{\undla}}$ and $\mathtt{c}_{\mt}^{\mt_{\undla}}.$
\begin{prop}\label{combina. formulae of c}
	Let $\undla\in\mathscr{P}^{f}_{n}$ and $\mt\in\Std(\undla).$
	Then we have
	\begin{align}\label{combina. formula of row c-coeffi.}
		\mathtt{c}_{\mt}^{\mt^{\undla}}
		=\prod_{k=1}^{n}\prod_{\alpha\in\mathscr{A}_{\mt^{\undla}}^{\rhd}(k)}\left(\mathtt{q}(\res_{\mt^{\undla}}(k))-\mathtt{q}(\res(\alpha))\right)^{-1}
		\cdot\prod_{k=1}^{n}\frac{\prod_{\alpha\in\mathscr{A}_{\mt}^{\rhd}(k)}\left(\mathtt{q}(\res_{\mt}(k))-\mathtt{q}(\res(\alpha))\right)}{\prod_{\beta\in\mathscr{R}_{\mt}^{\rhd}(k)\setminus\mathcal{D}}\left(\mathtt{q}(\res_{\mt}(k))-\mathtt{q}(\res(\beta))\right)},
	\end{align}
	\begin{align}\label{combina. formula of column c-coeffi.}
		\mathtt{c}_{\mt}^{\mt_{\undla}}
		=\prod_{k=1}^{n}\frac{\prod_{\beta\in\mathscr{R}_{\mt_{\undla}}^{\lhd}(k)\setminus\mathcal{D}}\left(\mathtt{q}(\res_{\mt_{\undla}}(k))-\mathtt{q}(\res(\beta))\right)}{\prod_{\alpha\in\mathscr{A}_{\mt_{\undla}}^{\lhd}(k)}\left(\mathtt{q}(\res_{\mt_{\undla}}(k))-\mathtt{q}(\res(\alpha))\right)}	\cdot\prod_{k=1}^{n}\frac{\prod_{\alpha\in\mathscr{A}_{\mt}^{\lhd}(k)}\left(\mathtt{q}(\res_{\mt}(k))-\mathtt{q}(\res(\alpha))\right)}{\prod_{\beta\in\mathscr{R}_{\mt}^{\lhd}(k)\setminus\mathcal{D}}\left(\mathtt{q}(\res_{\mt}(k))-\mathtt{q}(\res(\beta))\right)}.
	\end{align}
\end{prop}

\begin{proof}
	We prove only the first equality, as the second one follows similarly.
	Induction on $\ell:=\ell(d(\mt,\mt^{\undla})).$ If $\ell=0,$ then $\mt=\mt^{\undla}$ and $\mathtt{c}_{\mt^{\undla}}^{\mt^{\undla}}=1.$
	Note that $\mathscr{R}_{\mt^{\undla}}^{\rhd}(k)=\emptyset$ for any $k\in[n],$ so \eqref{combina. formula of row c-coeffi.} holds for $\ell=0.$ We now assume that \eqref{combina. formula of row c-coeffi.} holds for $\ell>0$ and $d(s_i\mt,\mt^{\undla})=s_id(\mt,\mt^{\undla})>d(\mt,\mt^{\undla})$ for some $i\in[n-1].$
	
	Denote by $\mathtt{\widetilde{c}}_{\mt}^{\mt^{\undla}}$ the right hand side of \eqref{combina. formula of row c-coeffi.} and let
	$\theta=\mt^{-1}(i+1)=(a,b,c),$ $x=(a-1,b,c),$ $y=(a,b+1,c),$ $z=(a+1,b,c),$ $w=(a,b-1,c)$ which can be displayed in (maybe extended) Young diagram as follows
	$$\young(:x,w\theta y,:z).$$
We set
	\begin{align*}
		A:=\mathscr{A}_{s_i\mt}^{\rhd}(i+1)\cap \{y,z\},\qquad R:=\mathscr{R}_{\mt}^{\rhd}(i)\cap \{x,w\}.
	\end{align*}
It's easy to check that
	\begin{align*}
		\mathscr{A}_{s_i\mt}^{\rhd}(i)&=\mathscr{A}_{\mt}^{\rhd}(i+1),\quad
		\mathscr{A}_{s_i\mt}^{\rhd}(i+1)=\left(\mathscr{A}_{\mt}^{\rhd}(i)\setminus\{\theta\}\right)\cup A,\\
		\mathscr{R}_{s_i\mt}^{\rhd}(i)&=\mathscr{R}_{\mt}^{\rhd}(i+1),\quad
		\mathscr{R}_{s_i\mt}^{\rhd}(i+1)=\left(\mathscr{R}_{\mt}^{\rhd}(i)\cup\{\theta\}\right)\setminus R
	\end{align*}

and
	\begin{align*}
		\mathscr{R}_{s_i\mt}^{\rhd}(i)\setminus\mathcal{D}&=\mathscr{R}_{\mt}^{\rhd}(i+1)\setminus\mathcal{D},\\
		\mathscr{R}_{s_i\mt}^{\rhd}(i+1)\setminus\mathcal{D}&=\left((\mathscr{R}_{\mt}^{\rhd}(i)\setminus\mathcal{D})\cup(\{\theta\}\setminus\mathcal{D})\right)\setminus (R\setminus\mathcal{D}).
	\end{align*}
	
	By definition, we have
	\begin{align}\label{eq1. combina. formula of row c-coeffi.}
		\frac{\mathtt{\widetilde{c}}_{s_i\mt}^{\mt^{\undla}}}{\mathtt{\widetilde{c}}_{\mt}^{\mt^{\undla}}}
		=&\frac{\prod_{\alpha\in\mathscr{A}_{s_i\mt}^{\rhd}(i)}\left(\mathtt{q}(\res_{s_i\mt}(i))-\mathtt{q}(\res(\alpha))\right)\cdot
			\prod_{\alpha\in\mathscr{A}_{s_i\mt}^{\rhd}(i+1)}\left(\mathtt{q}(\res_{s_i\mt}(i+1))-\mathtt{q}(\res(\alpha))\right)}{\prod_{\alpha\in\mathscr{A}_{\mt}^{\rhd}(i)}
			\left(\mathtt{q}(\res_{\mt}(i))-\mathtt{q}(\res(\alpha))\right)\cdot\prod_{\alpha\in\mathscr{A}_{\mt}^{\rhd}(i+1)}\left(\mathtt{q}(\res_{\mt}(i+1))-\mathtt{q}(\res(\alpha))\right)}\\
		&\cdot \frac{\prod_{\beta\in\mathscr{R}_{\mt}^{\rhd}(i)\setminus\mathcal{D}}\left(\mathtt{q}(\res_{\mt}(k))-\mathtt{q}(\res(\beta))\right)
			\cdot\prod_{\beta\in\mathscr{R}_{\mt}^{\rhd}(i+1)\setminus\mathcal{D}}\left(\mathtt{q}(\res_{\mt}(i+1))-\mathtt{q}(\res(\beta))\right)}{\prod_{\beta\in\mathscr{R}_{s_i\mt}^{\rhd}(i)\setminus\mathcal{D}}
			\left(\mathtt{q}(\res_{s_i\mt}(i))-\mathtt{q}(\res(\beta))\right)\cdot\prod_{\beta\in\mathscr{R}_{s_i\mt}^{\rhd}(i+1)\setminus\mathcal{D}}
			\left(\mathtt{q}(\res_{s_i\mt}(i+1))-\mathtt{q}(\res(\beta))\right)}.\nonumber
	\end{align}
	Case 1. If $\theta\notin\mathcal{D},$ then it's easy to check that
	\begin{align*}
		x\in R\setminus \mathcal{D} \Leftrightarrow y\notin A,\quad w\in R\setminus \mathcal{D} \Leftrightarrow z\notin A.
	\end{align*}
	By \eqref{eq1. combina. formula of row c-coeffi.}, we have
	\begin{align*}
		\frac{\mathtt{\widetilde{c}}_{s_i\mt}^{\mt^{\undla}}}{\mathtt{\widetilde{c}}_{\mt}^{\mt^{\undla}}}
		&=\frac{\prod_{\alpha\in A}
			\left(\mathtt{q}(\res_{\mt}(i))-\mathtt{q}(\res(\alpha))\right)\cdot \prod_{\beta\in R\setminus \mathcal{D}}
			\left(\mathtt{q}(\res_{\mt}(i))-\mathtt{q}(\res(\beta))\right)}{\left(\mathtt{q}(\res_{\mt}(i))-\mathtt{q}(\res_{\mt}(i+1))\right)^2}\\
		&=\frac{\left(\mathtt{q}(\res_{\mt}(i))-\mathtt{q}(q^2\res_{\mt}(i+1))\right)\left(\mathtt{q}(\res_{\mt}(i))
			-\mathtt{q}(q^{-2}\res_{\mt}(i+1))\right)}{\left(\mathtt{q}(\res_{\mt}(i))-\mathtt{q}(\res_{\mt}(i+1))\right)^2}\\
		&=\mathtt{c}_{\mt}(i),
	\end{align*}
	where in the last equality we have used Lemma \ref{rewrite cti}.
	
	Case 2. If $\theta\in\mathcal{D},$ then $\res(\theta)\in\{\pm 1\}$ and hence $\mathtt{q}(\res(\theta))=\mathtt{q}(q^{-2}\res(\theta)).$
	We can similarly check that
	\begin{align*}
		x\in R\setminus \mathcal{D} \Leftrightarrow y\notin A,\quad w\notin R\setminus \mathcal{D},\quad z\notin A.
	\end{align*}
	By \eqref{eq1. combina. formula of row c-coeffi.}, we have
	\begin{align*}
		\frac{\mathtt{\widetilde{c}}_{s_i\mt}^{\mt^{\undla}}}{\mathtt{\widetilde{c}}_{\mt}^{\mt^{\undla}}}
		&=\frac{\prod_{\alpha\in A}
			\left(\mathtt{q}(\res_{\mt}(i))-\mathtt{q}(\res(\alpha))\right)\cdot \prod_{\beta\in R\setminus \mathcal{D}}
			\left(\mathtt{q}(\res_{\mt}(i))-\mathtt{q}(\res(\beta))\right)}{\mathtt{q}(\res_{\mt}(i))-\mathtt{q}(\res_{\mt}(i+1))}\\
		&=\frac{\mathtt{q}(\res_{\mt}(i))-\mathtt{q}(q^2\res_{\mt}(i+1))}{\mathtt{q}(\res_{\mt}(i))-\mathtt{q}(\res_{\mt}(i+1))}\\
		&=\frac{\mathtt{q}(\res_{\mt}(i))-\mathtt{q}(q^2\res_{\mt}(i+1))}{\mathtt{q}(\res_{\mt}(i))-\mathtt{q}(\res_{\mt}(i+1))}
		\cdot \frac{\mathtt{q}(\res_{\mt}(i))-\mathtt{q}(q^{-2}\res_{\mt}(i+1))}{\mathtt{q}(\res_{\mt}(i))-\mathtt{q}(\res_{\mt}(i+1))}\\
		&=\mathtt{c}_{\mt}(i),
	\end{align*}
	where in the last equality we have used Lemma \ref{rewrite cti}.
	
	By the induction hypothesis, we have $\mathtt{\widetilde{c}}_{\mt}^{\mt^{\undla}}=\mathtt{c}_{\mt}^{\mt^{\undla}}.$
	Combining this with above Case 1 and Case 2,
	we can deduce that
    $\mathtt{\widetilde{c}}_{s_i\mt}^{\mt^{\undla}}=\mathtt{\widetilde{c}}_{\mt}^{\mt^{\undla}}\cdot \mathtt{c}_{\mt}(i)=\mathtt{c}_{\mt}^{\mt^{\undla}}\cdot \mathtt{c}_{\mt}(i)=\mathtt{c}_{s_i\mt}^{\mt^{\undla}}.$
    This completes the proof.
\end{proof}

\begin{lem}\label{lem. cundla}
	Let $\undla\in\mathscr{P}^{f}_{n}$,
	then we have $\mathtt{c}_{\mt}^{\mt^{\undla}}\mathtt{c}_{\mt}^{\mt_{\undla}}=\mathtt{c}_{\mt_{\undla}}^{\mt^{\undla}}$ for any $\mt\in\Std(\undla).$
\end{lem}
\begin{proof}
	For any $i\in[n-1]$ such that $s_i\mt\in\Std(\undla),$ by \eqref{recurrence relations of c}, we have $\mathtt{c}_{\mt}^{\mt^{\undla}}\mathtt{c}_{\mt}^{\mt_{\undla}}=\mathtt{c}_{s_i\mt}^{\mt^{\undla}}\mathtt{c}_{s_i\mt}^{\mt_{\undla}}.$
It follows that $$\mathtt{c}_{\mt}^{\mt^{\undla}}\mathtt{c}_{\mt}^{\mt_{\undla}}=\mathtt{c}_{\mt_{\undla}}^{\mt^{\undla}}\mathtt{c}_{\mt_{\undla}}^{\mt_{\undla}}=\mathtt{c}_{\mt_{\undla}}^{\mt^{\undla}}$$
	for any $\mt\in\Std(\undla).$
\end{proof}

\label{pag:qt}
For $\undla\in\mathscr P_n^f$, choose $\mt\in\Std(\undla)$ and set
	\begin{align*}
		\kappa_{\undla}
		:=\prod\limits_{k=1}^{n}\frac{\prod\limits_{\beta\in\Rem(\mt\downarrow_{k-1})\setminus \mathcal{D}}\left(\mathtt{q}(\res_{\mt}(k))-\mathtt{q}(\res(\beta))\right)}{
			\prod\limits_{\alpha\in\Add(\mt\downarrow_{k-1})\setminus \{\mt^{-1}(k)\}}\left(\mathtt{q}(\res_{\mt}(k))-\mathtt{q}(\res(\alpha))\right)}\in\mathbb{K}^*
	\end{align*}
\begin{cor}\label{prop. qt}
The scalar $\kappa_{\undla}$ is independent of the choice of $\mt\in\Std(\undla)$.
\end{cor}
\begin{proof}
	It follows from \eqref{combina. formula of row c-coeffi.}, \eqref{combina. formula of column c-coeffi.} and Lemma \ref{lem. cundla} that
	$\mathtt{c}_{\mt_{\undla}}^{\mt^{\undla}}=\eta(\undla)\cdot \kappa_{\undla}^{-1}$ for some scalar $\eta(\undla)\in\mathbb{K}^*,$ which is independent of $\mt.$ This proves the Corollary.
\end{proof}

\begin{lem}\label{lem. ct and cst}
	Let $\undla\in\mathscr{P}^{f}_{n-1}$ and $\underline{\mu}\in\mathscr{P}^{f}_{n}$, $\mt\in\Std(\undla),$ $\mathfrak{u}\in\Std(\underline{\mu})$
	such that $\mathfrak{u}\downarrow_{n-1}=\mt.$
	If $d(s_i \mt,\mt^{\undla})>d(\mt,\mt^{\undla})$ for some $i\in[n-2],$ then
	\begin{align}
		\frac{\mathtt{c}_{\mt,\mt^{\undla}}}{\mathtt{c}_{\mathfrak{u},\mt^{\underline{\mu}}}}
		=\frac{\mathtt{c}_{s_i\mt,\mt^{\undla}}}{\mathtt{c}_{s_i\mathfrak{u},\mt^{\underline{\mu}}}}.
	\end{align}
	\begin{proof}
		It is clear that $d(\mathfrak{u},\mt^{\underline{\mu}})=d(\mt,\mt^{\undla})$ and $d(s_i\mathfrak{u},\mt^{\underline{\mu}})=d(s_i\mt,\mt^{\undla}).$
		By the definition \eqref{c-coefficients. non-dege.} and Lemma \ref{Phist. lem} (3), we have
		\begin{align*}
			\mathtt{c}_{s_i\mt,\mt^{\undla}}
			=\mathtt{c}_{\mt,\mt^{\undla}}\cdot \sqrt{\mathtt{c}_{\mt}(i)},\quad
			\mathtt{c}_{s_i\mathfrak{u},\mt^{\underline{\mu}}}
			=\mathtt{c}_{\mathfrak{u},\mt^{\underline{\mu}}}\cdot\sqrt{\mathtt{c}_{\mathfrak{u}}(i)}.
		\end{align*}
		But $\mathtt{c}_{\mathfrak{u}}(i)=\mathtt{c}_{\mt}(i)$ for $i\in[n-2]$ and $\mathfrak{u}\downarrow_{n-1}=\mt.$ This completes the proof.
	\end{proof}
\end{lem}

\begin{cor}\label{cor. ct and cst}
	Let $\undla\in\mathscr{P}^{f}_{n-1}$ and $\underline{\mu}\in\mathscr{P}^{f}_{n}$.  Let $\mathfrak{a}\in\Std(\underline{\mu})$ be the unique tableau such that
	$\mathfrak{a}\downarrow_{n-1}=\mt^{\undla}.$ Then
	\begin{align}
		\frac{\mathtt{c}_{\mt,\mt^{\undla}}}{\mathtt{c}_{\mathfrak{u},\mt^{\underline{\mu}}}}
		=\frac{1}{\mathtt{c}_{\mathfrak{a},\mt^{\underline{\mu}}}}
	\end{align}
	holds for any $\mt\in\Std(\undla),$ $\mathfrak{u}\in\Std(\underline{\mu})$
	with $\mathfrak{u}\downarrow_{n-1}=\mt.$
\end{cor}
\begin{proof}
	Applying the Lemma \ref{lem. ct and cst} repeatedly and note that $\mathtt{c}_{\mt^{\undla},\mt^{\undla}}=1.$
\end{proof}
\subsection{Cyclotomic Sergeev algebras and their seminormal forms}\label{dege basic}
All the above constructions and results have analogs for the cyclotomic Sergeev algebras. For the reader's convenience, we briefly outline them in this subsection.
	\subsubsection{Affine Sergeev algebra $\mhcn$}
\label{pag:ASA}
For $n\in\mathbb{N}$, the affine Sergeev (or degenerate Hecke-Clifford) algebra $\mhcn$ is
the ${\rm R}$-superalgebra generated by even generators
$s_1,\ldots,s_{n-1},$ $x_1,\ldots,x_n$ and odd generators
$c_1,\ldots,c_n$ subject to the following relations
\begin{align}
	s_i^2=1,\quad s_is_j =s_js_i, \quad
	s_is_{i+1}s_i&=s_{i+1}s_is_{i+1}, \quad|i-j|>1,\label{braid}\\
	x_ix_j&=x_jx_i, \quad 1\leq i,j\leq n, \label{poly}\\
	c_i^2=1,c_ic_j&=-c_jc_i, \quad 1\leq i\neq j\leq n, \label{clifford}\\
	s_ix_i&=x_{i+1}s_i-(1+c_ic_{i+1}),\label{px1}\\
	s_ix_j&=x_js_i, \quad j\neq i, i+1, \label{px2}\\
	s_ic_i=c_{i+1}s_i, s_ic_{i+1}&=c_is_i,s_ic_j=c_js_i,\quad j\neq i, i+1, \label{pc}\\
	x_ic_i=-c_ix_i, x_ic_j&=c_jx_i,\quad 1\leq i\neq j\leq n.
	\label{xc}
\end{align}

For $\alpha=(\alpha_1,\ldots,\alpha_n)\in\mathbb{Z}_+^n$ and
$\beta=(\beta_1,\ldots,\beta_n)\in\mathbb{Z}_2^n$, set
$x^{\alpha}:=x_1^{\alpha_1}\cdots x_n^{\alpha},$
$c^{\beta}:=c_1^{\beta_1}\cdots c_n^{\beta_n}$ and recall that
$\supp(\beta):=\{1 \leq k \leq n:\beta_{k}=\bar{1}\},$
$|\beta|:=\sum_{i=1}^{n}\beta_i \in \mathbb{Z}_2.$
Then we have the
following.
\begin{lem}\cite[Theorem 2.2]{BK}\label{lem:PBW}
	The set $\{x^{\alpha}c^{\beta}w~|~ \alpha\in\mathbb{Z}_+^n,
	\beta\in\mathbb{Z}_2^n, w\in \mathfrak{S}_n\}$ forms an ${\rm R}$-basis of $\mhcn$.
\end{lem}

Let $\mathcal{P}_n$ \label{pag:subalg Pn} be the superalgebra generated by even generators $x_1,\ldots,x_n$ and odd generators $c_1,\ldots,c_n$.
By Lemma~\ref{lem:PBW},
$\mathcal{P}_n$ can be identified with the superalgebra generated by even generators $x_1,\ldots,x_n$ and odd generators $c_1,\ldots,c_n$ subject to relations \eqref{poly}, \eqref{clifford}, \eqref{xc}. Clifford algebra $\mathcal{C}_n$ also can be identified with the superalgebra generated by odd generators $c_1,\ldots,c_n$ with relations \eqref{clifford}.

In the rest of this subsection, we assume that
${\rm R}=\mathbb{K}$ is the algebraically closed field of characteristic different from $2.$
Following~\cite{Na2}, we define the intertwining elements as
\label{pag:dege BK intertwining elements}
\begin{align}
	\tilde{\phi}_i:=s_i(x_i^2-x^2_{i+1})+(x_i+x_{i+1})+c_ic_{i+1}(x_i-x_{i+1}),\quad
	1\leq i<n.\label{intertw}
\end{align}
We also have the intertwining elements of the following version
\label{pag:dege N intertwining elements}
\begin{align*}
	\phi_{i}:=\tilde{\phi}_i (x_{i}^{2}-x_{i+1}^{2})^{-1}&=s_i+\frac{1}{x_i-x_{i+1}}+c_ic_{i+1}\frac{1}{x_i+x_{i+1}}\\
	&\in \mathbb{K}(x_1,\ldots,x_n)\otimes_{\mathbb{K}[x_1,\ldots, x_n]} \mhcn,
\end{align*}
for $i=1,\ldots,n-1.$

For any $i=1,2,\ldots,n-1$ and $x,y \in \mathbb{K}$ satisfying $y\neq \pm x,$ let (\cite{Na2})
\label{pag:dege N phi function}
$$\phi_i(x,y):=s_{i}+\frac{1}{x-y}+\frac{c_{i}c_{i+1}}{x+y} \in \mHcn.$$

For any pair of $(x,y)\in \mathbb{K}^2$ and $y\neq \pm x$, we consider the following (degenerate) idempotency condition on $(x,y)$
\begin{align}\label{invertible dege}
	\frac{1}{(x-y)^2}+\frac{1}{(x+y)^2}=1.
\end{align}
Note that the equation \eqref{invertible dege} holds for the pair $(x,y)$ if and only if it holds for one of these four pairs $(\pm x,\pm y).$	
For any  $\iota \in \mathbb{K}$, we set
\label{pag:dege q-function and u-function}
\begin{align}
	\mathtt{q}(\iota):=\iota(\iota+1) \label{qi},\quad \mathtt{u}_{\pm}(\iota):=\pm \sqrt{\iota(\iota+1)}.
\end{align}

\subsubsection{Cyclotomic Sergeev algebra $\mhgcn$}\label{CSA}
\label{pag:CSA}
Fix $m\geq0$ and $\undQ=(Q_1,\ldots,Q_m)\in{\rm R}^m$.\label{pag:dege Q-parameters}
Recall the two defining polynomials $g$ in \eqref{deg-cyc-even} and~\eqref{deg-cyc-odd}, of levels $r=\deg g=2m$ and $2m+1$, respectively.


With the factorization in \eqref{deg-cyc-even} or~\eqref{deg-cyc-odd} and the parameter order fixed, put
\[
\mathscr P_n^g:=\begin{cases}
\mathscr P_n^m,&g\text{ as in }\eqref{deg-cyc-even},\\
\displaystyle\bigsqcup_{a+b=n}(\mathscr{SP}_a\times\mathscr P_b^m),&g\text{ as in }\eqref{deg-cyc-odd}.
\end{cases}
\]
The component labels and tableau conventions are the same as for the cyclotomic Hecke-Clifford algebras defined by \eqref{cyc-even} and~\eqref{cyc-odd}, respectively. We use $\mathcal D$, $\mathcal D_{\undla}$ and $\operatorname{Tri}(\undla)$ with these component labels, replacing $\mathscr P_n^f$ by $\mathscr P_n^g$ in their definitions.

The cyclotomic Sergeev (or degenerate cyclotomic Hecke-Clifford) algebra $\mhgcn$ is defined as $$\mhgcn:=\mhcn/\mathcal{J}_g,
$$ where $\mathcal{J}_g$ is the two-sided ideal of $\mhcn$ generated by $g(x_1)$. We call $r:=\deg(g)$ the level of $\mhgcn.$
\label{pag:dege level}
We shall denote the images of $x^{\alpha}, c^{\beta}, w$ in the cyclotomic quotient $\mhgcn$ by the same symbols.

Then we have the following due to \cite{BK}.
\begin{lem}\cite[Theorem 3.6]{BK}
	The set $\{x^{\alpha}c^{\beta}w~|~ \alpha\in\{0,1,\cdots,r-1\}^n,
	\beta\in\mathbb{Z}_2^n, w\in {\mathfrak{S}_n}\}$ forms an ${\rm R}$-basis of $\mhgcn$, where $r=\deg(g)$.
\end{lem}

We set $Q_0:=0$. Now we can also define the (degenerate) residues of boxes.
\label{pag:dege residue}
\begin{defn}Suppose  $\undla\in\mathscr{P}^{g}_{n}$  and $(i,j,l)\in \undla$, we define the residue of box $(i,j,l)$ with respect to the parameter $\undQ=(Q_1,\ldots,Q_m)\in {\rm R}^m$ as follows: \begin{equation}\label{eq:deg-residue}
		\res(i,j,l):=Q_l+j-i \in {\rm R}.
	\end{equation} If $\mathfrak{t}\in \Std(\undla)$ and $\mathfrak{t}(i,j,l)=a$, we set \begin{align}\label{res-dege}
		\res_\mathfrak{t}(a)&:=Q_l+j-i \in{\rm R},\\
		\res(\mathfrak{t})&:=(\res_\mathfrak{t}(1),\cdots,\res_\mathfrak{t}(n))\in {\rm R}^m,
	\end{align}
	then the (degenerate) $\mathtt{q}$-sequence of $\mt$ is
	\begin{align}\label{res-dege2}
		\mathtt{q}(\res(\mathfrak{t})):=(\mathtt{q}(\res_{\mathfrak{t}}(1)), \mathtt{q}(\res_{\mathfrak{t}}(2)),\ldots, \mathtt{q}(\res_{\mathfrak{t}}(n)))\in {\rm R}^m.
	\end{align}
\end{defn}

{\bf In the rest of this subsection, we assume that
${\rm R}=\mathbb{K}$ is the algebraically closed field of characteristic different from $2$.}
We shall recall the separation condition in \cite[Definition 5.8]{SW} on the choice of the parameters $\underline{Q}$ and $g$, $r:=\deg(g)$.
\begin{defn}\label{important condition2}\cite[Definition 5.8]{SW}
	Let $\undQ=(Q_1,\ldots,Q_m)\in\mathbb{K}^m$ and $\undla\in\mathscr{P}^{g}_{n}$. The parameter $\undQ$ is said to be {\em separate} with respect to $\undla$  if for any $\mathfrak{t}\in \Std(\undla)$, the $\mathtt{q}$-sequence for $\mathfrak{t}$ defined via \eqref{res-dege2} satisfies the following condition:
	$$
	\mathtt{q}(\res_{\mathfrak{t}}(k))\neq\mathtt{q}(\res_{\mathfrak{t}}(k+1)) \text{ for any $k=1,\cdots,n-1.$ }
	$$
\end{defn}
	
We write $P_n^g(\undQ)$ for the polynomial in \cite[Proposition 5.12]{SW} corresponding to the fixed presentation of $g$.\label{pag:dege Pioncare poly} Its role is parallel to that of $P_n^f(q^2,\undQ)$: the following criterion is again imposed on all shapes of size $n+1$.
	\begin{prop}\label{separate formula dege}\cite[Proposition 5.13]{SW}
		Let $n\geq 1,\,m\geq 0$,  $\undQ=(Q_1,\ldots,Q_m)\in\mathbb{K}^m$. Then the parameter $\undQ$ is separate with respect to $\undla$ for any $\undla\in\mathscr{P}^{g}_{n+1}$ if and only if $P_n^{g}(\undQ)\neq 0$.
	\end{prop}

	\begin{lem}\label{important conditionequi2}\cite[Lemma 2.27]{LS2}
		Let $\undQ=(Q_1,\ldots,Q_m)\in\mathbb{K}^m$. Suppose $P_n^{g}(\undQ)\neq 0$. Then
		\begin{enumerate}
			\item For any $\undla\in\mathscr{P}^{g}_{n},$ $\mathfrak{t}\in\Std(\undla)$, we have  $\mathtt{u}_{\pm}(\res_{\mathfrak{t}}(k))\neq 0$ for $k\notin \mathcal{D}_{\mathfrak{t}}$;
			\item For any $\undla,\underline{\mu} \in\mathscr{P}^{g}_{n},$ $\mt\in\Std(\undla), \mt' \in \Std(\underline{\mu}),$
			if $\mt\neq \mt',$ then we have $\mathtt{q}(\res(\mt))\neq \mathtt{q}(\res(\mt'))$;
			\item For any $\undla\in\mathscr{P}^{g}_{n},$ $\mathfrak{t}\in\Std(\undla)$ and $k\in [n-1]$, the four pairs $(\mathtt{u}_{\pm}(\res_{\mathfrak{t}}(k)),\mathtt{u}_{\pm}(\res_{\mathfrak{t}}(k+1)))$ do not satisfy \eqref{invertible dege} if $k,k+1$ are not in the adjacent diagonals of $\mathfrak{t}$.
		\end{enumerate}
	\end{lem}
	
	Again, when $Q_1,\ldots,Q_m$ are algebraically independent over $\Z$, and $\mathbb{K}$ is the algebraic closure of $\mathbb{Q}(Q_1,\ldots,Q_m)$, i.e., for generic degenerate cyclotomic Hecke-Clifford algebra, the condition $P_n^{g}(\undQ)\neq 0$ clearly holds.

Suppose that the condition $P_n^{g}(\undQ)\neq 0$ holds in $\mathbb{K}.$ Then for each $\undla\in\mathscr{P}^{g}_{n},$ we can associate $\undla$ with a simple $\mhgcn$-module $D(\undla),$ \label{pag:dege simple module} see \cite[Theorem 4.18]{SW}. Furthermore, we have the following result.
	\begin{thm}\cite[Theorem 5.21]{SW}
		Let $\undQ=(Q_1,Q_2,\ldots,Q_m)\in\mathbb{K}^m$.  Assume $P_n^{g}(\undQ)\neq 0$. Then  $\mhgcn$ is a (split) semisimple algebra
and $$\{D(\undla) \mid \undla\in\mathscr{P}^{g}_{n}\}$$ forms a complete set of pairwise non-isomorphic irreducible $\mhgcn$-modules. Moreover, $D(\undla)$ is of type $\texttt{M}$ if and only if $\sharp \mathcal{D}_{\undla}$ is even and is of type $\texttt{Q}$ if and only if $\sharp \mathcal{D}_{\undla}$ is odd.
	\end{thm}
	
{\bf In the rest of this subsection, we fix the parameter $\undQ=(Q_1,Q_2,\ldots,Q_m)\in \mathbb{K}^m$ and $g$ with $P_n^{g}(\undQ)\neq 0$.  Accordingly, we define the residues of boxes in the Young diagram $\undla$ via \eqref{eq:deg-residue} as well as $\res(\mathfrak{t})$ for each $\mathfrak{t}\in\Std(\undla)$ with $\undla\in\mathscr{P}^{g}_{n}$ with $m\geq 0.$}

We fix $\undla \in \mathscr{P}^{g}_{n}.$
For each $\mt\in \Std(\undla),$ recall Definition \ref{Dt,ODt,Z2ODt}.
		
\begin{defn}	
For any $i\in [n], \mt\in \Std(\undla),$ we denote
			 $$\mathtt{u}_{\mt,i}:=\mathtt{u}_{+}(\res_{\mt}(i))\in\mathbb{K}.$$
If $i\in [n-1]$, we define
\label{pag:dege coeffi cti}	
\begin{align}
	\mathfrak{c}_{\mt}(i):=1-\frac{1}{(\mathtt{u}_{\mt,i}-\mathtt{u}_{\mt,i+1})^2}
    -\frac{1}{(\mathtt{u}_{\mt,i}+\mathtt{u}_{\mt,i+1})^2}\in \mathbb{K}.
\end{align}
\end{defn}

Since $\mt \in \Std(\undla),$ $\mathtt{u}_{\mt,i}\neq \pm\mathtt{u}_{\mt,i+1}$ by Definition \ref{important condition2} and Proposition \ref{separate formula dege}, which immediately implies that $\mathfrak{c}_{\mt}(i)$ is well-defined. If $s_i$ is admissible with respect to $\mt$, i.e., $\delta(s_i\mt)=1$, then $\mathfrak{c}_{\mt}(i)\in \mathbb{K}^{*}$ by
the third part of Lemma \ref{important conditionequi2}. It is clear that $\mathfrak{c}_{\mt}(i)=\mathfrak{c}_{s_i\mt}(i).$	

By \cite[Proposition 5.6]{LS2},
The simple $\mhgcn$-supermodule $D(\undla)$ also has a $\mathbb{K}$-basis of the form
\begin{align}\label{basis of D(undla)}
\bigsqcup_{\mt\in \Std(\undla)}\Biggl\{c^{\beta_{\mt}}c^{\alpha_{\mt}}v_{\mt}\biggm|\begin{matrix}\beta_{\mt} \in \Z_2([n]\setminus \mathcal{D}_{\mt})  \\
				\alpha_{\mt}\in \Z_2(\mathcal{OD}_{\mt})
			\end{matrix}\Biggr\}
\end{align}
satisfying properties analogous to Proposition \ref{actions of generators on L basis}.

\subsubsection{Primitive idempotents of $\mhgcn$}
Now we shall define the primitive idempotents in the degenerate case.
\begin{defn}\cite[Definition 5.7]{LS2}

           For $k\in[n]$, let
           $$\mathtt{U}(k):=\{ \mathtt{u}_{\pm}(\res_{\ms}(k)) \mid \ms \in \Std(\mathscr{P}^{g}_{n}) \}.$$
			For any ${\rm T}=(\mt, \alpha_{\mt}, \beta_{\mt})\in {\rm Tri}_{\bar{0}}(\undla),$ we define
            \label{pag:dege primitive idempotents}
			\begin{align}\label{definition of primitive idempotents. dege}
				 \mathcal{F}_{\rm T}:=\left(c^{\alpha_{\mt}}\gamma_{\mt}(c^{\alpha_{\mt}})^{-1} \right) \cdot \left(\prod_{k=1}^{n}\prod_{\mathtt{u}\in \mathtt{U}(k)\atop \mathtt{u}\neq \mathtt{u}_{+}(\res_{\mt}(k))}\frac{\nu_{\beta_{\mt}}(k)x_k-\mathtt{u}}{\mathtt{u}_{+}(\res_{\mt}(k))-\mathtt{u}}\right)\in \mhgcn.
\end{align}
We define
\begin{align}
				 \mathcal{F}_{\undla}&:=\sum_{{\rm T}\in {\rm Tri}_{\bar{0}}(\undla)}  \mathcal{F}_{\rm T},
			\end{align}
and the left ideal of $\mhgcn$
    \label{pag:dege simple blocks}
	\begin{align}
				 \mathcal{B}_{\undla}:=\mhgcn   \mathcal{F}_{\undla}\subseteq \mhgcn.
			\end{align}
		\end{defn}

The following is an analog of Lemma \ref{idempotent action. non-dege} in the degenerate case, whose proof is completely similar to \cite[Lemma 4.14]{LS2}.
\begin{lem}\label{idempotent action. dege}	
Let ${\rm T}=(\mt, \alpha_{\mt}, \beta_{\mt})\in {\rm Tri}_{\bar{0}}(\undla),$
${\rm S}=(\ms, \alpha_{\ms}', \beta_{\ms}')\in {\rm Tri}(\mathscr{P}^{g}_{n})$.
We have
\begin{align*}
\mathcal{F}_{\rm T}\cdot c^{\beta_{\ms}^{'}} c^{\alpha_{\ms}^{'}} v_{\ms}=
\begin{cases}
c^{\beta_{\ms}^{'}} c^{\alpha_{\ms}^{'}}  v_{\ms}, & \text{ if } d_{\undla}=0 \text{ and } {\rm S}={\rm T}, \\
c^{\beta_{\ms}^{'}} c^{\alpha_{\ms}^{'}}  v_{\ms}, & \text{ if $d_{\undla}=1$ and } {\rm S}={\rm T}_{a}\text{ for some $a\in \mathbb{Z}_2,$  } \\
0, & \text{ otherwise.}
\end{cases}
\end{align*}
\end{lem}

The following is the analog of Theorem \ref{primitive idempotents}.
\begin{thm}\label{primitive idempotents. dege}\cite[Theorem 5.8]{LS2}
			Suppose $P_n^{g}(\undQ)\neq 0$. Then we have the following.
			
			(a) $\{\mathcal{F}_{\rm T} \mid {\rm T}\in {\rm Tri}_{\bar{0}}(\mathscr{P}^{g}_{n})\}$ is a complete set of (super) primitive orthogonal idempotents of $\mhgcn.$
			
			(b) $\{\mathcal{F}_{\undla} \mid \undla \in \mathscr{P}^{g}_{n} \}$ is a complete set of (super) primitive central idempotents of $\mhgcn.$
		\end{thm}
An alternative construction of $\mathcal{F}_{\rm T} $ can be found in \cite{KMS}.

\subsubsection{Seminormal bases of $\mhgcn$}
 In this subsection, we fix $\undla\in\mathscr{P}^{g}_{n}$.
 Next we sketch the construction of seminormal bases for block $\mathcal{B}_{\undla}$ of $\mhgcn$.

 With slight modifications, for $\ms,\mt \in \Std(\undla),$ we can similarly define  the element $\phi_{\ms,\mt}\in \mhgcn$ \label{pag:phist and cst} and the coefficient $\mathfrak{c}_{\ms,\mt}\in\mathbb{K}^*$ for $\ms,\mt \in \Std(\undla),$ which are the degenerate analogs of $\eqref{Phist}$
 and \eqref{c-coefficients. non-dege.} respectively (\cite[Definition 5.11]{LS2}).
Then we can define the seminormal bases \label{pag:dege seminormal basis}	
$\{\mathfrak{f}_{{\rm S},{\rm T}}^{\mathfrak{w}}\}$
and $\{\mathfrak{f}_{{\rm S},{\rm T}}\}$
in the degenerate case, see \cite[Definition 5.12]{LS2} for details.

Let
\label{pag:dege cT}
$\mathfrak{c}_{\rm T}^{\mathfrak{w}}=\mathfrak{c}_{\mt}^{\mathfrak{w}}:=(\mathfrak{c}_{\mt,\mathfrak{w}})^2\in \mathbb{K}^*.$
We have the following result as a degenerate version of Theorem \ref{seminormal basis}.
\begin{thm}\label{seminormal basis. dege}\cite[Theorem 5.13]{LS2}
			Suppose that $P_n^{g}(\undQ)\neq 0$.
            We fix $\mathfrak{w}\in\Std(\undla)$. Then the following two sets
			\begin{align}\label{deg seminormal1}
			\left\{ \mathfrak{f}_{{\rm S},{\rm T}}^\mathfrak{w} \Biggm|
			{\rm S}=(\ms, \alpha_{\ms}', \beta_{\ms}')\in {\rm Tri}_{\bar{0}}(\undla),
			{\rm T}=(\mt, \alpha_{\mt}, \beta_{\mt})\in {\rm Tri}(\undla)
			\right\}
		\end{align} and \begin{align}\label{deg seminormal2}
			\left\{ \mathfrak{f}_{{\rm S},{\rm T}} \Biggm|
		{\rm S}=(\ms, \alpha_{\ms}', \beta_{\ms}')\in {\rm Tri}_{\bar{0}}(\undla),
		{\rm T}=(\mt, \alpha_{\mt}, \beta_{\mt})\in {\rm Tri}(\undla)
		\right\}
			\end{align}  form two $\mathbb{K}$-bases of the block $\mathcal B_{\undla}$ of $\mhgcn$.
			
			Moreover, for ${\rm S}=(\ms, \alpha_{\ms}', \beta_{\ms}')\in {\rm Tri}_{\bar{0}}(\undla),
			{\rm T}=(\mt, \alpha_{\mt}, \beta_{\mt})\in {\rm Tri}(\undla),$ we have \begin{equation}\label{dege-fst and re. fst}
			\mathfrak{f}_{{\rm S},{\rm T}}
			=\frac{\mathfrak{c}_{\ms,\mt}}{\mathfrak{c}_{\ms,\mathfrak{w} }\mathfrak{c}_{\mathfrak{w},\mt }} \mathfrak{f}_{{\rm S},{\rm T}}^\mathfrak{w}\in \mathcal{F}_{\rm S}\mhgcn \mathcal{F}_{\rm T}.
			\end{equation} The multiplications of basis elements in \eqref{deg seminormal1} are given as follows.
			
			(1) Suppose $d_{\undla}=0.$  Then for any
			${\rm S}=(\ms, \alpha_{\ms}', \beta_{\ms}'),
			{\rm T}=(\mt, \alpha_{\mt}, \beta_{\mt}),
			{\rm U}=(\mfku,\alpha_{\mfku}^{''},\beta_{\mfku}^{''}),
			{\rm V}=(\mfkv,\alpha_{\mfkv}^{'''},\beta_{\mfkv}^{'''})\in {\rm Tri}(\undla),$ we have
			\begin{align}\label{deg multiplication1}
				\mathfrak{f}_{{\rm S},{\rm T}}^\mathfrak{w} f_{{\rm U},{\rm V}}^\mathfrak{w}
				=\delta_{{\rm T},{\rm U}} \mathfrak{c}_{\rm T}^\mathfrak{w} \mathfrak{f}_{{\rm S},{\rm V}}^\mathfrak{w}.
			\end{align}

			(2) Suppose $d_{\undla}=1.$ Then for any $a,b\in \mathbb{Z}_2$ and
			\begin{align*}
				{\rm S}&=(\ms, \alpha_{\ms}', \beta_{\ms}')\in {\rm Tri}_{\bar{0}}(\undla), \quad
				{\rm T}_{a}=(\mt, \alpha_{\mt,a}, \beta_{\mt})\in {\rm Tri}_{a}(\undla),\nonumber\\
				{\rm U}&=(\mfku,\alpha_{\mfku}^{''},\beta_{\mfku}^{''})\in {\rm Tri}_{\bar{0}}(\undla), \quad
				{\rm V}_{b}=(\mfkv,{\alpha_{\mfkv,b}^{'''}},\beta_{\mfkv}^{'''})\in {\rm Tri}_{b}(\undla),\nonumber
			\end{align*} we have
			\begin{align}\label{deg multiplication2}
				\mathfrak{f}_{{\rm S},{\rm T}_{a}}^\mathfrak{w} f_{{\rm U},{\rm V}_{b}}^\mathfrak{w}
				=\delta_{{\rm T}_{\bar{0}},{\rm U}} \mathfrak{c}_{\rm T}^\mathfrak{w} \mathfrak{f}_{{\rm S},{\rm V}_{a+b}}^\mathfrak{w}.
			\end{align}
		\end{thm}

Finally, we have the action formulae of generators on seminormal basis, in which
we also need supplementary terms $f(i,{\rm T},{\rm S})$ for $i\in[n-1]$, ${\rm T}\in{\rm Tri}_{\bar0}(\undla)$ and ${\rm S}\in{\rm Tri}(\undla)$, the degenerate analogs of $F(i,{\rm T},{\rm S})$ in Definition~\ref{F(i,T,S)}. Their definitions and the complete action formulae are given in \cite[Definition 5.15 and Proposition 5.16]{LS2}.

\begin{rem}\label{prop. qt'}
We can deduce the analog of Proposition \ref{combina. formulae of c} and its corollaries for the degenerate case with suitable modifications. In particular, let $\mt\in\Std(\undla),$ $\undla\in\mathscr{P}^{g}_{n}$,
then the following element (a degenerate analog of $\kappa_{\undla}$)
\label{pag:qt'}
\begin{align*}
	\kappa_{\undla}^{\mathrm{dege}}
	:=\prod\limits_{k=1}^{n}\frac{\prod\limits_{\beta\in\Rem(\mt\downarrow_{k-1})\setminus \mathcal{D}}\left(\mathtt{q}(\res_{\mt}(k))-\mathtt{q}(\res(\beta))\right)}{
		\prod\limits_{\alpha\in\Add(\mt\downarrow_{k-1})\setminus \{\mt^{-1}(k)\}}\left(\mathtt{q}(\res_{\mt}(k))-\mathtt{q}(\res(\alpha))\right)}\in\mathbb{K}^*
\end{align*}
is independent of the choice of $\mt.$
\end{rem}

\section{The embedding from $\mathcal{H}^f_c(n-1)$ to $\mHfcn$}\label{embedding}
\subsection{Restriction of seminormal bases of cyclotomic Hecke-Clifford algebras}
Throughout this subsection, we shall fix the parameter $\undQ=(Q_1,Q_2,\ldots,Q_m)\in(\mathbb{K}^*)^m$ and $f$ with $P_n^{f}(q^2,\undQ)\neq 0$. Accordingly, we define the residues of boxes in the Young diagram $\undla$ via \eqref{eq:residue} as well as $\res(\mathfrak{t})$ for each $\mathfrak{t}\in\Std(\undla)$ with $\undla\in\mathscr{P}^{f}_{n}$ with $m\geq 0.$

Let $\iota_{n-1}: \mathcal{H}^f_c(n-1)\rightarrow \mHfcn$ be the natural embedding which maps $X_j$ to $X_j$, $C_j$ to $C_j$ and $T_i$ to $T_i$ for any $i\in [n-2]$ and $j\in [n-1]$. In this subsection, we aim to compute the images of the $(n-1)$-seminormal basis under $\iota_{n-1}$ via the $n$-seminormal basis, which can be viewed as the branching rule for seminormal bases.

Fix $\undla\in\mathscr{P}^{f}_{n}$. For any ${\rm T}=(\mt, \alpha_{\mt}, \beta_{\mt})\in {\rm Tri}_{\bar{0}}(\undla),$ we recall the primitive idempotent $F_{\rm T}$ in \eqref{definition of primitive idempotents. non-dege} and let $\mathtt{q}_{\mt,i}:=\mathtt{q}(\res_{\mt}(i)),$ $i\in [n].$ \label{pag:nondege q-values}
For any $\ms,\mt\in \Std(\undla),$ we recall that $d(\ms,\mt)\in \mathfrak{S}_n$ is the unique element
such that $\ms=d(\ms,\mt)\mt.$
Recall the intertwining elements $\tilde{\Phi}_i$ and $z_i$ in \eqref{intertwinNon-dege}.
For any reduced expression $d(\ms,\mt)=s_{k_p}\cdots s_{k_1},$ we denote
$$\tilde{\Phi}_{\ms,\mt}:=\tilde{\Phi}_{d(\ms,\mt)}=\tilde{\Phi}_{k_p}\cdots \tilde{\Phi}_{k_1}\in \mHfcn,$$
$$\mathtt{q}_{\ms,\mt}:=\prod_{1\leq i \leq p}\left(\mathtt{q}_{s_{k_{i-1}}\cdots s_{k_1}\mt,k_i}-\mathtt{q}_{s_{k_{i-1}}\cdots s_{k_1}\mt,k_{i}+1}\right)^2 \in \mathbb{K}^*.$$
Then we have the following properties.
\begin{lem}\label{properties of tildePhi}
Let $\undla\in\mathscr{P}^{f}_{n}$ and ${\rm T}=(\mt, \alpha_{\mt}, \beta_{\mt})\in {\rm Tri}_{\bar{0}}(\undla).$
\begin{enumerate}
\item $F_{\rm T} \tilde{\Phi}_{\mt,\ms}=\tilde{\Phi}_{\mt,\ms} F_{d(\ms,\mt) \cdot {\rm T}},$
where $d(\ms,\mt) \cdot {\rm T}:=(\ms,d(\ms,\mt)\cdot \alpha_{\mt},d(\ms,\mt)\cdot \beta_{\mt})\in {\rm Tri}_{\bar{0}}(\undla).$
\item $\tilde{\Phi}_{\ms,\mt}v_{\mt}=\mathtt{q}_{\ms,\mt}\mathtt{c}_{\ms,\mt}v_{\ms}.$
\end{enumerate}
\end{lem}
\begin{proof}
\begin{enumerate}
\item It follows from the relations \eqref{Xinter} and \eqref{Cinter}.
\item Notice that $z_i^2 v_{\mt}=(X_i+X^{-1}_i-X_{i+1}-X^{-1}_{i+1})^2 v_{\mt}=(\mathtt{q}_{\mt,i}-\mathtt{q}_{\mt,i+1})^2 v_{\mt}$ for any $i\in [n],$ by using \eqref{X eigenvalues}. It follows from Lemma \ref{Phist. lem} (1) and the equation \eqref{universal-Phi}.
\end{enumerate}
\end{proof}
For $\undla\in\mathscr{P}^{f}_{n}$, we denote $${\rm T}^{\undla}:=(\mt^{\undla},0,0)\in {\rm Tri}_{\bar{0}}(\undla).
 $$

For the rest of this subsection, $f$ is as in \eqref{cyc-even} or \eqref{cyc-odd}.
\begin{prop}\label{eq.intertwining and f}
Let $\undla\in\mathscr{P}^{f}_{n}$, ${\rm S}=(\ms, \alpha_{\ms}',\beta_{\ms}')\in {\rm Tri}_{\bar{0}}(\undla)$ and ${\rm T}=(\mt, \alpha_{\mt}, \beta_{\mt})\in {\rm Tri}(\undla)$. Then we have the following

\begin{align*}
f_{{\rm S},{\rm T}}^{\mathfrak{t}^{\undla}}=\mathtt{q}_{\ms,\mathfrak{t}^{\undla}}^{-1}\mathtt{q}_{\mathfrak{t}^{\undla},\mt}^{-1}
C^{\beta_{\ms}'}C^{\alpha_{\ms}'} \tilde{\Phi}_{\ms,\mathfrak{t}^{\undla}} F_{{\rm T}^{\undla}} \tilde{\Phi}_{\mathfrak{t}^{\undla},\mt} (C^{\alpha_{\mt}})^{-1}(C^{\beta_{\mt}})^{-1}.
\end{align*}

\end{prop}
\begin{proof}
It suffices to check that both sides act as the same linear operator on all $\mathbb{D}(\underline{\mu})$, $\underline{\mu}\in \mathscr{P}^{f}_{n}.$
We first assume $d_{\undla}=1$.
Let
$$\widetilde{f_{{\rm S},{\rm T}}^{\mathfrak{t}^{\undla}}}:=C^{\beta_{\ms}'}C^{\alpha_{\ms}'} \tilde{\Phi}_{\ms,\mathfrak{t}^{\undla}} F_{{\rm T}^{\undla}} \tilde{\Phi}_{\mathfrak{t}^{\undla},\mt} (C^{\alpha_{\mt}})^{-1}(C^{\beta_{\mt}})^{-1}.$$
Suppose ${\rm T}={\rm U}_{a}=(\mathfrak{u},\alpha_{\mathfrak{u},a}'', \beta_{\mathfrak{u}}'')$ for some ${\rm U}\in{\rm Tri}_{\bar{0}}(\undla)$ and $i:=\max \left(\mathcal{OD}_{\mt^{\undla}}\right)$.
For $b\in \mathbb{Z}_2,$ we have
\begin{align*}
\widetilde{f_{{\rm S},{\rm U}_a}^{\mathfrak{t}^{\undla}}}\cdot C^{\beta_{\mathfrak{u}}''}C^{\alpha_{\mathfrak{u},b}''}v_{\mathfrak{u}}
&=C^{\beta_{\ms}'}C^{\alpha_{\ms}'} \tilde{\Phi}_{\ms,\mathfrak{t}^{\undla}} F_{{\rm T}^{\undla}} \tilde{\Phi}_{\mathfrak{t}^{\undla},\mathfrak{u}} (C^{\alpha_{\mathfrak{u},a}''})^{-1}(C^{\alpha_{\mathfrak{u},b}''})v_{\mathfrak{u}}\\
&=C^{\beta_{\ms}'}C^{\alpha_{\ms}'} \tilde{\Phi}_{\ms,\mathfrak{t}^{\undla}} F_{{\rm T}^{\undla}}\tilde{\Phi}_{\mathfrak{t}^{\undla},\mathfrak{u}} C_{d(\mathfrak{u},\mathfrak{t}^{\undla})(i)}^{a+b}v_{\mathfrak{u}}\\
&=C^{\beta_{\ms}'}C^{\alpha_{\ms,a+b}'} \tilde{\Phi}_{\ms,\mathfrak{t}^{\undla}} F_{{\rm T}^{\undla}}\tilde{\Phi}_{\mathfrak{t}^{\undla},\mathfrak{u}}v_{\mathfrak{u}}\\
&=
\mathtt{q}_{\ms,\mathfrak{t}^{\undla}}\mathtt{q}_{\mathfrak{t}^{\undla},\mathfrak{u}}\mathtt{c}_{\ms,\mathfrak{t}^{\undla}}\mathtt{c}_{\mathfrak{t}^{\undla},\mathfrak{u}}
C^{\beta_{\ms}'}C^{\alpha_{\ms,a+b}'} v_{\ms}\\
&=
\mathtt{q}_{\ms,\mathfrak{t}^{\undla}}\mathtt{q}_{\mathfrak{t}^{\undla},\mathfrak{u}} f_{{\rm S},{\rm U}_a}^{\mathfrak{t}^{\undla}}  C^{\beta_{\mathfrak{u}}''}C^{\alpha_{\mathfrak{u},b}''} v_{\mathfrak{u}},
\end{align*}
where in the third equation we have used  \eqref{Cinter} and Lemma \ref{change sign in diagonal}, in the fourth equation, we have used \eqref{SNB. eq3} and Lemma \ref{properties of tildePhi} (ii).
Using Lemma \ref{properties of tildePhi} (i), we have
\begin{align*}
\widetilde{f_{{\rm S},{\rm U}_a}^{\mathfrak{t}^{\undla}}}
=C^{\beta_{\ms}'}C^{\alpha_{\ms}'}F_{(\ms,0,0)} \tilde{\Phi}_{\ms,\mathfrak{t}^{\undla}}F_{{\rm T}^{\undla}} \tilde{\Phi}_{\mathfrak{t}^{\undla},\mt} F_{(\mathfrak{u},0,0)} (C^{\alpha_{\mathfrak{u},a}''})^{-1}(C^{\beta_{\mathfrak{u}}''})^{-1},
\end{align*}
hence, we deduce that $\widetilde{f_{{\rm S},{\rm U}_a}^{\mathfrak{t}^{\undla}}}\cdot C^{\beta_{\mfkv}^{'''}} C^{\alpha_{\mfkv}^{'''}} v_{\mfkv}=0,$
for any $V=(\mfkv,\alpha_{\mfkv}^{'''},\beta_{\mfkv}^{'''}) \in {\rm Tri}(\mathscr{P}^{f}_{n})\setminus \{{\rm U}_b \mid b\in \mathbb{Z}_2\}$ by \eqref{SNB. eq3}.
Thus it follows from \eqref{SNB. eq3} that
$$f_{{\rm S},{\rm T}}^{\mathfrak{t}^{\undla}}=f_{{\rm S},{\rm U}_a}^{\mathfrak{t}^{\undla}}=
\mathtt{q}_{\ms,\mathfrak{t}^{\undla}}^{-1}\mathtt{q}_{\mathfrak{t}^{\undla},\mathfrak{u}}^{-1}\widetilde{f_{{\rm S},{\rm U}_a}^{\mathfrak{t}^{\undla}}}=\mathtt{q}_{\ms,\mathfrak{t}^{\undla}}^{-1}\mathtt{q}_{\mathfrak{t}^{\undla},\mathfrak{t}}^{-1}\widetilde{f_{{\rm S},{\rm T}}^{\mathfrak{t}^{\undla}}}.$$
For $d_{\undla}=0$, the proof is similar.
\end{proof}

For any $\alpha=(\alpha_1,\ldots,\alpha_{n-1})\in \mathbb{Z}_2^{n-1}$, we abuse the same notation $\alpha=(\alpha_1,\ldots,\alpha_{n-1},\bar{0})\in\mathbb{Z}_2^{n}$ for the obvious embedding.

\begin{prop}\label{eq.intertwining and f 2}
	Let $\undla\in\mathscr{P}^{f}_{n-1},\,\underline{\mu}\in\mathscr{P}^{f}_{n}$, ${\rm S}=(\ms, \alpha_{\ms}', \beta_{\ms}')\in {\rm Tri}_{\bar{0}}(\undla)$ and ${\rm T}=(\mt, \alpha_{\mt}, \beta_{\mt})\in {\rm Tri}(\undla)$. Suppose $ {\rm U}=(\mathfrak{u},\alpha_{\mathfrak{u}}'', \beta_{\mathfrak{u}}''),\,{\rm P}=(\mathfrak{p},0,\beta_{\mathfrak{p}}'''')\in {\rm Tri}_{\bar{0}}(\underline{\mu})$ and ${\rm V}=(\mathfrak{v},\alpha_{\mathfrak{v}}''', \beta_{\mathfrak{v}}''')\in {\rm Tri}(\underline{\mu})$ such that the following hold
	\begin{enumerate}
	\item $\mathfrak{u}\downarrow_{n-1}=\mathfrak{s},\,\alpha_{\mathfrak{u}}''\downarrow_{n-1}=\alpha_{\mathfrak{s}}',\,\beta_{\mathfrak{u}}''\downarrow_{n-1}=\beta_{\mathfrak{s}}';$
	\item $\mathfrak{v}\downarrow_{n-1}=\mathfrak{t},\,\alpha_{\mathfrak{v}}'''\downarrow_{n-1}=\alpha_{\mathfrak{t}},\,\beta_{\mathfrak{v}}'''\downarrow_{n-1}=\beta_{\mathfrak{t}};$
	\item  $\mathfrak{p}\downarrow_{n-1}=\mathfrak{t}^{\undla},\,\beta_{\mathfrak{p}}''''\downarrow_{n-1}=0;$
	\item $\beta_{\mathfrak{u}}''(n)=\beta_{\mathfrak{p}}''''(n)=\beta_{\mathfrak{v}}'''(n).$
	\end{enumerate}Then we have
		\begin{align*}
			f_{{\rm U},{\rm V}}^{\mathfrak{t}^{\underline{\mu}}}
=(-1)^{\delta_{\beta_{\mathfrak{p}}''''}(n)(|\alpha_{\ms}'|+|\alpha_{\mt}|)}\mathtt{q}_{\ms,\mathfrak{t}^{\undla}}^{-1}\mathtt{q}_{\mathfrak{t}^{\undla},\mt}^{-1}
 \frac{\mathtt{c}_{\mathfrak{u},\mt^{\underline{\mu}}}\mathtt{c}_{\mt^{\underline{\mu}},\mathfrak{v}}}{\mathtt{c}_{\ms,\mt^{\undla}}\mathtt{c}_{\mt^{\undla},\mt}}
  C^{\beta_{\ms}'}C^{\alpha_{\ms}'} \tilde{\Phi}_{\ms,\mathfrak{t}^{\undla}}F_{\rm P}\tilde{\Phi}_{\mathfrak{t}^{\undla},\mt} (C^{\alpha_{\mt}})^{-1}(C^{\beta_{\mt}})^{-1}.
		\end{align*}
\end{prop}

\begin{proof}
We shall consider the following four cases
		\begin{enumerate}
			\item $d_{\undla}=0$, $d_{\underline{\mu}}=0,$
			\item $d_{\undla}=0$, $d_{\underline{\mu}}=1,$
			\item $d_{\undla}=1$, $d_{\underline{\mu}}=0,$
			\item $d_{\undla}=1$, $d_{\underline{\mu}}=1.$
			\end{enumerate}
	In each case, the proof is similar to Proposition \ref{eq.intertwining and f}. Hence we omit the proof.
	\end{proof}

For $\undla\in\mathscr{P}^{f}_{n-1},$, we define
\begin{align}
\delta_{\lambda^{(0)}}:=\begin{cases}
	0,&\qquad \text{if the last part of $\lambda^{(0)}$ is $1$,}\\
	1,&\qquad\text{otherwise.}
	\end{cases}
\end{align}
\begin{lem}\label{cor.iota fTT}
	Let $\undla\in\mathscr{P}^{f}_{n-1}$, and ${\rm T}=(\mt,\alpha_\mt,\beta_\mt)\in {\rm Tri}_{\bar{0}}(\undla).$
	
	(1) If $d_{\undla}=0,$ then
	\begin{align}\label{eq1.cor.iota fTT}
		\iota_{n-1}(F_{{\rm T}})
		=\delta_{\lambda^{(0)}}
		F_{{\rm P}}
		+\sum_{\substack{\underline{\mu}\in \mathscr{P}^{f}_{n}, d_{\underline{\mu}}=0 \\ {\rm U}=(\mathfrak{u},\alpha_\mt,\beta_{\mathfrak{u}}'')\in
				{\rm Tri}(\underline{\mu}) \\
				\mathfrak{u}\downarrow_{n-1}=\mt,\beta_{\mathfrak{u}}''\downarrow_{n-1}=\beta_\mt}}
		F_{{\rm U}}
	\end{align}
	where ${\rm P}=(\mathfrak{p},\alpha_\mt,\beta_\mt) \in {\rm Tri}_{\bar{0}}(\underline{\nu})$ $(\underline{\nu}\in\mathscr{P}^{f}_{n}, d_{\underline{\nu}}=1)$ such that $\mathfrak{p}$ is the unique tableau satisfying $\mathfrak{p}\downarrow_{n-1}=\mt$ and $n\in\mathcal{D}_{\mathfrak{p}},$ if $\delta_{\lambda^{(0)}}=1.$
	
	(2) If $d_{\undla}=1,$ then
	\begin{align}\label{eq2.cor.iota fTT}
		\iota_{n-1}(F_{{\rm T}})
		=\delta_{\lambda^{(0)}}
	(F_{{\rm U}}+F_{{\rm U}'})
		+\sum_{\substack{\underline{\mu}\in \mathscr{P}^{f}_{n}, d_{\underline{\mu}}=1 \\ {\rm P}=(\mathfrak{p},\alpha_\mt,\beta_{\mathfrak{p}}'''')\in
				{\rm Tri}_{\bar{0}}(\mu) \\ \mathfrak{p}\downarrow_{n-1}=\mt, \beta_{\mathfrak{p}}''''\downarrow_{n-1}=\beta_\mt}}
		F_{{\rm P}}
	\end{align}
	where ${\rm U}=(\mathfrak{u},\alpha_\mt,\beta_\mt),{\rm U}'=(\mathfrak{u},\alpha_{\mt,\overline{1}},\beta_{\mt}) \in {\rm Tri}(\underline{\nu})$ $(\underline{\nu}\in\mathscr{P}^{f}_{n}, d_{\underline{\nu}}=0)$ such that $\mathfrak{u}\in \Std(\underline{\nu})$ is the unique tableau satisfying $\mathfrak{u}\downarrow_{n-1}=\mt$ and $n\in\mathcal{D}_{\mathfrak{u}},$ if $\delta_{\lambda^{(0)}}=1.$
\end{lem}

\begin{proof}
{\bf Case 1: $d_{\undla}=0.$}  Then for any ${\rm T}=(\mt,\alpha_\mt,\beta_\mt)\in {\rm Tri}(\undla),$
we can write
\begin{align}\label{eq.image of f}
\iota_{n-1}(F_{\rm T})
=\sum_{\substack{\underline{\mu}\in \mathscr{P}^{f}_{n}, d_{\underline{\mu}}=0 \\ {\rm U},{\rm V}\in {\rm Tri}(\underline{\mu})}}
\beta_{{\rm U},{\rm V}}^{{\rm T}} f_{{\rm U},{\rm V}}
+\sum_{\substack{\underline{\mu}\in \mathscr{P}^{f}_{n}, d_{\underline{\mu}}=1 \\ {\rm P},{\rm Q}\in {\rm Tri}_{\bar{0}}(\underline{\mu}) \\ b\in \mathbb{Z}_2}}
\beta_{{\rm P},{\rm Q}_b}^{{\rm T}} f_{{\rm P},{\rm Q}_b}
\end{align}
where
$
{\rm U}=(\mathfrak{u},\alpha_{\mathfrak{u}}'', \beta_{\mathfrak{u}}''), {\rm V}=(\mathfrak{v},\alpha_{\mathfrak{v}}''', \beta_{\mathfrak{v}}'''),
{\rm P}=(\mathfrak{p},\alpha_{\mathfrak{p}}'''', \beta_{\mathfrak{p}}''''), {\rm Q}=(\mathfrak{q},\alpha_{\mathfrak{q}}''''', \beta_{\mathfrak{q}}''''')\in {\rm Tri}_{\bar{0}}(\mathscr{P}^{f}_{n}).
$
For any $k\in[n-1]$, multiplying both sides of \eqref{eq.image of f} on the left by $X_k$ gives
\begin{align}\label{eq.X act on iotaf}
\mathtt{b}_{\mt,k}^{-\nu_{\beta_\mt}(k)}\iota_{n-1}(F_{\rm T})
=\sum_{\substack{\underline{\mu}\in \mathscr{P}^{f}_{n}, d_{\underline{\mu}}=0 \\ {\rm U},{\rm V}\in {\rm Tri}(\underline{\mu})}}
\beta_{{\rm U},{\rm V}}^{{\rm T}} \mathtt{b}_{\mathfrak{u},k}^{-\nu_{\beta_{\mathfrak{u}}''}(k)} f_{{\rm U},{\rm V}}
+\sum_{\substack{\underline{\mu}\in \mathscr{P}^{f}_{n}, d_{\underline{\mu}}=1 \\ {\rm P},{\rm Q}\in {\rm Tri}_{\bar{0}}(\underline{\mu}) \\ b\in \mathbb{Z}_2}}
\beta_{{\rm P},{\rm Q}_b}^{{\rm T}} \mathtt{b}_{\mathfrak{p},k}^{-\nu_{\beta_{\mathfrak{p}}''''}(k)} f_{{\rm P},{\rm Q}_b}
\end{align}
by Proposition \ref{generators action on seminormal basis} (1).
Combining \eqref{eq.image of f} and \eqref{eq.X act on iotaf}, we deduce that  $\beta_{{\rm U},{\rm V}}^{{\rm T}}=0$
unless $\mathfrak{u}\downarrow_{n-1}=\mt,$ $\beta_{\mathfrak{u}}''\downarrow_{n-1}=\beta_\mt$
and $\beta_{{\rm P},{\rm Q}_b}^{{\rm T}}=0$
unless $\mathfrak{p}\downarrow_{n-1}=\mt,$ $\beta_{\mathfrak{p}}''''\downarrow_{n-1}=\beta_\mt$ and $n\in \mathcal{D}_{\mathfrak{p}},$
which implies that $\beta_{\mathfrak{p}}''''=\beta_\mt.$
Similarly, by using Proposition \ref{generators action on seminormal basis} (1), we have
$\beta_{{\rm U},{\rm V}}^{{\rm T}}=0$
unless $\mathfrak{v}\downarrow_{n-1}=\mt$, $\beta_{\mathfrak{v}}'''\downarrow_{n-1}=\beta_\mt$
and $\beta_{{\rm P},{\rm Q}_b}^{{\rm T}}=0$
unless $\mathfrak{q}\downarrow_{n-1}=\mt,$ $\beta_{\mathfrak{q}}'''''=\beta_\mt.$

It follows from \eqref{definition of primitive idempotents. non-dege} and \eqref{comm. form of FT} that
$C^{\alpha_{\mt}} \gamma_{\mt} (C^{\alpha_{\mt}})^{-1} F_{\rm T}=F_{\rm T}=F_{\rm T}C^{\alpha_{\mt}} \gamma_{\mt} (C^{\alpha_{\mt}})^{-1}$, hence we have
\begin{align*}
C^{\alpha_{\mt}} \gamma_{\mt} (C^{\alpha_{\mt}})^{-1} \iota_{n-1}(F_{\rm T})=\iota_{n-1}(F_{\rm T})
=\iota_{n-1}(F_{\rm T})C^{\alpha_{\mt}} \gamma_{\mt} (C^{\alpha_{\mt}})^{-1}.
\end{align*}
By Lemma \ref{lem:clifford rep}, the Clifford idempotents $\{C^{\alpha_\mt}\gamma_{\mt}(C^{\alpha_\mt})^{-1}\mid\alpha_{\mt}\in\mathbb Z_2(\mathcal{OD}_{\mt})_{\bar0}\}$ are pairwise orthogonal. Consequently, we deduce the following.

(i) $\beta_{{\rm U},V}^{{\rm T}}=0$
unless  $\alpha_{\mathfrak{u}}''\downarrow_{n-1}=\alpha_\mt,$ hence $\alpha_{\mathfrak{u}}''=\alpha_\mt$  (since $n\notin \mathcal{D}_{\mathfrak{u}});$

(ii) $\beta_{{\rm U},V}^{{\rm T}}=0$
unless $\alpha_{\mathfrak{v}}'''\downarrow_{n-1}=\alpha_\mt,$ hence $\alpha_{\mathfrak{v}}'''=\alpha_\mt$ (since $n\notin \mathcal{D}_{\mathfrak{v}});$

(iii) $\beta_{{\rm P},{\rm Q}_b}^{{\rm T}}=0$
unless $\alpha_{\mathfrak{p}}''''\downarrow_{n-1}=\alpha_\mt,$ hence $\alpha_{\mathfrak{p}}''''=\alpha_\mt$ (since $\alpha_{\mathfrak{p}}''''\in \mathbb{Z}_2(\mathcal{OD}_{\mathfrak{p}})_{\bar{0}});$

(iv) $\beta_{{\rm P},{\rm Q}_b}^{{\rm T}}=0$
unless $\alpha_{\mathfrak{q}}'''''\downarrow_{n-1}=\alpha_\mt,$ hence $\alpha_{\mathfrak{q}}'''''=\alpha_\mt$ (since $\alpha_{\mathfrak{q}}'''''\in \mathbb{Z}_2(\mathcal{OD}_{\mathfrak{q}})_{\bar{0}});$

(v) $\beta_{{\rm P},{\rm Q}_b}^{{\rm T}}=0$ unless $b=\overline{0}$ (since the parity $|\iota_{n-1}(F_{\rm T})|=\bar{0}$).

To sum up, we have
\begin{align}\label{eq2.iota fTT}
\iota_{n-1}(F_{{\rm T}})
=\delta_{\lambda^{(0)}} \beta_{{\rm P},{\rm P}_{\bar{0}}}^{{\rm T}} F_{{\rm P}}
+\sum_{\substack{\underline{\mu}\in \mathscr{P}^{f}_{n}, d_{\underline{\mu}}=0 \\ {\rm U}=(\mathfrak{u},\alpha_\mt,\beta_{\mathfrak{u}}'')\in
{\rm Tri}(\underline{\mu}) \\
\mathfrak{u}\downarrow_{n-1}=\mt,\beta_{\mathfrak{u}}''\downarrow_{n-1}=\beta_\mt}}
\beta_{{\rm U},{\rm U}}^{{\rm T}} F_{{\rm U}},
\end{align}
	where ${\rm P}=(\mathfrak{p},\alpha_\mt,\beta_\mt) \in {\rm Tri}_{\bar{0}}(\underline{\nu})$ $(\underline{\nu}\in\mathscr{P}^{f}_{n}, d_{\underline{\nu}}=1)$ such that $\mathfrak{p}$ is the unique tableau satisfying $\mathfrak{p}\downarrow_{n-1}=\mt$ and $n\in\mathcal{D}_{\mathfrak{p}},$ if $\delta_{\lambda^{(0)}}=1.$ Finally, \eqref{eq1.cor.iota fTT} follows from the fact that the idempotents in $\{F_{{\rm U}}\mid {\rm U}\in {\rm Tri}_{\bar{0}}(\mathscr{P}^{f}_{n})\}$  are complete and pairwise orthogonal.

{\bf Case 2: $d_{\undla}=1.$}  Then for any ${\rm T}=(\mt, \alpha,\beta)\in {\rm Tri}(\undla),$
we can write
\begin{align}\label{eq2.image of f}
	\iota_{n-1}(F_{\rm T})
	=\sum_{\substack{\underline{\mu}\in \mathscr{P}^{f}_{n}, d_{\underline{\mu}}=0 \\ {\rm U},{\rm V}\in {\rm Tri}(\underline{\mu})}}
	\beta_{{\rm U},{\rm V}}^{{\rm T}} f_{{\rm U},{\rm V}}
	+\sum_{\substack{\underline{\mu}\in \mathscr{P}^{f}_{n}, d_{\underline{\mu}}=1 \\ {\rm P},{\rm Q}\in {\rm Tri}_{\bar{0}}(\underline{\mu}) \\ b\in \mathbb{Z}_2}}
	\beta_{{\rm P},{\rm Q}_b}^{{\rm T}} f_{{\rm P},{\rm Q}_b}
\end{align}
where
$
{\rm U}=(\mathfrak{u},\alpha_{\mathfrak{u}}'', \beta_{\mathfrak{u}}''), {\rm V}=(\mathfrak{v},\alpha_{\mathfrak{v}}''', \beta_{\mathfrak{v}}'''),
{\rm P}=(\mathfrak{p},\alpha_{\mathfrak{p}}'''', \beta_{\mathfrak{p}}''''), {\rm Q}=(\mathfrak{q},\alpha_{\mathfrak{q}}''''', \beta_{\mathfrak{q}}''''')\in {\rm Tri}_{\bar{0}}(\mathscr{P}^{f}_{n}).
$
Similarly, we can deduce that

(1) The coefficient $\eta_{{\rm U},{\rm V}}^{{\rm T}}=0$
unless $\mathfrak{u}\downarrow_{n-1}=\mathfrak{v}\downarrow_{n-1}=\mt, \alpha_{\mathfrak{u}}''\downarrow_{n-1}=\alpha_{\mathfrak{v}}'''\downarrow_{n-1}=\alpha_\mt,\,\beta_{\mathfrak{u}}''=\beta_{\mathfrak{v}}'''=\beta_\mt.$

(2) The coefficient $\beta_{{\rm P},{\rm Q}_b}^{{\rm T}}=0$
unless $\mathfrak{p}\downarrow_{n-1}=\mathfrak{q}\downarrow_{n-1}=\mt, \alpha_{\mathfrak{p}}''''=\alpha_{\mathfrak{q}}'''''=\alpha_\mt,\,\beta_{\mathfrak{p}}''''\downarrow_{n-1}=\beta_{\mathfrak{q}}'''''\downarrow_{n-1}=\beta_\mt.$

Again, by comparing parity, we have
\begin{align}\label{eq3.iota fTT}
\iota_{n-1}(F_{\rm T})
=
\delta_{\lambda^{(0)}}\left( \eta_{{\rm U},{\rm U}}^{{\rm T}} F_{{\rm U}}
+ \eta_{{\rm U}',{\rm U}'}^{{\rm T}} F_{{\rm U}'} \right)
+\sum_{\substack{\underline{\mu}\in \mathscr{P}^{f}_{n}, d_{\underline{\mu}}=1 \\ {\rm P}=(\mathfrak{p},\alpha_\mt,\beta_{\mathfrak{p}}'''')\in
{\rm Tri}_{\bar{0}}(\underline{\mu}) \\ \mathfrak{p}\downarrow_{n-1}=\mt, \beta_{\mathfrak{p}}''''\downarrow_{n-1}=\beta_\mt}}
\eta_{{\rm P},{\rm P}}^{{\rm T}} F_{\rm P}
\end{align}
	where ${\rm U}=(\mathfrak{u},\alpha_\mt,\beta_\mt),{\rm U}'=(\mathfrak{u},\alpha_{\mt,\overline{1}},\beta_\mt) \in {\rm Tri}(\underline{\nu})$ $(\underline{\nu}\in\mathscr{P}^{f}_{n}, d_{\underline{\nu}}=0),$ such that $\mathfrak{u}\in \Std(\underline{\nu})$ is the unique tableau satisfying $\mathfrak{u}\downarrow_{n-1}=\mt$ and $n\in\mathcal{D}_{\mathfrak{u}},$ if $\delta_{\lambda^{(0)}}=1.$ Finally, \eqref{eq2.cor.iota fTT} follows from the fact that the idempotents in $\{F_{{\rm U}}\mid {\rm U}\in {\rm Tri}_{\bar{0}}(\mathscr{P}^{f}_{n})\}$ are complete and pairwise orthogonal.

\end{proof}

The main result of this subsection is the following.
\begin{thm}\label{maththm of embedding}
	Let $\undla\in\mathscr{P}^{f}_{n-1}$.

(1) If $d_{\undla}=0,$ then, for any ${\rm S}=(\ms, \alpha_{\ms}', \beta_{\ms}'),{\rm T}=(\mt, \alpha_{\mt}, \beta_{\mt})\in {\rm Tri}(\undla),$ we have
\begin{align}\label{embedding1}
\iota_{n-1}(f_{{\rm S},{\rm T}})
=\delta_{\lambda^{(0)}}
f_{{\rm P},{\rm Q}_{\bar{0}}}
+\sum_{\substack{\underline{\mu}\in \mathscr{P}^{f}_{n}, d_{\underline{\mu}}=0 \\
{\rm U}=(\mathfrak{u},\alpha_{\ms}',\beta_{\ms}'),{\rm V}=(\mathfrak{v},\alpha_{\mt},\beta_{\mt})\in
{\rm Tri}(\underline{\mu}) \\
{\rm U}'=(\mathfrak{u},\alpha_{\ms}',\beta_{\ms}'+e_n),{\rm V}'=(\mathfrak{v},\alpha_{\mt},\beta_{\mt}+e_n)\in
{\rm Tri}(\underline{\mu})\\
\mathfrak{u}\downarrow_{n-1}=\ms,\mathfrak{v}\downarrow_{n-1}=\mt}}
\left(f_{{\rm U},{\rm V}}+(-1)^{|\alpha_{\ms}'|+|\alpha_\mt|}f_{{\rm U}',{\rm V}'}\right),
\end{align}
where ${\rm P}=(\mathfrak{p},\alpha_{\ms}',\beta_{\ms}'),$ ${\rm Q}=(\mathfrak{q},\alpha_{\mt},\beta_{\mt}) \in {\rm Tri}_{\bar{0}}(\underline{\nu})$ $(\underline{\nu}\in\mathscr{P}^{f}_{n}, d_{\underline{\nu}}=1),$
$\mathfrak{p}, \mathfrak{q}\in \Std(\underline{\nu})$ are the unique tableaux such that $\mathfrak{p}\downarrow_{n-1}=\ms,$  $\mathfrak{q}\downarrow_{n-1}=\mt$ and $n\in\mathcal{D}_{\mathfrak{p}}\cap \mathcal{D}_{\mathfrak{q}},$ if $\delta_{\lambda^{(0)}}=1.$

(2) If $d_{\undla}=1,$ then, for any ${\rm S}=(\ms, \alpha_{\ms}', \beta_{\ms}'),{\rm T}=(\mt, \alpha_{\mt}, \beta_{\mt})\in {\rm Tri}_{\bar{0}}(\undla),$ $a\in \mathbb{Z}_2,$ we have
\begin{align}\label{embedding2}
\iota_{n-1}(f_{{\rm S},{\rm T}_a})
=\delta_{\lambda^{(0)}}
(f_{{\rm P},{\rm Q}}+f_{{\rm P}',{\rm Q}'})
+\sum_{\substack{\underline{\mu}\in \mathscr{P}^{f}_{n}, d_{\underline{\mu}}=1 \\
{\rm U}=(\mathfrak{u},\alpha_{\ms}',\beta_{\ms}'),{\rm V}=(\mathfrak{v},\alpha_{\mt},\beta_{\mt})\in
{\rm Tri}_{\bar{0}}(\underline{\mu}) \\
{\rm U}'=(\mathfrak{u},\alpha_{\ms}',\beta_{\ms}'+e_n),{\rm V}'=(\mathfrak{v},\alpha_{\mt},\beta_{\mt}+e_n)\in
{\rm Tri}_{\bar{0}}(\underline{\mu})\\
\mathfrak{u}\downarrow_{n-1}=\ms,\mathfrak{v}\downarrow_{n-1}=\mt}}
\left(f_{{\rm U},{\rm V}_a}+(-1)^{|\alpha_{\ms}'|+|\alpha_{\mt,a}|}f_{{\rm U}',{\rm V}'_a}\right),
\end{align}
where ${\rm P}=(\mathfrak{p},\alpha_{\ms}',\beta_{\ms}'), {\rm Q}=(\mathfrak{q},\alpha_{\mt,a},\beta_{\mt}),
{\rm P}'=(\mathfrak{p},\alpha_{\ms,\bar{1}}',\beta_{\ms}'), {\rm Q}'=(\mathfrak{q},\alpha_{\mt,a+\bar{1}},\beta_{\mt})\in {\rm Tri}(\underline{\nu})$ $(\underline{\nu}\in\mathscr{P}^{f}_{n}, d_{\underline{\nu}}=0),$
$\mathfrak{p}, \mathfrak{q}\in \Std(\underline{\nu})$ are the unique tableaux such that $\mathfrak{p}\downarrow_{n-1}=\ms,$  $\mathfrak{q}\downarrow_{n-1}=\mt$ and $n\in\mathcal{D}_{\mathfrak{p}}\cap \mathcal{D}_{\mathfrak{q}},$ if $\delta_{\lambda^{(0)}}=1.$
\end{thm}
\begin{proof}
We only prove \eqref{embedding1}, since the proof of \eqref{embedding2} is completely similar to \eqref{embedding1}.

By Proposition \ref{eq.intertwining and f} (1), we have
\begin{align}\label{eq1.mainthm of embeding}
\iota_{n-1}(f_{{\rm S},{\rm T}}^{\mt^{\undla}})
&=\mathtt{q}_{\ms,\mt^{\undla}}^{-1}\mathtt{q}_{\mt^{\undla},\mt}^{-1}
C^{\beta_{\ms}'}C^{\alpha_{\ms}'} \tilde{\Phi}_{\ms,\mt^{\undla}} \iota_{n-1}(F_{{\rm T}^{\undla}}) \tilde{\Phi}_{\mt^{\undla},\mt} (C^{\alpha_{\mt}})^{-1}(C^{\beta_{\mt}})^{-1}.
\end{align}
By \eqref{eq1.cor.iota fTT}, we have
\begin{align}\label{eq2.mainthm of embeding}
\iota_{n-1}(F_{{\rm T}^{\undla}})
=\delta_{\lambda^{(0)}}F_{{\rm A}}
+\sum_{\substack{\underline{\mu}\in \mathscr{P}^{f}_{n}, d_{\underline{\mu}}=0 \\ {\rm B}=(\mathfrak{b},0,0),{\rm B}'=(\mathfrak{b},0,e_n)\in
{\rm Tri}(\underline{\mu}) \\
\mathfrak{b}\downarrow_{n-1}=\mt^{\undla}}}
\left(F_{{\rm B}}+F_{{\rm B}'}\right)
\end{align}
where ${\rm A}=(\mathfrak{a},0,0) \in {\rm Tri}_{\bar{0}}(\underline{\nu})$ $(\underline{\nu}\in\mathscr{P}^{f}_{n}, d_{\underline{\nu}}=1),$ $\mathfrak{a}\in \Std(\underline{\nu})$ is the unique tableau such that $\mathfrak{a}\downarrow_{n-1}=\mt^{\undla}$ and $n\in\mathcal{D}_{\mathfrak{a}},$ when $\delta_{\lambda^{(0)}}=1.$

Using Proposition \ref{eq.intertwining and f 2}, we have
\begin{equation}\label{eq3.mainthm of embeding}
\mathtt{q}_{\ms,\mt^{\undla}}^{-1}\mathtt{q}_{\mt^{\undla},\mt}^{-1}
C^{\beta_{\ms}'}C^{\alpha_{\ms}'} \tilde{\Phi}_{\ms,\mt^{\undla}}
F_{\rm A}
\tilde{\Phi}_{\mt^{\undla},\mt} (C^{\alpha_{\mt}})^{-1}(C^{\beta_{\mt}})^{-1}
=\frac{\mathtt{c}_{\ms,\mt^{\undla}}\mathtt{c}_{\mt^{\undla},\mt}}{\mathtt{c}_{\mathfrak{p},\mt^{\underline{\nu}}}\mathtt{c}_{\mt^{\underline{\nu}},\mathfrak{q}}}f^{\mt^{\underline{\nu}}}_{{\rm P},{\rm Q}_{\bar{0}}},
\end{equation}
where ${\rm P}=(\mathfrak{p},\alpha_{\ms}',\beta_{\ms}'),$ ${\rm Q}=(\mathfrak{q},\alpha_{\mt},\beta_{\mt}) \in {\rm Tri}_{\bar{0}}(\underline{\nu})$,  $(\underline{\nu}\in\mathscr{P}^{f}_{n}, d_{\underline{\nu}}=1),$
$\mathfrak{p}, \mathfrak{q}\in \Std(\underline{\nu})$ are unique tableaux such that $\mathfrak{p}\downarrow_{n-1}=\ms,$  $\mathfrak{q}\downarrow_{n-1}=\mt$  and $n\in\mathcal{D}_{\mathfrak{p}}\cap \mathcal{D}_{\mathfrak{q}},$ if $\delta_{\lambda^{(0)}}=1$.

Similarly, for any ${\rm B}=(\mathfrak{b},0,0),{\rm B}'=(\mathfrak{b},0,e_n)\in
{\rm Tri}(\underline{\mu})$, where $\underline{\mu}\in \mathscr{P}^{f}_{n}, d_{\underline{\mu}}=0$ such that
$\mathfrak{b}\downarrow_{n-1}=\mt^{\undla}$, by Proposition \ref{eq.intertwining and f 2}, we have
\begin{equation}\label{eq4.mainthm of embeding}
\mathtt{q}_{\ms,\mt^{\undla}}^{-1}\mathtt{q}_{\mt^{\undla},\mt}^{-1}
C^{\beta_{\ms}'}C^{\alpha_{\ms}'} \tilde{\Phi}_{\ms,\mt^{\undla}}
\left(F_{\rm B}+F_{\rm B'}\right)
\tilde{\Phi}_{\mt^{\undla},\mt} (C^{\alpha_{\mt}})^{-1}(C^{\beta_{\mt}})^{-1}
=\frac{\mathtt{c}_{\ms,\mt^{\undla}}\mathtt{c}_{\mt^{\undla},\mt}}{\mathtt{c}_{\mathfrak{u},\mt^{\underline{\mu}}}\mathtt{c}_{\mt^{\underline{\mu}},\mathfrak{v}}}\left(f^{\mt^{\underline{\mu}}}_{{\rm U},{\rm V}}+(-1)^{|\alpha_\mt|+|\alpha_{\ms}'|}f^{\mt^{\underline{\mu}}}_{{\rm U}',{\rm V}'}\right)\nonumber
\end{equation}
where ${\rm U}=(\mathfrak{u},\alpha_{\ms}',\beta_{\ms}'),\,{\rm V}=(\mathfrak{v},\alpha_{\mt},\beta_{\mt}),\,
{\rm U}'=(\mathfrak{u},\alpha_{\ms}',\beta_{\ms}'+e_n),\,{\rm V}'=(\mathfrak{v},\alpha_{\mt},\beta_{\mt}+e_n)\in
{\rm Tri}(\underline{\mu})$ such that $\mathfrak{u}\downarrow_{n-1}=\ms,\mathfrak{v}\downarrow_{n-1}=\mt$.

Combining \eqref{eq1.mainthm of embeding}, \eqref{eq2.mainthm of embeding}, \eqref{eq3.mainthm of embeding} and \eqref{eq4.mainthm of embeding}, we deduce that

\begin{align} \label{eq5.mainthm of embeding}
	\iota_{n-1}(f_{{\rm S},{\rm T}}^{\mt^{\undla}})=
&\delta_{\lambda^{(0)}}\frac{\mathtt{c}_{\ms,\mt^{\undla}}\mathtt{c}_{\mt^{\undla},\mt}}{\mathtt{c}_{\mathfrak{p},\mt^{\underline{\nu}}}\mathtt{c}_{\mt^{\underline{\nu}},\mathfrak{q}}}
	f^{\mt^{\underline{\nu}}}_{{\rm P},{\rm Q}_{\bar{0}}} \\
&+\sum_{\substack{\underline{\mu}\in \mathscr{P}^{f}_{n}, d_{\underline{\mu}}=0 \\
			{\rm U}=(\mathfrak{u},\alpha_{\ms}',\beta_{\ms}'),V=(\mathfrak{v},\alpha_{\mt},\beta_{\mt})\in
			{\rm Tri}(\underline{\mu}) \\
			{\rm U}'=(\mathfrak{u},\alpha_{\ms}',\beta_{\ms}'+e_n),V'=(\mathfrak{v},\alpha_{\mt},\beta_{\mt}+e_n)\in
			{\rm Tri}(\underline{\mu})\\
			\mathfrak{u}\downarrow_{n-1}=\ms,\mathfrak{v}\downarrow_{n-1}=\mt}} \frac{\mathtt{c}_{\ms,\mt^{\undla}}\mathtt{c}_{\mt^{\undla},\mt}}{\mathtt{c}_{\mathfrak{u},\mt^{\underline{\mu}}}\mathtt{c}_{\mt^{\underline{\mu}},\mathfrak{v}}}\left(f^{\mt^{\underline{\mu}}}_{{\rm U},{\rm V}}+(-1)^{|\alpha_\mt|+|\alpha_{\ms}'|}f^{\mt^{\underline{\mu}}}_{{\rm U}',{\rm V}'}\right).\nonumber
\end{align}
Note that for $\mathfrak{u}\downarrow_{n-1}=\ms$ and $\mathfrak{v}\downarrow_{n-1}=\mt$, we have $c_{\ms,\mt}=c_{\mathfrak{p},\mathfrak{q}}=c_{\mathfrak{u},\mathfrak{v}}$. Combining this with \eqref{eq5.mainthm of embeding} and \eqref{fst and re. fst}, we complete the proof of \eqref{embedding1}.
\end{proof}

\begin{rem}
By Corollary \ref{cor. ct and cst}, one can see that the coefficient appearing in equation \eqref{eq5.mainthm of embeding} satisfies
\begin{align*}
\frac{\mathtt{c}_{\ms,\mt^{\undla}}\mathtt{c}_{\mt^{\undla},\mt}}{\mathtt{c}_{\mathfrak{p},\mt^{\underline{\nu}}}\mathtt{c}_{\mt^{\underline{\nu}},\mathfrak{q}}}
=\frac{\mathtt{c}_{\ms,\mt^{\undla}}\mathtt{c}_{\mt^{\undla},\ms}}{\mathtt{c}_{\mathfrak{p},\mt^{\underline{\nu}}}\mathtt{c}_{\mt^{\underline{\nu}},\mathfrak{p}}}
=\frac{\mathtt{c}_{\mt,\mt^{\undla}}\mathtt{c}_{\mt^{\undla},\mt}}{\mathtt{c}_{\mathfrak{q},\mt^{\underline{\nu}}}\mathtt{c}_{\mt^{\underline{\nu}},\mathfrak{q}}}.
\end{align*}
Similar equalities for the coefficients in the second term of the right-hand side of \eqref{eq5.mainthm of embeding} can be obtained.
From these we can deduce the analogs of formulae for the embedding of seminormal bases as in \cite[Theorem 1.3, Theorem 1.6]{HW}.
\end{rem}

\subsection{Modifications in the degenerate case}
Recall the cyclotomic Sergeev algebras $\mhgcn$, with $g$ as in \eqref{deg-cyc-even} or \eqref{deg-cyc-odd}. In this subsection, we fix $\undQ=(Q_1,\ldots,Q_m)\in\mathbb K^m$ and assume $P_n^g(\undQ)\ne0$. We use the residue convention in \eqref{eq:deg-residue}.

Let $\iota_{n-1}: \mathfrak{H}^g_c(n-1)\rightarrow \mhgcn$ be the natural embedding, which maps $x_j$ to $x_j$, $c_j$ to $c_j$ and $s_i$ to $s_i$ for any $i\in [n-2]$ and $j\in [n-1].$
Then Theorem \ref{maththm of embedding} has an analog to cyclotomic Sergeev algebras and the proof is similar to Theorem \ref{maththm of embedding}.

\begin{thm}\label{maththm of embedding. dege}
Let $\undla\in\mathscr{P}^{g}_{n-1}$.

(1) If $d_{\undla}=0,$ then, for any ${\rm S}=(\ms, \alpha_{\ms}', \beta_{\ms}'),{\rm T}=(\mt, \alpha_{\mt}, \beta_{\mt})\in {\rm Tri}(\undla),$ we have
\begin{align*}
\iota_{n-1}(\mathfrak{f}_{{\rm S},{\rm T}})
=\delta_{\lambda^{(0)}}
\mathfrak{f}_{{\rm P},{\rm Q}_{\bar{0}}}
+\sum_{\substack{\underline{\mu}\in \mathscr{P}^{g}_{n}, d_{\underline{\mu}}=0 \\
{\rm U}=(\mathfrak{u},\alpha_{\ms}',\beta_{\ms}'),{\rm V}=(\mathfrak{v},\alpha_{\mt},\beta_{\mt})\in
{\rm Tri}(\underline{\mu}) \\
{\rm U}'=(\mathfrak{u},\alpha_{\ms}',\beta_{\ms}'+e_n),{\rm V}'=(\mathfrak{v},\alpha_{\mt},\beta_{\mt}+e_n)\in
{\rm Tri}(\underline{\mu})\\
\mathfrak{u}\downarrow_{n-1}=\ms,\mathfrak{v}\downarrow_{n-1}=\mt}}
\left(\mathfrak{f}_{{\rm U},{\rm V}}+(-1)^{|\alpha_{\ms}'|+|\alpha_\mt|}\mathfrak{f}_{{\rm U}',{\rm V}'}\right)
\end{align*}
where ${\rm P}=(\mathfrak{p},\alpha_{\ms}',\beta_{\ms}'),$ ${\rm Q}=(\mathfrak{q},\alpha_{\mt},\beta_{\mt}) \in {\rm Tri}_{\bar{0}}(\underline{\nu})$ $(\underline{\nu}\in\mathscr{P}^{g}_{n}, d_{\underline{\nu}}=1),$
$\mathfrak{p}, \mathfrak{q}\in \Std(\underline{\nu})$ are the unique tableaux such that $\mathfrak{p}\downarrow_{n-1}=\ms,$  $\mathfrak{q}\downarrow_{n-1}=\mt$ and $n\in\mathcal{D}_{\mathfrak{p}}\cap \mathcal{D}_{\mathfrak{q}},$ if $\delta_{\lambda^{(0)}}=1.$

(2) If $d_{\undla}=1,$ then, for any ${\rm S}=(\ms, \alpha_{\ms}', \beta_{\ms}'),{\rm T}=(\mt, \alpha_{\mt}, \beta_{\mt})\in {\rm Tri}_{\bar{0}}(\undla),$ $a\in \mathbb{Z}_2,$ we have
\begin{align*}
\iota_{n-1}(\mathfrak{f}_{{\rm S},{\rm T}_a})
=\delta_{\lambda^{(0)}}
(\mathfrak{f}_{{\rm P},{\rm Q}}+\mathfrak{f}_{{\rm P}',{\rm Q}'})
+\sum_{\substack{\underline{\mu}\in \mathscr{P}^{g}_{n}, d_{\underline{\mu}}=1 \\
{\rm U}=(\mathfrak{u},\alpha_{\ms}',\beta_{\ms}'),{\rm V}=(\mathfrak{v},\alpha_{\mt},\beta_{\mt})\in
{\rm Tri}_{\bar{0}}(\underline{\mu}) \\
{\rm U}'=(\mathfrak{u},\alpha_{\ms}',\beta_{\ms}'+e_n),{\rm V}'=(\mathfrak{v},\alpha_{\mt},\beta_{\mt}+e_n)\in
{\rm Tri}_{\bar{0}}(\underline{\mu})\\
\mathfrak{u}\downarrow_{n-1}=\ms,\mathfrak{v}\downarrow_{n-1}=\mt}}
\left(\mathfrak{f}_{{\rm U},{\rm V}_a}+(-1)^{|\alpha_{\ms}'|+|\alpha_{\mt,a}|}\mathfrak{f}_{{\rm U}',{\rm V}'_a}\right)
\end{align*}
where ${\rm P}=(\mathfrak{p},\alpha_{\ms}',\beta_{\ms}'), {\rm Q}=(\mathfrak{q},\alpha_{\mt,a},\beta_{\mt}),
{\rm P}'=(\mathfrak{p},\alpha_{\ms,\bar{1}}',\beta_{\ms}'), {\rm Q}'=(\mathfrak{q},\alpha_{\mt,a+\bar{1}},\beta_{\mt})\in {\rm Tri}(\underline{\nu})$ $(\underline{\nu}\in\mathscr{P}^{g}_{n}, d_{\underline{\nu}}=0),$
$\mathfrak{p}, \mathfrak{q}\in \Std(\underline{\nu})$ are the unique tableaux such that $\mathfrak{p}\downarrow_{n-1}=\ms,$  $\mathfrak{q}\downarrow_{n-1}=\mt$ and $n\in\mathcal{D}_{\mathfrak{p}}\cap \mathcal{D}_{\mathfrak{q}},$ if $\delta_{\lambda^{(0)}}=1.$
\end{thm}

\section{Frobenius forms of cyclotomic Hecke-Clifford algebras}\label{frobenius}
The Mackey decomposition expresses the Frobenius form recursively in the rank. We first recall this decomposition, then reduce the Clifford indices so that only values on primitive idempotents remain to be computed.
For any two rings $A,B$ and a ring homomorphism $\iota: B\rightarrow A$, we can regard $A$ as an $(B,B)$-bimodule. For any $f\in\{a\in A\mid ab=ba,\ \forall\,b\in B\}$, we define (following \cite[(1)]{SVV})
$$\mu_f: A\otimes_{B}A\rightarrow A,\ \ \sum_{(a)}a_{(1)}\otimes a_{(2)}\mapsto\sum_{(a)}a_{(1)}fa_{(2)}.$$

\subsection{Frobenius form and Mackey decomposition}
In this subsection, let $\mHfcn$ be the cyclotomic Hecke-Clifford algebra defined over the integral domain ${\rm R},$ where $f$ is the defining polynomial determined by $\undQ=(Q_1,Q_2,\ldots,Q_m)\in({\rm R}^\times)^m$.
We first recall the Mackey decomposition of $\mHfcn.$
\begin{lem}\label{BK. Lemma 3.8}\cite[Lemma 3.8]{BK}
(1) As an $(\mathcal H_c^f(n-1),\mathcal H_c^f(n-1))$-bimodule, we have
\begin{align}
\mHfcn=\mathcal{H}^f_c(n-1)T_{n-1}\mathcal{H}^f_c(n-1)\oplus\bigoplus_{0\leq k\leq r-1, a\in \mathbb{Z}_2} X_n^k C_n^a \mathcal{H}^f_c(n-1). \label{Mackey nondege}
\end{align}

(2) For any $0 \leq k <r,$ there are isomorphisms
\begin{align*}
X_n^k \mathcal{H}^f_c(n-1)\simeq \mathcal{H}^f_c(n-1),&\qquad X_n^k C_n \mathcal{H}^f_c(n-1)\simeq \Pi\mathcal{H}^f_c(n-1),\\
\mathcal{H}^f_c(n-1)T_{n-1}\mathcal{H}^f_c(n-1)&\simeq \mathcal{H}^f_c(n-1)\otimes_{\mathcal{H}^f_c(n-2)}\mathcal{H}^f_c(n-1),
\end{align*}
as $(\mathcal{H}^f_c(n-1),\mathcal{H}^f_c(n-1))$-superbimodules.
\end{lem}
By the Lemma \ref{BK. Lemma 3.8}, we see that for any $h\in \mHfcn,$ there are unique elements $\pi(h)=\pi^{(n)}(h)\in \mathcal{H}^f_c(n-1)\otimes_{\mathcal{H}^f_c(n-2)}\mathcal{H}^f_c(n-1),$ and $p_{k,a}(h)=p^{(n)}_{k,a}(h)\in\mathcal{H}^f_c(n-1),$ $k=0,1,\ldots,r-1,$
$a\in \mathbb{Z}_2$ such that
\begin{align}\label{nondegMadecomequation}
h=\mu_{T_{n-1}}(\pi(h))+\sum_{\substack{0\leq k\leq r-1\\ a\in \mathbb{Z}_2}}X_n^k C_n^a p_{k,a}(h).
\end{align}
We call \eqref{nondegMadecomequation} the (cyclotomic) Mackey decomposition of $\mHfcn.$
Note that $\pi,$ $p_{k,a}$ are all ($\mathcal{H}^f_c(n-1),\mathcal{H}^f_c(n-1)$)-superbilinear maps with
$|\pi|=\bar{0},$ $|p_{k,a}|=a,$ for any $k=0,1,\ldots,r-1,$
$a\in \mathbb{Z}_2$.
In particular, we have the following
\begin{align}\label{superbilinearity of pk1}
p_{k,\bar{1}}(xhy)=(-1)^{|x||p_{k,\bar{1}}|}x p_{k,\bar{1}}(h)y=(-1)^{|x|}x p_{k,\bar{1}}(h)y,
\end{align}	
for any $x,y\in\mathcal{H}^f_c(n-1), h\in\mHfcn$ and
$k\in\{0,1,\ldots,r-1\}.$

For any $n\geq 1$, we denote $\varepsilon_n:=p^{(n)}_{0,\bar{0}}$, which is the projection map from $\mHfcn$ to $\mathcal{H}^f_c(n-1).$ Then we obtain the following ${\rm R}$-linear map \label{pag:Frobenius form. non-dege}
\begin{align}
\tau_{r,n}:=\varepsilon_1\circ \cdots \circ \varepsilon_n:\quad \mHfcn \rightarrow {\rm R}.
\end{align}		

\begin{lem} \cite[Corollary 3.14]{BK}
	The map $\tau_{r,n}$ is a Frobenius form of $\mHfcn.$
	\end{lem}

\begin{lem}\cite[(2.23)]{BK}
	For any $j\geq 1$ and $1 \leq i \leq n-1,$ we have
	\begin{align}
		X_{i+1}^jT_i=(T_i+\epsilon C_iC_{i+1})X_i^j+\epsilon X_{i+1}^j+ \epsilon\sum_{k=1}^{j-1}\left( X_i^{j-k}X_{i+1}^k + X_i^{-k}X_{i+1}^{j-k}C_iC_{i+1}\right). \label{BK(2.23)}
	\end{align}
\end{lem}
The following Proposition gives a closed formula for $\tau_{r,n}$ under the monomial basis of $\mHfcn.$

\begin{prop}\label{closed formula frob}
Let $\alpha=(\alpha_1,\ldots,\alpha_n)\in [0,r-1]^n,$ $\beta=(\beta_1,\ldots,\beta_n)\in\mathbb{Z}_2^n$ and $w\in \mathfrak{S}_n,$
then we have
\begin{align}
\tau_{r,n}( X^\alpha C^\beta T_w )=\delta_{(\alpha,\beta,w),(0,0,1)}.
\end{align}
\end{prop}
\begin{proof}
(i) If $w\in \mathfrak{S}_{n-1},$ then by Lemma \ref{BK. Lemma 3.8} (2), we have
\begin{align*}
\tau_{r,n}( X^\alpha C^\beta T_w )
&=\tau_{r,n-1}(X_1^{\alpha_1}\cdots X_{n-1}^{\alpha_{n-1}}C_1^{\beta_1}\cdots C_{n-1}^{\beta_n-1} \varepsilon_n(X_n^{\alpha_n}C_n^{\beta_n})  T_w)\\
&=\delta_{(\alpha_n,\beta_n),(0,0)}\tau_{r,n-1}(X_1^{\alpha_1}\cdots X_{n-1}^{\alpha_{n-1}}C_1^{\beta_1}\cdots C_{n-1}^{\beta_n-1} T_w);
\end{align*}
(ii) If $w=s_j\cdots s_{n-1}w'$ for some $1\leq j \leq n-1,$ $w'\in\mathfrak{S}_{n-1},$ then by Lemma \ref{BK. Lemma 3.8} (2), we have
\begin{align*}
\tau_{r,n}( X^\alpha C^\beta T_w )
&=\tau_{r,n-1}(X_1^{\alpha_1}\cdots X_{n-1}^{\alpha_{n-1}}C_1^{\beta_1}\cdots C_{n-1}^{\beta_n-1}T_j\cdots T_{n-2}\varepsilon_n(X_n^{\alpha_n}C_n^{\beta_n}T_{n-1}) T_{w'})\\
&=\tau_{r,n-1}(X_1^{\alpha_1}\cdots X_{n-1}^{\alpha_{n-1}}C_1^{\beta_1}\cdots C_{n-1}^{\beta_n-1}T_j\cdots T_{n-2}\varepsilon_n(X_n^{\alpha_n}T_{n-1})C_{n-1}^{\beta_n} T_{w'}),
\end{align*} where in the last equation we have used relation \eqref{PC}.
By \eqref{BK(2.23)}, we have
\begin{align*}
X_n^{\alpha_n}T_{n-1}\in T_{n-1}X_{n-1}^{\alpha_n}+\bigoplus_{1\leq a\leq r-1} X_n^a \mathcal{H}^f_c(n-1)\oplus\bigoplus_{0\leq a\leq r-1} X_n^a C_n\mathcal{H}^f_c(n-1).
\end{align*}
This, combined with Lemma \ref{BK. Lemma 3.8} (1), implies $\varepsilon_n(X_n^{\alpha_n}T_{n-1})=0$ and $\tau_{r,n}( X^\alpha C^\beta T_w )=0$.

Now the Proposition follows by using an induction on $n.$
\end{proof}

\subsection{Reduction of the Clifford factors}\label{reduction}
Here $f$ is restricted to \eqref{cyc-even} or \eqref{cyc-odd}. Reducing the Clifford factors leaves the values on primitive idempotents needed to compute the Schur elements.
The aim of this subsection is to study these values of $\tau_{r,n}$ at the elements of seminormal basis over $\mathbb{K}.$
{\bf In this section, we shall fix the parameter $\undQ=(Q_1,Q_2,\ldots,Q_m)\in(\mathbb{K}^*)^m$ and $f$ with $P_n^{f}(q^2,\undQ)\neq 0$.} Accordingly, we define the residues of boxes in the Young diagram $\undla$ via \eqref{eq:residue} as well as $\res(\mathfrak{t})$ for each $\mathfrak{t}\in\Std(\undla)$ with $\undla\in\mathscr{P}^{f}_{n}$ with $m\geq 0.$ {\bf Moreover, we shall fix $\undla\in\mathscr{P}^{f}_{n}$.}

Recall that \begin{align*}
	\mathcal{D}_{\mt^{\undla}}&:=\{i_1<i_2<\cdots<i_t\},\\
	\mathcal{OD}_{\mt^{\undla}}&:=\{i_1,i_3,\cdots,i_{2{\lceil t/2 \rceil}-1}\}\subset \mathcal{D}_{\mt^{\undla}},
\end{align*}
where $t:=\sharp\mathcal{D}_{\mt}.$  Hence for any ${\rm T}=(\mt,\alpha_{\mt},\beta_{\mt})\in{\rm Tri}_{\bar{0}}(\undla),$ we have $$\mathcal{OD}_{\mt}=\{d(\mt,\mt^{\undla})(i_1)<d(\mt,\mt^{\undla})(i_3)<\cdots<d(\mt,\mt^{\undla})(i_{2\lceil t/2\rceil-1})\}.$$

\begin{lem}\label{lem.taufST}
Let ${\rm S}=(\ms,\alpha_{\ms}',\beta_{\ms}'), {\rm T}=(\mt,\alpha_{\mt},\beta_{\mt})\in{\rm Tri}_{\bar{0}}(\undla).$ We set $\underline{\mu}:=\text{shape}(\mt\downarrow_{n-1}).$
\begin{enumerate}
\item  Suppose $d_{\undla}=0,$ $d_{\undmu}=0.$ Then we have
\begin{align}\label{eq00.lem.taufST}
\varepsilon_n(f_{{\rm S},{\rm T}})
\in \mathbb{K} f_{{\rm S}\downarrow_{n-1},{\rm T}\downarrow_{n-1}}.
\end{align}
Moreover, $\varepsilon_n(f_{{\rm S},{\rm T}})=0$ unless $\ms^{-1}(n)=\mt^{-1}(n)$ and
$$|\beta_{\ms}'|+|\beta_{\mt}| \equiv |\beta_{\ms}'\downarrow_{n-1}|+|\beta_{\mt}\downarrow_{n-1}| \pmod 2.$$

\item Suppose $d_{\undla}=0,$ $d_{\undmu}=1.$ Then we have
\begin{align}\label{eq01.lem.taufST}
\varepsilon_n(f_{{\rm S},{\rm T}})
\in \mathbb{K} f_{{\rm S}\downarrow_{n-1},({\rm T}\downarrow_{n-1})_{b}},
\end{align}
where $b\in\mathbb{Z}_2$ is determined by
$$|\alpha_{\ms}'|+|\alpha_{\mt}| \equiv |\alpha_{\ms}'|_{<d(\ms,\mt^{\undla})(i_{t-1})}+|\alpha_{\mt}|_{<d(\mt,\mt^{\undla})(i_{t-1})}+b\pmod 2.$$ Moreover, $\varepsilon_n(f_{{\rm S},{\rm T}})=0$ unless $\ms^{-1}(n)=\mt^{-1}(n)$.

\item Suppose $d_{\undla}=1,$ $d_{\undmu}=0.$ Then for any $a\in\mathbb{Z}_2,$ we have
\begin{align}\label{eq10.lem.taufST}
\varepsilon_n(f_{{\rm S},{\rm T}_{a}})
\in \mathbb{K} f_{{\rm S}\downarrow_{n-1},{\rm T}\downarrow_{n-1}}.
\end{align}
Moreover, $\varepsilon_n(f_{{\rm S},{\rm T}_{a}})=0$ unless $\ms^{-1}(n)=\mt^{-1}(n)$ and $a=\bar{0}.$

\item Suppose $d_{\undla}=1,$ $d_{\undmu}=1.$ Then for any $a\in\mathbb{Z}_2,$ we have
\begin{align}\label{eq11.lem.taufST}
\varepsilon_n(f_{{\rm S},{\rm T}_{a}})
\in \mathbb{K}  f_{{\rm S}\downarrow_{n-1},({\rm T}\downarrow_{n-1})_{c}},
\end{align}
where $c\in\mathbb{Z}_2$ is determined by
$$|\beta_{\ms}'|+|\beta_{\mt}|+a \equiv |\beta_{\ms}'\downarrow_{n-1}|+|\beta_{\mt}\downarrow_{n-1}|+c\pmod 2.$$ Moreover, $\varepsilon_n(f_{{\rm S},{\rm T}_{a}})=0$ unless $\ms^{-1}(n)=\mt^{-1}(n)$.
\end{enumerate}
\end{lem}
\begin{proof}
First, the equations \eqref{eq00.lem.taufST}, \eqref{eq01.lem.taufST}, \eqref{eq10.lem.taufST} and \eqref{eq11.lem.taufST} can be obtained similarly as in Lemma \ref{cor.iota fTT}. In fact, by Proposition \ref{generators action on seminormal basis} (1), Lemma \ref{important condition1} (2) and the ($\mathcal{H}^f_c(n-1),\mathcal{H}^f_c(n-1)$)-bilinearity of the $\varepsilon_n,$ we can determine all of the eigenvalues of $\varepsilon_n(f_{{\rm S},{\rm T}_{a}})$. Now, the four equations \eqref{eq00.lem.taufST}, \eqref{eq01.lem.taufST}, \eqref{eq10.lem.taufST} and \eqref{eq11.lem.taufST}  follow from Lemma \ref{lem:clifford rep}.

We only need to prove the equations for parities.  We only prove (2) and (4) since the proofs of (1) and (3) are similar and easier.

(2) Note that
\begin{align*}
|f_{{\rm S},{\rm T}}|
&\equiv |\alpha_{\ms}'|+|\beta_{\ms}'|+|\alpha_{\mt}|+|\beta_{\mt}| \pmod 2,\\
|f_{{\rm S}\downarrow_{n-1},({\rm T}\downarrow_{n-1})_{b}}|
&\equiv|\alpha_{\ms}'|_{<d(\ms,\mt^{\undla})(i_{t-1})}+|\beta_{\ms}'|+|\alpha_{\mt}|_{<d(\mt,\mt^{\undla})(i_{t-1})}+|\beta_{\mt}|+b \pmod 2,
\end{align*}
and $f_{({\rm S}\downarrow_{n-1})_{\bar{0}},({\rm T}\downarrow_{n-1})_{b}}$ makes no sense unless $\ms^{-1}(n)=\mt^{-1}(n).$ Now the equation for the parity follows from that $|\varepsilon_n|=\overline{0}$.

(4) Note that
\begin{align*}
|f_{{\rm S},{\rm T}_{a}}|
&\equiv |\alpha_{\ms}'|+|\beta_{\ms}'|+|\alpha_{\mt}|+|\beta_{\mt}|+a \pmod 2,\\
|f_{{{\rm S}\downarrow_{n-1}},({\rm T}\downarrow_{n-1})_{c}}|
&\equiv |\alpha_{\ms}'|+|\beta_{\ms}'\downarrow_{n-1}|+|\alpha_{\mt}|+|\beta_{\mt}\downarrow_{n-1}|+c \pmod 2,
\end{align*}
and $f_{({\rm S}\downarrow_{n-1})_{\bar{0}},({\rm T}\downarrow_{n-1})_{c}}$ makes no sense unless $\ms^{-1}(n)=\mt^{-1}(n).$ Now the equation for the parity follows from that $|\varepsilon_n|=\overline{0}$.
\end{proof}

\begin{lem}\label{tinylemma1}
	Let ${\rm S}=(\ms,\alpha_{\ms}',\beta_{\ms}'), {\rm T}=(\mt,\alpha_{\mt},\beta_{\mt})\in{\rm Tri}_{\bar{0}}(\undla).$ We set $\underline{\mu}:=\text{shape}(\mt\downarrow_{n-1}).$
Suppose $d_{\undla}=1,\,d_{\undmu}=1,$ $\ms=\mt$ and $\varepsilon_n(f_{{\rm S},{\rm T}_{a}})\neq0$. Then we have
	$$|\beta_{\ms}'|+|\beta_{\mt}| \equiv |\beta_{\ms}'\downarrow_{n-1}|+|\beta_{\mt}\downarrow_{n-1}|\pmod 2.$$
	\end{lem}
	
	\begin{proof}
	By \eqref{eq11.lem.taufST}, we have
		\begin{align}
			\varepsilon_n(f_{{\rm S},{\rm T}_{a}})
			\in \mathbb{K}  f_{{\rm S}\downarrow_{n-1},({\rm T}\downarrow_{n-1})_{c}},
		\end{align}
		where $c\in\mathbb{Z}_2$ is determined by
		$$|\beta_{\ms}'|+|\beta_{\mt}|+a \equiv |\beta_{\ms}'\downarrow_{n-1}|+|\beta_{\mt}\downarrow_{n-1}|+c\pmod 2.$$ We set $i:=d(\mt,\mt^{\undla})(i_t)$. Since $d_{\undla}=1,\,d_{\undmu}=1$, we have $i<n$. We first claim the following two equations hold.
		\begin{align}
			C_i f_{{\rm S},{\rm T}_a} C_i
			&=(-1)^{|\alpha_{\mt}'|+|\beta_{\mt}'|+|\alpha_{\mt}|+|\beta_{\mt}|} f_{{\rm S},{\rm T}_a},\label{eq1.prop.taufST}\\
			C_i f_{{\rm S}\downarrow_{n-1},({\rm T}\downarrow_{n-1})_{c}}C_i
			&=(-1)^{|\alpha_{\mt}'|+|\beta_{\mt}'\downarrow_{n-1}|+|\alpha_{\mt}|+|\beta_{\mt}\downarrow_{n-1}|} f_{{\rm S}\downarrow_{n-1},({\rm T}\downarrow_{n-1})_{c}}\label{eq2.prop.taufST}.
		\end{align}
	
		In fact, we have
		\begin{align*}
			&C_i f_{{\rm S},{\rm T}_a} C_i\\
			&=C_i F_{\rm S}C^{\beta_{\mt}'}C^{\alpha_{\mt}'}\Phi_{\mt,\mt}
			(C^{{\alpha}_{\mt,a}})^{-1} (C^{\beta_{\mt}})^{-1} F_{\rm T}C_i\\
			&=(-1)^{|\alpha_{\mt}'|+|\beta_{\mt}'|+|\alpha_{\mt}|+|\beta_{\mt}|} F_{\rm S}C^{\beta_{\mt}'}C^{\alpha_{\mt}'}\Phi_{\mt,\mt}
			(C^{{\alpha}_{\mt,a+\bar{1}+\bar{1}}})^{-1} (C^{\beta_{\mt}})^{-1} F_{\rm T}\\
			&=(-1)^{|\alpha_{\mt}'|+|\beta_{\mt}'|+|\alpha_{\mt}|+|\beta_{\mt}|} F_{\rm S}C^{\beta_{\mt}'}C^{\alpha_{\mt}'}\Phi_{\mt,\mt}
			(C^{{\alpha}_{\mt,a}})^{-1} (C^{\beta_{\mt}})^{-1} F_{\rm T}\\
			&=(-1)^{|\alpha_{\mt}'|+|\beta_{\mt}'|+|\alpha_{\mt}|+|\beta_{\mt}|} f_{{\rm S},{\rm T}_a},
		\end{align*}
		where in the second equation, we have used \eqref{CFC=F} and the fact $\ms=\mt$. Similarly, we can compute \eqref{eq2.prop.taufST}.
		
		Suppose $\varepsilon_n(f_{{\rm S},{\rm T}_a})=\eta f_{({\rm S}\downarrow_{n-1})_{\bar{0}},({\rm T}\downarrow_{n-1})_{c}}$ for some $0\neq \eta\in  \mathbb{K}.$
		Then we have
		\begin{align*}
			(-1)^{|\alpha_{\mt}'|+|\beta_{\mt}'|+|\alpha_{\mt}|+|\beta_{\mt}|} & \eta f_{{\rm S}\downarrow_{n-1},({\rm T}\downarrow_{n-1})_{c}}\\
			&=\varepsilon_n ((-1)^{|\alpha_{\mt}'|+|\beta_{\mt}'|+|\alpha_{\mt}|+|\beta_{\mt}|} f_{{\rm S},{\rm T}_a} )\\
			&=\varepsilon_n (C_i f_{{\rm S},{\rm T}_a} C_i)\\
			&=C_i\varepsilon_n(f_{{\rm S},{\rm T}_a})C_i\\
			&=agC_i  f_{{\rm S}\downarrow_{n-1},({\rm T}\downarrow_{n-1})_{c}}C_i\\
			&=(-1)^{|\alpha_{\mt}'|+|\beta_{\mt}'\downarrow_{n-1}|+|\alpha_{\mt}|+|\beta_{\mt}\downarrow_{n-1}|}\eta f_{{\rm S}\downarrow_{n-1},({\rm T}\downarrow_{n-1})_{c}},
		\end{align*}
		where in the second and last equalities we have used \eqref{eq1.prop.taufST} and \eqref{eq2.prop.taufST}, in the third equality we have used the ($\mathcal{H}^f_c(n-1),\mathcal{H}^f_c(n-1)$)-bilinearity of map $\varepsilon_n.$ Now the Lemma follows from the fact that $\eta\neq0$.
	\end{proof}
\begin{prop}\label{prop.taufST}
Suppose ${\rm S}=(\ms,\alpha_{\ms}',\beta_{\mt}')\in{\rm Tri}_{\bar{0}}(\undla)$ and ${\rm T}=(\mt,\alpha_{\mt},\beta_{\mt})\in{\rm Tri}(\undla).$ If ${\rm S}\neq{\rm T},$ then $\tau_{r,n}(f_{{\rm S},{\rm T}})=0.$
\end{prop}
\begin{proof}
Assume $\tau_{r,n}(f_{{\rm S},{\rm T}})\neq0.$ Recall that $\tau_{r,n}=\varepsilon_1\circ \cdots \circ \varepsilon_n$ and $|\tau_{r,n}|=0$.
Hence we have $|f_{{\rm S},{\rm T}}|=\bar{0}$. We shall prove ${\rm S}={\rm T}$ to deduce contradiction.

(i) By applying Lemma \ref{lem.taufST} repeatedly, we deduce that $\ms=\mt.$

(ii) Now we prove $\beta_{\mt}'=\beta_{\mt}.$

Otherwise, we can choose a maximal number $k\in[n]\setminus\mathcal{D}_{\mt}$ such that $\nu_{\beta_{\mt}'}(k)\nu_{\beta_{\mt}}(k)=-1.$ Note that $\tau_{r,n}=\tau_{r,k}\circ\varepsilon_{k+1}\circ \cdots \circ \varepsilon_n$. By Lemmas \ref{lem.taufST}, \ref{tinylemma1}, we can assume that $k=n$ without loss of generality. Moreover, we set $\underline{\mu}=\text{shape}(\mt\downarrow_{n-1}).$

(ii.1) Suppose $d_{\undla}=0.$ Since $n\notin\mathcal{D}_{\mt},$ we have $d_{\undmu}=0$. By $\nu_{\beta_{\mt}'}(n)\nu_{\beta_{\mt}}(n)=-1,$ we deduce
$$|\beta_{\mt}'|+|\beta_{\mt}| \not\equiv |\beta_{\mt}'\downarrow_{n-1}|+|\beta_{\mt}\downarrow_{n-1}| \pmod 2.$$ This contradicts Lemma \ref{lem.taufST} (1).

(ii.2) Suppose $d_{\undla}=1.$ Since $n\notin\mathcal{D}_{\mt},$ we have $d_{\undmu}=1$.
Applying Lemma \ref{tinylemma1}, we have
\begin{align*}
|\beta_{\mt}'|+|\beta_{\mt}| \equiv |\beta_{\mt}'\downarrow_{n-1}|+|\beta_{\mt}\downarrow_{n-1}|\pmod 2.
\end{align*}
But $\nu_{\beta_{\mt}'}(n)\nu_{\beta_{\mt}}(n)=-1,$ this leads to a contradiction.

(iii) We finally prove $\alpha_{\mt}'=\alpha_{\mt}.$

Otherwise, we can choose a maximal number $k':=d(\mt,\mt^{\undla})(i_{2p-1})\in\mathcal{OD}_{\mt}$ (where $1\leq p\leq \lceil t/2\rceil$)
such that
\begin{align*}
\nu_{\alpha_{\mt}'}(k')\nu_{\alpha_{\mt}}(k')=-1.
\end{align*}
Suppose $p=1,$ then by definition of $p$ and $\beta_\mt=\beta_\mt'$, we deduce that $$|f_{{\rm S},{\rm T}}|=|\alpha_{\mt}'|+|\beta_{\mt}'|+|\alpha_{\mt}|+|\beta_{\mt}| \equiv 1\pmod 2,
$$  which is a contradiction. Hence $p>1$.
Since $\beta_\mt=\beta_\mt'$, $|f_{{\rm S},{\rm T}}|=\bar{0}$, we have
\begin{align}\label{eq:twoalphaparities}
|\alpha_{\mt}'|\equiv |\alpha_{\mt}|\pmod 2.
\end{align}
We set $h:=d(\mt,\mt^{\undla})(i_{2p-2})\in \mathcal{D}_{\mt}\setminus\mathcal{OD}_{\mt},$ then we have $$|\alpha_{\mt}'\downarrow_{h}|\not\equiv |\alpha_{\mt}\downarrow_{h}| \pmod 2$$
by the definition of $p$ and \eqref{eq:twoalphaparities}. Hence
$$|f_{{\rm S}\downarrow_{h},{\rm T}\downarrow_h}|=|f_{(\mt\downarrow_{h},\alpha_{\mt}'\downarrow_{h}, \beta_{\mt}\downarrow_{h}),(\mt\downarrow_{h},\alpha_{\mt}\downarrow_{h}, \beta_{\mt}\downarrow_{h})}|=\bar{1},$$
It follows that
\begin{align*}
\varepsilon_1\circ \cdots \circ \varepsilon_h\left(f_{{\rm S}\downarrow_{h},{\rm T}\downarrow_h}\right)=0.
\end{align*}
By Lemma \ref{lem.taufST}, this again contradicts $\tau_{r,n}(f_{{\rm S},{\rm T}})\neq 0.$
\end{proof}

By Proposition \ref{prop.taufST}, we only need to consider the value $\tau_{r,n}(F_{\rm T})$ for any ${\rm T}\in{\rm Tri}_{\bar{0}}(\undla).$

\begin{lem}\label{typeos, scalars ak}
Let ${\rm T}=(\mt,\alpha_{\mt},\beta_{\mt})\in{\rm Tri}_{\bar{0}}(\undla),$
\begin{enumerate}
\item If $d_{\underline{\mu}}=0,$ then we have
\begin{align*}
p_{k,b}(F_{{\rm T}})\in\mathbb{K}\delta_{b,\bar{0}}F_{{\rm T}\downarrow_{n-1}}, \quad \text{for $0\leq k\leq r-1,$ $b\in\mathbb{Z}_2.$}
\end{align*}

\item If $d_{\underline{\mu}}=1,$ then we have
\begin{align*}
p_{k,b}(F_{\rm T})\in\mathbb{K}f_{({\rm T}\downarrow_{n-1})_{\bar{0}},({\rm T}\downarrow_{n-1})_{b}}, \quad \text{for $0\leq k\leq r-1,$ $b\in\mathbb{Z}_2.$}
\end{align*}
\end{enumerate}
\end{lem}
\begin{proof}
	The proof is similar to that of Lemma \ref{lem.taufST}. In fact, we first use Proposition \ref{generators action on seminormal basis} (1), Lemma \ref{important condition1} (2) and the ($\mathcal{H}^f_c(n-1),\mathcal{H}^f_c(n-1)$)-bilinearity of the $p_{k,b}$ to determine all of the eigenvalues of $p_{k,b}(F_{\rm T})$. Note that $|p_{k,b}|=b$ for any $0\leq k <r$. Hence, the Lemma follows from Lemma \ref{lem:clifford rep}.
\end{proof}

The next Lemma establishes  connections among coefficients appearing in the images of $p_{k,b}$ on different idempotents.
\begin{lem}\label{reductionlem}
Let ${\rm T}=(\mt,\alpha_{\mt},\beta_{\mt})\in{\rm Tri}_{\bar{0}}(\undla).$ We set $\underline{\mu}:=\text{shape}(\mt\downarrow_{n-1}).$
\begin{enumerate}
	\item Suppose that $d_{\underline{\mu}}=0$ and
\begin{align*}
p_{k,\overline{0}}(F_{{\rm T}})=a_{k,\overline{0}}F_{{\rm T}\downarrow_{n-1}}, \quad \text{for $a_{k,b}\in\mathbb{K},$ $0\leq k\leq r-1.$}
\end{align*}
Then for any $0\leq k\leq r-1,$ the following hold.
\begin{enumerate}
\item Suppose $i\in [n-1]\cap(\mathcal{OD}_{\mt}\setminus\{d(\mt,\mt^{\undla})(i_t)\}),$ we have\begin{align*}
p_{k,\overline{0}}\left(F_{(\mt,\alpha_{\mt}+e_i, \beta_{\mt})}\right)
=a_{k,\overline{0}}F_{(\mt\downarrow_{n-1},\alpha_{\mt}\downarrow_{n-1}+e_i,\beta_{\mt}\downarrow_{n-1})},
\end{align*}
\item Suppose $j\in[n-1]\setminus\mathcal{D}_{\mt},$  we have
\begin{align*}
p_{k,\overline{0}}\left(F_{(\mt,\alpha_{\mt}, \beta_{\mt}+e_j)}\right)
=a_{k,\overline{0}}F_{(\mt\downarrow_{n-1},\alpha_{\mt}\downarrow_{n-1},\beta_{\mt}\downarrow_{n-1}+e_j)},
\end{align*}

\end{enumerate}
\item Suppose that $d_{\underline{\mu}}=1$ and
\begin{align*}
p_{k,b}(F_{{\rm T}})=a_{k,b}f_{{\rm T}\downarrow_{n-1},({\rm T}\downarrow_{n-1})_{b}}, \quad \text{for $a_{k,b}\in\mathbb{K},$ $0\leq k\leq r-1,$ $b\in\mathbb{Z}_2.$}
\end{align*}
Then for any $0\leq k\leq r-1,$ $b\in\mathbb{Z}_2,$ the following hold.
\begin{enumerate}
\item Suppose $i\in[n-1]\cap(\mathcal{OD}_{\mt}\setminus\{d(\mt,\mt^{\undla})(i_{t-1}),d(\mt,\mt^{\undla})(i_t)\}),$  we have
\begin{align*}
p_{k,b}\left(F_{(\mt,\alpha_{\mt}+e_i, \beta_{\mt})}\right)
=(-1)^{b}a_{k,b}f_{(\mt\downarrow_{n-1},\alpha_{\mt\downarrow_{n-1},\bar{0}}+e_i,\beta_{\mt}\downarrow_{n-1}),
(\mt\downarrow_{n-1},\alpha_{\mt\downarrow_{n-1},b}+e_i,\beta_{\mt}\downarrow_{n-1})}.
\end{align*}
\item Suppose $d_{\undla}=0,$ we have
\begin{align*}
	p_{k,b}\left(F_{(\mt,\alpha_{\mt}+e_{d(\mt,\mt^{\undla})(i_{t-1})}, \beta_{\mt})}\right)
	=(-1)^{b}a_{k,b}f_{{\rm T}\downarrow_{n-1},({\rm T}\downarrow_{n-1})_{b}};
\end{align*}
\item Suppose $j\in[n-1]\setminus\mathcal{D}_{\mt},$  we have
\begin{align*}
	p_{k,b}\left(F_{(\mt,\alpha_{\mt}, \beta_{\mt}+e_j)}\right)
	=(-1)^{b}a_{k,b}f_{(\mt\downarrow_{n-1},\alpha_{\mt\downarrow_{n-1},\bar{0}},\beta_{\mt}\downarrow_{n-1}+e_j),
		(\mt\downarrow_{n-1},\alpha_{\mt\downarrow_{n-1},b},\beta_{\mt}\downarrow_{n-1}+e_j)}.
\end{align*}
\end{enumerate}

\item Suppose that $d_{\undla}=1$ and $n\notin\mathcal{D}_{\mt},$ then we have
$
p_{k,\bar{1}}(F_{\rm T})
=0.
$
\end{enumerate}
\end{lem}
\begin{proof}
We only prove (2), since the proof of (1) is similar and simpler.

For any $i\in[n-1]\cap(\mathcal{OD}_{\mt}\setminus\{d(\mt,\mt^{\undla})(i_{t-1}),d(\mt,\mt^{\undla})(i_t)\}),$ we do the following computation as for \eqref{CFC=F}.

\begin{align*}
F_{(\mt,\alpha_{\mt}+e_i, \beta_{\mt})}
&=\left(C^{\alpha_{\mt}+e_i}\gamma_{\mt}(C^{\alpha_{\mt}+e_i})^{-1} \right)
 \prod_{k=1}^{n}\prod_{\mathtt{b}\in \mathtt{B}(k)\atop \mathtt{b}\neq \mathtt{b}_{+}(\res_{\mt}(k))}\frac{X_k^{\nu_{\beta_{\mt}}(k)}-\mathtt{b}}{\mathtt{b}_{+}(\res_{\mt}(k))-\mathtt{b}}\\
&=\left(C_i C^{\alpha_{\mt}}\gamma_{\mt}(C^{\alpha_{\mt}})^{-1} C_i \right)
 \prod_{k=1}^{n}\prod_{\mathtt{b}\in \mathtt{B}(k)\atop \mathtt{b}\neq \mathtt{b}_{+}(\res_{\mt}(k))}\frac{X_k^{\nu_{\beta_{\mt}}(k)}-\mathtt{b}}{\mathtt{b}_{+}(\res_{\mt}(k))-\mathtt{b}}\\
&=C_i\left(C^{\alpha_{\mt}}\gamma_{\mt}(C^{\alpha_{\mt}})^{-1} \right)
 \prod_{k=1}^{n}\prod_{\mathtt{b}\in \mathtt{B}(k)\atop \mathtt{b}\neq \mathtt{b}_{+}(\res_{\mt}(k))}\frac{X_k^{\nu_{\beta_{\mt}}(k)(-1)^{\delta_{k,i}}}-\mathtt{b}}{\mathtt{b}_{+}(\res_{\mt}(k))-\mathtt{b}}C_i\\
&=C_i\left(C^{\alpha_{\mt}}\gamma_{\mt}(C^{\alpha_{\mt}})^{-1} \right)
 \prod_{k=1}^{n}\prod_{\mathtt{b}\in \mathtt{B}(k)\atop \mathtt{b}\neq \mathtt{b}_{+}(\res_{\mt}(k))}\frac{X_k^{\nu_{\beta_{\mt}}(k)}-\mathtt{b}}{\mathtt{b}_{+}(\res_{\mt}(k))-\mathtt{b}}C_i\\
&=C_i F_{{\rm T}} C_i,
\end{align*} where in the penultimate equality, we have used Lemma \ref{change sign in diagona}.
By a similar computation, we have
\begin{align*}
&C_i f_{{\rm T}\downarrow_{n-1},({\rm T}\downarrow_{n-1})_{b}} C_i\\
&=C_i F_{{\rm T}\downarrow_{n-1}}C^{\beta_{\mt}\downarrow_{n-1}}C^{\alpha_{\mt\downarrow_{n-1},\bar{0}}}\Phi_{\mt\downarrow_{n-1},\mt\downarrow_{n-1}}
(C^{\alpha_{\mt\downarrow_{n-1},b}})^{-1}(C^{\beta_{\mt}\downarrow_{n-1}})^{-1}F_{{\rm T}\downarrow_{n-1}}C_i \\
&=F_{{\rm T}'} C^{\beta_{\mt}\downarrow_{n-1}}C^{\alpha_{\mt\downarrow_{n-1},\bar{0}}+e_i}\Phi_{\mt\downarrow_{n-1},\mt\downarrow_{n-1}}(C^{\alpha_{\mt\downarrow_{n-1},b}+e_i})^{-1}(C^{\beta_{\mt}\downarrow_{n-1}})^{-1}F_{{\rm T}'} \\
&=f_{{\rm T}',{\rm T}'_{b}},
\end{align*}
where ${\rm T}':=(\mt\downarrow_{n-1},\alpha_{\mt\downarrow_{n-1},\bar{0}}+e_i,\beta_{\mt}\downarrow_{n-1})\in {\rm Tri}_{\overline{0}}(\underline{\mu})$ and in the second equation, we used the fact that $i\notin\{d(\mt,\mt^{\undla})(i_{t-1}),d(\mt,\mt^{\undla})(i_t)\}$.
Then we can deduce that
\begin{align*}
&p_{k,b}\left(F_{(\mt,\alpha_{\mt}+e_i, \beta_{\mt})}\right)\\
&=p_{k,b}\left(C_i F_{{\rm T}} C_i \right)\\
&=(-1)^{b}C_i p_{k,b}(F_{{\rm T}}) C_i\\
&=(-1)^{b}a_{k,b}C_i f_{{\rm T}\downarrow_{n-1},({\rm T}\downarrow_{n-1})_{b}} C_i\\
&=(-1)^{b}a_{k,b}f_{{\rm T}',{\rm T}'_{b}},
\end{align*}
where in the second equation, we have used the ($\mathcal{H}^f_c(n-1),\mathcal{H}^f_c(n-1)$)-superbilinearity of maps $p_{k,b}$ (see \eqref{superbilinearity of pk1}).
This proves (2.a). The proofs of (2.b) and (2.c) are analogous to that of (2.a). Hence we omit them.

(3) Suppose $d_{\undla}=1$ and $n\notin\mathcal{D}_{\mt}.$ Then we have $d_{\undmu}=1.$
By Lemma \ref{typeos, scalars ak} (2), we have
\begin{align*}
	p_{k,b}(F_{{\rm T}})=a_{k,b}f_{{\rm T}\downarrow_{n-1},({\rm T}\downarrow_{n-1})_{b}}, \quad \text{for $a_{k,b}\in\mathbb{K},$ $0\leq k\leq r-1,$ $b\in\mathbb{Z}_2.$}
\end{align*}
Let $i:=d(\mt,\mt^{\undla})(i_{t})<n.$ Use the similar computations with (2), we have
\begin{align*}
C_iF_{{\rm T}}C_i=F_{{\rm T}},\quad
C_if_{{\rm T}\downarrow_{n-1},({\rm T}\downarrow_{n-1})_{\bar{1}}}C_i=f_{{\rm T}\downarrow_{n-1},({\rm T}\downarrow_{n-1})_{\bar{1}}}.
\end{align*}
Then we have
\begin{align*}
&a_{k,\bar{1}}F_{{\rm T}\downarrow_{n-1},({\rm T}\downarrow_{n-1})_{\bar{1}}}\\
&=p_{k,\bar{1}}( F_{{\rm T}})\\
&=p_{k,\bar{1}}(C_iF_{{\rm T}}C_i)\\
&=-C_ip_{k,\bar{1}}(F_{{\rm T}})C_i\\
&=-a_{k,\bar{1}}C_if_{{\rm T}\downarrow_{n-1},({\rm T}\downarrow_{n-1})_{\bar{1}}}C_i\\
&=-a_{k,\bar{1}}f_{{\rm T}\downarrow_{n-1},({\rm T}\downarrow_{n-1})_{\bar{1}}},
\end{align*}
where in the third equality we have used the ($\mathcal{H}^f_c(n-1),\mathcal{H}^f_c(n-1)$)-superbilinearity of maps $p_{k,\bar{1}}$ (see \eqref{superbilinearity of pk1}).
It follows from char $\mathbb{K}\neq2$ that $a_{k,\bar{1}}=0$ for $0\leq k\leq r-1.$
\end{proof}

In particular, we deduce the following Corollary, which is key in the computations of next two sections.
\begin{cor}\label{reductionrem}
Let ${\rm T}=(\mt,\alpha_{\mt},\beta_{\mt})\in{\rm Tri}_{\bar{0}}(\undla).$
Suppose that $\varepsilon_n(F_{\rm T})=a_{0,\bar{0}} F_{{\rm T}\downarrow_{n-1}}$ for some scalar $a_{0,\bar{0}}\in\mathbb{K},$ then we have the following.
	\begin{enumerate}
		\item Suppose $\delta_{\beta_\mt}(n)=0$, then $\varepsilon_n \left(F_{(\mt,0,0)}\right)=a_{0,\bar{0}} F_{(\mt\downarrow_{n-1},0,0)}.$
		\item Suppose $\delta_{\beta_\mt}(n)=1$, then $\varepsilon_n \left(F_{(\mt,0,e_n)}\right)=a_{0,\bar{0}} F_{(\mt\downarrow_{n-1},0,0)}.$
		\end{enumerate}
\end{cor}
\begin{proof}
	We can repeatedly apply Lemma \ref{reductionlem} to add some $e_i$ and finally make all but the $n$-th components in $\alpha_\mt$ and  $\beta_\mt$ to be zero. Moreover, in each step, the coefficient $a_{0,\bar{0}}$ appearing in the image of $\varepsilon_n$ remains stable by Lemma \ref{reductionlem}. This proves the Corollary.
\end{proof}
By Corollary \ref{reductionrem}, to compute $\varepsilon_n(F_{\rm T})$
for general ${\rm T}=(\mt,\alpha_{\mt},\beta_{\mt})\in{\rm Tri}_{\bar{0}}(\undla)$, we may assume that $\alpha_{\mt}=0$ and $\beta_{\mt}\in\{0,e_n\}$ without loss of generality.

\subsection{(Super)symmetrizing form of cyclotomic Sergeev algebra}
Let $\undQ=(Q_1,Q_2,\ldots,Q_m)\in{\rm R}^m$. In this subsection, we consider the cyclotomic Sergeev algebra $\mhgcn$ defined over the integral domain ${\rm R}$, where $g$ is as in \eqref{deg-cyc-even} or \eqref{deg-cyc-odd}.
Recall the level $r:=\text{deg}(g).$ In the degenerate case,
there is also a Frobenius form $\mathtt{t}_{r,n}:\mhgcn \rightarrow {\rm R}$ \label{pag:(super)sym. form. dege} via the degenerate Mackey decomposition (\cite[Section 3-e]{BK}). The closed formula for $\mathtt{t}_{r,n}$ was obtained in \cite[Theorem 3.10]{LS1}:

	$$
	\mathtt{t}_{r,n}(x^{\alpha}c^\beta w)=
	\begin{cases}
		1,& \text{{\rm if} $\alpha_{1}= \cdots = \alpha_{n}=r-1, \beta=0, w=1$};\\
		0, & \text{{\rm otherwise}},
	\end{cases}
	$$

Recall Definition \ref{(super)symmetrizing form}.
\begin{thm}\label{LS1:Theorem 1.2}\cite[Theorem 1.2]{LS1}

(i) If the level $r$ is odd, then $\mhgcn$ is symmetric with symmetrizing form $\mathtt{t}_{r,n}.$
	
(ii) If the level $r$ is even, then $\mhgcn$ is supersymmetric with supersymmetrizing form $\mathtt{t}_{r,n}.$
\end{thm}

{\bf From now on, we assume that the separation condition $P_n^{g}(\undQ)\neq 0$ holds for $\mhgcn$ over $\mathbb{K}$.}
Then we have the analog of Proposition \ref{prop.taufST} in the degenerate case.
\begin{prop}\label{prop.taufST.dege}
Let $\undla\in\mathscr{P}^{g}_{n}$.
Suppose ${\rm S}=(\ms,\alpha_{\ms}',\beta_{\mt}')\in{\rm Tri}_{\bar{0}}(\undla)$ and ${\rm T}=(\mt,\alpha_{\mt},\beta_{\mt})\in{\rm Tri}(\undla).$ If ${\rm S}\neq{\rm T},$ then $\mathtt{t}_{r,n}(\mathfrak{f}_{{\rm S},{\rm T}})=0.$
\end{prop}
\begin{proof}
It is completely similar to Proposition \ref{prop.taufST} or alternatively, we can use the (super)symmetric property to deduce the Proposition.
\end{proof}

Again, we only need to consider the values $\mathtt{t}_{r,n}(\mathcal{F}_{\rm T})$ for ${\rm T}\in{\rm Tri}_{\bar{0}}(\undla),$ which
are almost the inverses of Schur elements with respect to the (super)symmetrizing form $\mathtt{t}_{r,n}.$
\label{pag:Schur elements.dege}
\begin{lem}\label{Schur elements and tF}
Suppose the separation condition $P_n^{g}(\undQ)\neq 0$ holds for $\mhgcn$.
Then we have the following.

(1) If the level $r$ is odd, let $\undla\in\mathscr{P}^{g}_{n},$ then the Schur element $\mathtt{s}_{\undla}$ of simple module $D(\undla)$ is equal to $1/\mathtt{t}_{r,n}(\mathcal{F}_{\rm T})$ for any ${\rm T}\in{\rm Tri}_{\bar{0}}(\undla).$

(2) If the level $r$ is even, let $\undla\in\mathscr{P}^{g}_{n},$ then the Schur element $\mathtt{s}_{\undla}$ of simple module $D(\undla)$ is equal to $(-1)^{|\beta_{\mt}|}/\mathtt{t}_{r,n}(\mathcal{F}_{\rm T})$ for any ${\rm T}=(\mt,\beta_{\mt})\in{\rm Tri}(\undla).$
\end{lem}
\begin{proof}
(1) By Proposition \ref{schur formula 1}, the symmetrizing form $\mathtt{t}_{r,n}$ can be written as a linear combination of all irreducible characters $\chi_{\undla}$ for $\undla\in\mathscr{P}^{g}_{n},$
\begin{align*}
\mathtt{t}_{r,n}=\sum_{\undla\in\mathscr{P}^{g}_{n}}\frac{1}{2^{d_{\undla}}\mathtt{s}_{\undla}}\chi_{\undla}.
\end{align*}
Recall the $\mathbb{K}-$basis of $D(\undla)$ in \eqref{basis of D(undla)} and use Lemma \ref{idempotent action. dege}, we have $\chi_{\undla}(\mathcal{F}_{\rm T})=2^{d_{\undla}}$ for any ${\rm T}\in{\rm Tri}_{\bar{0}}(\undla).$
Then we deduce that $\mathtt{s}_{\undla}=1/\mathtt{t}_{r,n}(\mathcal{F}_{\rm T})$ for each $\undla\in\mathscr{P}^{g}_{n}.$

(2) By Proposition \ref{supersym. schur formula 2}, the supersymmetrizing form $\mathtt t_{r,n}$ can be written as a linear combination of all irreducible supercharacters $\chi_{\undla}'$ for $\undla\in\mathscr{P}^{g}_{n},$
\begin{align*}
\mathtt{t}_{r,n}=\sum_{\undla\in\mathscr{P}^{g}_{n}}\frac{1}{\mathtt{s}_{\undla}}\chi_{\undla}'.
\end{align*}
Similarly, by using the $\mathbb{K}-$basis of $D(\undla)$ in \eqref{basis of D(undla)} and Lemma \ref{idempotent action. dege}, we have $\chi'_{\undla}(\mathcal{F}_{\rm T})=(-1)^{|\beta_{\mt}|}$ for any ${\rm T}=(\mt,\beta_{\mt})\in{\rm Tri}(\undla).$
Then we deduce that $\mathtt{s}_{\undla}=(-1)^{|\beta_{\mt}|}/\mathtt{t}_{r,n}(\mathcal{F}_{\rm T}).$
\end{proof}

To close this section, we shall present two technical Lemmas, which are useful for our computations in the next two sections.
\begin{lem}\label{TechLem1}\cite[Lemma 5.1]{HLL}
For any $k,\ell\in\mathbb{N},$ let $\mathbb{X}_1,\ldots,\mathbb{X}_{\ell+k},$ ${\rm Z}_1,\ldots,{\rm Z}_\ell$ be $2\ell+k$ indeterminates over $\mathbb{Z}$, then we have
$$\begin{aligned}
&{\rm det}\left(\begin{matrix}
\frac{\mathbb{X}_1}{\mathbb{X}_1-{\rm Z}_1} & \frac{\mathbb{X}_1}{\mathbb{X}_1-{\rm Z}_2}&\cdots  &\frac{\mathbb{X}_1}{\mathbb{X}_1-{\rm Z}_\ell}
&1 & \mathbb{X}_1 & \mathbb{X}_1^{2} &\cdots  & \mathbb{X}_1^{k-1}\\
\frac{\mathbb{X}_2}{\mathbb{X}_2-{\rm Z}_1} &  \frac{\mathbb{X}_2}{\mathbb{X}_2-{\rm Z}_2}&\cdots  &\frac{\mathbb{X}_2}{\mathbb{X}_2-{\rm Z}_\ell}
&1 & \mathbb{X}_2 & \mathbb{X}_2^{2} &\cdots  & \mathbb{X}_2^{k-1}\\
\vdots&\vdots  & &\vdots &\vdots&\vdots &\vdots  & &\vdots  \\
\frac{\mathbb{X}_{\ell+k}}{\mathbb{X}_{\ell+k}-{\rm Z}_{1}} &  \frac{\mathbb{X}_{\ell+k}}{\mathbb{X}_{\ell+k}-{\rm Z}_{2}}&\cdots  &\frac{\mathbb{X}_{\ell+k}}{\mathbb{X}_{\ell+k}-{\rm Z}_\ell}
&1 & \mathbb{X}_{\ell+k} & \mathbb{X}_{\ell+k}^{2} &\cdots  & \mathbb{X}_{\ell+k}^{k-1}
\end{matrix}\right)\\
&\qquad = \,\frac{\prod\limits_{1\leq t\leq \ell}{\rm Z}_t\prod\limits_{1\leq p< q\leq \ell}({\rm Z}_p-{\rm Z}_q)\prod\limits_{1\leq u<v\leq \ell+k}(\mathbb{X}_v-\mathbb{X}_u)}
{\prod\limits_{r=1}^{\ell+k}\prod\limits_{s=1}^{\ell}(\mathbb{X}_r-{\rm Z}_s)}.
\end{aligned}$$
\end{lem}

\begin{lem}\label{TechLem2}\cite[Lemma 5.4]{HLL}
For any $k,\ell\in\mathbb{N},$ let $\mathbb{X}_1,\ldots,\mathbb{X}_{\ell+k},$ ${\rm Z}_1,\ldots,{\rm Z}_\ell$ be $2\ell+k$ indeterminates over $\mathbb{Z}$, then we have
$$\begin{aligned}
&{\rm det}\left(\begin{matrix}
\frac{1}{\mathbb{X}_1-{\rm Z}_1} & \frac{1}{\mathbb{X}_1-{\rm Z}_2}&\cdots  &\frac{1}{\mathbb{X}_1-{\rm Z}_\ell}
&1 & \mathbb{X}_1 & \mathbb{X}_1^{2} &\cdots  & \mathbb{X}_1^{k-1}\\
\frac{1}{\mathbb{X}_2-{\rm Z}_1} &  \frac{1}{\mathbb{X}_2-{\rm Z}_2}&\cdots  &\frac{1}{\mathbb{X}_2-{\rm Z}_\ell}
&1 & \mathbb{X}_2 & \mathbb{X}_2^{2} &\cdots  & \mathbb{X}_2^{k-1}\\
\vdots&\vdots  & &\vdots &\vdots&\vdots &\vdots  & &\vdots  \\
\frac{1}{\mathbb{X}_{\ell+k}-{\rm Z}_{1}} &  \frac{1}{\mathbb{X}_{\ell+k}-{\rm Z}_{2}}&\cdots  &\frac{1}{\mathbb{X}_{\ell+k}-{\rm Z}_\ell}
&1 & \mathbb{X}_{\ell+k} & \mathbb{X}_{\ell+k}^{2} &\cdots  & \mathbb{X}_{\ell+k}^{k-1}
\end{matrix}\right)\\
&\qquad = \,\frac{\prod\limits_{1\leq p< q\leq \ell}({\rm Z}_p-{\rm Z}_q)\prod\limits_{1\leq u<v\leq \ell+k}(\mathbb{X}_v-\mathbb{X}_u)}
{\prod\limits_{r=1}^{\ell+k}\prod\limits_{s=1}^{\ell}(\mathbb{X}_r-{\rm Z}_s)}.
\end{aligned}$$
\end{lem}

\section{(Super)symmetrizing forms and Schur elements at even level}\label{dot=0}
\noindent\textbf{This section is restricted to $f$ in \eqref{cyc-even} and, in the final subsection, $g$ in \eqref{deg-cyc-even}. The polynomial in \eqref{cyc-two-special-factors} is not included, although its level is also even.}

We first compute at separated parameters and then descend the supersymmetrizing form to general parameters. When every $X_i-X_i^{-1}$ is invertible, a further modification gives a symmetrizing form.
We fix $\undla\in\mathscr{P}^{f}_{n}$ and then fix ${\rm T}=(\mt,0,\beta_\mt)\in{\rm Tri}(\mathscr{P}^{f}_{n})$.
Since $\mathscr{P}^{f}_{n}$ is the set of $m$-multipartitions of $n$ and
${\rm Tri}(\mathscr{P}^{f}_{n})=\Std(\mathscr{P}^{f}_{n})\times \mathbb{Z}_2^n,$  we can abbreviate ${\rm T}=(\mt,\beta_\mt)$ in this section.

\subsection{Values of the Frobenius form on primitive idempotents}
In this subsection, we assume $P_n^{f}(q^2,\undQ)\ne0$ over $\mathbb K$. We determine the coefficients of the Mackey decomposition by a linear system and use them to compute $\varepsilon_n(F_{\rm T})$. Applying the maps $\varepsilon_n,\ldots,\varepsilon_1$ successively gives $\tau_{r,n}(F_{\rm T})$ and proves the following product formula.

\begin{thm}\label{mainthm type0 nondege}
	Suppose $f$ is as in \eqref{cyc-even} and $\undla\in\mathscr{P}^{f}_{n}.$ For any ${\rm T}=(\mt,\beta_{\mt})\in{\rm Tri}(\undla),$ we have
	\begin{align}\label{type=0. tauFT}
		\tau_{r,n}(F_{\rm T})
		=\prod\limits_{k=1}^{n}\frac{\mathtt{b}_{\mt,k}^{\nu_{\beta_{\mt}}(k)m}}{\mathtt{b}_{\mt,k}^{\nu_{\beta_{\mt}}(k)}-\mathtt{b}_{\mt,k}^{-\nu_{\beta_{\mt}}(k)}}
		\cdot \prod\limits_{k=1}^{n}\frac{\prod\limits_{\beta\in\Rem(\mt\downarrow_{k-1})}\left(\mathtt{q}(\res_{\mt}(k))-\mathtt{q}(\res(\beta))\right)}{
			\prod\limits_{\alpha\in\Add(\mt\downarrow_{k-1})\setminus \{\mt^{-1}(k)\}}\left(\mathtt{q}(\res_{\mt}(k))-\mathtt{q}(\res(\alpha))\right)}.
	\end{align}
\end{thm}

Recall the definition of $\tau_{r,n}$. In order to prove Theorem \ref{mainthm type0 nondege}, we need to compute $\varepsilon_n(F_{\rm T})$ first. To this end, {\bf we can further assume that $\beta_{\mt}\in\{0,e_n\},$ i.e., $\beta_{\mt}\downarrow_{n-1}=0$ by Corollary \ref{reductionrem}.}

Recall that
$\mathtt{c}_{\mt}^{\ms}:=\mathtt{c}_{\mt,\ms}\mathtt{c}_{\ms,\mt}$
for any $\ms\in\Std(\undla).$
Moreover, for any $(n-1)$-dimensional vector $\alpha=(\alpha_1,\ldots,\alpha_{n-1})\in\mathbb{Z}_2^{n-1}$, we use the same notation $\alpha$ to denote the $n$-dimensional vector $\alpha=(\alpha_1,\ldots,\alpha_{n-1},\bar{0})\in\mathbb{Z}_2^{n}$ by the natural embedding.

Let $\underline{\mu}:={\rm shape}(\mt\downarrow_{n-1})\in \mathscr{P}^{f}_{n-1}.$
\begin{lem}\label{HLL. lem2.13}\cite[Lemma 2.13]{HLL}
For $\underline{\mu}\in\mathscr{P}^{f}_{n-1},$ we have $\sharp\Add(\underline{\mu})=\sharp\Rem(\underline{\mu})+m.$
\end{lem}
By Lemma \ref{HLL. lem2.13}, we can set
\begin{align*}
\Add(\undmu)&:=\{\alpha_1=\mt^{-1}(n),\alpha_2,\ldots,\alpha_N \},\\
\Rem(\undmu)&:=\{\beta_1=\mt^{-1}(n-1),\beta_2,\ldots,\beta_{N-m} \}.
\end{align*}
Let $j_{\iota}\in[n-1]$ be the number such that $\mt^{-1}(j_{\iota})=\beta_{\iota}$, where $\iota\in[N-m].$  In particular, $j_1=n-1.$
For any $\iota\in[N-m],$ we set
\begin{align*}
{\rm U}_{\iota}=(\mathfrak{u}_{\iota}, 0)
:=\begin{cases}s_{n-2}s_{n-3}\cdots s_{j_{\iota}}\cdot (\mt\downarrow_{n-1}, 0) \in {\rm Tri}(\undmu),&\text{if $1<\iota\leq N-m$,}\\
	(\mt\downarrow_{n-1},0), &\text{otherwise.}
	\end{cases}
\end{align*}
For any $\iota\in[N-m]$ and $k\in[N],$ let $\mathfrak{u}_{\iota,k}$ be the unique standard tableau of shape $\undla_{k}:=\undmu\cup\{\alpha_{k}\}\in\mathscr{P}^{f}_{n}$ such that
$\mathfrak{u}_{\iota,k}\downarrow_{n-1}=\mathfrak{u}_{\iota}$ and $\mathfrak{u}_{\iota,k}^{-1}(n)=\alpha_k.$
By definition, we have
$\mathfrak{u}_{\iota,k}^{-1}(n-1)=\mathfrak{u}_{\iota}^{-1}(n-1)=\beta_{\iota}.$ For convenience, we set $\mt_k:=\mathfrak{u}_{1,k}\in\Std(\undla_k)$ for $k\in[N].$

Recall that we have fixed ${\rm T}=(\mt,\beta_{\mt})\in{\rm Tri}(\undla)$ such that $\beta_{\mt}\downarrow_{n-1}=0.$ By the Mackey decomposition \eqref{nondegMadecomequation}, we have
\begin{align}\label{Madecomequation of fTT}
	F_{\rm T}=\mu_{T_{n-1}}(\pi(F_{\rm T}))+\sum_{\substack{0\leq k\leq r-1\\ a\in \mathbb{Z}_2}}X_n^k C_n^a p_{k,a}(F_{\rm T}).
\end{align}

For simplicity, we set $\mathbb{X}_k:=\mathtt{b}_{+}(\res(\alpha_k)),$ ${\rm Y}_{\iota}:=\mathtt{b}_{+}(\res(\beta_{\iota}))$
for $k\in[N],$ $\iota\in[N-m].$ The following Proposition is key to the main result in this section.
\begin{prop}\label{type=0, mainprop}
Let $\undla\in\mathscr{P}^{f}_{n}.$ Suppose ${\rm T}=(\mt,\beta_{\mt})\in{\rm Tri}(\undla)$ with $\beta_{\mt}\downarrow_{n-1}=0$ and ${\rm T}\downarrow_{n-1}=(\mt\downarrow_{n-1},0)\in {\rm Tri}(\underline{\mu}).$ We have the following.
\begin{enumerate}
  \item There are some scalars $a_0,\ldots,a_{r-1},b_{1,\bar{0}},\ldots,b_{N-m,\bar{0}},b_{1,\bar{1}},\ldots,b_{N-m,\bar{1}}\in\mathbb{K}$ such that
$$
\pi(F_{\rm T})
=\sum_{\substack{1\leq\iota\leq N-m\\ \delta\in\mathbb{Z}_2}}b_{\iota,\delta}
  f_{{\rm T}\downarrow_{n-1},(\mathfrak{u}_{\iota}, \delta e_{n-1})}
  \otimes f_{(\mathfrak{u}_{\iota}, \delta e_{n-1}),{\rm T}\downarrow_{n-1}},
$$
$$
p_{k,\bar{0}}(F_{\rm T})=a_k F_{{\rm T}\downarrow_{n-1}},
\quad 0\leq k <r.
$$
  \item We have
\begin{align}\label{type=0, veFT}
\varepsilon_n(F_{\rm T})
=a_0 F_{\rm T\downarrow_{n-1}}
=\frac{\mathtt{b}_{\mt,n}^{\nu_{\beta_{\mt}}(n)m}}{\mathtt{b}_{\mt,n}^{\nu_{\beta_{\mt}}(n)}-\mathtt{b}_{\mt,n}^{-\nu_{\beta_{\mt}}(n)}}
\frac{\prod\limits_{\beta\in\Rem(\mt\downarrow_{n-1})}\left(\mathtt{q}(\res_{\mt}(n))-\mathtt{q}(\res(\beta))\right)}{
	\prod\limits_{\alpha\in\Add(\mt\downarrow_{n-1})\setminus \{\mt^{-1}(n)\}}\left(\mathtt{q}(\res_{\mt}(n))-\mathtt{q}(\res(\alpha))\right)} F_{\rm T\downarrow_{n-1}}.
\end{align}
\end{enumerate}
\end{prop}
\begin{proof}
We define
\begin{align}\label{eq.widehatpi}
\widehat{\pi}(F_{\rm T})
:=\sum_{\substack{1\leq\iota\leq N-m\\ \delta\in\mathbb{Z}_2}}b_{\iota,\delta}&
  f_{{\rm T}\downarrow_{n-1},(\mathfrak{u}_{\iota}, \delta e_{n-1})}
  \otimes f_{(\mathfrak{u}_{\iota}, \delta e_{n-1}),{\rm T}\downarrow_{n-1}}\\
  &\in \mathcal{H}^f_c(n-1)\otimes_{\mathcal{H}^f_c(n-2)}\mathcal{H}^f_c(n-1)\nonumber
\end{align}
for some unknown scalars $b_{\iota,\delta}\in \mathbb{K}.$
Then we get that
\begin{align*}
\mu_{T_{n-1}}(\widehat{\pi}(F_{\rm T}))
=\sum_{\substack{1\leq\iota\leq N-m\\ \delta\in\mathbb{Z}_2}}b_{\iota,\delta}
  \iota_{n-1}\left(f_{{\rm T}\downarrow_{n-1},(\mathfrak{u}_{\iota}, \delta e_{n-1})}\right)
  T_{n-1} \iota_{n-1}\left(f_{(\mathfrak{u}_{\iota}, \delta e_{n-1}),{\rm T}\downarrow_{n-1}}\right).
\end{align*}
By Theorem \ref{maththm of embedding} (1), for each $1\leq\iota\leq N-m$, we have
\begin{align*}
\iota_{n-1}\left(f_{{\rm T}\downarrow_{n-1},(\mathfrak{u}_{\iota}, \delta e_{n-1})}\right)
&=\sum_{\substack{1 \leq k\leq N\\ \delta_1\in\mathbb{Z}_2}}
  f_{(\mt_k,\delta_1 e_n),(\mathfrak{u}_{\iota,k}, \delta e_{n-1}+\delta_1 e_n)},\\
\iota_{n-1}\left(f_{(\mathfrak{u}_{\iota}, \delta e_{n-1}),{\rm T}\downarrow_{n-1}}\right)
&=\sum_{\substack{1 \leq k\leq N\\ \delta_2\in\mathbb{Z}_2}}
  f_{(\mathfrak{u}_{\iota,k}, \delta e_{n-1}+\delta_2 e_n),(\mt_k,\delta_2 e_n)}.
\end{align*}
Combining with Proposition \ref{generators action on seminormal basis} (2) and equation \eqref{Non-deg multiplication1}, we can compute
\begin{align}\label{eq. muTn-1}
&\mu_{T_{n-1}}(\widehat{\pi}(F_{\rm T}))\\
&=\sum_{\substack{1\leq\iota\leq N-m\\ \delta\in\mathbb{Z}_2}}b_{\iota,\delta}
  \sum_{\substack{1\leq k \leq N\\ \delta_1\in\mathbb{Z}_2}}
       \frac{\epsilon}{1-\mathtt{b}_{+}(\res(\beta_{\iota}))^{\nu(\delta)}\mathtt{b}_{+}(\res(\alpha_k))^{-\nu(\delta_1)}}
        \mathtt{c}^{\mt_k}_{\mathfrak{u}_{\iota,k}}
         F_{(\mt_k,\delta_1 e_n)}\nonumber\\
&=\sum_{\substack{1\leq\iota\leq N-m\\ \delta\in\mathbb{Z}_2}}b_{\iota,\delta}
\sum_{\substack{1\leq k \leq N\\ \delta_1\in\mathbb{Z}_2}}
\frac{\epsilon}{1-\mathtt{b}_{+}(\res(\beta_{\iota}))^{\nu(\delta)}\mathtt{b}_{+}(\res(\alpha_k))^{-\nu(\delta_1)}}
\mathtt{c}^{\mathfrak{u}_1}_{\mathfrak{u}_{\iota}}
F_{(\mt_k,\delta_1 e_n)}\nonumber
\end{align}
where
\begin{align*}
\nu(\delta):=
\begin{cases}
-1, &\text{if $\delta=\bar{1}$,}\\
1, &\text{if $\delta=\bar{0}$,}
\end{cases}
\qquad\text{and}\qquad
\nu(\delta_1):=
\begin{cases}
-1, &\text{if $\delta_1=\bar{1}$,}\\
1, &\text{if $\delta_1=\bar{0}$.}
\end{cases}
\end{align*}
By Lemma \ref{typeos, scalars ak} (1), we have
\begin{align}\label{scalars ak}p_{k',\bar{1}}(F_{\rm T})=0, \quad
p_{k',\bar{0}}(F_{\rm T})=a_{k'} F_{{\rm T}\downarrow_{n-1}}
\end{align}
for some $a_{k'}\in \mathbb{K},$ $0\leq k' <r.$

It follows from Theorem \ref{maththm of embedding} (1) and Proposition \ref{generators action on seminormal basis} (1) that
\begin{align}\label{eq. Xnk'pk'0}
&\sum_{k'=0}^{r-1}X_n^{k'}\iota_{n-1}(p_{k',\bar{0}}(F_{\rm T}))\\
&=\sum_{k'=0}^{r-1}a_{k'} X_n^{k'}\iota_{n-1}(F_{{\rm T}\downarrow_{n-1}}) \nonumber\\
&=\sum_{k'=0}^{r-1}a_{k'} \sum_{\substack{1\leq k\leq N\\ \delta_1\in\mathbb{Z}_2}}
    X_n^{k'} F_{(\mt_k,\delta_1 e_n)}\nonumber\\
&= \sum_{\substack{1\leq k\leq N\\ \delta_1\in\mathbb{Z}_2}}\sum_{k'=0}^{r-1}
   a_{k'}
    \mathtt{b}_{+}(\res(\alpha_k))^{\nu(\delta_1)k'}
    F_{(\mt_k,\delta_1 e_n)}.\nonumber
\end{align}

We now regard those scalars
\begin{align}\label{nondegunknownquantities}
a_0,\ldots,a_{r-1},b_{1,\bar{0}},\ldots,b_{N-m,\bar{0}},b_{1,\bar{1}},\ldots,b_{N-m,\bar{1}}
\end{align}
appearing in \eqref{scalars ak} and \eqref{eq.widehatpi} as unknown quantities.
Combining \eqref{Madecomequation of fTT} with \eqref{eq. muTn-1}, \eqref{eq. Xnk'pk'0} and comparing the coefficients of
$F_{(\mt_k,\delta_1 e_n)},$
$k\in [N],$ $\delta_1\in\mathbb{Z}_2$
on both sides of \eqref{Madecomequation of fTT} with $\pi(F_{\rm T})$ replaced
with $\widehat{\pi}(F_{\rm T}),$ we get the following linear system of equations in two cases:
\begin{align}\label{type0 nondegLinearEquations}
& A\cdot \text{diag}(\underbrace{1,\dots,1}_{r\,\text{copies}},
                    \epsilon\mathtt{c}^{\mathfrak{u}_{1}}_{\mathfrak{u}_{1}},\dots,
                     \epsilon\mathtt{c}^{\mathfrak{u}_{1}}_{\mathfrak{u}_{N-m}},
                      \epsilon\mathtt{c}^{\mathfrak{u}_{1}}_{\mathfrak{u}_{1}},\dots,
                       \epsilon\mathtt{c}^{\mathfrak{u}_{1}}_{\mathfrak{u}_{N-m}})\nonumber\\
& \cdot \left(a_0,\dots,a_{r-1},b_{1,\bar{0}},\dots,b_{N-m,\bar{0}},b_{1,\bar{1}},\dots,b_{N-m,\bar{1}}\right)^{T}\nonumber\\
&\quad=\begin{cases}
  (1,\underbrace{0,\dots,0}_{2N-1\,\text{copies}})^T, &\text{if $\delta_{\beta_{\mt}}(n)=\bar{0}$,}\\
  (\underbrace{0,\dots,0}_{N\,\text{copies}},1,\underbrace{0,\dots,0}_{N-1\,\text{copies}})^T, &\text{if $\delta_{\beta_{\mt}}(n)=\bar{1}$,}
  \end{cases}\nonumber
\end{align}
where
{\small
$$A=\begin{pmatrix}
1 & \mathbb{X}_1 & \cdots & \mathbb{X}_1^{r-1} &
\frac{\mathbb{X}_1}{\mathbb{X}_1-{\rm Y}_1} & \cdots & \frac{\mathbb{X}_1}{\mathbb{X}_1-{\rm Y}_{N-m}} &
\frac{\mathbb{X}_1}{\mathbb{X}_1-{\rm Y}_1^{-1}} & \cdots & \frac{\mathbb{X}_1}{\mathbb{X}_1-{\rm Y}_{N-m}^{-1}}\\
\vdots & \vdots &  & \vdots & \vdots &  & \vdots & \vdots &  & \vdots\\
1 & \mathbb{X}_N & \cdots & \mathbb{X}_N^{r-1} &
\frac{\mathbb{X}_N}{\mathbb{X}_N-{\rm Y}_1} & \cdots & \frac{\mathbb{X}_N}{\mathbb{X}_N-{\rm Y}_{N-m}} &
\frac{\mathbb{X}_N}{\mathbb{X}_N-{\rm Y}_1^{-1}} & \cdots & \frac{\mathbb{X}_N}{\mathbb{X}_N-{\rm Y}_{N-m}^{-1}}\\
1 & \mathbb{X}_1^{-1} & \cdots & \mathbb{X}_1^{-(r-1)} &
\frac{\mathbb{X}_1^{-1}}{\mathbb{X}_1^{-1}-{\rm Y}_1} & \cdots & \frac{\mathbb{X}_1^{-1}}{\mathbb{X}_1^{-1}-{\rm Y}_{N-m}} &
\frac{\mathbb{X}_1^{-1}}{\mathbb{X}_1^{-1}-{\rm Y}_1^{-1}} & \cdots & \frac{\mathbb{X}_1^{-1}}{\mathbb{X}_1^{-1}-{\rm Y}_{N-m}^{-1}}\\
\vdots & \vdots &  & \vdots & \vdots &  & \vdots & \vdots &  & \vdots\\
1 & \mathbb{X}_N^{-1} & \cdots & \mathbb{X}_N^{-(r-1)} &
\frac{\mathbb{X}_N^{-1}}{\mathbb{X}_N^{-1}-{\rm Y}_1} & \cdots & \frac{\mathbb{X}_N^{-1}}{\mathbb{X}_N^{-1}-{\rm Y}_{N-m}} &
\frac{\mathbb{X}_N^{-1}}{\mathbb{X}_N^{-1}-{\rm Y}_1^{-1}} & \cdots & \frac{\mathbb{X}_N^{-1}}{\mathbb{X}_N^{-1}-{\rm Y}_{N-m}^{-1}}
\end{pmatrix}_{2N\times 2N}.$$ }
If the matrix $A$ is invertible, then we can deduce that there exist scalars \eqref{nondegunknownquantities} such that the corresponding elements $\widehat{\pi}(F_{\rm T})$ and $p_{k,\bar{0}}(F_{\rm T}),\ k=0,1,\dots,r-1$ satisfy the equality \eqref{Madecomequation of fTT} with $\pi(F_{\rm T})$ replaced by $\widehat{\pi}(F_{\rm T})$.
Then by the uniqueness statement of \eqref{Madecomequation of fTT}, we can deduce that $\widehat{\pi}(F_{\rm T})=\pi(F_{\rm T}).$

To apply Lemmas \ref{TechLem1}, \ref{TechLem2}, we set $\mathbb{X}_{N+k}:=\mathbb{X}_k^{-1},$ $k\in[N]$ and ${\rm Z}_{\iota}:={\rm Y}_{\iota},$ ${\rm Z}_{N-m+\iota}:={\rm Y}_{\iota}^{-1},$ $\iota\in[N-m].$ Then it follows from Lemma \ref{TechLem1} that
\begin{align*}
{\rm det}(A)
&=\frac{\prod\limits_{1\leq t\leq 2(N-m)}{\rm Z}_t \prod\limits_{1\leq p< q\leq 2(N-m)}({\rm Z}_p-{\rm Z}_q)\prod\limits_{1\leq u<v\leq 2N}(\mathbb{X}_v-\mathbb{X}_u)}
{\prod\limits_{j=1}^{2N}\prod\limits_{s=1}^{2(N-m)}(\mathbb{X}_j-{\rm Z}_s)}\\
&=\frac{\prod\limits_{1\leq p< q\leq 2(N-m)}({\rm Z}_p-{\rm Z}_q)\prod\limits_{1\leq u<v\leq 2N}(\mathbb{X}_v-\mathbb{X}_u)}
{\prod\limits_{j=1}^{2N}\prod\limits_{s=1}^{2(N-m)}(\mathbb{X}_j-{\rm Z}_s)},
\end{align*}
where we have used the fact that
$$\prod\limits_{1\leq t\leq 2(N-m)}{\rm Z}_t=\prod\limits_{1\leq \iota \leq N-m}\left({\rm Y}_{\iota}{\rm Y}_{\iota}^{-1}\right)=1.$$
This implies that ${\rm det}(A)$ is invertible by Proposition \ref{compare eigenvalue}.

It remains to prove \eqref{type=0, veFT}.
By Lemma \ref{TechLem2},
the algebraic cofactor $A_{1,1}$ of ${\rm det}(A)$ at position $(1,1)$ is
\begin{align*}
A_{1,1}
&=\left(\prod\limits_{2\leq k\leq 2N}\mathbb{X}_k \right)\cdot\frac{\prod\limits_{1\leq p< q\leq 2(N-m)}({\rm Z}_p-{\rm Z}_q)\prod\limits_{2\leq u<v\leq 2N}(\mathbb{X}_v-\mathbb{X}_u)}{\prod\limits_{j=2}^{2N}\prod\limits_{s=1}^{2(N-m)}(\mathbb{X}_j-{\rm Z}_s)}\\
&=\mathbb{X}_1^{-1}\cdot\frac{\prod\limits_{1\leq p< q\leq 2(N-m)}({\rm Z}_p-{\rm Z}_q)\prod\limits_{2\leq u<v\leq 2N}(\mathbb{X}_v-\mathbb{X}_u)}{\prod\limits_{j=2}^{2N}\prod\limits_{s=1}^{2(N-m)}(\mathbb{X}_j-{\rm Z}_s)},
\end{align*}
where we have used the fact that $\prod_{1\leq k\leq 2N}\mathbb{X}_k=1.$

Similarly, by Lemma \ref{TechLem1}, the algebraic cofactor $A_{N+1,1}$ of ${\rm det}(A)$ at position $(N+1,1)$ is
\begin{align*}
A_{N+1,1}
&=(-1)^N\left(\prod\limits_{\substack{1\leq k\leq 2N\\ k\neq N+1}}\mathbb{X}_k \right)\cdot\frac{\prod\limits_{1\leq p< q\leq 2(N-m)}({\rm Z}_p-{\rm Z}_q)\prod\limits_{\substack{1\leq u<v\leq 2N \\ u,v\neq N+1}}(\mathbb{X}_v-\mathbb{X}_u)}{\prod\limits_{\substack{1\leq j\leq n \\ j\neq N+1}}\prod\limits_{s=1}^{2(N-m)}(\mathbb{X}_j-{\rm Z}_s)}\\
&=(-1)^N\mathbb{X}_1\cdot\frac{\prod\limits_{1\leq p< q\leq 2(N-m)}({\rm Z}_p-{\rm Z}_q)\prod\limits_{\substack{1\leq u<v\leq 2N \\ u,v\neq N+1}}(\mathbb{X}_v-\mathbb{X}_u)}{\prod\limits_{\substack{1\leq j\leq n \\ j\neq N+1}}\prod\limits_{s=1}^{2(N-m)}(\mathbb{X}_j-{\rm Z}_s)},
\end{align*}
where we have also used the fact that $\prod_{1\leq k\leq 2N}\mathbb{X}_k=1.$

\begin{caselist}
	\item $\delta_{\beta_{\mt}}(n)=\bar{0}$\label{delta=0} . Then we have
\begin{align}
&a_0
=\frac{A_{1,1}}{{\rm det}(A)}
=\mathbb{X}_1^{-1}
  \frac{\prod\limits_{s=1}^{2(N-m)}(\mathbb{X}_1-{\rm Z}_s)}{\prod\limits_{v=2}^{2N}(\mathbb{X}_v-\mathbb{X}_1)}\\
&=\frac{\mathtt{b}_{-}(\res_{\mt}(n))}{\mathtt{b}_{-}(\res_{\mt}(n))-\mathtt{b}_{+}(\res_{\mt}(n))}
    \frac{\prod\limits_{\beta\in\Rem(\mt\downarrow_{n-1}),\ast\in\{\pm\}}\left(\mathtt{b}_{+}(\res_{\mt}(n))-\mathtt{b}_{\ast}(\res(\beta))\right)}{
    \prod\limits_{\alpha\in\Add(\mt\downarrow_{n-1})\setminus \{\mt^{-1}(n)\},\ast\in\{\pm\}}\left(\mathtt{b}_{+}(\res_{\mt}(n))-\mathtt{b}_{\ast}(\res(\alpha))\right)}\nonumber\\
&=\frac{\mathtt{b}_{+}(\res_{\mt}(n))^{-m}}{\mathtt{b}_{-}(\res_{\mt}(n))-\mathtt{b}_{+}(\res_{\mt}(n))}
    \frac{\prod\limits_{\beta\in\Rem(\mt\downarrow_{n-1})}\left(\mathtt{q}(\res_{\mt}(n))-\mathtt{q}(\res(\beta))\right)}{
    \prod\limits_{\alpha\in\Add(\mt\downarrow_{n-1})\setminus \{\mt^{-1}(n)\}}\left(\mathtt{q}(\res_{\mt}(n))-\mathtt{q}(\res(\alpha))\right)},\nonumber
\end{align}
where in the last equation we have used the elementary identity
$$(b'-b)(b'-b^{-1})=b'(b'+b'^{-1}-b-b^{-1}),\qquad \forall b,b'\in\mathbb{K}^*,$$
and Lemma \ref{HLL. lem2.13}.

\item  $\delta_{\beta_{\mt}}(n)=\bar{1}.$\label{delta=1} Then we have
\begin{align}
&a_0
=\frac{A_{N+1,1}}{{\rm det}(A)}
=(-1)^N
  \frac{\mathbb{X}_1\cdot \prod\limits_{s=1}^{2(N-m)}(\mathbb{X}_{N+1}-{\rm Z}_s)}{\prod\limits_{u=1}^{N}(\mathbb{X}_{N+1}-\mathbb{X}_u)\prod\limits_{v=N+2}^{2N}(\mathbb{X}_v-\mathbb{X}_{N+1})}\\
&=
   \frac{\mathtt{b}_{+}(\res_{\mt}(n))}{\mathtt{b}_{+}(\res_{\mt}(n))-\mathtt{b}_{-}(\res_{\mt}(n))}
    \frac{\prod\limits_{\beta\in\Rem(\mt\downarrow_{n-1}),\ast\in\{\pm\}}\left(\mathtt{b}_{-}(\res_{\mt}(n))-\mathtt{b}_{\ast}(\res(\beta))\right)}{
    \prod\limits_{\alpha\in\Add(\mt\downarrow_{n-1})\setminus \{\mt^{-1}(n)\},\ast\in\{\pm\}}\left(\mathtt{b}_{-}(\res_{\mt}(n))-\mathtt{b}_{\ast}(\res(\alpha))\right)}\nonumber\\
&=
   \frac{\mathtt{b}_{-}(\res_{\mt}(n))^{-m}}{\mathtt{b}_{+}(\res_{\mt}(n))-\mathtt{b}_{-}(\res_{\mt}(n))}
    \frac{\prod\limits_{\beta\in\Rem(\mt\downarrow_{n-1})}\left(\mathtt{q}(\res_{\mt}(n))-\mathtt{q}(\res(\beta))\right)}{
    \prod\limits_{\alpha\in\Add(\mt\downarrow_{n-1})\setminus \{\mt^{-1}(n)\}}\left(\mathtt{q}(\res_{\mt}(n))-\mathtt{q}(\res(\alpha))\right)}.\nonumber
\end{align}
\end{caselist}
Now \eqref{type=0, veFT} follows from Case \ref{delta=0} and Case \ref{delta=1}.
\end{proof}

\medskip
{\bf Proof of Theorem \ref{mainthm type0 nondege}}: By Corollary \ref{reductionrem}, the equation \eqref{type=0, veFT} holds for general ${\rm T}=(\mt,\beta_{\mt})\in{\rm Tri}(\undla)$. Since $\tau_{r,n}(F_{\rm T})=\varepsilon_1\circ\cdots\circ\varepsilon_n(F_{\rm T}),$ now the Theorem follows by an induction on $n$.\qed
\medskip

\subsection{(Super)symmetrizing forms and Schur elements}

In this subsection, we apply Theorem \ref{mainthm type0 nondege} to construct a supersymmetrizing form and, under an additional invertibility condition, a symmetrizing form. This proves Theorem \ref{Nondengerate} (i). We then compute the Schur elements relative to the supersymmetrizing form. {\bf We emphasise that, unless otherwise stated, it is not necessary to assume that the separation condition holds in this subsection.}

\label{pag:supersym. form. non-dege}
\begin{cor}\label{supersym. form. non-dege}
	Let ${\rm R}$ be an integral domain of characteristic different from 2. Suppose $q\in{\rm R^\times}\setminus\{\pm 1\},$ $q+ q^{-1}\in{\rm R^\times}$ and $\undQ\in ({\rm R^\times})^m.$ Let $f$ be as in \eqref{cyc-even}, then the cyclotomic Hecke-Clifford algebra $\mHfcn$ defined over ${\rm R}$ is a supersymmetric algebra with supersymmetrizing form
$$t_{r,n}:=\tau_{r,n}\Bigl(-\cdot (X_1X_2\cdots X_n)^m \Bigr).$$
\end{cor}
\begin{proof}
	Note that $t_{r,n}$ is defined over an arbitrary integral domain of characteristic different from 2 with $q, q+q^{-1}\in{\rm R^\times}$ and $\undQ\in ({\rm R^\times})^m.$ Since the element $(X_1X_2\cdots X_n)^m\in\mHfcn$ is invertible and $\tau_{r,n}$ is nondegenerate, it's easy to check that $t_{r,n}$ is also nondegenerate. We only need to prove
\begin{align}\label{supersymm}
		t_{r,n}(xy)=(-1)^{|x||y|}t_{r,n}(yx)	
\end{align} for all homogeneous $x,y\in \mathcal{H}^{f}_{c}(n)$. In the following, we shall use $t_{r,n}^{\rm R}$ to emphasise the base ring we are considering. Set ${\rm R}':=\Z[q^{\pm 1},(q+q^{-1})^{-1},Q'^{\pm 1}_1,\cdots,Q'^{\pm 1}_m]$ to be the generic base ring, so $q,\undQ'$ are algebraically independent. Let $f'$ be the polynomial in \eqref{cyc-even} with $\undQ$ replaced by $\undQ'$. Then there is a natural isomorphism
$${\rm R}\otimes_{{\rm R}'}	\mathcal{H}^{f'}_{c}(n)\cong 	\mathcal{H}^{f}_{c}(n)$$
and we have $t_{r,n}^{\rm R}=1\otimes_{\rm R'}t_{r,n}^{\rm R'}$. 	
Hence we only need to prove \eqref{supersymm} for $	\mathcal{H}^{f'}_{c}(n)$. Let $\mathbb{K}$ be the algebraic closure of the fraction field of ${\rm R}'$. Then it's easy to check the separation condition holds over $\mathbb{K}$, hence the cyclotomic Hecke-Clifford algebra $\mathbb{K}\otimes_{\rm R'}\mathcal{H}^{f'}_{c}(n)$ defined over $\mathbb{K}$ is semisimple.

Now we check that \eqref{supersymm} holds true for $\mathbb{K}\otimes_{\rm R'}\mathcal{H}^{f'}_{c}(n)$ by doing this on seminormal basis.
For any $\undla\in\mathscr{P}^{f}_{n},$ and ${\rm S}=(\ms,\beta_{\ms}'),{\rm T}=(\mt,\beta_{\mt})\in{\rm Tri}(\undla),$
by Theorem \ref{generators action on seminormal basis} (1) and Theorem \ref{mainthm type0 nondege}, we can deduce that
\begin{align}\label{eq. supersym. form. non-dege}
(-1)^{|\beta_{\ms}'|}t_{r,n}(F_{\rm S})
=\kappa_{\undla} \cdot \prod\limits_{\alpha\in\undla}\left(\mathtt{b}_{-}(\res(\alpha))-\mathtt{b}_{+}(\res(\alpha))\right)^{-1}
=(-1)^{|\beta_{\mt}|}t_{r,n}(F_{\rm T}).
\end{align}

If $|\beta_{\ms}'|\equiv |\beta_{\mt}| \pmod 2,$
then the parity $|f_{{\rm S},{\rm T}}|=\bar{0},$ and
\begin{align*}
t_{r,n}(f_{{\rm S},{\rm T}}f_{{\rm T},{\rm S}})
=\mathtt{c}_{\mt}^{\ms}t_{r,n}(F_{\rm S})
=\mathtt{c}_{\ms}^{\mt}t_{r,n}(F_{\rm T})
=t_{r,n}(f_{{\rm T},{\rm S}}f_{{\rm S},{\rm T}}),
\end{align*}
where in the first and last equalities we have used \eqref{Non-deg multiplication1}, in the second equality we have used \eqref{idempotent1}, \eqref{eq. supersym. form. non-dege}
and the assumption $|\beta_{\ms}'|\equiv |\beta_{\mt}| \pmod 2.$
Similarly, if $|\beta_{\ms}'|\not\equiv |\beta_{\mt}| \pmod 2,$
then the parity $|f_{{\rm S},{\rm T}}|=\bar{1},$ and
\begin{align*}
t_{r,n}(f_{{\rm S},{\rm T}}f_{{\rm T},{\rm S}})
=\mathtt{c}_{\mt}^{\ms}t_{r,n}(F_{\rm S})
=-\mathtt{c}_{\ms}^{\mt}t_{r,n}(F_{\rm T})
=-t_{r,n}(f_{{\rm T},{\rm S}}f_{{\rm S},{\rm T}}).
\end{align*}
Combining Proposition \ref{prop.taufST}, it follows that \eqref{supersymm} holds true for $\mathbb{K}\otimes_{\rm R'}\mathcal{H}^{f'}_{c}(n)$. Hence the Frobenius form $t_{r,n}$ is supersymmetric.
\end{proof}
\medskip
The supersymmetrizing form above can be modified further when the following invertibility condition holds.
\begin{cor}\label{sym. form. even. non-dege}
Under the assumptions of Corollary~\ref{supersym. form. non-dege}, suppose that $X_i-X_i^{-1}$ is invertible for every $1\leq i\leq n$. Then
\[
t_{r,n}\Bigl(-\cdot\prod_{i=1}^{n}(X_i-X_i^{-1})\Bigr)
\]
is a symmetrizing form on $\mHfcn$.
\end{cor}
\begin{proof}
The proof is analogous to that of Corollary~\ref{supersym. form. non-dege}, using seminormal bases and specialization. We omit the details.
\end{proof}
\medskip
{\bf Proof of Theorem \ref{Nondengerate} (i)}: This follows from Corollaries~\ref{supersym. form. non-dege} and~\ref{sym. form. even. non-dege}.\qed\par
\medskip
We now return to the supersymmetrizing form $t_{r,n}$ of Corollary~\ref{supersym. form. non-dege} to compute its Schur elements.
\label{pag:Schur elements.non-dege} 	
\begin{lem}\label{Schur elements and tauF}
Suppose the separation condition $P_n^{f}(q^2,\undQ)\neq 0$ holds for $\mHfcn,$ where $\undQ=(Q_1,Q_2,\ldots,Q_m)\in({\mathbb{K}}^*)^m$ and $f$ is as in \eqref{cyc-even}. Let $\undla\in\mathscr{P}^{f}_{n}.$
Then the Schur element $s_{\undla}$ of simple module $\mathbb{D}(\undla)$ with respect to $t_{r,n}$ is equal to $(-1)^{|\beta_{\mt}|}/t_{r,n}(F_{\rm T})$ for any ${\rm T}=(\mt,\beta_{\mt})\in{\rm Tri}(\undla).$
\end{lem}
\begin{proof}
It is completely similar to Lemma \ref{Schur elements and tF} (2).
\end{proof}

\medskip
\begin{cor}\label{schur-even-f}
Assume $f$ is as in \eqref{cyc-even} and $P_n^{f}(q^2,\undQ)\ne0$. For $\undla\in\mathscr P_n^{f}$, the Schur element of the fixed simple module relative to $t_{r,n}$ is
\[
s_{\undla}=\prod_{\alpha\in\undla}\bigl(\mathtt b_-(\res(\alpha))-\mathtt b_+(\res(\alpha))\bigr)\kappa_{\undla}^{-1}.
\]
\end{cor}
\begin{proof}
This follows from Lemma~\ref{Schur elements and tauF} and Theorem~\ref{mainthm type0 nondege}.
\end{proof}
\medskip
\subsection{Schur elements of the even-level cyclotomic Sergeev algebra}

In this subsection, we consider the cyclotomic Sergeev algebra $\mhgcn$, where $g$ is as in \eqref{deg-cyc-even} and $\undQ=(Q_1,Q_2,\ldots,Q_m)\in\mathbb{K}^m$ . {\bf We assume the degenerate separation condition $P_n^{g}(\undQ)\neq 0$ holds.} Following the argument in the non-degenerate case above, it is easy to obtain the parallel result for $\mhgcn$ with some suitable modifications.
For $g$ as in \eqref{deg-cyc-even}, i.e., the level $r=2m$ is even, the scalar
\begin{align*}
\prod\limits_{k=1}^{n}\frac{\mathtt{b}_{\mt,k}^{\nu_{\beta_{\mt}}(k)m}}{\mathtt{b}_{\mt,k}^{\nu_{\beta_{\mt}}(k)}-\mathtt{b}_{\mt,k}^{-\nu_{\beta_{\mt}}(k)}}
\end{align*}
appeared in \eqref{type=0. tauFT} should be replaced by
\begin{align*}
&\prod\limits_{k=1}^{n}\frac{\nu_{\beta_{\mt}}(k)}{\mathtt{u}_{-}(\res_{\mt}(k))-\mathtt{u}_{+}(\res_{\mt}(k))}\\
&=(-1)^{n+|\beta_{\mt}|}\prod\limits_{k=1}^{n}(2\mathtt{u}_{+}(\res_{\mt}(k)))^{-1}\\
&=(-1)^{n+|\beta_{\mt}|}2^{-n}\prod_{\alpha\in\undla}(\mathtt{u}_{+}(\res(\alpha)))^{-1}.
\end{align*}

Then we have the following.
\begin{thm}\label{dege. type=0. tFT}
Suppose $g$ is as in \eqref{deg-cyc-even} and $\undla\in\mathscr{P}^{g}_{n}.$ For any ${\rm T}=(\mt,\beta_{\mt})\in{\rm Tri}(\undla),$ we have
\begin{align*}
\mathtt{t}_{r,n}(\mathcal{F}_{\rm T})
=(-1)^{n+|\beta_{\mt}|}2^{-n}\prod_{\alpha\in\undla}\mathtt{u}_{+}(\res(\alpha))^{-1}
   \cdot \prod\limits_{k=1}^{n}\frac{\prod\limits_{\beta\in\Rem(\mt\downarrow_{k-1})}\left(\mathtt{q}(\res_{\mt}(k))-\mathtt{q}(\res(\beta))\right)}{
    \prod\limits_{\alpha\in\Add(\mt\downarrow_{k-1})\setminus \{\mt^{-1}(k)\}}\left(\mathtt{q}(\res_{\mt}(k))-\mathtt{q}(\res(\alpha))\right)}.
\end{align*}
\end{thm}

\medskip
\begin{cor}\label{schur-even-g}
Assume $g$ is as in \eqref{deg-cyc-even} and $P_n^{g}(\undQ)\ne0$. For $\undla\in\mathscr P_n^{g}$, the Schur element of the fixed simple module relative to $\mathtt t_{r,n}$ is
\[
\mathtt s_{\undla}=(-2)^n\prod_{\alpha\in\undla}\mathtt u_+(\res(\alpha))\bigl(\kappa_{\undla}^{\mathrm{dege}}\bigr)^{-1}.
\]
\end{cor}
\begin{proof}
This follows from Lemma~\ref{Schur elements and tF} (2) and Theorem~\ref{dege. type=0. tFT}.
\end{proof}

\begin{rem}
We can now give an alternative proof of Theorem \ref{LS1:Theorem 1.2} (ii) using the same argument as in
Corollary \ref{supersym. form. non-dege}.
\end{rem}

\section{Symmetrizing forms and Schur elements at odd level}\label{dot=s}
\noindent\textbf{This section is restricted to $f$ in \eqref{cyc-odd} and, in the final subsection, $g$ in \eqref{deg-cyc-odd}. References to odd level in this section concern these polynomials.}

The separated-parameter calculation produces a trace at general parameters; its non-degeneracy requires the invertibility condition in Corollary~\ref{sym. form. non-dege}. The set $\mathcal D$ consists of the diagonal positions in the strict component. We fix $\undla\in\mathscr{P}^{f}_{n}$ and then fix ${\rm T}=(\mt,\alpha_\mt,\beta_{\mt})\in{\rm Tri}_{\bar{0}}(\undla).$
\subsection{Values of the Frobenius form on primitive idempotents}
{\bf Throughout this subsection and the next four subsections, we always assume the separation condition $P_n^{f}(q^2,\undQ)\neq 0$ holds. }The aim of these subsections is to prove the following Theorem.

\begin{thm}\label{mainthm types nondege}
	Suppose $f$ is as in \eqref{cyc-odd} and $\undla\in\mathscr{P}^{f}_{n}.$ For any ${\rm T}=(\mt,\alpha_{\mt},\beta_{\mt})\in{\rm Tri}_{\bar{0}}(\undla),$ we have
	$$\tau_{r,n}(F_{\rm T})
	=2^{\lceil \sharp\mathcal{D}_{\undla}/2 \rceil}\prod\limits_{k=1}^{n}\frac{\mathtt{b}_{\mt,k}^{\nu_{\beta_{\mt}}(k)m}}{1+\mathtt{b}_{\mt,k}^{-\nu_{\beta_{\mt}}(k)}}\cdot \kappa_{\undla}\in\mathbb{K}^*.$$
\end{thm}

 Again, by the definition of $\tau_{r,n}$, in order to prove Theorem \ref{mainthm types nondege}, we need to compute $\varepsilon_n(F_{\rm T})$ first. To this end, {\bf we can further assume that $\alpha_\mt=0$ and $\beta_{\mt}\in\{0,e_n\},$ i.e., $\beta_{\mt}\downarrow_{n-1}=0$ by Corollary \ref{reductionrem}.}

Recall that
$\mathtt{c}_{\mt}^{\ms}:=\mathtt{c}_{\mt,\ms}\mathtt{c}_{\ms,\mt}$
for any $\ms\in\Std(\undla).$
Moreover, for any $(n-1)$-dimensional vector $\alpha=(\alpha_1,\ldots,\alpha_{n-1})\in\mathbb{Z}_2^{n-1}$, we use the same notation $\alpha$ to denote the $n$-dimensional vector $\alpha=(\alpha_1,\ldots,\alpha_{n-1},\bar{0})\in\mathbb{Z}_2^{n}$ by the natural embedding.

Let $\underline{\mu}:={\rm shape}(\mt\downarrow_{n-1})\in \mathscr{P}^{f}_{n-1}.$ For $\underline{\mu}=(\mu^{(0)},\mu^{(1)},\ldots,\mu^{(m)})\in\mathscr{P}^{f}_{n-1},$ we recall the notation
\begin{align*}
\delta_{\mu^{(0)}}:=
\begin{cases}
0, &\text{if the last part of $\mu^{(0)}$ is $1,$}\\
1, &\text{otherwise.}
\end{cases}
\end{align*}

The four cases below are determined by the length of the strict component $\mu^{(0)}$ and by $\delta_{\mu^{(0)}}$. Here $d_{\undmu}\equiv\ell(\mu^{(0)})\pmod2$; for the empty strict partition, $d_{\undmu}=0$ and $\delta_{\mu^{(0)}}=1$. The value $\delta_{\mu^{(0)}}=1$ means that a diagonal box can be added to the strict component, whereas $\delta_{\mu^{(0)}}=0$ means that a diagonal box can be removed. The length parity determines which Clifford terms occur under restriction.
\begin{center}
\begin{tabular}{ccc}
Section & $\ell(\mu^{(0)})$ & $\delta_{\mu^{(0)}}$\\
\hline
\ref{type=s, case(1) ,section} & even & 1\\
\ref{type=s, case(3) ,section} & odd & 1\\
\ref{type=s, case(2) ,section} & even & 0\\
\ref{type=s, case(4) ,section} & odd & 0
\end{tabular}
\end{center}
\begin{lem}\label{HLL. lem2.13. type=s}
For $\underline{\mu}\in\mathscr{P}^{f}_{n-1},$ we have $\sharp\Add(\underline{\mu})=\sharp\Rem(\underline{\mu})+m+\delta_{\mu^{(0)}}.$
\end{lem}
\begin{proof}
It is completely similar to \cite[Lemma 2.13]{HLL}.
\end{proof}

By Lemma \ref{HLL. lem2.13. type=s}, we can set
\begin{align*}
\Add(\undmu)&:=\{\alpha_1=\mt^{-1}(n),\alpha_2,\ldots,\alpha_N \},\\
\Rem(\undmu)&:=\{\beta_1=\mt^{-1}(n-1),\beta_2,\ldots,\beta_{N-m-\delta_{\mu^{(0)}}} \}.
\end{align*}
Let $j_{\iota}\in[n-1]$ be the number such that $\mt^{-1}(j_{\iota})=\beta_{\iota}$ for $\iota\in[N-m-\delta_{\mu^{(0)}}].$ In particular, $j_1=n-1.$
Again, for any $\iota\in[N-m-\delta_{\mu^{(0)}}],$ we set
\begin{align*}
	{\rm U}_{\iota}=(\mathfrak{u}_{\iota}, 0,0)
	:=\begin{cases}s_{n-2}s_{n-3}\cdots s_{j_{\iota}}\cdot (\mt\downarrow_{n-1}, 0,0) \in {\rm Tri}(\undmu)&\text{if $1<\iota\leq N-m-\delta_{\mu^{(0)}}$,}\\
		(\mt\downarrow_{n-1},0,0) &\text{otherwise.}
	\end{cases}
\end{align*}

The following definition will be used in the computation for $\varepsilon_n(F_{\rm T})$.

\begin{defn}\label{eta. type=s}
	Let $\undla\in\mathscr{P}^{f}_{n}.$ For any ${\rm T}=(\mt,\alpha_{\mt},\beta_{\mt})\in{\rm Tri}_{\bar{0}}(\undla),$ if $\undmu:={\rm shape}(\mt\downarrow_{n-1}),$ we define
	\begin{align*}
		\eta_{\mt,\beta_{\mt}}(n)
		:=\begin{cases}
			\frac{\mathtt{b}_{\mt,n}^{\nu_{\beta_{\mt}}(n)m}}{1+\mathtt{b}_{\mt,n}^{-\nu_{\beta_{\mt}}(n)}},
			&\text{if $d_{\undla}=d_{\undmu}$,}\\
			1, &\text{if $d_{\undla}=1,$ $d_{\undmu}=0$,}\\
			\frac{1}{2}, &\text{if $d_{\undla}=0,$ $d_{\undmu}=1$,}
		\end{cases}
	\end{align*}
	and
	\begin{align*}
		\eta_{\mt,\beta_{\mt}}
		:=\prod_{k=1}^{n} \eta_{\mt\downarrow_k,\beta_{\mt}\downarrow_k}(k)
	\end{align*}
\end{defn}

For any $\iota\in[N-m-\delta_{\mu^{(0)}}]$ and $k\in[N],$ let $\mathfrak{u}_{\iota,k}$ be the unique standard tableau of shape $\undla_{k}:=\undmu\cup\{\alpha_{k}\}\in\mathscr{P}^{f}_{n}$ such that
$\mathfrak{u}_{\iota,k}\downarrow_{n-1}=\mathfrak{u}_{\iota}$ and $\mathfrak{u}_{\iota,k}^{-1}(n)=\alpha_k.$ Note that
$\mathfrak{u}_{\iota,k}^{-1}(n-1)=\mathfrak{u}_{\iota}^{-1}(n-1)=\beta_{\iota}.$ In particular, we set $\mt_k:=\mathfrak{u}_{1,k}\in\Std(\undla_k)$ for $k\in[N].$

Recall that we have fixed ${\rm T}=(\mt,0,\beta_{\mt})\in{\rm Tri}_{\bar{0}}(\undla)$ such that $\beta_{\mt}\downarrow_{n-1}=0.$ By the Mackey decomposition \eqref{nondegMadecomequation}, we can write
\begin{align}\label{type=s.Madecomequation of fTT0}
F_{{\rm T}}=\mu_{T_{n-1}}(\pi(F_{{\rm T}}))+\sum_{\substack{0\leq k\leq r-1\\ a\in \mathbb{Z}_2}}X_n^k C_n^a p_{k,a}(F_{{\rm T}}).
\end{align}

For simplicity, we set $\mathbb{X}_k:=\mathtt{b}_{+}(\res(\alpha_k)),$ ${\rm Y}_{\iota}:=\mathtt{b}_{+}(\res(\beta_{\iota}))$
for $k\in[N],$ $\iota\in[N-m-\delta_{\mu^{(0)}}].$  We shall divide the computation for $\varepsilon_n(F_{\rm T})=p_{0,\bar{0}}(F_{\rm T})$ into four subsections.

\subsection{Even length with $\delta_{\mu^{(0)}}=1$}\label{type=s, case(1) ,section}
	In this subsection, we shall compute $\varepsilon_n(F_{\rm T})=p_{0,\bar{0}}(F_{\rm T})$ when $d_{\undmu}=0,$ $\delta_{\mu^{(0)}}=1$.
	 Since $\delta_{\mu^{(0)}}=1$, we have $\Rem(\undmu)\cap\mathcal{D}=\emptyset$. We set $k_0\in[N]$ to be the unique number such that $\alpha_{k_0}\in\Add(\undmu)\cap\mathcal{D}.$

The main result of this subsection is the following.
\begin{prop}\label{type=s, case(1) ,mainprop}
Let $\undla\in\mathscr{P}^{f}_{n}.$ Suppose ${\rm T}=(\mt,0,\beta_{\mt})\in{\rm Tri}_{\bar{0}}(\undla)$ with $\beta_{\mt}\downarrow_{n-1}=0$ and ${\rm T}\downarrow_{n-1}=(\mt\downarrow_{n-1},0,0)\in {\rm Tri}_{\bar{0}}(\underline{\mu}).$ If $d_{\undmu}=0$ and $\delta_{\mu^{(0)}}=1,$ we have the following.
\begin{enumerate}
  \item There are some scalars $a_0,\ldots,a_{r-1},b_{1,\bar{0}},\ldots,b_{N-m-1,\bar{0}},b_{1,\bar{1}},\ldots,b_{N-m-1,\bar{1}}\in \mathbb{K}$ such that
$$
\pi(F_{{\rm T}})
=\sum_{\substack{1\leq\iota\leq N-m-1\\ \delta\in\mathbb{Z}_2}}b_{\iota,\delta}
  f_{{\rm T}\downarrow_{n-1},(\mathfrak{u}_{\iota},0,\delta e_{n-1})}
  \otimes f_{(\mathfrak{u}_{\iota},0,\delta e_{n-1}),{\rm T}\downarrow_{n-1}},
$$and
$$
p_{k,\bar{0}}(F_{{\rm T}})=a_k F_{{\rm T}\downarrow_{n-1}},
\quad 0\leq k <r.
$$
  \item We have
\begin{align}\label{type=s, case(1) ,veFT}
\varepsilon_n(F_{\rm T})
=\eta_{\mt,\beta_{\mt}}(n)
    \frac{\prod\limits_{\beta\in\Rem(\mt\downarrow_{n-1})\setminus\mathcal{D}}\left(\mathtt{q}(\res_{\mt}(n))-\mathtt{q}(\res(\beta))\right)}{
    \prod\limits_{\alpha\in\Add(\mt\downarrow_{n-1})\setminus \{\mt^{-1}(n)\}}\left(\mathtt{q}(\res_{\mt}(n))-\mathtt{q}(\res(\alpha))\right)} F_{\rm T\downarrow_{n-1}}.
\end{align}
\end{enumerate}
\end{prop}

\begin{proof}
We define
\begin{align}\label{type=s.(1).eq.widehatpi}
\widehat{\pi}(F_{{\rm T}})
:=\sum_{\substack{1\leq\iota\leq N-m-1\\ \delta\in\mathbb{Z}_2}}b_{\iota,\delta}&
  f_{{\rm T}\downarrow_{n-1},(\mathfrak{u}_{\iota},0,\delta e_{n-1})}
  \otimes f_{(\mathfrak{u}_{\iota},0,\delta e_{n-1}),{\rm T}\downarrow_{n-1}}\\
  &\in \mathcal{H}^f_c(n-1)\otimes_{\mathcal{H}^f_c(n-2)}\mathcal{H}^f_c(n-1)\nonumber
\end{align}
for some unknown scalars $b_{\iota,\delta}\in \mathbb{K}.$
Then we get that
\begin{align*}
\mu_{T_{n-1}}(\widehat{\pi}(F_{{\rm T}}))
=\sum_{\substack{1\leq\iota\leq N-m-1\\ \delta\in\mathbb{Z}_2}}b_{\iota,\delta}
  \iota_{n-1}\left(f_{{\rm T}\downarrow_{n-1},(\mathfrak{u}_{\iota},0,\delta e_{n-1})}\right)
  T_{n-1} \iota_{n-1}\left(f_{(\mathfrak{u}_{\iota},0,\delta e_{n-1}),{\rm T}\downarrow_{n-1}}\right).
\end{align*}
By Theorem \ref{maththm of embedding} (1), for each $1\leq\iota\leq N-m-1$, we have
\begin{align*}
\iota_{n-1}\left(f_{{\rm T}\downarrow_{n-1},(\mathfrak{u}_{\iota},0,\delta e_{n-1})}\right)
&=f_{(\mt_{k_0},0,0),(\mathfrak{u}_{\iota,{k_0}},0,\delta e_{n-1})}
  +\sum_{\substack{1 \leq k\leq N, k\neq k_0\\ \delta_1\in\mathbb{Z}_2}}
    f_{(\mt_k,0,\delta_1 e_n),(\mathfrak{u}_{\iota,k},0,\delta e_{n-1}+\delta_1 e_n)},\\
\iota_{n-1}\left(f_{(\mathfrak{u}_{\iota},0,\delta e_{n-1}),{\rm T}\downarrow_{n-1}}\right)
&=
 f_{(\mathfrak{u}_{\iota,{k_0}},0,\delta e_{n-1}),(\mt_{k_0},0,0)}
  +\sum_{\substack{1 \leq k\leq N, k\neq k_0\\ \delta_2\in\mathbb{Z}_2}}
  f_{(\mathfrak{u}_{\iota,k},0,\delta e_{n-1}+\delta_2 e_n),(\mt_k,0,\delta_2 e_n)}.
\end{align*}
Combining these with Proposition \ref{generators action on seminormal basis} (2) and equations \eqref{Non-deg multiplication1}, \eqref{Non-deg multiplication2}, we can compute
\begin{align}\label{type=s.(1).eq. muTn-1}
&\mu_{T_{n-1}}(\widehat{\pi}(F_{{\rm T}}))\\
&=\sum_{\substack{1\leq\iota\leq N-m-1\\ \delta\in\mathbb{Z}_2}}b_{\iota,\delta}
   \sum_{\substack{1\leq k \leq N, k\neq k_0\\ \delta_1\in\mathbb{Z}_2}}
     \frac{\epsilon}{1-\mathtt{b}_{+}(\res(\beta_{\iota}))^{\nu(\delta)}\mathtt{b}_{+}(\res(\alpha_k))^{-\nu(\delta_1)}}
        \mathtt{c}^{\mt_k}_{\mathfrak{u}_{\iota,k}}
         f_{(\mt_k,0,\delta_1 e_n),(\mt_k,0,\delta_1 e_n)} \nonumber\\
&\qquad\qquad\quad
      +\sum_{\substack{1\leq\iota\leq N-m-1\\ \delta\in\mathbb{Z}_2}}b_{\iota,\delta}
         \frac{\epsilon}{1-\mathtt{b}_{+}(\res(\beta_{\iota}))^{\nu(\delta)}\mathtt{b}_{+}(\res(\alpha_{k_0}))^{-1}}
           \mathtt{c}^{\mt_{k_0}}_{\mathfrak{u}_{\iota,{k_0}}}
           f_{(\mt_{k_0},0,0),(\mt_{k_0},0,0)} \nonumber\\
           &=\sum_{\substack{1\leq\iota\leq N-m-1\\ \delta\in\mathbb{Z}_2}}b_{\iota,\delta}
           \sum_{\substack{1\leq k \leq N, k\neq k_0\\ \delta_1\in\mathbb{Z}_2}}
           \frac{\epsilon}{1-\mathtt{b}_{+}(\res(\beta_{\iota}))^{\nu(\delta)}\mathtt{b}_{+}(\res(\alpha_k))^{-\nu(\delta_1)}}
           \mathtt{c}^{\mathfrak{u}_{1}}_{\mathfrak{u}_{\iota}}
           f_{(\mt_k,0,\delta_1 e_n),(\mt_k,0,\delta_1 e_n)} \nonumber\\
           &\qquad\qquad\quad
           +\sum_{\substack{1\leq\iota\leq N-m-1\\ \delta\in\mathbb{Z}_2}}b_{\iota,\delta}
           \frac{\epsilon}{1-\mathtt{b}_{+}(\res(\beta_{\iota}))^{\nu(\delta)}\mathtt{b}_{+}(\res(\alpha_{k_0}))^{-1}}
          \mathtt{c}^{\mathfrak{u}_{1}}_{\mathfrak{u}_{\iota}}
           f_{(\mt_{k_0},0,0),(\mt_{k_0},0,0)} \nonumber
\end{align}
where $\mathtt{b}_{+}(\res(\alpha_{k_0}))=1$ and
\begin{align*}
\nu(\delta):=
\begin{cases}
-1, &\text{if $\delta=\bar{1}$,}\\
1, &\text{if $\delta=\bar{0}$,}
\end{cases}
\quad
\nu(\delta_1):=
\begin{cases}
-1, &\text{if $\delta_1=\bar{1}$,}\\
1, &\text{if $\delta_1=\bar{0}$.}
\end{cases}
\end{align*}

By Lemma \ref{typeos, scalars ak} (1), we have
	\begin{align}\label{scalars ak. types, (1)}p_{k',\bar{1}}(F_{{\rm T}})=0,\quad
	p_{k',\bar{0}}(F_{{\rm T}})=a_{k'} F_{{\rm T}\downarrow_{n-1}}
	\end{align}
	for some $a_{k'}\in \mathbb{K},$ $0\leq k' <r.$
	
It follows from Theorem \ref{maththm of embedding} (1) and Proposition \ref{generators action on seminormal basis} (1) that
\begin{align}\label{type=s.(1).eq. Xnk'pk'0}
&\sum_{k'=0}^{r-1}X_n^{k'}\iota_{n-1}(p_{k',\bar{0}}(F_{{\rm T}}))\\
&=\sum_{k'=0}^{r-1}a_{k'} X_n^{k'}\iota_{n-1}(F_{{\rm T}\downarrow_{n-1}}) \nonumber\\
&=\sum_{k'=0}^{r-1}a_{k'}
\left(
          X_n^{k'}F_{(\mt_{k_0},0,0)}
          +\sum_{\substack{1\leq k\leq N, k\neq k_0\\ \delta_1\in\mathbb{Z}_2}}
            X_n^{k'} F_{(\mt_k,0,\delta_1 e_n)} \right) \nonumber\\
&=\sum_{k'=0}^{r-1}a_{k'}
   F_{(\mt_{k_0},0,0)}
 +\sum_{\substack{1\leq k\leq N, k\neq k_0\\ \delta_1\in\mathbb{Z}_2}}\sum_{k'=0}^{r-1}
   a_{k'}
    \mathtt{b}_{+}(\res(\alpha_k))^{\nu(\delta_1)k'}
     F_{(\mt_k,0,\delta_1 e_n)}.\nonumber
\end{align}

We now regard those scalars
\begin{align}\label{type=s.(1).nondegunknownquantities}
a_0,\ldots,a_{r-1},b_{1,\bar{0}},\ldots,b_{N-m-1,\bar{0}},b_{1,\bar{1}},\ldots,b_{N-m-1,\bar{1}}
\end{align}
appearing in  \eqref{scalars ak. types, (1)} and \eqref{type=s.(1).eq.widehatpi} as unknown quantities.
Combining \eqref{type=s.Madecomequation of fTT0}  with \eqref{type=s.(1).eq. muTn-1}, \eqref{type=s.(1).eq. Xnk'pk'0} and comparing the coefficients of
$F_{(\mt_{k_0},0,0)}$ and $F_{(\mt_k,0,\delta_1 e_n)},$
$k\in [N]\setminus\{k_0\},$ $\delta_1\in\mathbb{Z}_2$
on both sides of \eqref{type=s.Madecomequation of fTT0} with $\pi(F_{{\rm T}})$ replaced
with $\widehat{\pi}(F_{{\rm T}}),$ we get the following linear system of equations in three cases:
\begin{align}\label{type=s.(1).nondegLinearEquations}
	& A\cdot \text{diag}(\underbrace{1,\dots,1}_{r\,\text{copies}},
	\epsilon\mathtt{c}^{\mathfrak{u}_{1}}_{\mathfrak{u}_{1}},\dots,
	\epsilon\mathtt{c}^{\mathfrak{u}_{1}}_{\mathfrak{u}_{N-m-1}},
	\epsilon\mathtt{c}^{\mathfrak{u}_{1}}_{\mathfrak{u}_{1}},\dots,
	\epsilon\mathtt{c}^{\mathfrak{u}_{1}}_{\mathfrak{u}_{N-m-1}})\nonumber\\
	& \cdot \left(a_0,\dots,a_{r-1},b_{1,\bar{0}},\dots,b_{N-m-1,\bar{0}},b_{1,\bar{1}},\dots,b_{N-m-1,\bar{1}}\right)^{T}\\
	&\quad=\begin{cases}
		(1,\underbrace{0,\dots,0}_{2N-2\,\text{copies}})^T, &\text{if $d_{\undla}=1$  and $\delta_{\beta_{\mt}}(n)=\bar{0}$ (in this case, $k_0=1$),}\\
		(1,\underbrace{0,\dots,0}_{2N-2\,\text{copies}})^T, &\text{if $d_{\undla}=0$  and $\delta_{\beta_{\mt}}(n)=\bar{0}$,}\\
		(\underbrace{0,\dots,0}_{N\,\text{copies}},1,\underbrace{0,\dots,0}_{N-2\,\text{copies}})^T, &\text{if $d_{\undla}=0$ and $\delta_{\beta_{\mt}}(n)=\bar{1}$,}
	\end{cases}\nonumber
\end{align}
where
{\footnotesize
$$A=\begin{pmatrix}
1 & \mathbb{X}_1 & \cdots & \mathbb{X}_1^{r-1} &
\frac{\mathbb{X}_1}{\mathbb{X}_1-{\rm Y}_1} & \cdots & \frac{\mathbb{X}_1}{\mathbb{X}_1-{\rm Y}_{N-m-1}} &
\frac{\mathbb{X}_1}{\mathbb{X}_1-{\rm Y}_1^{-1}} & \cdots & \frac{\mathbb{X}_1}{\mathbb{X}_1-{\rm Y}_{N-m-1}^{-1}}\\
\vdots & \vdots &  & \vdots & \vdots &  & \vdots & \vdots &  & \vdots\\
1 & \mathbb{X}_N & \cdots & \mathbb{X}_N^{r-1} &
\frac{\mathbb{X}_N}{\mathbb{X}_N-{\rm Y}_1} & \cdots & \frac{\mathbb{X}_N}{\mathbb{X}_N-{\rm Y}_{N-m-1}} &
\frac{\mathbb{X}_N}{\mathbb{X}_N-{\rm Y}_1^{-1}} & \cdots & \frac{\mathbb{X}_N}{\mathbb{X}_N-{\rm Y}_{N-m-1}^{-1}}\\
1 & \mathbb{X}_1^{-1} & \cdots & \mathbb{X}_1^{-(r-1)} &
\frac{\mathbb{X}_1^{-1}}{\mathbb{X}_1^{-1}-{\rm Y}_1} & \cdots & \frac{\mathbb{X}_1^{-1}}{\mathbb{X}_1^{-1}-{\rm Y}_{N-m-1}} &
\frac{\mathbb{X}_1^{-1}}{\mathbb{X}_1^{-1}-{\rm Y}_1^{-1}} & \cdots & \frac{\mathbb{X}_1^{-1}}{\mathbb{X}_1^{-1}-{\rm Y}_{N-m-1}^{-1}}\\
\vdots & \vdots &  & \vdots & \vdots &  & \vdots & \vdots &  & \vdots\\
1 & \mathbb{X}_{k_0-1}^{-1} & \cdots & \mathbb{X}_{k_0-1}^{-(r-1)} &
\frac{\mathbb{X}_{k_0-1}^{-1}}{\mathbb{X}_{k_0-1}^{-1}-{\rm Y}_1} & \cdots & \frac{\mathbb{X}_{k_0-1}^{-1}}{\mathbb{X}_{k_0-1}^{-1}-{\rm Y}_{N-m-1}} &
\frac{\mathbb{X}_{k_0-1}^{-1}}{\mathbb{X}_{k_0-1}^{-1}-{\rm Y}_1^{-1}} & \cdots & \frac{\mathbb{X}_{k_0-1}^{-1}}{\mathbb{X}_{k_0-1}^{-1}-{\rm Y}_{N-m-1}^{-1}}\\
1 & \mathbb{X}_{k_0+1}^{-1} & \cdots & \mathbb{X}_{k_0+1}^{-(r-1)} &
\frac{\mathbb{X}_{k_0+1}^{-1}}{\mathbb{X}_{k_0+1}^{-1}-{\rm Y}_1} & \cdots & \frac{\mathbb{X}_{k_0+1}^{-1}}{\mathbb{X}_{k_0+1}^{-1}-{\rm Y}_{N-m-1}} &
\frac{\mathbb{X}_{k_0+1}^{-1}}{\mathbb{X}_{k_0+1}^{-1}-{\rm Y}_1^{-1}} & \cdots & \frac{\mathbb{X}_{k_0+1}^{-1}}{\mathbb{X}_{k_0+1}^{-1}-{\rm Y}_{N-m-1}^{-1}}\\
\vdots & \vdots &  & \vdots & \vdots &  & \vdots & \vdots &  & \vdots\\
1 & \mathbb{X}_N^{-1} & \cdots & \mathbb{X}_N^{-(r-1)} &
\frac{\mathbb{X}_N^{-1}}{\mathbb{X}_N^{-1}-{\rm Y}_1} & \cdots & \frac{\mathbb{X}_N^{-1}}{\mathbb{X}_N^{-1}-{\rm Y}_{N-m-1}} &
\frac{\mathbb{X}_N^{-1}}{\mathbb{X}_N^{-1}-{\rm Y}_1^{-1}} & \cdots & \frac{\mathbb{X}_N^{-1}}{\mathbb{X}_N^{-1}-{\rm Y}_{N-m-1}^{-1}}
\end{pmatrix}$$ }
is a $(2N-1)\times (2N-1)$ matrix.
If the matrix $A$ is invertible, then we can deduce that there exist scalars \eqref{type=s.(1).nondegunknownquantities} such that the corresponding elements $\widehat{\pi}(F_{\rm T})$ and $p_{k,\bar{0}}(F_{\rm T}),\ k=0,1,\dots,r-1$ satisfy the equality \eqref{type=s.Madecomequation of fTT0} with $\pi(F_{\rm T})$ replaced by $\widehat{\pi}(F_{\rm T})$ (Note again that all $p_{k,\bar{1}}(F_{\rm T})=0$).
Then by the uniqueness statement of \eqref{type=s.Madecomequation of fTT0}, we can deduce that $\widehat{\pi}(F_{\rm T})=\pi(F_{\rm T}).$

To use Lemmas \ref{TechLem1}, \ref{TechLem2} conveniently, we further denote $\mathbb{X}_{N+k}:=\mathbb{X}_k^{-1},$ $k\in[N]\setminus\{k_0\}$ and ${\rm Z}_{\iota}:={\rm Y}_{\iota},$ ${\rm Z}_{N-m-1+\iota}:={\rm Y}_{\iota}^{-1},$ $\iota\in[N-m-1].$ Then it follows from
Lemma \ref{TechLem1} that
\begin{align}\label{type=s.(1).detA}
{\rm det}(A)
&=\frac{\prod\limits_{1\leq t\leq 2(N-m-1)}{\rm Z}_t \prod\limits_{1\leq p< q\leq 2(N-m-1)}({\rm Z}_p-{\rm Z}_q)
\prod\limits_{1\leq u<v\leq 2N, u,v\neq N+k_0}(\mathbb{X}_v-\mathbb{X}_u)}
{\prod\limits_{\substack{1\leq j\leq 2N \\ j\neq N+k_0}}\prod\limits_{s=1}^{2(N-m-1)}(\mathbb{X}_j-{\rm Z}_s)}\\
&=\frac{\prod\limits_{1\leq p< q\leq 2(N-m-1)}({\rm Z}_p-{\rm Z}_q) \prod\limits_{1\leq u<v\leq 2N, u,v\neq N+k_0}(\mathbb{X}_v-\mathbb{X}_u)}{\prod\limits_{\substack{1\leq j\leq 2N \\ j\neq N+k_0}}\prod\limits_{s=1}^{2(N-m-1)}(\mathbb{X}_j-{\rm Z}_s)},\nonumber
\end{align}
where we have used the fact that
$$\prod\limits_{1\leq t\leq 2(N-m-1)}{\rm Z}_t=\prod\limits_{1\leq \iota \leq N-m-1}\left({\rm Y}_{\iota}{\rm Y}_{\iota}^{-1}\right)=1.$$
This implies that $A$ is invertible by Proposition \ref{compare eigenvalue}.

It remains to prove \eqref{type=s, case(1) ,veFT}. By Lemma \ref{TechLem2},
the algebraic cofactor $A_{1,1}$ of ${\rm det}(A)$ at position $(1,1)$ is
\begin{align*}
A_{1,1}
&=\left(\prod\limits_{\substack{2\leq k\leq 2N\\ k\neq N+k_0}}\mathbb{X}_k \right)\cdot\frac{\prod\limits_{1\leq p< q\leq 2(N-m-1)}({\rm Z}_p-{\rm Z}_q) \prod\limits_{\substack{2\leq u<v\leq 2N\\ u,v\neq N+k_0}}(\mathbb{X}_v-\mathbb{X}_u)}{\prod\limits_{\substack{2\leq j\leq 2N \\ j\neq N+k_0}}\prod\limits_{s=1}^{2(N-m-1)}(\mathbb{X}_j-{\rm Z}_s)}\\
&=\mathbb{X}_1^{-1}\cdot\frac{\prod\limits_{1\leq p< q\leq 2(N-m-1)}({\rm Z}_p-{\rm Z}_q) \prod\limits_{\substack{2\leq u<v\leq 2N\\ u,v\neq N+k_0}}(\mathbb{X}_v-\mathbb{X}_u)}{\prod\limits_{\substack{2\leq j\leq 2N \\ j\neq N+k_0}}\prod\limits_{s=1}^{2(N-m-1)}(\mathbb{X}_j-{\rm Z}_s)},
\end{align*}
where in the second equation, we have used the fact that $\prod_{\substack{1\leq k\leq 2N\\k\neq N+k_0}}\mathbb{X}_k=1$ and $\mathbb{X}_{k_0}=1.$

When $k_0\neq 1,$ similarly, by Lemma \ref{TechLem2}, the algebraic cofactor $A_{N+1,1}$ of ${\rm det}(A)$ at position $(N+1,1)$ is
\begin{align*}
A_{N+1,1}
&=(-1)^N\left(\prod\limits_{\substack{1\leq k\leq 2N \\ k\neq N+1,N+k_0}}\mathbb{X}_k \right)\cdot\frac{\prod\limits_{1\leq p< q\leq 2(N-m-1)}({\rm Z}_p-{\rm Z}_q)\prod\limits_{\substack{1\leq u<v\leq 2N \\ u,v\neq N+1,N+k_0}}(\mathbb{X}_v-\mathbb{X}_u)}{\prod\limits_{\substack{1\leq j\leq 2N \\ j\neq N+1,N+k_0}}\prod\limits_{s=1}^{2(N-m-1)}(\mathbb{X}_j-{\rm Z}_s)}\\
&=(-1)^N\mathbb{X}_1\cdot\frac{\prod\limits_{1\leq p< q\leq 2(N-m-1)}({\rm Z}_p-{\rm Z}_q)\prod\limits_{\substack{1\leq u<v\leq 2N\\ u,v\neq N+1,N+k_0}}(\mathbb{X}_v-\mathbb{X}_u)}{\prod\limits_{\substack{1\leq j\leq 2N \\ j\neq N+1,N+k_0}}\prod\limits_{s=1}^{2(N-m-1)}(\mathbb{X}_j-{\rm Z}_s)},
\end{align*}
where in the second equation, we have used the fact that $\prod_{\substack{1\leq k\leq 2N\\k\neq N+k_0}}\mathbb{X}_k=1$ and $\mathbb{X}_{k_0}=1.$

\begin{caselist}
	\item $d_{\undla}=1$ and $\delta_{\beta_{\mt}}(n)=\bar{0}.$\label{d=1,delta=0}
Then $k_0=1$, ${\rm T}=(\mt_1,0,0),$ $\mt^{-1}(n)=\alpha_{1}=\alpha_{k_0}$ and $\mathbb{X}_1=\mathtt{b}_{+}(\res_{\mt}(n))=1$. We have
\begin{align}\label{a0.case(1.a)}
&a_0
=\frac{A_{1,1}}{{\rm det}(A)}
= \frac{\prod\limits_{s=1}^{2(N-m-1)}(\mathbb{X}_1-{\rm Z}_s)}{\prod\limits_{\substack{2\leq v\leq 2N \\ v\neq N+1}}(\mathbb{X}_v-\mathbb{X}_1)}\\
&= \frac{\prod\limits_{\beta\in\Rem(\mt\downarrow_{n-1}),\ast\in\{\pm\}}\left(\mathtt{b}_{+}(\res_{\mt}(n))-\mathtt{b}_{\ast}(\res(\beta))\right)}{
    \prod\limits_{\alpha\in\Add(\mt\downarrow_{n-1})\setminus \{\mt^{-1}(n)\},\ast\in\{\pm\}}\left(\mathtt{b}_{+}(\res_{\mt}(n))-\mathtt{b}_{\ast}(\res(\alpha))\right)}\nonumber\\
&= \frac{\prod\limits_{\beta\in\Rem(\mt\downarrow_{n-1})}\left(\mathtt{q}(\res_{\mt}(n))-\mathtt{q}(\res(\beta))\right)}{
    \prod\limits_{\alpha\in\Add(\mt\downarrow_{n-1})\setminus \{\mt^{-1}(n)\}}\left(\mathtt{q}(\res_{\mt}(n))-\mathtt{q}(\res(\alpha))\right)},\nonumber
\end{align}
where in the last equation we have used the elementary identity
$$(b'-b)(b'-b^{-1})=b'(b'+b'^{-1}-b-b^{-1}),\qquad \forall b,b'\in\mathbb{K}^*.$$
\item $d_{\undla}=0$ and $\delta_{\beta_{\mt}}(n)=\bar{0}.$ \label{d=0,delta=0}
Then $k_0\neq 1$, ${\rm T}=(\mt_1,0,0),$ $\mt^{-1}(n)=\alpha_{1}$ and $\mathbb{X}_1=\mathtt{b}_{+}(\res_{\mt}(n))\neq 1$. We have
\begin{align}\label{a0.case(1.b)}
&a_0
=\frac{A_{1,1}}{{\rm det}(A)}
=\mathbb{X}_1^{-1}
  \frac{\prod\limits_{s=1}^{2(N-m-1)}(\mathbb{X}_1-{\rm Z}_s)}{\prod\limits_{\substack{2\leq v\leq 2N \\ v\neq N+k_0}}(\mathbb{X}_v-\mathbb{X}_1)}\\
&=
   \frac{\mathtt{b}_{-}(\res_{\mt}(n))(\mathtt{b}_{+}(\res_{\mt}(n))-1)}{\mathtt{b}_{+}(\res_{\mt}(n))-\mathtt{b}_{-}(\res_{\mt}(n))}
    \frac{\prod\limits_{\beta\in\Rem(\mt\downarrow_{n-1}),\ast\in\{\pm\}}\left(\mathtt{b}_{+}(\res_{\mt}(n))-\mathtt{b}_{\ast}(\res(\beta))\right)}{
     \prod\limits_{\alpha\in\Add(\mt\downarrow_{n-1})\setminus \{\mt^{-1}(n)\},\ast\in\{\pm\}}\left(\mathtt{b}_{+}(\res_{\mt}(n))-\mathtt{b}_{\ast}(\res(\alpha))\right)}\nonumber\\
&=
   \frac{\mathtt{b}_{+}(\res_{\mt}(n))^{-m}}{1+\mathtt{b}_{+}(\res_{\mt}(n))}
    \frac{\prod\limits_{\beta\in\Rem(\mt\downarrow_{n-1})}\left(\mathtt{q}(\res_{\mt}(n))-\mathtt{q}(\res(\beta))\right)}{
     \prod\limits_{\alpha\in\Add(\mt\downarrow_{n-1})\setminus \{\mt^{-1}(n)\}}\left(\mathtt{q}(\res_{\mt}(n))-\mathtt{q}(\res(\alpha))\right)},\nonumber
\end{align}
where in the last equation we have used the elementary identity
$$(b'-b)(b'-b^{-1})=b'(b'+b'^{-1}-b-b^{-1}),\qquad \forall b,b'\in\mathbb{K}^*,$$
and Lemma \ref{HLL. lem2.13. type=s}.
\item $d_{\undla}=0$ and $\delta_{\beta_{\mt}}(n)=\bar{1}.$\label{d=0,delta=1}
Then $k_0\neq 1$, ${\rm T}=(\mt_1,0,e_n),$ $\mt^{-1}(n)=\alpha_{1},$ $\mathbb{X}_1=\mathtt{b}_{+}(\res_{\mt}(n))\neq 1$ and $\mathbb{X}_{N+1}=\mathtt{b}_{-}(\res_{\mt}(n))\neq 1$. We have
\begin{align}\label{a0.case(1.c)}
&a_0
=\frac{A_{N+1,1}}{{\rm det}(A)}
=(-1)^{N}\mathbb{X}_1
  \frac{\prod\limits_{s=1}^{2(N-m-1)}(\mathbb{X}_{N+1}-{\rm Z}_s)}{\prod\limits_{\substack{N+1< v\leq 2N \\ v\neq N+k_0}}(\mathbb{X}_v-\mathbb{X}_{N+1})\prod\limits_{\substack{1\leq u<N+1 \\ u\neq N+k_0}}(\mathbb{X}_{N+1}-\mathbb{X}_u)}\\
&=
   \frac{\mathtt{b}_{+}(\res_{\mt}(n))(\mathtt{b}_{-}(\res_{\mt}(n))-1)}{\mathtt{b}_{-}(\res_{\mt}(n))-\mathtt{b}_{+}(\res_{\mt}(n))}
    \frac{\prod\limits_{\beta\in\Rem(\mt\downarrow_{n-1}),\ast\in\{\pm\}}\left(\mathtt{b}_{-}(\res_{\mt}(n))-\mathtt{b}_{\ast}(\res(\beta))\right)}{
     \prod\limits_{\alpha\in\Add(\mt\downarrow_{n-1})\setminus \{\mt^{-1}(n)\},\ast\in\{\pm\}}\left(\mathtt{b}_{-}(\res_{\mt}(n))-\mathtt{b}_{\ast}(\res(\alpha))\right)}\nonumber\\
&=
   \frac{\mathtt{b}_{-}(\res_{\mt}(n))^{-m}}{1+\mathtt{b}_{-}(\res_{\mt}(n))}
    \frac{\prod\limits_{\beta\in\Rem(\mt\downarrow_{n-1})}\left(\mathtt{q}(\res_{\mt}(n))-\mathtt{q}(\res(\beta))\right)}{
     \prod\limits_{\alpha\in\Add(\mt\downarrow_{n-1})\setminus \{\mt^{-1}(n)\}}\left(\mathtt{q}(\res_{\mt}(n))-\mathtt{q}(\res(\alpha))\right)},\nonumber
\end{align}
where in the last equation we have used the elementary identity
$$(b'-b)(b'-b^{-1})=b'(b'+b'^{-1}-b-b^{-1}),\qquad \forall b,b'\in\mathbb{K}^*,$$
and Lemma \ref{HLL. lem2.13. type=s}.
\end{caselist}
Finally, note that $\Rem(\undmu)\cap\mathcal{D}=\emptyset$, since $\delta_{\mu^{(0)}}=1$. Combining this with Case \ref{d=1,delta=0}, Case \ref{d=0,delta=0} and Case \ref{d=0,delta=1}, we complete the proof of \eqref{type=s, case(1) ,veFT}.
\end{proof}

\subsection{Odd length with $\delta_{\mu^{(0)}}=1$}\label{type=s, case(3) ,section}
In this subsection, we shall compute $\varepsilon_n(F_{\rm T})=p_{0,\bar{0}}(F_{\rm T})$ when $d_{\undmu}=1$, $\delta_{\mu^{(0)}}=1$. Since $\delta_{\mu^{(0)}}=1$, we have $\Rem(\undmu)\cap\mathcal{D}=\emptyset$ again. We set $k_0\in[N]$ to be the unique number such that $\alpha_{k_0}\in\Add(\undmu)\cap\mathcal{D}.$

Since $d_{\undmu}=1$, we have $({\rm T}\downarrow_{n-1})_{\overline{1}}=(\mt\downarrow_{n-1},e_i,0)\in{\rm Tri}_{\bar{1}}(\undmu)$, where $i:=\max(\mathcal{OD}_{\mt\downarrow_{n-1}})$. The coefficient matrix is the same as in Subsection~\ref{type=s, case(1) ,section}, but the restriction now includes an odd Clifford component. We construct an element from the unique solution of the linear system and identify its Mackey components to obtain the following result.
\begin{prop}\label{type=s, case(3) ,mainprop}
Let $\undla\in\mathscr{P}^{f}_{n}.$ Suppose ${\rm T}=(\mt,0,\beta_{\mt})\in{\rm Tri}_{\bar{0}}(\undla)$ with $\beta_{\mt}\downarrow_{n-1}=0$ and ${\rm T}\downarrow_{n-1}=(\mt\downarrow_{n-1},0,0)\in {\rm Tri}_{\bar{0}}(\underline{\mu}).$ If $d_{\undmu}=1$ and $\delta_{\mu^{(0)}}=1,$ we have
	\begin{align}\label{type=s, case(3) ,veFT}
		\varepsilon_n(F_{\rm T})
		=\eta_{\mt,\beta_{\mt}}(n)
		\frac{\prod\limits_{\beta\in\Rem(\mt\downarrow_{n-1})\setminus\mathcal{D}}\left(\mathtt{q}(\res_{\mt}(n))-\mathtt{q}(\res(\beta))\right)}{
			\prod\limits_{\alpha\in\Add(\mt\downarrow_{n-1})\setminus \{\mt^{-1}(n)\}}\left(\mathtt{q}(\res_{\mt}(n))-\mathtt{q}(\res(\alpha))\right)} F_{\rm T\downarrow_{n-1}}.
	\end{align}
\end{prop}

\begin{proof}
	We consider the following system of linear equations
	\begin{align}\label{type=s.(3).nondegLinearEquations}
		& A\cdot \text{diag}(\underbrace{1,\dots,1}_{r\,\text{copies}},
		\epsilon\mathtt{c}^{\mathfrak{u}_{1}}_{\mathfrak{u}_{1}},\dots,
		\epsilon\mathtt{c}^{\mathfrak{u}_{1}}_{\mathfrak{u}_{N-m-1}},
		\epsilon\mathtt{c}^{\mathfrak{u}_{1}}_{\mathfrak{u}_{1}},\dots,
		\epsilon\mathtt{c}^{\mathfrak{u}_{1}}_{\mathfrak{u}_{N-m-1}})\nonumber\\
		& \cdot \left(a_0,\dots,a_{r-1},b_{1,\bar{0}},\dots,b_{N-m-1,\bar{0}},b_{1,\bar{1}},\dots,b_{N-m-1,\bar{1}}\right)^{T}\\
		&\quad=\begin{cases}
			(1,\underbrace{0,\dots,0}_{2N-2\,\text{copies}})^T, &\text{if $d_{\undla}=0$  and $\delta_{\beta_{\mt}}(n)=\bar{0}$ (in this case, $k_0=1$),}\\
			(1,\underbrace{0,\dots,0}_{2N-2\,\text{copies}})^T, &\text{if $d_{\undla}=1$  and $\delta_{\beta_{\mt}}(n)=\bar{0}$,}\\
			(\underbrace{0,\dots,0}_{N\,\text{copies}},1,\underbrace{0,\dots,0}_{N-2\,\text{copies}})^T, &\text{if $d_{\undla}=1$  and $\delta_{\beta_{\mt}}(n)=\bar{1}$,}
		\end{cases}\nonumber
	\end{align}
	where the $(2N-1)\times (2N-1)$ matrix $A$ is the same as in \eqref{type=s.(1).nondegLinearEquations},
{\footnotesize
	$$A=\begin{pmatrix}
		1 & \mathbb{X}_1 & \cdots & \mathbb{X}_1^{r-1} &
		\frac{\mathbb{X}_1}{\mathbb{X}_1-{\rm Y}_1} & \cdots & \frac{\mathbb{X}_1}{\mathbb{X}_1-{\rm Y}_{N-m-1}} &
		\frac{\mathbb{X}_1}{\mathbb{X}_1-{\rm Y}_1^{-1}} & \cdots & \frac{\mathbb{X}_1}{\mathbb{X}_1-{\rm Y}_{N-m-1}^{-1}}\\
		\vdots & \vdots &  & \vdots & \vdots &  & \vdots & \vdots &  & \vdots\\
		1 & \mathbb{X}_N & \cdots & \mathbb{X}_N^{r-1} &
		\frac{\mathbb{X}_N}{\mathbb{X}_N-{\rm Y}_1} & \cdots & \frac{\mathbb{X}_N}{\mathbb{X}_N-{\rm Y}_{N-m-1}} &
		\frac{\mathbb{X}_N}{\mathbb{X}_N-{\rm Y}_1^{-1}} & \cdots & \frac{\mathbb{X}_N}{\mathbb{X}_N-{\rm Y}_{N-m-1}^{-1}}\\
		1 & \mathbb{X}_1^{-1} & \cdots & \mathbb{X}_1^{-(r-1)} &
		\frac{\mathbb{X}_1^{-1}}{\mathbb{X}_1^{-1}-{\rm Y}_1} & \cdots & \frac{\mathbb{X}_1^{-1}}{\mathbb{X}_1^{-1}-{\rm Y}_{N-m-1}} &
		\frac{\mathbb{X}_1^{-1}}{\mathbb{X}_1^{-1}-{\rm Y}_1^{-1}} & \cdots & \frac{\mathbb{X}_1^{-1}}{\mathbb{X}_1^{-1}-{\rm Y}_{N-m-1}^{-1}}\\
		\vdots & \vdots &  & \vdots & \vdots &  & \vdots & \vdots &  & \vdots\\
		1 & \mathbb{X}_{k_0-1}^{-1} & \cdots & \mathbb{X}_{k_0-1}^{-(r-1)} &
		\frac{\mathbb{X}_{k_0-1}^{-1}}{\mathbb{X}_{k_0-1}^{-1}-{\rm Y}_1} & \cdots & \frac{\mathbb{X}_{k_0-1}^{-1}}{\mathbb{X}_{k_0-1}^{-1}-{\rm Y}_{N-m-1}} &
		\frac{\mathbb{X}_{k_0-1}^{-1}}{\mathbb{X}_{k_0-1}^{-1}-{\rm Y}_1^{-1}} & \cdots & \frac{\mathbb{X}_{k_0-1}^{-1}}{\mathbb{X}_{k_0-1}^{-1}-{\rm Y}_{N-m-1}^{-1}}\\
		1 & \mathbb{X}_{k_0+1}^{-1} & \cdots & \mathbb{X}_{k_0+1}^{-(r-1)} &
		\frac{\mathbb{X}_{k_0+1}^{-1}}{\mathbb{X}_{k_0+1}^{-1}-{\rm Y}_1} & \cdots & \frac{\mathbb{X}_{k_0+1}^{-1}}{\mathbb{X}_{k_0+1}^{-1}-{\rm Y}_{N-m-1}} &
		\frac{\mathbb{X}_{k_0+1}^{-1}}{\mathbb{X}_{k_0+1}^{-1}-{\rm Y}_1^{-1}} & \cdots & \frac{\mathbb{X}_{k_0+1}^{-1}}{\mathbb{X}_{k_0+1}^{-1}-{\rm Y}_{N-m-1}^{-1}}\\
		\vdots & \vdots &  & \vdots & \vdots &  & \vdots & \vdots &  & \vdots\\
		1 & \mathbb{X}_N^{-1} & \cdots & \mathbb{X}_N^{-(r-1)} &
		\frac{\mathbb{X}_N^{-1}}{\mathbb{X}_N^{-1}-{\rm Y}_1} & \cdots & \frac{\mathbb{X}_N^{-1}}{\mathbb{X}_N^{-1}-{\rm Y}_{N-m-1}} &
		\frac{\mathbb{X}_N^{-1}}{\mathbb{X}_N^{-1}-{\rm Y}_1^{-1}} & \cdots & \frac{\mathbb{X}_N^{-1}}{\mathbb{X}_N^{-1}-{\rm Y}_{N-m-1}^{-1}}
	\end{pmatrix}.$$ }
	It follows from equation \eqref{type=s.(1).detA} and Proposition \ref{compare eigenvalue} that the above system of linear equations has a unique solution $$(a_0,\ldots,a_{r-1},b_{1,\bar{0}},\ldots,b_{N-m-1,\bar{0}},b_{1,\bar{1}},\ldots,b_{N-m-1,\bar{1}})^{T}\in\mathbb{K}^{2N-1}.$$
	Now we define an element $h\in\mHfcn$ via Mackey decomposition \eqref{nondegMadecomequation} as follows:
	\begin{align}
		h
	&=\sum_{\substack{1\leq\iota\leq N-m-1\\ \delta\in\mathbb{Z}_2}}b_{\iota,\delta}
	\iota_{n-1}\left(f_{{\rm T}\downarrow_{n-1},(\mathfrak{u}_{\iota},0,\delta e_{n-1})}\right)
	T_{n-1} \iota_{n-1}\left(f_{(\mathfrak{u}_{\iota},0,\delta e_{n-1}),{\rm T}\downarrow_{n-1}}\right)\nonumber\\
	&\qquad\qquad\qquad+\sum_{k'=0}^{r-1}a_{k'} X_n^{k'}\iota_{n-1}\left(F_{{\rm T}\downarrow_{n-1}}\right).\nonumber
	\end{align}

	By Theorem \ref{maththm of embedding} (2), and recall that $i:=\max(\mathcal{OD}_{\mt\downarrow_{n-1}}),$ we have
	\begin{align*}
		&\iota_{n-1}\left(f_{{\rm T}\downarrow_{n-1},(\mathfrak{u}_{\iota},0,\delta e_{n-1})}\right)\\
		&\qquad= \left(f_{(\mt_{k_0},0,0),(\mathfrak{u}_{\iota,{k_0}},0,\delta e_{n-1})}
		+f_{(\mt_{k_0},e_{i},0),(\mathfrak{u}_{\iota,{k_0}},e_{i},\delta e_{n-1})}\right)\\
		&\qquad\qquad\qquad +\sum_{\substack{1 \leq k\leq N, k\neq k_0\\ \delta_1\in\mathbb{Z}_2}}
		f_{(\mt_k,0,\delta_1 e_n),(\mathfrak{u}_{\iota,k},0,\delta e_{n-1}+\delta_1 e_n)},\\
		&\iota_{n-1}\left(f_{(\mathfrak{u}_{\iota},0,\delta e_{n-1}),{\rm T}\downarrow_{n-1}}\right)\\
		&\qquad=\left(f_{(\mathfrak{u}_{\iota,{k_0}},0,\delta e_{n-1}),(\mt_{k_0},0,0)}
		+f_{(\mathfrak{u}_{\iota,{k_0}},e_{i},\delta e_{n-1}),(\mt_{k_0},e_{i},0)}\right)\\
		&\qquad\qquad\qquad +\sum_{\substack{1 \leq k\leq N, k\neq k_0\\ \delta_2\in\mathbb{Z}_2}}
		f_{(\mathfrak{u}_{\iota,k},0,\delta e_{n-1}+\delta_2 e_n),(\mt_k,0,\delta_2 e_n)},\\
		&\iota_{n-1}\left(F_{{\rm T}\downarrow_{n-1}}\right)=\left(F_{(\mt_{k_0},0,0)}+F_{(\mt_{k_0},e_{i},0)}\right)+\sum_{\substack{1 \leq k\leq N, k\neq k_0\\ \delta_3\in\mathbb{Z}_2}}
		F_{(\mt_k,0,\delta_3 e_n)}.
	\end{align*}
	
	Combining with Proposition \ref{generators action on seminormal basis} (2) and equations \eqref{Non-deg multiplication1}, \eqref{Non-deg multiplication2}, we can deduce that
	\begin{align}
		h=\sum_{\substack{1\leq k \leq N, k\neq k_0\\ \delta_1\in\mathbb{Z}_2}}
		A_{(\mt_k,0,\delta_1 e_n)}
		F_{(\mt_k,0,\delta_1 e_n)}
		+A_{(\mt_{k_0},0,0)}\left(F_{(\mt_{k_0},0,0)}
		+F_{(\mt_{k_0},e_{i},0)}\right),\nonumber
	\end{align}
	where
	\begin{align*}
		A_{(\mt_k,0,\delta_1 e_n)}
		:=&\sum_{k'=0}^{r-1}
		a_{k'} \mathtt{b}_{+}(\res(\alpha_k))^{\nu(\delta_1)k'}\\
		&+\sum_{\substack{1\leq\iota\leq N-m-1\\ \delta\in\mathbb{Z}_2}}b_{\iota,\delta} \frac{\epsilon\mathtt{c}^{\mathfrak{u}_{1}}_{\mathfrak{u}_{\iota}}}{1-\mathtt{b}_{+}(\res(\beta_{\iota}))^{\nu(\delta)}\mathtt{b}_{+}(\res(\alpha_k))^{-\nu(\delta_1)}}\in\mathbb{K}
	\end{align*}
	for $k\in[N]\setminus\{k_0\},$ $\delta_1\in\mathbb{Z}_2,$ and
	\begin{align*}
		A_{(\mt_{k_0},0,0)}
		:=&\sum_{k'=0}^{r-1}a_{k'}
		+\sum_{\substack{1\leq\iota\leq N-m-1\\ \delta\in\mathbb{Z}_2}}b_{\iota,\delta}	\frac{\epsilon\mathtt{c}^{\mathfrak{u}_{1}}_{\mathfrak{u}_{\iota}}}{1-\mathtt{b}_{+}(\res(\beta_{\iota}))^{\nu(\delta)}\mathtt{b}_{+}(\res(\alpha_{k_0}))^{-1}}\in\mathbb{K}.
	\end{align*}

\begin{caselist}
	\item $d_{\undla}=0,$ $\delta_{\beta_{\mt}}(n)=\bar{0}.$\label{d=0,delta n=0}
	 In this case, we have $k_0=1$ and $\mt=\mt_1=\mt_{k_0}$.
	By  \eqref{type=s.(3).nondegLinearEquations}, we have
	\begin{align*}
		A_{(\mt_{k_0},0,0)}=1, \text{ and }  A_{(\mt_k,0,\delta_1 e_n)}=0
	\end{align*}
	for $k\in[N]\setminus\{k_0\},$ $\delta_1\in\mathbb{Z}_2.$ It follows that
	$$h=F_{\rm T}+F_{(\mt,e_{i},0)}.$$
	By the uniqueness of Mackey decomposition \eqref{type=s.Madecomequation of fTT0}, we have
	\begin{align*}
		\varepsilon_n\left(F_{\rm T}\right)
		+\varepsilon_n\left(F_{(\mt,e_{i},0)}\right)
		=\varepsilon_n(h)=a_0F_{{\rm T}\downarrow_{n-1}}.
	\end{align*}
On the other hand, it follows from Lemma \ref{reductionlem} (2.b) that
	\begin{align*}
		\varepsilon_n\left(F_{\rm T}\right)
		=\varepsilon_n\left(F_{(\mt,e_{i},0)}\right).
	\end{align*}
Hence we deduce from char $\mathbb{K}\neq 2$ that
	\begin{align*}
		\varepsilon_n(F_{{\rm T}})=\frac{1}{2}a_0F_{{\rm T}\downarrow_{n-1}}.
	\end{align*}
	Following a similar computation with \eqref{a0.case(1.a)}, we have
	\begin{align*}
		a_0=  \frac{\prod\limits_{\beta\in\Rem(\mt\downarrow_{n-1})}\left(\mathtt{q}(\res_{\mt}(n))-\mathtt{q}(\res(\beta))\right)}{
			\prod\limits_{\alpha\in\Add(\mt\downarrow_{n-1})\setminus \{\mt^{-1}(n)\}}\left(\mathtt{q}(\res_{\mt}(n))-\mathtt{q}(\res(\alpha))\right)}.
	\end{align*}
\item $d_{\undla}=1,$ $\delta_{\beta_{\mt}}(n)=\bar{0}.$\label{d=1,delta n=0}
	 In this case, we have $k_0\neq 1,$ and ${\rm T}=(\mt_1,0,0)=(\mt,0,0).$
	By \eqref{type=s.(3).nondegLinearEquations}, we have
	\begin{align*}
		A_{(\mt_1,0,0)}=1, \quad A_{(\mt_{k_0},0,0)}=0, \text{ and }  A_{(\mt_k,0,\delta_1 e_n)}=0
	\end{align*}
	for $(k,\delta_1)\in \biggl(\bigl([N]\setminus\{k_0\}\bigr)\times\mathbb{Z}_2\biggr)\setminus \{(1,\bar{0})\}.$ It follows that
	$$h=F_{(\mt_1,0,0)}=F_{{\rm T}}.$$
By the uniqueness of Mackey decomposition \eqref{type=s.Madecomequation of fTT0}, we have
	\begin{align*}
		\varepsilon_n(F_{{\rm T}})=\varepsilon_n(h)=a_0F_{{\rm T}\downarrow_{n-1}}.
	\end{align*}
	Following a similar computation with \eqref{a0.case(1.b)}, we have
	\begin{align*}
		a_0
		=
		\frac{\mathtt{b}_{+}(\res_{\mt}(n))^{-m}}{1+\mathtt{b}_{+}(\res_{\mt}(n))}
		\frac{\prod\limits_{\beta\in\Rem(\mt\downarrow_{n-1})}\left(\mathtt{q}(\res_{\mt}(n))-\mathtt{q}(\res(\beta))\right)}{
			\prod\limits_{\alpha\in\Add(\mt\downarrow_{n-1})\setminus \{\mt^{-1}(n)\}}\left(\mathtt{q}(\res_{\mt}(n))-\mathtt{q}(\res(\alpha))\right)}.
	\end{align*}
\item $d_{\undla}=1,$ $\delta_{\beta_{\mt}}(n)=\bar{1}.$\label{d=1,delta n=1}
 In this case, we have $k_0\neq 1,$ and ${\rm T}=(\mt_1,0,e_n)=(\mt,0,e_n).$ Following a similar computation with Case \ref{d=1,delta n=0}, we have
	\begin{align*}
		\varepsilon_n(F_{{\rm T}})=\varepsilon_n(h)=a_0F_{{\rm T}\downarrow_{n-1}}
	\end{align*}
and
	\begin{align*}
		a_0
		=
		\frac{\mathtt{b}_{-}(\res_{\mt}(n))^{-m}}{1+\mathtt{b}_{-}(\res_{\mt}(n))}
		\frac{\prod\limits_{\beta\in\Rem(\mt\downarrow_{n-1})}\left(\mathtt{q}(\res_{\mt}(n))-\mathtt{q}(\res(\beta))\right)}{
			\prod\limits_{\alpha\in\Add(\mt\downarrow_{n-1})\setminus \{\mt^{-1}(n)\}}\left(\mathtt{q}(\res_{\mt}(n))-\mathtt{q}(\res(\alpha))\right)}.
	\end{align*}
	\end{caselist}
	Finally, note that $\Rem(\mt\downarrow_{n-1})\cap\mathcal{D}=\emptyset$, since $\delta_{\mu^{(0)}}=1$. Combining this with Case \ref{d=0,delta n=0}, Case \ref{d=1,delta n=0}, Case \ref{d=1,delta n=1}, we deduce the formula \eqref{type=s, case(3) ,veFT}.
\end{proof}

\subsection{Even length with $\delta_{\mu^{(0)}}=0$}\label{type=s, case(2) ,section}
In this subsection, we shall compute $\varepsilon_n(F_{\rm T})=p_{0,\bar{0}}(F_{\rm T})$ when $d_{\undmu}=0$, $\delta_{\mu^{(0)}}=0$. Since $\delta_{\mu^{(0)}}=0$, we have $d_{\undla}=0$. We set $\iota_{0}\in[N-m]$ to be the unique number such that $\beta_{\iota_0}\in\Rem(\undmu)\cap\mathcal{D}$.

The main result of this subsection is the following.
\begin{prop}\label{type=s, case(2) ,mainprop}
Let $\undla\in\mathscr{P}^{f}_{n}.$ Suppose ${\rm T}=(\mt,0,\beta_{\mt})\in{\rm Tri}_{\bar{0}}(\undla)$ with $\beta_{\mt}\downarrow_{n-1}=0$ and ${\rm T}\downarrow_{n-1}=(\mt\downarrow_{n-1},0,0)\in {\rm Tri}_{\bar{0}}(\underline{\mu}).$ If $d_{\undmu}=0$ and $\delta_{\mu^{(0)}}=0,$ we have the following.
\begin{enumerate}
  \item There are some scalars $a_0,\ldots,a_{r-1},b_{1,\bar{0}},\ldots,b_{N-m,\bar{0}},
      b_{1,\bar{1}},\ldots,b_{\iota_{0}-1,\bar{1}},b_{\iota_{0}+1,\bar{1}},\ldots,b_{N-m,\bar{1}}\in\mathbb{K}$
such that
\begin{align*}
\pi(F_{{\rm T}})
&=\sum_{\substack{1\leq\iota\leq N-m, \iota\neq\iota_{0}\\ \delta\in\mathbb{Z}_2}}b_{\iota,\delta}
  f_{{\rm T}\downarrow_{n-1},(\mathfrak{u}_{\iota},0,\delta e_{n-1})}
  \otimes f_{(\mathfrak{u}_{\iota},0,\delta e_{n-1}),{\rm T}\downarrow_{n-1}} \\
&\quad\quad + b_{\iota_0,\bar{0}} f_{{\rm T}\downarrow_{n-1},(\mathfrak{u}_{\iota_0},0,0)}\otimes f_{(\mathfrak{u}_{\iota_0},0,0),{\rm T}\downarrow_{n-1}},
\end{align*}
and $$
p_{k,\bar{0}}(F_{{\rm T}})=a_k F_{{\rm T}\downarrow_{n-1}},
\quad 0\leq k <r.
$$
  \item We have
\begin{align}\label{type=s, case(2) ,veFT}
\varepsilon_n(F_{\rm T})
=\eta_{\mt,\beta_{\mt}}(n)
    \frac{\prod\limits_{\beta\in\Rem(\mt\downarrow_{n-1})\setminus\mathcal{D}}\left(\mathtt{q}(\res_{\mt}(n))-\mathtt{q}(\res(\beta))\right)}{
    \prod\limits_{\alpha\in\Add(\mt\downarrow_{n-1})\setminus \{\mt^{-1}(n)\}}\left(\mathtt{q}(\res_{\mt}(n))-\mathtt{q}(\res(\alpha))\right)} F_{\rm T\downarrow_{n-1}}.
\end{align}
\end{enumerate}
\end{prop}

\begin{proof}
We define
\begin{align}\label{type=s.(2).eq.widehatpi}
\widehat{\pi}(F_{{\rm T}})
&=\sum_{\substack{1\leq\iota\leq N-m, \iota\neq\iota_{0}\\ \delta\in\mathbb{Z}_2}}b_{\iota,\delta}
  f_{{\rm T}\downarrow_{n-1},(\mathfrak{u}_{\iota},0,\delta e_{n-1})}
  \otimes f_{(\mathfrak{u}_{\iota},0,\delta e_{n-1}),{\rm T}\downarrow_{n-1}} \\
&\quad + b_{\iota_0,\bar{0}} f_{{\rm T}\downarrow_{n-1},(\mathfrak{u}_{\iota_0},0,0)}\otimes f_{(\mathfrak{u}_{\iota_0},0,0),{\rm T}\downarrow_{n-1}}\in \mathcal{H}^f_c(n-1)\otimes_{\mathcal{H}^f_c(n-2)}\mathcal{H}^f_c(n-1)\nonumber
\end{align}
for some unknown scalars $b_{\iota,\delta}\in \mathbb{K}.$
Then we get that
\begin{align*}
\mu_{T_{n-1}}(\widehat{\pi}(F_{{\rm T}}))
&=\sum_{\substack{1\leq\iota\leq N-m,\iota\neq\iota_0 \\ \delta\in\mathbb{Z}_2}}b_{\iota,\delta}
  \iota_{n-1}\left(f_{{\rm T}\downarrow_{n-1},(\mathfrak{u}_{\iota},0,\delta e_{n-1})}\right)
  T_{n-1} \iota_{n-1}\left(f_{(\mathfrak{u}_{\iota},0,\delta e_{n-1}),{\rm T}\downarrow_{n-1}}\right)\\
&\qquad\qquad\qquad + b_{\iota_0,\bar{0}}
  \iota_{n-1}\left(f_{{\rm T}\downarrow_{n-1},(\mathfrak{u}_{\iota_0},0,0)}\right)
  T_{n-1} \iota_{n-1}\left(f_{(\mathfrak{u}_{\iota_0},0,0),{\rm T}\downarrow_{n-1}}\right).
\end{align*}
By Theorem \ref{maththm of embedding} (1), we have
\begin{align*}
\iota_{n-1}\left(f_{{\rm T}\downarrow_{n-1},(\mathfrak{u}_{\iota},0,\delta e_{n-1})}\right)
&=\sum_{\substack{1 \leq k\leq N \\ \delta_1\in\mathbb{Z}_2}}
  f_{(\mt_k,0,\delta_1 e_n),(\mathfrak{u}_{\iota,k},0,\delta e_{n-1}+\delta_1 e_n)},\\
\iota_{n-1}\left(f_{(\mathfrak{u}_{\iota},0,\delta e_{n-1}),{\rm T}\downarrow_{n-1}}\right)
&=\sum_{\substack{1 \leq k\leq N \\ \delta_2\in\mathbb{Z}_2}}
  f_{(\mathfrak{u}_{\iota,k},0,\delta e_{n-1}+\delta_2 e_n),(\mt_k,0,\delta_2 e_n)},
\end{align*}
for $(\iota,\delta)\in\biggl(\big([N-m]\setminus\{\iota_0\}\bigr)\times\mathbb{Z}_2\biggr)\cup\{(\iota_0,\bar{0})\}$.
Combining with Proposition \ref{generators action on seminormal basis} (2) and equation \eqref{Non-deg multiplication1}, \eqref{Non-deg multiplication2}, we have
\begin{align}\label{type=s.(2).eq. muTn-1}
&\mu_{T_{n-1}}(\widehat{\pi}(F_{{\rm T}}))\\
&=\sum_{\substack{1\leq\iota\leq N-m,\iota\neq\iota_0\\ \delta\in\mathbb{Z}_2}}b_{\iota,\delta}
   \sum_{\substack{1\leq k \leq N \\ \delta_1\in\mathbb{Z}_2}}
       \frac{\epsilon}{1-\mathtt{b}_{+}(\res(\beta_{\iota}))^{\nu(\delta)}\mathtt{b}_{+}(\res(\alpha_k))^{-\nu(\delta_1)}}
        \mathtt{c}^{\mathfrak{u}_{1,k}}_{\mathfrak{u}_{\iota,k}}
         F_{(\mt_k,0,\delta_1 e_n)} \nonumber\\
&\qquad\qquad\quad
      +\sum_{\substack{1\leq k\leq N\\ \delta_1\in\mathbb{Z}_2}}b_{\iota_0,\bar{0}}
          \frac{\epsilon}{1-\mathtt{b}_{+}(\res(\beta_{\iota_0}))\mathtt{b}_{+}(\res(\alpha_k))^{-\nu(\delta_1)}}
           \mathtt{c}^{\mathfrak{u}_{1,k}}_{\mathfrak{u}_{\iota_0,k}}
            F_{(\mt_k,0,\delta_1 e_n)} \nonumber\\
&=\sum_{\substack{1\leq\iota\leq N-m,\iota\neq\iota_0\\ \delta\in\mathbb{Z}_2}}b_{\iota,\delta}
\sum_{\substack{1\leq k \leq N \\ \delta_1\in\mathbb{Z}_2}}
\frac{\epsilon}{1-\mathtt{b}_{+}(\res(\beta_{\iota}))^{\nu(\delta)}\mathtt{b}_{+}(\res(\alpha_k))^{-\nu(\delta_1)}}
\mathtt{c}^{\mathfrak{u}_{1}}_{\mathfrak{u}_{\iota}}
F_{(\mt_k,0,\delta_1 e_n)} \nonumber\\
&\qquad\qquad\quad
+\sum_{\substack{1\leq k\leq N\\ \delta_1\in\mathbb{Z}_2}}b_{\iota_0,\bar{0}}
\frac{\epsilon}{1-\mathtt{b}_{+}(\res(\beta_{\iota_0}))\mathtt{b}_{+}(\res(\alpha_k))^{-\nu(\delta_1)}}
\mathtt{c}^{\mathfrak{u}_{1}}_{\mathfrak{u}_{\iota_0}}
F_{(\mt_k,0,\delta_1 e_n)}, \nonumber
\end{align}
where $\mathtt{b}_{+}(\res(\beta_{\iota_0}))=1$ and
\begin{align*}
\nu(\delta):=
\begin{cases}
-1, &\text{if $\delta=\bar{1}$,}\\
1, &\text{if $\delta=\bar{0}$,}
\end{cases}
\quad
\nu(\delta_1):=
\begin{cases}
-1, &\text{if $\delta_1=\bar{1}$,}\\
1, &\text{if $\delta_1=\bar{0}$.}
\end{cases}
\end{align*}

	By Lemma \ref{typeos, scalars ak} (1), we have
	\begin{align}\label{scalars ak. types, (2)}p_{k',\bar{1}}(F_{{\rm T}})=0,\quad
	p_{k',\bar{0}}(F_{{\rm T}})=a_{k'} F_{{\rm T}\downarrow_{n-1}}
	\end{align}
	for some $a_{k'}\in \mathbb{K},$ $0\leq k' <r.$

It follows from Theorem \ref{maththm of embedding} (1) and Proposition \ref{generators action on seminormal basis} (1) that
\begin{align}\label{type=s.(2).eq. Xnk'pk'0}
&\sum_{k'=0}^{r-1}X_n^{k'}\iota_{n-1}(p_{k',\bar{0}}(F_{{\rm T}}))\\
&=\sum_{k'=0}^{r-1}a_{k'} X_n^{k'}\iota_{n-1}(F_{{\rm T}\downarrow_{n-1}}) \nonumber\\
&=\sum_{k'=0}^{r-1}a_{k'} \sum_{\substack{1\leq k\leq N\\ \delta_1\in\mathbb{Z}_2}}
    X_n^{k'} F_{(\mt_k,0,\delta_1 e_n)}\nonumber\\
&= \sum_{\substack{1\leq k\leq N\\ \delta_1\in\mathbb{Z}_2}}\sum_{k'=0}^{r-1}
   a_{k'}\mathtt{b}_{+}(\res(\alpha_k))^{\nu(\delta_1)k'}
    F_{(\mt_k,0,\delta_1 e_n)}.\nonumber
\end{align}We now regard those scalars
\begin{align}\label{type=s.(2).nondegunknownquantities}
a_0,\ldots,a_{r-1},b_{1,\bar{0}},\ldots,b_{N-m,\bar{0}},
      b_{1,\bar{1}},\ldots,b_{\iota_{0}-1,\bar{1}},b_{\iota_{0}+1,\bar{1}},\ldots,b_{N-m,\bar{1}}
\end{align}
appeared in  \eqref{scalars ak. types, (2)} and \eqref{type=s.(2).eq.widehatpi} as unknown quantities.
Combining \eqref{type=s.Madecomequation of fTT0} with \eqref{type=s.(2).eq. muTn-1}, \eqref{type=s.(2).eq. Xnk'pk'0} and comparing the coefficients of
$F_{(\mt_k,0,\delta_1 e_n)},$
$k\in [N],$ $\delta_1\in\mathbb{Z}_2$
on both sides of \eqref{type=s.Madecomequation of fTT0} with $\pi(F_{{\rm T}})$ replaced
with $\widehat{\pi}(F_{{\rm T}}),$ then we get the following linear system of equations in two cases:
\begin{align}\label{type=s.(2).nondegLinearEquations}
& A \cdot \text{diag}(\underbrace{1,\dots,1}_{r\,\text{copies}},
                    \epsilon\mathtt{c}^{\mathfrak{u}_{1}}_{\mathfrak{u}_{1}},\dots,
                     \epsilon\mathtt{c}^{\mathfrak{u}_{1}}_{\mathfrak{u}_{N-m}},
                      \epsilon\mathtt{c}^{\mathfrak{u}_{1}}_{\mathfrak{u}_{1}},\dots,
                       \epsilon\mathtt{c}^{\mathfrak{u}_{1}}_{\mathfrak{u}_{\iota_0-1}},
                        \epsilon\mathtt{c}^{\mathfrak{u}_{1}}_{\mathfrak{u}_{\iota_0+1}},\dots,
                         \epsilon\mathtt{c}^{\mathfrak{u}_{1}}_{\mathfrak{u}_{N-m}})\nonumber\\
& \cdot \left(a_0,\dots,a_{r-1},b_{1,\bar{0}},\dots,b_{N-m,\bar{0}},b_{1,\bar{1}},\dots,b_{\iota_0-1,\bar{1}},b_{\iota_0+1,\bar{1}},\dots,b_{N-m,\bar{1}}\right)^{T}\\
&\quad=\begin{cases}
  (1,\underbrace{0,\dots,0}_{2N-1\,\text{copies}})^T, &\text{if $\delta_{\beta_{\mt}}(n)=\bar{0}$,}\\
  (\underbrace{0,\dots,0}_{N\,\text{copies}},1,\underbrace{0,\dots,0}_{N-1\,\text{copies}})^T, &\text{if $\delta_{\beta_{\mt}}(n)=\bar{1}$,}
  \end{cases}\nonumber
\end{align}
where the $2N\times 2N$ matrix $A=\left(B,C\right)$ is given by
$$B=\begin{pmatrix}
1 & \mathbb{X}_1 & \cdots & \mathbb{X}_1^{r-1} \\
\vdots & \vdots &  & \vdots\\
1 & \mathbb{X}_N & \cdots & \mathbb{X}_N^{r-1} \\
1 & \mathbb{X}_1^{-1} & \cdots & \mathbb{X}_1^{-(r-1)}\\
\vdots & \vdots &  & \vdots\\
1 & \mathbb{X}_N^{-1} & \cdots & \mathbb{X}_N^{-(r-1)}
\end{pmatrix}_{2N\times r},$$
{\footnotesize
$$C=\begin{pmatrix}
\frac{\mathbb{X}_1}{\mathbb{X}_1-{\rm Y}_1} & \cdots & \frac{\mathbb{X}_1}{\mathbb{X}_1-{\rm Y}_{N-m}} &
\frac{\mathbb{X}_1}{\mathbb{X}_1-{\rm Y}_1^{-1}} & \cdots & \frac{\mathbb{X}_1}{\mathbb{X}_1-{\rm Y}_{\iota_0-1}^{-1}} & \frac{\mathbb{X}_1}{\mathbb{X}_1-{\rm Y}_{\iota_0+1}^{-1}} & \cdots & \frac{\mathbb{X}_1}{\mathbb{X}_1-{\rm Y}_{N-m}^{-1}}\\
\vdots &  & \vdots & \vdots &  & \vdots & \vdots &  &\vdots\\
\frac{\mathbb{X}_N}{\mathbb{X}_N-{\rm Y}_1} & \cdots & \frac{\mathbb{X}_N}{\mathbb{X}_N-{\rm Y}_{N-m}} &
\frac{\mathbb{X}_N}{\mathbb{X}_N-{\rm Y}_1^{-1}} & \cdots & \frac{\mathbb{X}_N}{\mathbb{X}_N-{\rm Y}_{\iota_0-1}^{-1}} & \frac{\mathbb{X}_N}{\mathbb{X}_N-{\rm Y}_{\iota_0+1}^{-1}} & \cdots & \frac{\mathbb{X}_N}{\mathbb{X}_N-{\rm Y}_{N-m}^{-1}}\\
\frac{\mathbb{X}_1^{-1}}{\mathbb{X}_1^{-1}-{\rm Y}_1} & \cdots & \frac{\mathbb{X}_1^{-1}}{\mathbb{X}_1^{-1}-{\rm Y}_{N-m}} &
\frac{\mathbb{X}_1^{-1}}{\mathbb{X}_1^{-1}-{\rm Y}_1^{-1}} & \cdots & \frac{\mathbb{X}_1^{-1}}{\mathbb{X}_1^{-1}-{\rm Y}_{\iota_0-1}^{-1}} & \frac{\mathbb{X}_1^{-1}}{\mathbb{X}_1^{-1}-{\rm Y}_{\iota_0+1}^{-1}} & \cdots & \frac{\mathbb{X}_1^{-1}}{\mathbb{X}_1^{-1}-{\rm Y}_{N-m}^{-1}}\\
\vdots &  & \vdots & \vdots &  & \vdots & \vdots &  &\vdots\\
\frac{\mathbb{X}_N^{-1}}{\mathbb{X}_N^{-1}-{\rm Y}_1} & \cdots & \frac{\mathbb{X}_N^{-1}}{\mathbb{X}_N^{-1}-{\rm Y}_{N-m}} &
\frac{\mathbb{X}_N^{-1}}{\mathbb{X}_N^{-1}-{\rm Y}_1^{-1}} & \cdots & \frac{\mathbb{X}_N^{-1}}{\mathbb{X}_N^{-1}-{\rm Y}_{\iota_0-1}^{-1}} & \frac{\mathbb{X}_1}{\mathbb{X}_N^{-1}-{\rm Y}_{\iota_0+1}^{-1}} & \cdots & \frac{\mathbb{X}_N^{-1}}{\mathbb{X}_N^{-1}-{\rm Y}_{N-m}^{-1}}
\end{pmatrix}_{2N\times (2N-r)}.$$ }

If the matrix $A$ is invertible, then we can deduce that there exist scalars \eqref{type=s.(2).nondegunknownquantities} such that the corresponding elements $\widehat{\pi}(F_{{\rm T}})$ and $p_{k,\bar{0}}(F_{{\rm T}}),\ k=0,1,\dots,r-1$ satisfy the equality \eqref{type=s.Madecomequation of fTT0} with $\pi(F_{{\rm T}})$ replaced by $\widehat{\pi}(F_{{\rm T}})$ (Note again that all $p_{k,\bar{1}}(F_{{\rm T}})=0$).
Then by the uniqueness statement of \eqref{type=s.Madecomequation of fTT0}, we can deduce that $\widehat{\pi}(F_{{\rm T}})=\pi(F_{{\rm T}}).$

To use Lemmas \ref{TechLem1}, \ref{TechLem2} conveniently, we further denote $\mathbb{X}_{N+k}:=\mathbb{X}_k^{-1},$ $k\in[N]$ and ${\rm Z}_{\iota}:={\rm Y}_{\iota},$ ${\rm Z}_{N-m+\iota}:={\rm Y}_{\iota}^{-1},$ $\iota\in[N-m]\setminus\{\iota_0\}.$ Then it follows from
Lemma \ref{TechLem1} that
\begin{align}\label{type=s.(2).detA}
{\rm det}(A)
&=(-1)^r\frac{\prod\limits_{\substack{1\leq t\leq 2(N-m) \\ t\neq N-m+\iota_0}}{\rm Z}_t \prod\limits_{\substack{1\leq p< q\leq 2(N-m)\\
p,q\neq N-m+\iota_0}}({\rm Z}_p-{\rm Z}_q)\prod\limits_{1\leq u<v\leq 2N}(\mathbb{X}_v-\mathbb{X}_u)}
{\prod\limits_{j=1}^{2N}\prod\limits_{\substack{1\leq s\leq 2(N-m)\\ s\neq N-m+\iota_0}}(\mathbb{X}_j-{\rm Z}_s)}\\
&=(-1)^r\frac{\prod\limits_{\substack{1\leq p< q\leq 2(N-m)\\ p,q\neq N-m+\iota_0}}({\rm Z}_p-{\rm Z}_q)\prod\limits_{1\leq u<v\leq 2N}(\mathbb{X}_v-\mathbb{X}_u)}
{\prod\limits_{j=1}^{2N}\prod\limits_{\substack{1\leq s\leq 2(N-m)\\ s\neq N-m+\iota_0}}(\mathbb{X}_j-{\rm Z}_s)},\nonumber
\end{align}
where we have used the fact that
$$\prod\limits_{\substack{1\leq t\leq 2(N-m) \\ t\neq N-m+\iota_0}}{\rm Z}_t={\rm Y}_{\iota_0}=1.$$
It follows from Proposition \ref{compare eigenvalue} that $A$ is invertible. Hence
By Lemma \ref{TechLem2},
the algebraic cofactor $A_{1,1}$ of ${\rm det}(A)$ at position $(1,1)$ is
\begin{align*}
A_{1,1}
&=(-1)^{r-1}\left(\prod\limits_{2\leq k\leq 2N}\mathbb{X}_k \right)\cdot\frac{\prod\limits_{\substack{1\leq p< q\leq 2(N-m)\\ p,q\neq N-m+\iota_0}}({\rm Z}_p-{\rm Z}_q)\prod\limits_{2\leq u<v\leq 2N}(\mathbb{X}_v-\mathbb{X}_u)}{\prod\limits_{j=2}^{2N}\prod\limits_{\substack{1\leq s\leq 2(N-m)\\ s\neq N-m+\iota_0}}(\mathbb{X}_j-{\rm Z}_s)}\\
&=(-1)^{r-1}\mathbb{X}_1^{-1}\cdot\frac{\prod\limits_{\substack{1\leq p< q\leq 2(N-m)\\ p,q\neq N-m+\iota_0}}({\rm Z}_p-{\rm Z}_q)\prod\limits_{2\leq u<v\leq 2N}(\mathbb{X}_v-\mathbb{X}_u)}{\prod\limits_{j=2}^{2N}\prod\limits_{\substack{1\leq s\leq 2(N-m)\\ s\neq N-m+\iota_0}}(\mathbb{X}_j-{\rm Z}_s)},
\end{align*}
where we have used the fact that $\prod_{1\leq k\leq 2N}\mathbb{X}_k=1.$

Similarly, by Lemma \ref{TechLem2}, the algebraic cofactor $A_{N+1,1}$ of ${\rm det}(A)$ at position $(N+1,1)$ is
\begin{align*}
A_{N+1,1}
&=(-1)^{N+r-1}\left(\prod\limits_{\substack{1\leq k\leq 2N\\ k\neq N+1}}\mathbb{X}_k \right)\cdot\frac{\prod\limits_{\substack{1\leq p< q\leq 2(N-m)\\ p,q\neq N-m+\iota_0}}({\rm Z}_p-{\rm Z}_q)\prod\limits_{\substack{1\leq u<v\leq 2N \\ u,v\neq N+1}}(\mathbb{X}_v-\mathbb{X}_u)}{\prod\limits_{\substack{1\leq j\leq n \\ j\neq N+1}}\prod\limits_{\substack{1\leq s\leq 2(N-m)\\ s\neq N-m+\iota_0}}(\mathbb{X}_j-{\rm Z}_s)}\\
&=(-1)^{N+r-1}\mathbb{X}_1\cdot\frac{\prod\limits_{\substack{1\leq p< q\leq 2(N-m)\\ p,q\neq N-m+\iota_0}}({\rm Z}_p-{\rm Z}_q)\prod\limits_{\substack{1\leq u<v\leq 2N \\ u,v\neq N+1}}(\mathbb{X}_v-\mathbb{X}_u)}{\prod\limits_{\substack{1\leq j\leq n \\ j\neq N+1}}\prod\limits_{\substack{1\leq s\leq 2(N-m)\\ s\neq N-m+\iota_0}}(\mathbb{X}_j-{\rm Z}_s)},
\end{align*}
where we have also used the fact that $\prod_{1\leq k\leq 2N}\mathbb{X}_k=1.$

\begin{caselist}
		\item $d_{\undla}=0,$ $\delta_{\beta_{\mt}}(n)=\bar{0}.$\label{d=0delta n=0}
	Then
\begin{align}\label{a0.case(2.a)}
&a_0
=\frac{A_{1,1}}{{\rm det}(A)}
=\mathbb{X}_1^{-1}
  \frac{\prod\limits_{\substack{1\leq s\leq 2(N-m)\\ s\neq N-m+\iota_0}}(\mathbb{X}_1-{\rm Z}_s)}{\prod\limits_{v=2}^{2N}(\mathbb{X}_1-\mathbb{X}_v)}\\
&=   \frac{\mathtt{b}_{-}(\res_{\mt}(n))\left(\mathtt{b}_{+}(\res_{\mt}(n))-1\right)}{\mathtt{b}_{+}(\res_{\mt}(n))-\mathtt{b}_{-}(\res_{\mt}(n))}\nonumber\\
  &\qquad\qquad\cdot\frac{\prod\limits_{\beta\in\Rem(\mt\downarrow_{n-1})\setminus\{\beta_{\iota_0}\},\ast\in\{\pm\}}\left(\mathtt{b}_{+}(\res_{\mt}(n))-\mathtt{b}_{\ast}(\res(\beta))\right)}{
         \prod\limits_{\alpha\in\Add(\mt\downarrow_{n-1})\setminus \{\mt^{-1}(n)\},\ast\in\{\pm\}}\left(\mathtt{b}_{+}(\res_{\mt}(n))-\mathtt{b}_{\ast}(\res(\alpha))\right)}\nonumber\\
&= \frac{\mathtt{b}_{+}(\res_{\mt}(n))^{-m}}{\mathtt{b}_{+}(\res_{\mt}(n))+1}\nonumber\\
   &\qquad\qquad\cdot\frac{\prod\limits_{\beta\in\Rem(\mt\downarrow_{n-1})\setminus\mathcal{D}}\left(\mathtt{q}(\res_{\mt}(n))-\mathtt{q}(\res(\beta))\right)}{
          \prod\limits_{\alpha\in\Add(\mt\downarrow_{n-1})\setminus \{\mt^{-1}(n)\}}\left(\mathtt{q}(\res_{\mt}(n))-\mathtt{q}(\res(\alpha))\right)},\nonumber
\end{align}
where in the last equation we have used the elementary identity
$$(b'-b)(b'-b^{-1})=b'(b'+b'^{-1}-b-b^{-1}),\qquad \forall b,b'\in\mathbb{K}^*,$$
and Lemma \ref{HLL. lem2.13. type=s}.
\item $d_{\undla}=0,$ $\delta_{\beta_{\mt}}(n)=\bar{1}.$\label{d=0delta n=1}
Then
\begin{align}\label{a0.case(2.b)}
&a_0
=\frac{A_{N+1,1}}{{\rm det}(A)}
=(-1)^{N+1}\mathbb{X}_1\cdot
  \frac{\prod\limits_{\substack{1\leq s\leq 2(N-m)\\ s\neq N-m+\iota_0}}(\mathbb{X}_{N+1}-{\rm Z}_s)}{\prod\limits_{u=1}^{N}(\mathbb{X}_{N+1}-\mathbb{X}_u)\prod\limits_{v=N+2}^{2N}(\mathbb{X}_v-\mathbb{X}_{N+1})}\\
&=   \frac{\mathtt{b}_{+}(\res_{\mt}(n))\left(\mathtt{b}_{-}(\res_{\mt}(n))-1\right)}{\mathtt{b}_{-}(\res_{\mt}(n))-\mathtt{b}_{+}(\res_{\mt}(n))}\nonumber\\
  &\qquad\qquad\cdot
    \frac{\prod\limits_{\beta\in\Rem(\mt\downarrow_{n-1})\setminus\{\beta_{\iota_0}\},\ast\in\{\pm\}}\left(\mathtt{b}_{-}(\res_{\mt}(n))-\mathtt{b}_{\ast}(\res(\beta))\right)}{
    \prod\limits_{\alpha\in\Add(\mt\downarrow_{n-1})\setminus \{\mt^{-1}(n)\},\ast\in\{\pm\}}\left(\mathtt{b}_{-}(\res_{\mt}(n))-\mathtt{b}_{\ast}(\res(\alpha))\right)}\nonumber\\
&=
   \frac{\mathtt{b}_{-}(\res_{\mt}(n))^{-m}}{\mathtt{b}_{-}(\res_{\mt}(n))+1}\nonumber\\
  &\qquad\qquad\cdot
    \frac{\prod\limits_{\beta\in\Rem(\mt\downarrow_{n-1})\setminus\mathcal{D}}\left(\mathtt{q}(\res_{\mt}(n))-\mathtt{q}(\res(\beta))\right)}{
    \prod\limits_{\alpha\in\Add(\mt\downarrow_{n-1})\setminus \{\mt^{-1}(n)\}}\left(\mathtt{q}(\res_{\mt}(n))-\mathtt{q}(\res(\alpha))\right)}.\nonumber
\end{align}
\end{caselist}
To sum up, we obtain the formula \eqref{type=s, case(2) ,veFT} from Case \ref{d=0delta n=0} and Case \ref{d=0delta n=1}.
\end{proof}

\subsection{Odd length with $\delta_{\mu^{(0)}}=0$}\label{type=s, case(4) ,section}

In this subsection, we shall compute $\varepsilon_n(F_{\rm T})=p_{0,\bar{0}}(F_{\rm T})$ when $d_{\undmu}=1$, $\delta_{\mu^{(0)}}=0$. Since $\delta_{\mu^{(0)}}=0$, we have $d_{\undla}=1$ and $n\notin \mathcal{D}_{\mt}$. We set $\iota_{0}\in[N-m]$ to be the unique number such that $\beta_{\iota_0}\in\Rem(\undmu)\cap\mathcal{D}$.

Since $d_{\undmu}=1$, we have $({\rm T}\downarrow_{n-1})_{\overline{1}}=(\mt\downarrow_{n-1},e_i,0)\in{\rm Tri}_{\bar{1}}(\undmu)$ for $i:=\max(\mathcal{OD}_{\mt\downarrow_{n-1}}).$
The proof combines the coefficient matrix from Subsection~\ref{type=s, case(2) ,section} with the treatment of the odd Clifford component in Subsection~\ref{type=s, case(3) ,section}. This gives the following result.

\begin{prop}\label{type=s, case(4) ,mainprop}
Let $\undla\in\mathscr{P}^{f}_{n}.$ Suppose ${\rm T}=(\mt,0,\beta_{\mt})\in{\rm Tri}_{\bar{0}}(\undla)$ with $\beta_{\mt}\downarrow_{n-1}=0$ and ${\rm T}\downarrow_{n-1}=(\mt\downarrow_{n-1},0,0)\in {\rm Tri}_{\bar{0}}(\underline{\mu}).$ If $d_{\undmu}=1$ and $\delta_{\mu^{(0)}}=0,$ then we have
\begin{align}\label{type=s, case(4) ,veFT}
\varepsilon_n(F_{\rm T})
=\eta_{\mt,\beta_{\mt}}(n)
    \frac{\prod\limits_{\beta\in\Rem(\mt\downarrow_{n-1})\setminus\mathcal{D}}\left(\mathtt{q}(\res_{\mt}(n))-\mathtt{q}(\res(\beta))\right)}{
    \prod\limits_{\alpha\in\Add(\mt\downarrow_{n-1})\setminus \{\mt^{-1}(n)\}}\left(\mathtt{q}(\res_{\mt}(n))-\mathtt{q}(\res(\alpha))\right)} F_{\rm T\downarrow_{n-1}}.
\end{align}
\end{prop}
\begin{proof}
We consider the following system of linear equations
\begin{align}\label{type=s.(4).nondegLinearEquations}
& A \cdot \text{diag}(\underbrace{1,\dots,1}_{r\,\text{copies}},
                    \epsilon\mathtt{c}^{\mathfrak{u}_{1}}_{\mathfrak{u}_{1}},\dots,
                     \epsilon\mathtt{c}^{\mathfrak{u}_{1}}_{\mathfrak{u}_{N-m}},
                      \epsilon\mathtt{c}^{\mathfrak{u}_{1}}_{\mathfrak{u}_{1}},\dots,
                       \epsilon\mathtt{c}^{\mathfrak{u}_{1}}_{\mathfrak{u}_{\iota_0-1}},
                        \epsilon\mathtt{c}^{\mathfrak{u}_{1}}_{\mathfrak{u}_{\iota_0+1}},\dots,
                         \epsilon\mathtt{c}^{\mathfrak{u}_{1}}_{\mathfrak{u}_{N-m}})\nonumber\\
& \cdot \left(a_0,\dots,a_{r-1},b_{1,\bar{0}},\dots,b_{N-m,\bar{0}},b_{1,\bar{1}},\dots,b_{\iota_0-1,\bar{1}},b_{\iota_0+1,\bar{1}},\dots,b_{N-m,\bar{1}}\right)^{T}\\
&\quad=\begin{cases}
  (1,\underbrace{0,\dots,0}_{2N-1\,\text{copies}})^T, &\text{if $\delta_{\beta_{\mt}}(n)=\bar{0}$,}\\
  (\underbrace{0,\dots,0}_{N\,\text{copies}},1,\underbrace{0,\dots,0}_{N-1\,\text{copies}})^T, &\text{if $\delta_{\beta_{\mt}}(n)=\bar{1}$,}
  \end{cases}\nonumber
\end{align}
where the $2N\times 2N$ matrix $A=\left(B,C\right)$ is the same as in \eqref{type=s.(2).nondegLinearEquations}
$$B=\begin{pmatrix}
1 & \mathbb{X}_1 & \cdots & \mathbb{X}_1^{r-1} \\
\vdots & \vdots &  & \vdots\\
1 & \mathbb{X}_N & \cdots & \mathbb{X}_N^{r-1} \\
1 & \mathbb{X}_1^{-1} & \cdots & \mathbb{X}_1^{-(r-1)}\\
\vdots & \vdots &  & \vdots\\
1 & \mathbb{X}_N^{-1} & \cdots & \mathbb{X}_N^{-(r-1)}
\end{pmatrix}_{2N\times r},$$
{\footnotesize
$$C=\begin{pmatrix}
\frac{\mathbb{X}_1}{\mathbb{X}_1-{\rm Y}_1} & \cdots & \frac{\mathbb{X}_1}{\mathbb{X}_1-{\rm Y}_{N-m}} &
\frac{\mathbb{X}_1}{\mathbb{X}_1-{\rm Y}_1^{-1}} & \cdots & \frac{\mathbb{X}_1}{\mathbb{X}_1-{\rm Y}_{\iota_0-1}^{-1}} & \frac{\mathbb{X}_1}{\mathbb{X}_1-{\rm Y}_{\iota_0+1}^{-1}} & \cdots & \frac{\mathbb{X}_1}{\mathbb{X}_1-{\rm Y}_{N-m}^{-1}}\\
\vdots &  & \vdots & \vdots &  & \vdots & \vdots &  &\vdots\\
\frac{\mathbb{X}_N}{\mathbb{X}_N-{\rm Y}_1} & \cdots & \frac{\mathbb{X}_N}{\mathbb{X}_N-{\rm Y}_{N-m}} &
\frac{\mathbb{X}_N}{\mathbb{X}_N-{\rm Y}_1^{-1}} & \cdots & \frac{\mathbb{X}_N}{\mathbb{X}_N-{\rm Y}_{\iota_0-1}^{-1}} & \frac{\mathbb{X}_N}{\mathbb{X}_N-{\rm Y}_{\iota_0+1}^{-1}} & \cdots & \frac{\mathbb{X}_N}{\mathbb{X}_N-{\rm Y}_{N-m}^{-1}}\\
\frac{\mathbb{X}_1^{-1}}{\mathbb{X}_1^{-1}-{\rm Y}_1} & \cdots & \frac{\mathbb{X}_1^{-1}}{\mathbb{X}_1^{-1}-{\rm Y}_{N-m}} &
\frac{\mathbb{X}_1^{-1}}{\mathbb{X}_1^{-1}-{\rm Y}_1^{-1}} & \cdots & \frac{\mathbb{X}_1^{-1}}{\mathbb{X}_1^{-1}-{\rm Y}_{\iota_0-1}^{-1}} & \frac{\mathbb{X}_1^{-1}}{\mathbb{X}_1^{-1}-{\rm Y}_{\iota_0+1}^{-1}} & \cdots & \frac{\mathbb{X}_1^{-1}}{\mathbb{X}_1^{-1}-{\rm Y}_{N-m}^{-1}}\\
\vdots &  & \vdots & \vdots &  & \vdots & \vdots &  &\vdots\\
\frac{\mathbb{X}_N^{-1}}{\mathbb{X}_N^{-1}-{\rm Y}_1} & \cdots & \frac{\mathbb{X}_N^{-1}}{\mathbb{X}_N^{-1}-{\rm Y}_{N-m}} &
\frac{\mathbb{X}_N^{-1}}{\mathbb{X}_N^{-1}-{\rm Y}_1^{-1}} & \cdots & \frac{\mathbb{X}_N^{-1}}{\mathbb{X}_N^{-1}-{\rm Y}_{\iota_0-1}^{-1}} & \frac{\mathbb{X}_1}{\mathbb{X}_N^{-1}-{\rm Y}_{\iota_0+1}^{-1}} & \cdots & \frac{\mathbb{X}_N^{-1}}{\mathbb{X}_N^{-1}-{\rm Y}_{N-m}^{-1}}
\end{pmatrix}_{2N\times (2N-r)}.$$ }
It follows from equation \eqref{type=s.(2).detA} and Proposition \ref{compare eigenvalue} that the above system of linear equations has a unique solution
$$\left(a_0,\dots,a_{r-1},b_{1,\bar{0}},\dots,b_{N-m,\bar{0}},b_{1,\bar{1}},\dots,b_{\iota_0-1,\bar{1}},b_{\iota_0+1,\bar{1}},\dots,b_{N-m,\bar{1}}\right)^{T}\in\mathbb{K}^{2N}.$$
Now we define an element $h\in\mHfcn$ via Mackey decomposition \eqref{nondegMadecomequation} as follows:
	\begin{align}
		h
	&=\sum_{\substack{1\leq\iota\leq N-m-1,\iota\neq\iota_0\\ \delta\in\mathbb{Z}_2}}b_{\iota,\delta}
	\iota_{n-1}\left(f_{{\rm T}\downarrow_{n-1},(\mathfrak{u}_{\iota},0,\delta e_{n-1})}\right)
	T_{n-1} \iota_{n-1}\left(f_{(\mathfrak{u}_{\iota},0,\delta e_{n-1}),{\rm T}\downarrow_{n-1}}\right)\nonumber\\
    &\qquad\qquad+b_{\iota_0,\bar{0}}
	\iota_{n-1}\left(f_{{\rm T}\downarrow_{n-1},(\mathfrak{u}_{\iota_0},0,0)}\right)
	T_{n-1} \iota_{n-1}\left(f_{(\mathfrak{u}_{\iota_0},0,0),{\rm T}\downarrow_{n-1}}\right)\nonumber\\
	&\qquad\qquad\qquad+\sum_{k'=0}^{r-1}a_{k'} X_n^{k'}\iota_{n-1}\left(F_{{\rm T}\downarrow_{n-1}}\right).\nonumber
	\end{align}
By Theorem \ref{maththm of embedding} (2), we have
\begin{align*}
\iota_{n-1}\left(f_{{\rm T}\downarrow_{n-1},(\mathfrak{u}_{\iota},0,\delta e_{n-1})}\right)
&=\sum_{\substack{1 \leq k\leq N \\ \delta_1\in\mathbb{Z}_2}}
  f_{(\mt_k,0,\delta_1 e_n),(\mathfrak{u}_{\iota,k},0,\delta e_{n-1}+\delta_1 e_n)},\\
\iota_{n-1}\left(f_{(\mathfrak{u}_{\iota},0,\delta e_{n-1}),{\rm T}\downarrow_{n-1}}\right)
&=\sum_{\substack{1 \leq k\leq N \\ \delta_2\in\mathbb{Z}_2}}
  f_{(\mathfrak{u}_{\iota,k},0,\delta e_{n-1}+\delta_2 e_n),(\mt_k,0,\delta_2 e_n)},\\
\iota_{n-1}\left(F_{{\rm T}\downarrow_{n-1}}\right)
&=\sum_{\substack{1 \leq k\leq N \\ \delta_3\in\mathbb{Z}_2}}
  F_{(\mt_k,0,\delta_3 e_n)},
\end{align*}
for $(\iota,\delta)\in\biggl(\big([N-m]\setminus\{\iota_0\}\bigr)\times\mathbb{Z}_2\biggr)\cup\{(\iota_0,\bar{0})\}.$

Combining with Proposition \ref{generators action on seminormal basis} (2), equations \eqref{Non-deg multiplication1}, \eqref{Non-deg multiplication2} and $\mathtt{b}_{+}(\res(\beta_{\iota_0}))=1,$ we can deduce that
\begin{align}
h
=\sum_{\substack{1\leq k \leq N\\ \delta_1\in\mathbb{Z}_2}}A_{(\mt_k,0,\delta_1 e_n)}F_{(\mt_k,0,\delta_1 e_n)}
 +\sum_{\substack{1\leq k \leq N\\ \delta_1\in\mathbb{Z}_2}}(-1)^{\delta_1}b_{\iota_0,\bar{0}}\frac{\epsilon\mathtt{c}^{\mathfrak{u}_{1}}_{\mathfrak{u}_{\iota_0}}}{1-\mathtt{b}_{+}(\res(\alpha_k))^{\nu(\delta_1)}}
 f_{(\mt_k,0,\delta_1 e_n),(\mt_k,e_i,(\delta_1 +\bar{1})e_n)},\nonumber
\end{align}
where $i:=\max(\mathcal{OD}_{\mt\downarrow_{n-1}})=\max(\mathcal{OD}_{\mt_k})$ for any $1\leq k \leq N,$ and
\begin{align*}
A_{(\mt_k,0,\delta_1 e_n)}
:=&\sum_{k'=0}^{r-1}
		a_{k'} \mathtt{b}_{+}(\res(\alpha_k))^{\nu(\delta_1)k'}\\
		&+\sum_{\substack{1\leq\iota\leq N-m-1\\ \delta\in\mathbb{Z}_2}}b_{\iota,\delta}		\frac{\epsilon\mathtt{c}^{\mathfrak{u}_{1}}_{\mathfrak{u}_{\iota}}}{1-\mathtt{b}_{+}(\res(\beta_{\iota}))^{\nu(\delta)}\mathtt{b}_{+}(\res(\alpha_k))^{-\nu(\delta_1)}}\in\mathbb{K}
\end{align*}
for $(k,\delta_1)\in [N]\times \mathbb{Z}_2.$
By \eqref{type=s.(4).nondegLinearEquations}, we have
\begin{align*}
A_{(\mt_1,0,\delta_{\beta_\mt}(n)e_n)}=1, \text{ and } A_{(\mt_k,0,\delta_1 e_n)}=0
\end{align*}
for any $(k,\delta_1)\in\left( [N]\times\mathbb{Z}_2 \right)\setminus \{(1,\delta_{\beta_\mt}(n))\}.$
Using Lemma \ref{tinylemma1}, we deduce that
\begin{align*}
\varepsilon_n\left(f_{(\mt_k,0,\delta_1 e_n),(\mt_k,e_i,(\delta_1 +\bar{1})e_n)}\right)=0
\end{align*}
for any $(k,\delta_1)\in [N]\times\mathbb{Z}_2.$ Recall that ${\rm T}=(\mt_1,0,\delta_{\beta_\mt}(n)e_n).$
It follows from the uniqueness of Mackey decomposition \eqref{type=s.Madecomequation of fTT0} of $h$ that
\begin{align*}
\varepsilon_n(F_{\rm T})=\varepsilon_n(h)=a_0F_{{\rm T}\downarrow_{n-1}}.
\end{align*}
Following the same computations as in \eqref{a0.case(2.a)}, \eqref{a0.case(2.b)}, we have
\begin{align*}
a_0=\frac{\mathtt{b}_{\mt,n}^{\nu_{\beta_{\mt}}(n)m}}{1+\mathtt{b}_{\mt,n}^{-\nu_{\beta_{\mt}}(n)}}
    \frac{\prod\limits_{\beta\in\Rem(\mt\downarrow_{n-1})\setminus\mathcal{D}}\left(\mathtt{q}(\res_{\mt}(n))-\mathtt{q}(\res(\beta))\right)}{
    \prod\limits_{\alpha\in\Add(\mt\downarrow_{n-1})\setminus \{\mt^{-1}(n)\}}\left(\mathtt{q}(\res_{\mt}(n))-\mathtt{q}(\res(\alpha))\right)}.
\end{align*}
Hence we obtain the formula \eqref{type=s, case(4) ,veFT}.
\end{proof}

\medskip
{\bf Proof of Theorem \ref{mainthm types nondege}}: By Corollary \ref{reductionrem}, the expressions of $\varepsilon_n(F_{\rm T})$ in Propositions \ref{type=s, case(1) ,mainprop}, \ref{type=s, case(3) ,mainprop}, \ref{type=s, case(2) ,mainprop},  \ref{type=s, case(4) ,mainprop}  hold for general ${\rm T}=(\mt,\alpha_{\mt},\beta_{\mt})\in{\rm Tri}_{\bar{0}}(\undla).$ Since $\tau_{r,n}(F_{\rm T})=\varepsilon_1\circ\cdots\circ\varepsilon_n(F_{\rm T}),$ we deduce $\tau_{r,n}(F_{\rm T})
=	\eta_{\mt,\beta_{\mt}}\cdot \kappa_{\undla}\in\mathbb{K}^*$ by an induction on $n$. On the other hand, one can easily check \begin{align*}
	\eta_{\mt,\beta_{\mt}}
	=\prod_{k=1}^{n} \eta_{\mt\downarrow_k,\beta_{\mt}\downarrow_k}(k)
	=2^{\lceil \sharp\mathcal{D}_{\undla}/2 \rceil}\prod\limits_{k=1}^{n}\frac{\mathtt{b}_{\mt,k}^{\nu_{\beta_{\mt}}(k)m}}{1+\mathtt{b}_{\mt,k}^{-\nu_{\beta_{\mt}}(k)}}
	\in\mathbb{K}^*
\end{align*} from Definition \ref{eta. type=s}. This completes the proof of Theorem \ref{mainthm types nondege}.
\qed
\medskip

\subsection{The symmetrizing form and Schur elements}

In this subsection, we shall apply Theorem \ref{mainthm types nondege} to give the proof of Theorem \ref{Nondengerate} (ii) and the formula for Schur elements in Corollary~\ref{schur-odd-f}. {\bf We emphasise that, unless otherwise stated, it is not necessary to assume that the separation condition holds in this subsection.}
\label{pag:sym. form. non-dege}
\begin{cor}\label{sym. form. non-dege}
	Let ${\rm R}$ be an integral domain of characteristic different from 2. Suppose $q\in{\rm R^\times}\setminus\{\pm 1\},$ $q+ q^{-1}\in{\rm R^\times}$ and $\undQ\in ({\rm R^\times})^m.$ Let $f$ be as in \eqref{cyc-odd} and we define the cyclotomic Hecke-Clifford algebra $\mHfcn$ over $\rm R$. Then the map
$$t_{r,n}:=\tau_{r,n}\Bigl(-\cdot (X_1X_2\cdots X_n)^m (1+X_1)(1+X_2)\cdots (1+X_n)\Bigr)$$
from $\mHfcn$ to ${\rm R}$ satisfies that $t_{r,n}(xy)=t_{r,n}(yx)$ for any $x,y\in\mHfcn.$ Moreover, if $(1+X_1)(1+X_2)\cdots (1+X_n)$ is invertible in $\mHfcn$, then $\mHfcn$ is a symmetric algebra with symmetrizing form $t_{r,n}$.
\end{cor}
\begin{proof}
The proof is similar to Corollary \ref{supersym. form. non-dege}. We only need to check the property $t_{r,n}(xy)=t_{r,n}(yx)$ for any $x,y\in\mHfcn$ in the semisimple case by applying Theorem \ref{mainthm types nondege}.
\end{proof}

\medskip
{\bf Proof of Theorem \ref{Nondengerate} (ii)}:
This follows from Corollary \ref{sym. form. non-dege}.
\qed
\medskip

Now suppose ${\rm R}=\mathbb{K}$ is a field. Let $f$ be as in \eqref{cyc-odd} and $\mHfcn$ be the cyclotomic Hecke-Clifford algebra defined on $\mathbb{K},$ which is not necessarily semisimple. We want to give another description for the invertibility condition in Theorem \ref{Nondengerate} (ii).
Recall $Q_0:=1.$ We define
\begin{align}\label{Gamma poly}
\Gamma_n(q^2,\underline{Q})
:=\prod_{i=0}^{m}\prod_{k=1-n}^{n}
\left(q^{2k}Q_i+1\right)\in\mathbb{K}.
\end{align}
\begin{prop}\label{Gamma condition}
Suppose $\Gamma_n(q^2,\underline{Q})\neq 0$ in $\mathbb{K},$ then $\mHfcn$ is a symmetric algebra with symmetrizing form $t_{r,n}.$
\end{prop}
\begin{proof}
We view $\mHfcn$ as the left regular $\mHfcn$-module. By \cite[the proof of Lemma 4.4]{BK} and \eqref{invertible2}, for $1\leq j\leq n-1,$ the eigenvalue of $X_{j+1}$ must be of the form $\mathtt{b}_{\pm}(u_{j+1}),$ where
$$u_{j+1}\in \left\{q^2u_j,u_j,q^{-2}u_j,u_j^{-1},q^{-2}u_j^{-1},q^{-4}u_j^{-1} \biggm| \begin{matrix}&\text{$\mathtt{b}_{\pm}(u_j)$ is an eigenvalue of $X_j$,} \\
		&\text{ $u_j\in \mathbb{K}^*$}\end{matrix}\right\}.$$

Note that the eigenvalue of $X_1$ is necessarily of the form $\mathtt{b}_{\pm}(u_1),$ where $u_1\in \{Q_i,q^{-2}Q_i^{-1}\mid 0\leq i\leq m\}.$
By induction on $1\leq j\leq n,$ it's easy to check that the eigenvalue of $X_j$ is necessarily of the form $\mathtt{b}_{\pm}(u_j),$ where
$$u_j\in\{q^{2k}Q_i,q^{2(k-1)}Q_i^{-1}\mid 0\leq i\leq m,1-j\leq k\leq j-1\}.$$
Since $\mathtt{b}_{\pm}(u)=-1$ if and only if $u\in\{-1,-q^{-2}\},$ the condition $\Gamma_n(q^2,\underline{Q})\neq 0$ implies that the eigenvalues of $1+X_j$ ($1\leq j\leq n$) are non-zero and hence $(1+X_1)\cdots(1+X_n)$ is invertible. The Proposition now follows from Theorem \ref{Nondengerate} (ii).
\end{proof}

\label{pag:Schur elements'.non-dege}
\begin{lem}\label{Schur elements and tauF'}
	Suppose the separation condition $P_n^{f}(q^2,\undQ)\neq 0$ holds for $\mHfcn,$ where $\undQ=(Q_1,Q_2,\ldots,Q_m)\in({\mathbb{K}}^*)^m$ and $f$ is as in \eqref{cyc-odd}. Let $\undla\in\mathscr{P}^{f}_{n}.$
	Then the Schur element $s_{\undla}$ of the simple module $\mathbb{D}(\undla)$ with respect to $t_{r,n}$ is equal to $1/t_{r,n}(F_{\rm T})$ for any ${\rm T}=(\mt,\alpha_{\mt},\beta_{\mt})\in{\rm Tri}_{\bar0}(\undla).$
\end{lem}
\begin{proof}
	It is completely similar to Lemma \ref{Schur elements and tF} (1).
\end{proof}

\medskip
\begin{cor}\label{schur-odd-f}
Assume $f$ is as in \eqref{cyc-odd} and $P_n^{f}(q^2,\undQ)\ne0$. For $\undla\in\mathscr P_n^{f}$, the Schur element of the fixed simple module relative to $t_{r,n}$ is
\[
s_{\undla}=2^{-\lceil\sharp\mathcal D_{\undla}/2\rceil}\kappa_{\undla}^{-1}.
\]
\end{cor}
\begin{proof}
This follows from Lemma~\ref{Schur elements and tauF'} and Theorem~\ref{mainthm types nondege}.
\end{proof}
\medskip

\subsection{Schur elements of the odd-level cyclotomic Sergeev algebra}
In this subsection, we consider the cyclotomic Sergeev algebra $\mhgcn$, where $g$ is as in \eqref{deg-cyc-odd} and $\undQ=(Q_1,Q_2,\ldots,Q_m)\in\mathbb{K}^m$ . {\bf We assume the degenerate separation condition $P_n^{g}(\undQ)\neq 0$ holds.} Following the argument in the non-degenerate case above, we can obtain the parallel result for $\mhgcn$ with some suitable modifications.
For $g$ as in \eqref{deg-cyc-odd}, i.e., the level $r=2m+1$ is odd, the scalar $\eta_{\mt,\beta_{\mt}}(n)$ and $\eta_{\mt,\beta_{\mt}}$ appeared in Definition \ref{eta. type=s} should be replaced by
\begin{align*}
\eta_{\mt,\beta_{\mt}}'(n)
=\begin{cases}
  1,           &\text{if $d_{\undla}=1,$ $d_{\undmu}=0$,}\\
 \frac{1}{2}  &\text{otherwise,}\\
  \end{cases}
\end{align*}
and
\begin{align*}
\eta_{\mt,\beta_{\mt}}'
:=\prod_{k=1}^{n} \eta_{\mt\downarrow_k,\beta_{\mt}\downarrow_k}'(k)=2^{\lceil \sharp\mathcal{D}_{\undla}/2 \rceil -n}.
\end{align*}

Then we have the following.
\begin{thm}\label{dege. type=s. tFT}
Suppose $g$ is as in \eqref{deg-cyc-odd} and $\undla\in\mathscr{P}^{g}_{n}.$ For any ${\rm T}=(\mt,\alpha_{\mt},\beta_{\mt})\in{\rm Tri}_{\bar{0}}(\undla),$ we have
\begin{align*}
\mathtt{t}_{r,n}(\mathcal{F}_{\rm T})
=2^{\lceil \sharp\mathcal{D}_{\undla}/2 \rceil -n}
  \prod\limits_{k=1}^{n}\frac{\prod\limits_{\beta\in\Rem(\mt\downarrow_{k-1})\setminus\mathcal{D}}\left(\mathtt{q}(\res_{\mt}(k))-\mathtt{q}(\res(\beta))\right)}{
    \prod\limits_{\alpha\in\Add(\mt\downarrow_{k-1})\setminus \{\mt^{-1}(k)\}}\left(\mathtt{q}(\res_{\mt}(k))-\mathtt{q}(\res(\alpha))\right)}.
\end{align*}
\end{thm}

\medskip
\begin{cor}\label{schur-odd-g}
Assume $g$ is as in \eqref{deg-cyc-odd} and $P_n^{g}(\undQ)\ne0$. For $\undla\in\mathscr P_n^{g}$, the Schur element of the fixed simple module relative to $\mathtt t_{r,n}$ is
\[
\mathtt s_{\undla}=2^{n-\lceil\sharp\mathcal D_{\undla}/2\rceil}\bigl(\kappa_{\undla}^{\mathrm{dege}}\bigr)^{-1}.
\]
\end{cor}
\begin{proof}
This follows from Lemma~\ref{Schur elements and tF} (1) and Theorem~\ref{dege. type=s. tFT}.
\end{proof}
\medskip

\begin{rem}
As another application, we can now give an alternative proof of Theorem \ref{LS1:Theorem 1.2} (i) using the same argument as in
Corollary \ref{supersym. form. non-dege}.
\end{rem}

\section{Applications}\label{application}

\subsection{Hecke-Clifford algebra and Sergeev algebra}\label{application:level-one} In this subsection, we apply our main result to the symmetrizing form on the Hecke-Clifford algebra $\mathcal{H}(n)$. We also compare the symmetrizing forms and Schur elements of $\mathcal{H}(n)$ with those of the Sergeev algebra $\mathfrak{H}(n)=\mathcal{C}_n\rtimes\mathbb C\mathfrak S_n$ after specializing $q$ to $1$.

The Hecke-Clifford algebra $\mathcal{H}(n)$ of \cite{O} is defined as $\mHfcn$, where $f=X_1-1$. Following \cite{JN}, $\mathcal{H}(n)$ can be viewed as a subalgebra of $\mHcn$ generated by $T_1,\cdots, T_{n-1}, C_1,\cdots,C_n$.

\medskip
\begin{prop}\label{Hecke-Cliffod sym}
	Suppose that $q$ is not a primitive $4l$-th root of unity for $l\in [n],$ then $\mathcal{H}(n)$ is a symmetric superalgebra with the symmetrizing form $t_{1,n}$.
\end{prop}
\begin{proof}
This follows from Proposition \ref{Gamma condition}.
\end{proof}
\medskip
	
The Sergeev algebra $\mathfrak{H}(n)$ is defined as $\mhgcn$,  for $g=x_1$. Following \cite{Na2}, $\mathfrak{H}(n)$ can be viewed as a subalgebra of $\mhcn$ generated by $s_1,\cdots, s_{n-1}, c_1,\cdots,c_n$.
	
{\bf In the rest of this subsection, $\mathcal{H}(n)$ is the Hecke-Clifford algebra defined over the algebraic closure of $\mathbb{C}(q)$. $\mathfrak{H}(n)$ is the Sergeev algebra defined over $\mathbb{C}$.} Both $\mathcal{H}(n)$ and $\mathfrak{H}(n)$ are semisimple, by \cite{JN,Na2} or checking the separation conditions directly. By Proposition \ref{Hecke-Cliffod sym} and Theorem \ref{LS1:Theorem 1.2}, $\mathcal{H}(n)$ and $\mathfrak{H}(n)$ are both symmetric superalgebras with symmetrizing forms $t_{1,n}$ and $\mathtt{t}_{1,n}$. For $\undla=(\lambda^{(0)})\in\mathscr{P}^{f}_{n}=\mathscr{SP}_{n}$, we have used $s_{\undla}$ and $\mathtt{s}_{\undla}$ to distinguish the Schur elements of $\mathcal{H}(n)$ and $\mathfrak{H}(n)$ with respect to $t_{1,n}$ and $\mathtt{t}_{1,n}$ respectively. We want to compare $t_{1,n}$ with $\mathtt{t}_{1,n}$ and $s_{\undla}$ with $\mathtt{s}_{\undla}$ after specializing $q$ to $1$.

We use $\mathcal{H}(n)|_{q=1}$ to denote the algebra obtained by specializing $q$ to $1$. Then $\mathcal{H}(n)|_{q=1}$ can be viewed as a $\mathbb{C}$-algebra and we have \begin{align}\label{specializing}
\mathcal{H}(n)|_{q=1}\cong \mathfrak{H}(n),\qquad T_i\mapsto s_i,\,1\leq i\leq n-1;\,C_j\mapsto c_j,\,1\leq j\leq n.
\end{align}
\begin{prop}\label{application:q1comparison}
	We have $t_{1,n}|_{q=1}=2^n\mathtt{t}_{1,n}$ and $s_{\undla}|_{q=1}=2^{-n}\mathtt{s}_{\undla}$ for $\undla=(\lambda^{(0)})\in\mathscr{SP}_{n}$.
	\end{prop}
	\begin{proof}
One can easily check under the above isomorphism \eqref{specializing}, $X_i\mapsto 1,\,1\leq i\leq n$, by the relation \eqref{PX1} or \eqref{PX2}. Hence the first equation follows from Proposition \ref{closed formula frob}, Proposition \ref{Hecke-Cliffod sym} and \cite[Theorem 3.10 ]{LS1}. It's easy to check the following\begin{align}\label{limit q to 1}
	\lim_{q\rightarrow 1}  \frac{q^{2a+1}+q^{-2a-1}-q^{2b+1}-q^{-2b-1}}{q^{2c+1}+q^{-2c-1}-q^{2d+1}-q^{-2d-1}}=\frac{a(a+1)-b(b+1)}{c(c+1)-d(d+1)}.
	\end{align}
By Lemma \ref{HLL. lem2.13. type=s}, for $k\in[n]$, we have
\begin{align}\label{equality}\sharp\left( \Rem(\mt\downarrow_{k-1})\setminus\mathcal{D}\right)=\sharp\left(\Add(\mt\downarrow_{k-1})\setminus \{\mt^{-1}(k)\}\right).
	\end{align}
 Now the second equation follows from Corollaries~\ref{schur-odd-f} and~\ref{schur-odd-g}, \eqref{limit q to 1} and \eqref{equality}.
\end{proof}

\subsection{Cyclotomic quiver Hecke algebras of types $A^{(1)}$ and $C^{(1)}$}\label{application:quiver} In this subsection, we use the supersymmetrizing form of Corollary~\ref{supersym. form. non-dege} and its modification in Corollary~\ref{sym. form. even. non-dege} to obtain symmetrizing forms on the cyclotomic quiver Hecke algebras of types $A^{(1)}_{e-1}$ and $C^{(1)}_e$, $e>1$.

In type $A^{(1)}$, the isomorphism with cyclotomic Hecke algebras in \cite{BK:GradedKL} also gives symmetrizing forms on the corresponding cyclotomic Hecke algebra.

Following \cite{KKT}, to an affine quiver of type $A^{(1)}_{e-1}$ or $C^{(1)}_e$ and a suitable dominant weight $\Lambda$, we associate two algebras ${\rm R}^\Lambda_n$ and ${\rm RC}^\Lambda_n$: the cyclotomic quiver Hecke algebra and the cyclotomic quiver Hecke-Clifford algebra, respectively. These are $\Z$-graded algebras and play important roles in the categorification of irreducible highest weight modules of quantum groups \cite{KKO1,KKO2}.

Kang, Kashiwara and Tsuchioka prove that there is an explicit idempotent $e^\dag\in{\rm RC}^\Lambda_n$ such that
	${\rm RC}^\Lambda_n$ is Morita superequivalent to $e^\dag{\rm RC}^\Lambda_ne^\dag$ (\cite[Below Definition 3.10]{KKT}) and $e^\dag{\rm RC}^\Lambda_ne^\dag\cong   {\rm R}^\Lambda_n$ (\cite[Theorem 3.13]{KKT}). Moreover, when $p\nmid e$, where $p$ is the characteristic of $\mathbb{K}$, they associated $\Lambda$ with a certain polynomial $f_\Lambda$ (\cite[Subsection 4.7]{KKT}) which is equivalent to some $f$ as in \eqref{cyc-even}, (i.e., $f_\Lambda=X_1^l f$, for some $l\geq 0$) where $\undQ=(Q_1,Q_2,\ldots,Q_m)\in(\mathbb{K}^*)^m$ and showed that $\mHfcn\cong {\rm RC}^\Lambda_n$  (\cite[Corollary 4.8]{KKT}). In conclusion, we have the Morita superequivalence
\begin{equation}\label{Superequivalent}
		\mHfcn \overset{\text{sMor}}{\sim}{\rm R}^\Lambda_n,
\end{equation}
where ${\rm R}^\Lambda_n$ is the cyclotomic quiver Hecke algebra of type $A^{(1)}_{e-1}$ or $C^{(1)}_e$ and $\mHfcn$ is the cyclotomic Hecke-Clifford algebra, where $f$ is as in \eqref{cyc-even} for some $\undQ=(Q_1,Q_2,\ldots,Q_m)\in(\mathbb{K}^*)^m$.

We shall use the following idempotent-truncation observation. Its proof is an easy exercise.
\begin{lem}\label{LS1:lem:even-idempotent-truncation}
Let $\mathcal{A}$ be a finite-dimensional superalgebra with a (super)symmetrizing form $t$,
and let $e\in \mathcal{A}$ be an even idempotent.
Then $t|_{e\mathcal{A}e}$ is also a (super)symmetrizing form on $e\mathcal{A}e$ with the same parity of $t$.
\end{lem}

We transport the forms through the isomorphism $\mHfcn\cong{\rm RC}^\Lambda_n$ above and retain the notation $X_i$ for the images of the affine generators.
\begin{prop}\label{new trace form}
The restrictions of
\[
t_{r,n}
\qquad\text{and}\qquad
t_{r,n}\Bigl(-\cdot\prod_{i=1}^{n}(X_i-X_i^{-1})\Bigr)
\]
to $e^\dag{\rm RC}^\Lambda_ne^\dag\cong{\rm R}^\Lambda_n$ are symmetrizing forms.
\end{prop}
\begin{proof}
By Corollary~\ref{supersym. form. non-dege} and Lemma~\ref{LS1:lem:even-idempotent-truncation}, the restriction of $t_{r,n}$ is supersymmetrizing. Since ${\rm R}^\Lambda_n$ has no odd part, this restriction is symmetrizing.

By \cite[Lemma 4.7(ii)]{KKT}, in the cases considered here, the eigenvalues of each $X_i$ lie in $\mathbb K^*\setminus\{\pm1\}$. Hence $X_i-X_i^{-1}$ is invertible for every $1\leq i\leq n$. The assertion for the second form follows from Corollary~\ref{sym. form. even. non-dege} and Lemma~\ref{LS1:lem:even-idempotent-truncation}.
\end{proof}

\begingroup

\subsection{Cyclotomic quiver Hecke superalgebras of twisted affine type $A$}
\label{LS1:subsec:twisted-affine-A}

Following \cite[Section 2.1d]{KL:RoCK}, let $e\geq1$ and let $\mathfrak g$
be the affine Kac--Moody algebra of type $A_{2e}^{(2)}$, with vertex set
$I=\{0,1,\ldots,e\}$ and unique odd vertex $0$, i.e., $I_{\bar1}=\{0\}$.  Let $P$ be its weight
lattice, let $\Lambda\in P^+$, and set
\[
 d_i:=\langle h_i,\Lambda\rangle,
 \quad
 K:=h_0+2\sum_{i=1}^{e}h_i.
\]
The level of $\Lambda$ is
\begin{equation}\label{LS1:eq:twisted-A-level}
 r:=\langle K,\Lambda\rangle=d_0+2\sum_{i=1}^{e}d_i.
\end{equation}

Following \cite[Section 3.7]{KKT}, let ${\rm R}_n^\Lambda$ and
${\rm RC}_n^\Lambda$ be the cyclotomic quiver Hecke superalgebra and the
cyclotomic quiver Hecke-Clifford superalgebra of type $A_{2e}^{(2)}$.
They have block decompositions
\begin{equation}\label{LS1:eq:twisted-A-block-decompositions}
 {\rm R}_n^\Lambda=
 \bigoplus_{\substack{\alpha\in Q_+\\\operatorname{ht}(\alpha)=n}}
 {\rm R}_\alpha^\Lambda,
 \quad
 {\rm RC}_n^\Lambda=
 \bigoplus_{\substack{\alpha\in Q_+\\\operatorname{ht}(\alpha)=n}}
 {\rm RC}_\alpha^\Lambda.
\end{equation}
In the rest of this subsection we fix
$\alpha=\sum_{i\in I}m_i\alpha_i\in Q_+$ such that
${\rm R}_\alpha^\Lambda\ne0$, and put
\[
 n:=\operatorname{ht}(\alpha),
 \quad
 \ell_\alpha:=\sum_{i\in I_{\bar1}}m_i=m_0.
\]

Kang, Kashiwara and Tsuchioka construct an even idempotent $e^\dag$ such
that ${\rm RC}_\alpha^\Lambda$ is Morita superequivalent to
$e^\dag{\rm RC}_\alpha^\Lambda e^\dag$ and
\begin{equation}\label{LS1:eq:twisted-A-idempotent-truncation}
 e^\dag{\rm RC}_\alpha^\Lambda e^\dag
 \cong {\rm R}_\alpha^\Lambda\otimes\mathcal C_{\ell_\alpha};
\end{equation}
see \cite[below Definition 3.10, Theorem 3.13 and Section 3.7]{KKT}.
By \cite[Corollary 4.8 and after the proof of Lemma 5.5]{KKT}, their cyclotomic
isomorphisms take the following two forms.
\begin{enumerate}
 \item[(1)] Let $\mathbb K$ be algebraically closed, with either
 $\operatorname{char}\mathbb K=0$, or
 $\operatorname{char}\mathbb K=p>0$ and $p\nmid2(2e+1)$, and let
 $q\in\mathbb K^*$ be such that $q^2$ is a primitive $(2e+1)$-th root of
 unity.  Then \cite[Corollary 4.8]{KKT} gives
 \begin{equation}\label{KKTiso1}
 {\rm RC}_n^\Lambda\cong\mHfcn,
  \end{equation}
 where
 \begin{equation}\label{LS1:eq:twisted-A-cyclotomic-polynomial-one}
  f(X_1)=(X_1-1)^{d_0}\prod_{i=1}^{e}
  \Bigl(X_1\bigl(X_1+X_1^{-1}-\mathtt q(q^{2i})\bigr)\Bigr)^{d_i}.
 \end{equation}
 \item[(2)] Suppose that $p:=2e+1$ is prime and
 $\operatorname{char}\mathbb K=p$.  The cyclotomic statement following
 \cite[Theorem 5.4]{KKT} gives
 \begin{equation}\label{KKTiso2}{\rm RC}_n^\Lambda\cong\mhgcn,
 \end{equation}
 where
 \begin{equation}\label{LS1:eq:twisted-A-cyclotomic-polynomial-two}
  g(x_1)=x_1^{d_0}\prod_{i=1}^{e}
  \bigl(x_1^2-\mathtt q(i)\bigr)^{d_i}.
 \end{equation}
\end{enumerate}

We first recall a result relating the symmetricity of $\mathcal A$ to that of $\mathcal A\otimes\mathcal C_s$, which we use in the odd-level case.

\begin{lem}[Kleshchev {\cite[Lemma 3.24]{Kle:RoCKSchur}}]
\label{LS1:lem:ordinary-Clifford-removal}
Let $\mathcal{A}$ be a finite-dimensional superalgebra over $\mathbb K$ and let
$s\in\mathbb Z_{>0}$.  Then $\mathcal{A}$ is a symmetric superalgebra if and only if
$\mathcal{A}\otimes\mathcal C_s$ is a symmetric superalgebra.
\end{lem}

For the even-level case we must keep the supersymmetric identity and record the
parity of the resulting form. A linear map $t:\mathcal{A}\to\mathbb K$ that is homogeneous in the $\Z_2$-grading
has parity $\varepsilon\in\mathbb Z_2$ if
$t(\mathcal{A}_\delta)=0$ for $\varepsilon\ne\delta\in\mathbb Z_2$.  We call it a
supersymmetrizing form of parity $\varepsilon$ if the pairing
$(a,b)\mapsto t(ab)$ is non-degenerate and
\begin{equation}\label{eq:homogeneous-supertrace}
 t(ab)=(-1)^{|a||b|}t(ba)
\end{equation}
for homogeneous $a,b\in \mathcal{A}$.  When $\varepsilon=\bar0$, this agrees with
Definition \ref{(super)symmetrizing form}.  We use
$\mathcal C_s\cong\mathcal C_1^{\otimes s}$ with ordered odd generators
$c_1,\ldots,c_s$, and put $\mathcal C_0=\mathbb K$.

\begin{lem}
\label{LS1:lem:super-Clifford-removal}
Let $\mathcal{A}$ be a finite-dimensional superalgebra over $\mathbb K$.
If $\mathcal{A}\otimes\mathcal C_s$ has a supersymmetrizing form of parity
$\varepsilon$, then $\mathcal{A}$ has a supersymmetrizing form of parity
$\varepsilon+\bar s$.  More precisely, if $t$ is the form on
$\mathcal{A}\otimes\mathcal C_s$, the required form is
\begin{equation}\label{LS1:eq:top-Clifford-trace}
 t_\mathcal{A}(a):=t(a\otimes c_1\cdots c_s).
\end{equation}
\end{lem}

\begin{proof}
We first prove the case $s=1$.  For $\mathcal{B}=\mathcal{A}\otimes\mathcal C_1$, write $c$ for
the odd Clifford generator, and define $t_\mathcal{A}(a)=t(a\otimes c)$, $a\in\mathcal{A}$.  If
$a\in \mathcal{A}_\delta$, then $a\otimes c$ has parity $\delta+\bar1$; hence
$t_\mathcal{A}$ has parity $\varepsilon+\bar1$.

For homogeneous elements, the super tensor product multiplication is
\begin{equation}\label{eq:Clifford-tensor-product}
 (a\otimes u)(b\otimes v)=(-1)^{|u||b|}ab\otimes uv.
\end{equation}
For homogeneous $a,b\in\mathcal A$, we therefore have
\begin{align*}
 t_\mathcal{A}(ab)
 &=t\bigl((a\otimes1)(b\otimes c)\bigr)\\
 &=(-1)^{|a|(|b|+1)}t\bigl((b\otimes c)(a\otimes1)\bigr)\\
 &=(-1)^{|a|(|b|+1)}(-1)^{|a|}t(ba\otimes c)\\
 &=(-1)^{|a||b|}t_\mathcal{A}(ba).
\end{align*}
The first equality uses \eqref{LS1:eq:top-Clifford-trace}; the second uses \eqref{eq:homogeneous-supertrace}; the third uses \eqref{eq:Clifford-tensor-product}; and the last simplifies the sign and applies \eqref{LS1:eq:top-Clifford-trace}. Thus $t_\mathcal A$ is supersymmetric.

To prove non-degeneracy, first let $a\in\mathcal A$ be homogeneous. Since $c^2=1$, we have
\begin{align*}
 t(a\otimes1)
 &=t((a\otimes c)(1\otimes c))\\
 &=(-1)^{|a|+1}t((1\otimes c)(a\otimes c))\\
 &=(-1)^{|a|+1}(-1)^{|a|}t(a\otimes1)\\
 &=-t(a\otimes1).
\end{align*}
Here the first equality uses \eqref{eq:Clifford-tensor-product} and $c^2=1$; the second uses \eqref{eq:homogeneous-supertrace}; the third again uses \eqref{eq:Clifford-tensor-product} and $c^2=1$; and the last simplifies the sign. Since $2$ is invertible, linearity gives
\begin{equation}\label{eq:Clifford-trace-vanishing}
 t(\mathcal A\otimes1)=0.
\end{equation}
If $a\in\mathcal A$ satisfies $t_\mathcal A(ab)=0$ for all $b\in\mathcal A$, then \eqref{eq:Clifford-trace-vanishing} and \eqref{LS1:eq:top-Clifford-trace} give, respectively,
\[
 t((a\otimes1)(b\otimes1))=0,\qquad
 t((a\otimes1)(b\otimes c))=0
 \quad\text{for all }b\in\mathcal A.
\]
Since $\mathcal B=(\mathcal A\otimes1)\oplus(\mathcal A\otimes c)$ and the pairing defined by $t$ is non-degenerate, $a\otimes1=0$, hence $a=0$. The finite-dimensional pairing defined by $t_\mathcal A$ is therefore non-degenerate.

For $s>1$, apply the rank-one case successively to the ordered tensor factors of $\mathcal C_s\cong\mathcal C_1^{\otimes s}$. Each step shifts the parity by $\bar1$, and the resulting form is \eqref{LS1:eq:top-Clifford-trace}. The case $s=0$ is immediate from $\mathcal C_0=\mathbb K$.
\end{proof}

The following proof follows the strategy of
\cite[Lemma 9.22]{Kle:RoCKSchur}, with the odd- and even-level cases kept
separate and Lemma \ref{LS1:lem:super-Clifford-removal} replacing ordinary
Clifford removal in the latter case.

\begin{prop}\label{LS1:prop:twisted-affine-A}
Let $\mathbb K$ be algebraically closed and suppose that one of the
following conditions holds:
\begin{enumerate}
 \item[(1)] either $\operatorname{char}\mathbb K=0$, or
 $\operatorname{char}\mathbb K=p>0$ and $p\nmid2(2e+1)$; moreover,
 $q\in\mathbb K^*$ is chosen so that $q^2$ is a primitive
 $(2e+1)$-th root of unity;
 \item[(2)] $p:=2e+1$ is prime and $\operatorname{char}\mathbb K=p$.
\end{enumerate}
Let $\Lambda\in P^+$ have level $r=\langle K,\Lambda\rangle$, let
$\alpha\in Q_+$, and let
${\rm R}_\alpha^\Lambda$ be a non-zero cyclotomic quiver Hecke
superalgebra of twisted affine type $A_{2e}^{(2)}$ as above.  Then the
following statements hold.
\begin{enumerate}
 \item[(i)] If the level $r$ is odd, then ${\rm R}_\alpha^\Lambda$ is a
 symmetric superalgebra.
 \item[(ii)] If the level $r$ is even, then ${\rm R}_\alpha^\Lambda$ has a
 supersymmetrizing form of parity $\overline{\ell_\alpha}\in\mathbb Z_2$.
\end{enumerate}
\end{prop}

\begin{proof}
Write $d_0=2a+\delta$, where $\delta\in\{0,1\}$.  By
\eqref{LS1:eq:twisted-A-level}, $r$ is odd if and only if $\delta=1$.

We first assume (1). Suppose that $\delta=1$, and let $\underline Q$ consist of $a$ copies of
$1$ and $d_i$ copies of $q^{2i}$ for $1\leq i\leq e$.  By
\eqref{LS1:eq:twisted-A-cyclotomic-polynomial-one},
the ordinary polynomial $f(X_1)$ is $X_1^{a+\sum_{i=1}^e d_i}$ times the Laurent polynomial in \eqref{cyc-odd} with these parameters. They define the same cyclotomic quotient.
Every parameter, including $Q_0=1$, belongs to the odd-order group
$\langle q^2\rangle$.  Hence $q^{2k}Q+1\ne0$ for every relevant $Q$ and
$k$, so $\Gamma_n(q^2,\underline Q)\ne0$.  By Proposition \ref{Gamma condition}, $\mHfcn$ is symmetric.
The isomorphism \eqref{KKTiso1} 
identifies ${\rm RC}_n^\Lambda$ with $\mHfcn$, so ${\rm RC}_n^\Lambda$ is symmetric.
Applying Lemma \ref{LS1:lem:even-idempotent-truncation} first to the central idempotent
of the $\alpha$-block and then to the even idempotent $e^\dag$, we obtain
a symmetrizing form on $e^\dag {\rm RC}_\alpha^\Lambda e^\dag$.
By \eqref{LS1:eq:twisted-A-idempotent-truncation},
${\rm R}_\alpha^\Lambda\otimes\mathcal C_{\ell_\alpha}$ is also symmetric.
If $\ell_\alpha>0$, Lemma \ref{LS1:lem:ordinary-Clifford-removal} removes the Clifford factor; if $\ell_\alpha=0$, there is no Clifford factor to remove. In either case, ${\rm R}_\alpha^\Lambda$ is symmetric.

If $\delta=0$, the same choice of $\underline Q$ makes $f(X_1)$ equal to $X_1^{a+\sum_{i=1}^e d_i}$ times the Laurent polynomial in \eqref{cyc-even}.  The same argument, using Theorem
\ref{Nondengerate}(i), the isomorphism
\eqref{KKTiso1}, Lemma \ref{LS1:lem:even-idempotent-truncation},
\eqref{LS1:eq:twisted-A-idempotent-truncation},
and Lemma
\ref{LS1:lem:super-Clifford-removal}, gives a supersymmetrizing form on
${\rm R}_\alpha^\Lambda$ of parity
$\bar0+\overline{\ell_\alpha}=\overline{\ell_\alpha}$.

The proof in the case (2) is similar.
Replace the isomorphism
\eqref{KKTiso1} by the isomorphism
\eqref{KKTiso2}, and replace
the Laurent polynomials in \eqref{cyc-odd} and \eqref{cyc-even} by the polynomials in \eqref{deg-cyc-odd} and \eqref{deg-cyc-even}, respectively, where $\underline Q$
consists of $a$ copies of $0$ and $d_i$ copies of $i$ for
$1\leq i\leq e$; this is exactly the polynomial in
\eqref{LS1:eq:twisted-A-cyclotomic-polynomial-two}.  Theorem
\ref{LS1:Theorem 1.2}(i) replaces Theorem \ref{Nondengerate}(ii) when
$\delta=1$, and Theorem \ref{LS1:Theorem 1.2}(ii) replaces Theorem
\ref{Nondengerate}(i) when $\delta=0$.  The two applications of Lemma
\ref{LS1:lem:even-idempotent-truncation}, \eqref{LS1:eq:twisted-A-idempotent-truncation}, and
Lemmas \ref{LS1:lem:ordinary-Clifford-removal} and
\ref{LS1:lem:super-Clifford-removal}, respectively, give results (i) and (ii).
\end{proof}

Proposition \ref{LS1:prop:twisted-affine-A} partially confirms the
Kleshchev--Livesey symmetricity conjecture
\cite[Remark 4.2.58]{KL:RoCK} for these cyclotomic quiver Hecke superalgebras of twisted affine type $A$ at odd
level. At even level, it gives a supersymmetrizing form of $\Z_2$-parity
$\overline{\ell_\alpha}$. The next section gives a cyclotomic quiver Hecke
superalgebra that admits no non-degenerate ordinary
trace.

\section{A counterexample to symmetricity}
\label{LS1:sec:counterexample}

Let $\Bbbk$ be a field with $\operatorname{char}\Bbbk\ne2$.  Consider the
one-vertex Cartan superdatum with unique vertex $i$ odd, and take
$\Lambda=4\Lambda_i$ and $\theta=2\alpha_i$.  The associated cyclotomic
quiver Hecke superalgebra is the two-strand level-four cyclotomic odd
nilHecke algebra
\[
 \mathcal{A}:=R^{4\Lambda_i}(2\alpha_i)=\operatorname{ONH}_2^4.
\]
Thus this example lies inside the class considered in the conjecture of
\cite[Remark 4.2.58]{KL:RoCK}.

Write $x=x_1$, $y=x_2$, and $\partial=\partial_1$.  All three generators are
odd.  Since the only idempotent is $e(i,i)=1$, the defining relations reduce
to
\begin{align}
 xy&=-yx, & \partial^2&=0, \label{LS1:eq:ONH-basic1}\\
 x\partial+\partial y&=1, & \partial x+y\partial&=1, \label{LS1:eq:ONH-basic2}
\end{align}
together with the cyclotomic relation
\begin{equation}\label{LS1:eq:ONH-x4}
 x^4=0.
\end{equation}
These are the two-strand specializations of the odd nilHecke relations in
\cite{EKL}.  By \cite[Corollary 1.7]{HuShi:super},
\begin{equation}\label{LS1:eq:ONH-basis}
 \left\{x^ay^b\partial^\varepsilon
 \mathrel{\big|}0\leq a\leq3,\ 0\leq b\leq2,\
 \varepsilon\in\{0,1\}\right\}
\end{equation}
is a $\Bbbk$-basis of $\mathcal{A}$.

We call a linear functional $t:\mathcal{A}\to\Bbbk$ an ordinary trace if
$t(ab)=t(ba)$ for all $a,b\in \mathcal{A}$.  Such a functional vanishes on the ordinary
commutator space
\[
 [\mathcal{A},\mathcal{A}]:=\operatorname{span}_{\Bbbk}\{ab-ba\mid a,b\in \mathcal{A}\}.
\]
We shall construct a non-zero right ideal contained in $[\mathcal{A},\mathcal{A}]$.  This forces
the bilinear form $(a,b)\mapsto t(ab)$ of every ordinary trace to be
degenerate.

We begin with the cyclotomic reductions needed in the construction.  From
$\partial x=1-y\partial$, repeated right multiplication by $x$, using
$yx=-xy$, gives
\begin{align}
 \partial x^2&=x-y+y^2\partial, \label{LS1:eq:ONH-dx2}\\
 \partial x^3&=x^2+xy+y^2-y^3\partial, \label{LS1:eq:ONH-dx3}\\
 \partial x^4&=x^3-x^2y+xy^2-y^3+y^4\partial. \label{LS1:eq:ONH-dx4}
\end{align}
Since $x^4=0$, equation \eqref{LS1:eq:ONH-dx4} yields
\begin{equation}\label{LS1:eq:ONH-dx4zero}
 0=x^3-x^2y+xy^2-y^3+y^4\partial.
\end{equation}
Right multiplication by $\partial$ and the relation $\partial^2=0$ give
\begin{equation}\label{LS1:eq:ONH-y3d}
 y^3\partial=(x^3-x^2y+xy^2)\partial.
\end{equation}

We next reduce $y^4\partial$.  Left multiplication of
\eqref{LS1:eq:ONH-y3d} by $y$, with $yx=-xy$, gives
\begin{equation}\label{LS1:eq:ONH-y4d-step}
 y^4\partial=-x^3y\partial-x^2y^2\partial-xy^3\partial.
\end{equation}
Left multiplication of \eqref{LS1:eq:ONH-y3d} by $x$, followed by
$x^4=0$, gives
\begin{equation}\label{LS1:eq:ONH-xy3d}
 xy^3\partial=(-x^3y+x^2y^2)\partial.
\end{equation}
Substituting \eqref{LS1:eq:ONH-xy3d} into
\eqref{LS1:eq:ONH-y4d-step}, we obtain
\begin{equation}\label{LS1:eq:ONH-y4d}
 y^4\partial=-2x^2y^2\partial.
\end{equation}
Substitution in \eqref{LS1:eq:ONH-dx4zero} now yields
\begin{equation}\label{LS1:eq:ONH-y3}
 y^3=x^3-x^2y+xy^2-2x^2y^2\partial.
\end{equation}
Left multiplication of \eqref{LS1:eq:ONH-y3} by $x^3$, using $x^4=0$,
gives
\begin{equation}\label{LS1:eq:ONH-x3y3}
 x^3y^3=0.
\end{equation}
Finally, $\partial y=1-x\partial$ implies
\begin{equation}\label{LS1:eq:ONH-dy2}
 \partial y^2=(1-x\partial)y=y-x+x^2\partial.
\end{equation}

Define
\begin{equation}\label{LS1:eq:ONH-zw}
 z:=-x^3y^2\partial,
 \quad
 w:=-x^3y^2.
\end{equation}
Both elements are non-zero and linearly independent by the basis
\eqref{LS1:eq:ONH-basis}.

\begin{lem}\label{LS1:lem:ONH-right-ideal}
The subspace $I:=\Bbbk z\oplus\Bbbk w$ is a non-zero right ideal of $\mathcal A$.  Its
multiplication by the generators is
\[
\begin{array}{c|ccc}
 &x&y&\partial\\ \hline
 z&w&w&0\\
 w&0&0&z
\end{array}.
\]
\end{lem}

\begin{proof}
Using $\partial x=1-y\partial$ and
\eqref{LS1:eq:ONH-x3y3}, we have
\[
 zx=-x^3y^2(1-y\partial)=w+x^3y^3\partial=w.
\]
Using $\partial y=1-x\partial$, $y^2x=xy^2$, and $x^4=0$, we get
\[
 zy=-x^3y^2(1-x\partial)=w+x^3y^2x\partial=w.
\]
The equality $z\partial=0$ follows from $\partial^2=0$.  Next,
$wx=-x^3y^2x=0$ by $y^2x=xy^2$ and $x^4=0$, while
$wy=-x^3y^3=0$ by \eqref{LS1:eq:ONH-x3y3}; finally
$w\partial=z$ by definition.  Since $x,y,\partial$ generate $\mathcal A$, the displayed
table proves that $I$ is a right ideal.
\end{proof}

\begin{lem}\label{LS1:lem:ONH-commutator-ideal}
One has $I\subseteq[\mathcal A,\mathcal A]$.
\end{lem}

\begin{proof}
Set
\[
 u:=y^2\partial,
 \quad
 v:=x^3y\partial.
\]
We first compute the ordinary commutator $[u,v]=uv-vu$.  Substituting
\eqref{LS1:eq:ONH-dx3} into $uv=y^2\partial x^3y\partial$, and using
$\partial y\partial=\partial$, which follows from
$\partial y=1-x\partial$ and $\partial^2=0$, the last two terms cancel.  Thus
\begin{equation}\label{LS1:eq:ONH-uv-reduction}
 uv=x^2y^3\partial+xy^4\partial.
\end{equation}
Left multiplication of \eqref{LS1:eq:ONH-y3d} by $x^2$ gives
$x^2y^3\partial=x^3y^2\partial$, since every other term contains $x^4$.
Left multiplication of \eqref{LS1:eq:ONH-y4d} by $x$ gives
$xy^4\partial=-2x^3y^2\partial$.  Hence
\begin{equation}\label{LS1:eq:ONH-uv}
 uv=-x^3y^2\partial=z.
\end{equation}

For the reverse product, use \eqref{LS1:eq:ONH-dy2}:
\begin{align*}
 vu
 &=x^3y\partial y^2\partial
 =x^3y(y-x+x^2\partial)\partial\\
 &=x^3y^2\partial=-z.
\end{align*}
Here the term containing $\partial^2$ is zero, and the middle term is zero
because $yx=-xy$ and $x^4=0$.  Therefore
\[
 [u,v]=uv-vu=2z.
\]
Since $2$ is invertible in $\Bbbk$, this proves
$z=\frac12[u,v]\in[\mathcal A,\mathcal A]$.  By Lemma
\ref{LS1:lem:ONH-right-ideal}, $zx=w$, whereas
$xz=-x^4y^2\partial=0$.  Consequently
\[
 w=[z,x]\in[\mathcal A,\mathcal A].
\]
Thus both basis elements of $I$ lie in $[\mathcal A,\mathcal A]$.
\end{proof}

\begin{proof}[Proof of Proposition~\ref{LS1:prop:ONH-not-symmetric}]
Let $t:\mathcal A\to\Bbbk$ be an ordinary trace.  Then $t([\mathcal A,\mathcal A])=0$, so Lemma
\ref{LS1:lem:ONH-commutator-ideal} gives $t(I)=0$.  The element $z$ is
non-zero by \eqref{LS1:eq:ONH-basis}.  Since $I$ is a right ideal and
$z\in I$, for every $a\in\mathcal A$ we have $za\in I$, and hence
\[
 t(za)=0.
\]
Thus $z$ is a non-zero element in the left radical of the bilinear form
$(a,b)\mapsto t(ab)$.  This form is degenerate.  The conclusion applies to
every ordinary trace, and therefore in particular to every even ordinary
trace that could serve as a symmetrizing form.
\end{proof}

Proposition \ref{LS1:prop:ONH-not-symmetric} disproves the conjecture in
\cite[Remark 4.2.58]{KL:RoCK}.
It does not rule out a supersymmetrizing form.  Together with Proposition
\ref{LS1:prop:twisted-affine-A}, it shows that supersymmetricity is not merely
a change of terminology but the natural replacement for symmetricity in the
quiver Hecke superalgebra setting.

\endgroup

\medskip
\begin{symbols}
\medskip
\item[\textbf{Generalities}] \hfill
\symitem{${\rm R}$}{An integral domain of characteristic different from $2$}{pag:R}
\symitem{${\rm R^\times}$}{The unit group of ${\rm R}$}{pag:R}
\symitem{$\mathbb{K}$}{An algebraically closed field of characteristic different from $2$}{pag:K}
\symitem{$\mathbb{K}^*$}{The set $\mathbb{K}\setminus\{0\}$}{pag:K}
\symitem{$\N$}{The set of positive integers $\{1,2,\ldots\}$}{pag:N}
\symitem{{$[n]$}}{The set of positive integers $\{1,2,\ldots,n\}$ for $n\in\N$}{pag:[n]}
\symitem{$\mathfrak{S}_n$}{The symmetric group of $n$ letters}{pag:Sn}
\symitem{$|v|$}{The parity (or superdegree) of vector $v$ in some vector superspace}{pag:||}
\symitem{$\Pi V$}{The parity shift of supermodule $V$}{pag:parity shift}
\symitem{$s_V$}{The Schur element of simple module $V$}{pag:Schur element of V}
\symitem{${\rm supp}(\beta)$}{The supporting set $\{1 \leq k \leq n:\beta_{k}=\bar{1}\}$ for $\beta=(\beta_1,\ldots,\beta_n)\in\mathbb{Z}_2^n$}{pag:suppot and sum}
\symitem{$|\beta|$}{$\sum_{i=1}^{n}\beta_i$ for $\beta=(\beta_1,\ldots,\beta_n)\in\mathbb{Z}_2^n$}{pag:suppot and sum}
\symitem{$\mathcal{C}_n$}{The Clifford algebra}{pag:Clifford algebra}
\symitem{$\overrightarrow{\prod}$}{The ordered product}{pag:ordered product}
\symitem{$\lfloor x \rfloor$}{The greatest integer less than or equal to the real number $x$}{pag:round down}
\symitem{$\lceil x \rceil$}{The smallest integer greater than or equal to the real number $x$}{pag:round up}
\symitem{$\mathscr P_n^f$, $\mathscr P_n^g$}{The shape sets for the fixed cyclotomic presentations}{pag:The types of combinatorics}
\symitem{$\mathscr{SP}_n$}{The set of strict partitions of $n$}{pag:The types of combinatorics}
\symitem{$\undla$}{An element in $\mathscr{P}^{f}_{n}$}{pag:multipartition}
\symitem{$\alpha\in\undla$}{A box (or node) of $\undla$}{pag:multipartition}
\symitem{$\Std(\undla)$}{The set of standard tableaux of shape $\undla$}{pag:standard tableaux}
\symitem{$\mt$}{An element in $\Std(\undla)$}{pag:standard tableaux}
\symitem{$\mt^{\undla},$ $\mt_{\undla}$}{Initial row tableau of shape $\undla,$ Initial column tableau of shape $\undla$}{pag:standard tableaux}
\symitem{$\Add(\undla)$}{The set of addable boxes of $\undla$}{pag:Add and Rem}
\symitem{$\Rem(\undla)$}{The set of removable boxes of $\undla$}{pag:Add and Rem}
\symitem{$\mathcal{D}_{\undla}$}{The set of boxes in the first diagonals of strict partition components of $\undla$}{pag:diag of undlam}
\symitem{$\mathcal{D}_{\mt}$}{The set of numbers in the first diagonals of strict partition components of $\mt$}{pag:diag of undlam}
\symitem{$d(\ms,\mt)$}{The unique element in $\mathfrak{S}_n$ such that $\ms=d(\ms,\mt)\mt$ for $\ms,\mt\in\Std(\undla)$}{pag:adimssible}
\symitem{$\sharp S$}{The number of elements in the set $S$}{pag:cardinality}
\symitem{$d_{\undla}$}{It equals $1$ if $\sharp \mathcal{D}_{\undla}$ is odd, otherwise $0$}{pag:dundla}
\symitem{$\mathcal{OD}_{\mt}$}{A key subset of $\mathcal{D}_{\mt}$}{pag:Dt,ODt,Z2ODt}
\symitem{$\mathbb{Z}_2(\mathcal{OD}_{\mt})$}{The subset of $\mathbb{Z}_2^n$ supported on $\mathcal{OD}_{\mt}$}{pag:Dt,ODt,Z2ODt}
\symitem{{$\mathbb{Z}_2([n]\setminus\mathcal{D}_{\mt})$}}{The subset of $\mathbb{Z}_2^n$ supported on $[n]\setminus\mathcal{D}_{\mt}$}{pag:Dt,ODt,Z2ODt}
\symitem{{$\mathbb{Z}_2(\mathcal{OD}_{\mt})_{a}$}}{There is a certain decomposition $\mathbb{Z}_2(\mathcal{OD}_{\mt})=\sqcup_{a\in \mathbb{Z}_2}\mathbb{Z}_2(\mathcal{OD}_{\mt})_{a}$}{pag:decomposition of OD}
\symitem{$\gamma_{\mt}$}{A certain idempotent of $\mathcal{C}_n$ related to $\mt$}{pag:gamma t}
\symitem{$\nu_{\beta}(k)$}{It equals $-1$ if $\beta_{k}=\bar{1}$ and equals $1$ if $\beta_{k}=\bar{0},$ for $\beta\in\mathbb{Z}_2^n$}{pag:nubetak}
\symitem{$|\beta|_{<i}$}{$\Sigma_{k=1}^{i-1}\beta_k$ for $\beta\in\mathbb{Z}_2^n$}{pag:nubetak}
\symitem{$\delta_{\beta}(k)$}{It equals 1 if $\beta_k=\bar{1}$ and equals $0$ if $\beta_k=\bar{0},$ for $\beta\in\mathbb{Z}_2^n$}{pag:deltabetak}
\symitem{$\delta(s_i\mt)$}{It equals 1 if $s_i\mt\in\Std(\undla)$ for $\mt\in\Std(\undla),$ otherwise $0$}{pag:nondege coeffi cti}
\symitem{${\rm Tri}(\undla)$}{The set of triples associated with standard tableaux of shape $\undla$}{pag:Tri}
\symitem{${\rm Tri}_{a}(\undla)$}{There is a certain decomposition ${\rm Tri}(\undla)=\sqcup_{a\in \mathbb{Z}_2}{\rm Tri}_{a}(\undla)$}{pag:Tri}
\symitem{$\mathcal{D}$}{The set of all boxes in the first diagonals of strict partition components}{pag:diagonalnodes}
\\
\item[\textbf{Non-degenerate case}] \hfill
\symitem{$\mHcn$}{The affine Hecke-Clifford algebra}{pag:AHCA}
\symitem{$q$}{The Hecke parameter in ${\rm R^\times}$}{pag:AHCA}
\symitem{$\epsilon$}{$q-q^{-1}$}{pag:AHCA}
\symitem{$\mathcal{A}_n$}{A certain subalgebra of $\mHcn$}{pag:subalg An}
\symitem{$\tilde{\Phi}_i$}{The intertwining element}{pag:BK intertwining element}
\symitem{$\Phi_i$}{Jones--Nazarov intertwining element}{pag:JN intertwining element}
\symitem{$\Phi_i(x,y)$}{An element in $\mHcn$ related to $\Phi_i$}{pag:Phi function}
\symitem{$\mathtt{q}(\iota)$}{$2(q\iota+(q\iota)^{-1})/(q+q^{-1})$ for $\iota\in{\rm R^\times}$}{pag:q-function and b-function}
\symitem{$\mathtt{b}_{\pm}(\iota)$}{The solutions of equation $x+x^{-1}=\mathtt{q}(\iota)$}{pag:q-function and b-function}
\symitem{$\mHfcn$}{The cyclotomic Hecke-Clifford algebra}{pag:CHCA}
\symitem{$\underline{Q}$}{The cyclotomic parameters $(Q_1,Q_2,\ldots,Q_m)\in({\rm R^\times})^m$}{pag:Q-parameters}
\symitem{$r$}{The level of $\mHfcn$}{pag:nondege level}
\symitem{$\res(\alpha)$}{The residue $Q_lq^{2(j-i)}$ of box $\alpha=(i,j,l)$}{pag:nondeg residue}
\symitem{$\res_{\mt}(k)$}{The residue of box $\mt^{-1}(k)$ for $\mt\in\Std(\undla)$}{pag:nondeg residue}
\symitem{$\res(\mt)$}{The residue sequence $(\res_{\mt}(1),\ldots,\res_{\mt}(n))$ of $\mt\in\Std(\undla)$}{pag:nondeg residue}
\symitem{$\mathtt{q}(\res(\mt))$}{The $\mathtt{q}$-sequence $(\mathtt{q}(\res_{\mt}(1)),\ldots,\mathtt{q}(\res_{\mt}(n)))$ of $\mt\in\Std(\undla)$}{pag:nondeg residue}
\symitem{$P_n^f(q^2,\undQ)$}{The separation polynomial for the defining polynomial $f$}{pag:nondege Pioncare poly}
\symitem{$\mathbb{D}(\undla)$}{The simple module of $\mHfcn$ indexed by $\undla$}{pag:nondege simple module}
\symitem{$\mathtt{b}_{\mt,i}$}{$\mathtt{b}_{-}(\res_{\mt}(i))$}{pag:nondege coeffi cti}
\symitem{$\mathtt{c}_{\mt}(i)$}{Some structure coefficient appeared in module $\mathbb{D}(\undla)$}{pag:nondege coeffi cti}
\symitem{$F_{\rm T}$}{The primitive idempotent indexed by ${\rm T}\in{\rm Tri}_{\bar{0}}(\undla)$}{pag:primitive idempotents and blocks}
\symitem{$F_{\undla}$}{The primitive central idempotent indexed by $\undla$}{pag:primitive idempotents and blocks}
\symitem{$B_{\undla}$}{The simple block of $\mHfcn$ indexed by $\undla$}{pag:primitive idempotents and blocks}
\symitem{$\Phi_{\ms,\mt}$}{A key element in $\mHfcn$ indexed by $\ms,\mt\in\Std(\undla)$}{pag:Phist and cst}
\symitem{$\mathtt{c}_{\ms,\mt}$}{A key coefficient indexed by $\ms,\mt\in\Std(\undla)$}{pag:Phist and cst}
\symitem{$f_{{\rm S},{\rm T}}^{\mathfrak{w}},$ $f_{{\rm S},{\rm T}_a}^{\mathfrak{w}}$}{The seminormal basis factoring through a fixed standard tableau $\mathfrak{w}$}{pag:nondege seminormal basis}
\symitem{$f_{{\rm S},{\rm T}},$ $f_{{\rm S},{\rm T}_a}$}{The (reduced) seminormal basis}{pag:nondege seminormal basis}
\symitem{$\mathtt{c}_{\rm T}^{\mathfrak{w}}$ ($=\mathtt{c}_{\mt}^{\mathfrak{w}}$)}{$(\mathtt{c}_{\mt,\mathfrak{w}})^2$ for ${\rm T}=(\mt, \alpha_{\mt}, \beta_{\mt})\in {\rm Tri}(\undla)$ and a fixed standard tableau $\mathfrak{w}$}{pag:nondege cT}
\symitem{$\kappa_{\undla}$}{A key element in $\mathbb{K}^*$ which only depends on $\undla$}{pag:qt}
\symitem{$\mathtt{q}_{\mt,i}$}{$\mathtt{q}(\res_{\mt}(i))$}{pag:nondege q-values}
\symitem{$\tau_{r,n}$}{The Frobenius form}{pag:Frobenius form. non-dege}
\symitemtwo{$t_{r,n}$}{The (super)symmetrizing form}{pag:supersym. form. non-dege}{pag:sym. form. non-dege}
\symitemtwo{$s_{\undla}$}{The Schur element of simple module $\mathbb{D}(\undla)$}{pag:Schur elements.non-dege}{pag:Schur elements'.non-dege}
\\
\item[\textbf{Degenerate case}] \hfill
\symitem{$\mhcn$}{The affine Sergeev algebra}{pag:ASA}
\symitem{$\mathcal{P}_n$}{A certain subalgebra of $\mhcn$}{pag:subalg Pn}
\symitem{$\tilde{\phi}_i$}{The intertwining element}{pag:dege BK intertwining elements}
\symitem{$\phi_i$}{Nazarov's intertwining element}{pag:dege N intertwining elements}
\symitem{$\phi_i(x,y)$}{An element in $\mhcn$ related to $\phi_i$}{pag:dege N phi function}
\symitem{$\mathtt{q}(\iota)$}{$\iota(\iota+1)$ for $\iota\in{\rm R}$}{pag:dege q-function and u-function}
\symitem{$\mathtt{u}_{\pm}(\iota)$}{$\pm \sqrt{\iota(\iota+1)}$ for $\iota\in\mathbb{K}$}{pag:dege q-function and u-function}
\symitem{$\mhgcn$}{The cyclotomic Sergeev algebra}{pag:CSA}
\symitem{$\underline{Q}$}{The cyclotomic parameters $(Q_1,Q_2,\ldots,Q_m)\in{\rm R}^m$}{pag:dege Q-parameters}
\symitem{$r$}{The level of $\mhgcn$}{pag:dege level}
\symitem{$\res(\alpha)$}{The residue $Q_l+j-i$ of box $\alpha=(i,j,l)$}{pag:dege residue}
\symitem{$\res_{\mt}(k)$}{The residue of box $\mt^{-1}(k)$ for $\mt\in\Std(\undla)$}{pag:dege residue}
\symitem{$\res(\mt)$}{The residue sequence $(\res_{\mt}(1),\ldots,\res_{\mt}(n))$ of $\mt\in\Std(\undla)$}{pag:dege residue}
\symitem{$\mathtt{q}(\res(\mt))$}{The $\mathtt{q}$-sequence $(\mathtt{q}(\res_{\mt}(1)),\ldots,\mathtt{q}(\res_{\mt}(n)))$ of $\mt\in\Std(\undla)$}{pag:dege residue}
\symitem{$P_n^g(\undQ)$}{The separation polynomial for the defining polynomial $g$}{pag:dege Pioncare poly}
\symitem{$D(\undla)$}{The simple module of $\mhgcn$ indexed by $\undla$}{pag:dege simple module}
\symitem{$\mathtt{u}_{\mt,i}$}{$\mathtt{u}_{+}(\res_{\mt}(i))$}{pag:dege coeffi cti}
\symitem{$\mathfrak{c}_{\mt}(i)$}{Some structure coefficient appeared in module $D(\undla)$}{pag:dege coeffi cti}
\symitem{$\mathcal{F}_{\rm T}$}{The primitive idempotent indexed by ${\rm T}\in{\rm Tri}_{\bar{0}}(\undla)$}{pag:dege primitive idempotents}
\symitem{$\mathcal{F}_{\undla}$}{The primitive central idempotent indexed by $\undla$}{pag:dege primitive idempotents}
\symitem{$\mathcal{B}_{\undla}$}{The simple block of $\mhgcn$ indexed by $\undla$}{pag:dege simple blocks}
\symitem{$\phi_{\ms,\mt}$}{A key element in $\mhgcn$ indexed by $\ms,\mt\in\Std(\undla)$}{pag:phist and cst}
\symitem{$\mathfrak{c}_{\ms,\mt}$}{A key coefficient indexed by $\ms,\mt\in\Std(\undla)$}{pag:phist and cst}
\symitem{$\mathfrak{f}_{{\rm S},{\rm T}}^{\mathfrak{w}},$ $\mathfrak{f}_{{\rm S},{\rm T}_a}^{\mathfrak{w}}$}{The seminormal basis factoring through a fixed standard tableau $\mathfrak{w}$}{pag:dege seminormal basis}
\symitem{$\mathfrak{f}_{{\rm S},{\rm T}},$ $\mathfrak{f}_{{\rm S},{\rm T}_a}$}{The (reduced) seminormal basis}{pag:dege seminormal basis}
\symitem{$\mathfrak{c}_{\rm T}^{\mathfrak{w}}$ ($=\mathfrak{c}_{\mt}^{\mathfrak{w}}$)}{$(\mathfrak{c}_{\mt,\mathfrak{w}})^2$ for ${\rm T}=(\mt, \alpha_{\mt}, \beta_{\mt})\in {\rm Tri}(\undla)$ and a fixed standard tableau $\mathfrak{w}$}{pag:dege cT}
\symitem{$\kappa_{\undla}^{\mathrm{dege}}$}{A key element in $\mathbb{K}^*$ which only depends on $\undla$}{pag:qt'}
\symitem{$\mathtt{t}_{r,n}$}{The (super)symmetrizing form}{pag:(super)sym. form. dege}
\symitem{$\mathtt{s}_{\undla}$}{The Schur element of simple module $D(\undla)$}{pag:Schur elements.dege}

\end{symbols}


\begin{thebibliography}{ABC}
\bibitem[Ar]{Ariki:can}
{\sc S.~Ariki},  {\em On the decomposition numbers of the {Hecke} algebra of {$G(m,1,n)$}}, J. Math. Kyoto Univ., {\bf 36} (1996), 789--808.

\bibitem[AK]{AK}
{\sc S.~Ariki and K.~Koike}, {\em  A Hecke algebra of $(\mathbb{Z}/r\mathbb{Z})\wr\mathfrak{S}_n$ and construction of its representations}, Adv. Math., {\bf 106} (1994), 216--243.

\bibitem[BK1]{BK:GradedKL}
{\sc J.~Brundan and A.~Kleshchev},  {\em Blocks of cyclotomic {H}ecke algebras and {K}hovanov-{L}auda algebras}, Invent. Math., {\bf 178} (2009), 451--484.

\bibitem[BK2]{BK}
\leavevmode\vrule height 2pt depth -1.6pt width 23pt, {\em Hecke-Clifford superalgebras, crystals of type $A^{(2)}_{2l}$, and modular branching rules for $\widehat{S}_n$}, Repr. Theory, {\bf 5} (2001), 317--403.

\bibitem[BM]{BM:cyc}
{\sc M.~Brou{\'e} and G.~Malle}, {\em Zyklotomische {H}eckealgebren}, in: Repr{\'e}sentations Unipotentes G{\'e}n{\'e}riques et Blocs des Groupes R{\'e}ductifs Finis, Ast\'erisque., {\bf 212} (1993), 119--189.

\bibitem[C]{C}
{\sc I.V.~Cherednik}, {\em A new interpretation of Gelfand-Tzetlin bases}, Duke Math. J., {\bf 54}(2) (1987), 563--577.

\bibitem[CJ]{CJ}
{\sc M.~Chlouveraki and N.~Jacon}, {\em Schur elements for the Ariki-Koike algebra and applications}, J. Algebr. Comb., {\bf 35} (2012), 291--311.

\bibitem[CW]{CW}
{\sc S.-J.~Cheng and W.~Wang}, {\em Dualities and representations of Lie superalgebras}, Graduate Studies in Mathematics, vol. 144. American Mathematical Society, Providence, RI, 2012.

\bibitem[EKL]{EKL}
{\sc A.~P. Ellis, M.~Khovanov and A.~D. Lauda}, {\em The odd nilHecke algebra and its diagrammatics}, Int. Math. Res. Not. IMRN, 2014, no.~4, 991--1062.

\bibitem[EM]{EM}
{\sc A.~Evseev and A.~Mathas}, {\em Content systems and deformations of cyclotomic KLR algebras of type $A$ and $C$}, Ann. Represent., {\bf 1}(2) (2024), 193--297.

\bibitem[GF]{GF}
{\sc M.~Geck and G.~Pfeiffer}, {\em Characters of finite coxeter groups and Iwahori--Hecke algebras}, London Mathematical Society Monographs, New Series 21, Oxford University Press, New York, 2000.

\bibitem[GIM]{GIM}
{\sc M.~Geck, L.~Iancu and G.~Malle}, {\em Weights of Markov traces and generic degrees}, Indag. Math., {\bf 11} (2000), 379--397.

\bibitem[HL]{HL}
{\sc J.~Hu and H.~Li}, {\em SVV's trace forms on cyclotomic KLR algebras with graded content systems}, preprint, submitted, 2024.

\bibitem[HLL]{HLL}
{\sc J.~Hu, H.~Li and S.~Li}, {\em New formulae for the Schur elements of the cyclotomic Hecke algebra of type $G(\ell,1,n)$}, J. Pure Appl. Algebra, {\bf 229}(11) (2025), Paper No. 108090, 29 pp.

\bibitem[HS]{HuShi:super}
{\sc J.~Hu and L.~Shi}, {\em Graded dimensions and monomial bases for the cyclotomic quiver Hecke superalgebras}, J. Algebra, {\bf 635} (2023), 642--670.

\bibitem[HW]{HW}
{\sc J.~Hu and S.~Wang}, {\em On the seminormal bases and dual seminormal bases of the cyclotomic Hecke algebras of type $G(r, 1, n)$}, J. Algebra, {\bf 600} (2022), 195--221.

\bibitem[JN]{JN}
{\sc A.~Jones, M.~Nazarov}, {\em Affine Sergeev algebra and $q$-analogues of the Young symmetrizers for projective representations of the symmetric group}, Proc. London Math. Soc. {\bf 78} (1999), 481--512.

\bibitem[KKT]{KKT}
{\sc S.~J. Kang, M.~Kashiwara and S.~Tsuchioka}, {\em Quiver Hecke Superalgebras}, J. Reine Angew. Math., {\bf 711} (2016), 1--54.

\bibitem[KKO1]{KKO1}
{\sc S.~J. Kang, M.~Kashiwara and S.~J. Oh}, {\em Supercategorification of quantum Kac-Moody algebras}, Adv. Math., {\bf 242} (2013), 116--162.

\bibitem[KKO2]{KKO2}
\leavevmode\vrule height 2pt depth -1.6pt width 23pt, {\em Supercategorification of quantum Kac-Moody algebras II}, Adv. Math., {\bf 265} (2014), 169--240.

\bibitem[KMS]{KMS}
{\sc I.~Kashuba, A.~Molev and V.~Serganova}, {\em On the Jucys-Murphy method and fusion procedure for the Sergeev superalgebra}, J. Lond. Math. Soc. (2), {\bf 112}(3) (2025), Paper No. e70302.

\bibitem[Kl]{Kle:RoCKSchur}
{\sc A.~Kleshchev}, {\em RoCK blocks of double covers of symmetric groups and generalized Schur algebras}, arXiv:2411.03653, 2024.

\bibitem[KL]{KL:RoCK}
{\sc A.~Kleshchev and M.~Livesey}, {\em RoCK blocks for double covers of symmetric groups and quiver Hecke superalgebras}, Mem. Amer. Math. Soc., {\bf 309} (2025), no.~1564, v+182 pp.

\bibitem[K]{K2}
{\sc A.~Kleshchev}, {\em Linear and Projective Representations of Symmetric Groups}, Cambridge University Press, 2005.

\bibitem[LS1]{LS1}
{\sc S.~Li, L.~Shi}, {\em  On the (super)cocenter of cyclotomic Sergeev algebras},  J. Algebra, {\bf 682} (2025), 824--858.

\bibitem[LS2]{LS2}
\leavevmode\vrule height 2pt depth -1.6pt width 23pt, {\em Seminormal bases of cyclotomic Hecke-Clifford algebras}, Lett. Math. Phys., {\bf 115}(5) (2025), Paper No. 110.

\bibitem[LS3]{LS3}
\leavevmode\vrule height 2pt depth -1.6pt width 23pt, {\em On the generalized graded cellular bases for cyclotomic quiver Hecke-Clifford superalgebras}, preprint, arXiv:2604.04058, 2026.

\bibitem[MM]{MM}
{\sc G.~Malle and A.~Mathas}, {\em Symmetric cyclotomic Hecke algebras}, {J. Algebra,} {\bf 205}(1) (1998), 275--293.

\bibitem[Ma]{Ma2}
{\sc A.~Mathas}, {\em Matrix units and generic degrees for the Ariki-Koike algebras}, {J. Algebra,} {\bf 281} (2004), 695--730.

\bibitem[N]{Na2}
{\sc M.~Nazarov}, {\em Young's symmetrizers for projective representations of the symmetric group}, Adv. Math. {\bf 127} (1997), no. 2, 190--257.

\bibitem[O]{O}
{\sc G.~I.~Olshanski}, {\em Quantized universal enveloping superalgebra of type $\mathtt{Q}$ and a super-extension of the Hecke algebra}, Lett. Math. Phys., {\bf 24} (1992), 93--102.

\bibitem[SVV]{SVV}
{\sc P.~Shan, M.~Varagnolo and E.~Vasserot}, {\em On the center of quiver-Hecke algebras}, Duke Math. J., {\bf 166}(6) (2017), 1005--1101.

\bibitem[Sh]{Sh}
{\sc L.~Shi}, {\em On the semisimplicity and Schur elements of (super)symmetric superalgebras}, preprint, arXiv:2605.04745, 2026.

\bibitem[SW]{SW}
{\sc L.~Shi, J.~Wan}, {\em On representation theory of cyclotomic Hecke-Clifford algebras}, J. Algebra, {\bf 701}C (2026), 179--218.

\bibitem[T]{T}
{\sc S.~Tsuchioka}, {\em Hecke-Clifford superalgebras and crystals of type $D_l^{(2)}$}, Publ. Res. Inst. Math. Sci. {\bf 46} (2010), 423--471.

\bibitem[WW]{WW2}
{\sc J.~Wan and W.~Wang}, {\em Frobenius character formula and spin generic degrees for Hecke--Clifford algebra}, Proc. London Math. Soc., {\bf 106}(3) (2013), 287--317.

\end{thebibliography}
	\end{document}